\documentclass[11pt,twoside]{amsart}
\usepackage{latexsym}
\usepackage{amssymb}
\usepackage{vmargin}
\usepackage{stmaryrd}
\usepackage{mathrsfs}
\usepackage{amscd}
\usepackage[all]{xy}
\usepackage{xr}

\newtheorem{theorem}{Theorem}[section]
\newtheorem{proposition}[theorem]{Proposition}
\newtheorem{corollary}[theorem]{Corollary}

\newtheorem{lemma}[theorem]{Lemma}
\newtheorem*{theorem*}{Theorem}
\newtheorem*{proposition*}{Proposition}
\newtheorem*{corollary*}{Corollary}
\newtheorem*{lemma*}{Lemma}
\newtheorem*{conjecture*}{Conjecture}

\theoremstyle{definition}
\newtheorem{definition}[theorem]{Definition}

\newtheorem*{definition*}{Definition}

\theoremstyle{remark}
\newtheorem{example}[theorem]{Example}
\newtheorem{examples}[theorem]{Examples}
\newtheorem{remark}[theorem]{Remark}
\newtheorem{remarks}[theorem]{Remarks}

\newtheorem*{example*}{Example}
\newtheorem*{examples*}{Examples}
\newtheorem*{remark*}{Remark}
\newtheorem*{remarks*}{Remarks}
\newtheorem*{exercise*}{Exercise}
\newtheorem*{property*}{Property}
\newtheorem*{properties*}{Properties}

\newcommand\da{\!\downarrow\!}
\newcommand\ra{\rightarrow}
\newcommand\la{\leftarrow}
\newcommand\lra{\longrightarrow}

\newcommand\id{\mathrm{id}}

\newcommand\ten{\otimes}
\newcommand\hten{\hat{\otimes}}
\newcommand\vareps{\varepsilon}
\newcommand\eps{\epsilon}

\newcommand\CC{\mathrm{C}}

\newcommand\OO{\mathrm{O}}

\newcommand\RR{\mathrm{R}}

\newcommand\Ru{\mathrm{R_u}}
\newcommand\Th{\mathrm{Th}\,}
\newcommand\CCC{\mathrm{CC}}

\renewcommand\H{\mathrm{H}}
\newcommand\z{\mathrm{Z}}

\newcommand\N{\mathbb{N}}
\newcommand\Z{\mathbb{Z}}
\newcommand\Q{\mathbb{Q}}

\newcommand\R{\mathbb{R}}
\newcommand\Cx{\mathbb{C}}

\newcommand\vv{\mathbb{V}}
\newcommand\ww{\mathbb{W}}
\newcommand\Bu{\mathbb{B}}

\newcommand\bA{\mathbb{A}}

\newcommand\bF{\mathbb{F}}
\newcommand\bG{\mathbb{G}}
\newcommand\bH{\mathbb{H}}

\newcommand\bL{\mathbb{L}}

\newcommand\bO{\mathbb{O}}
\newcommand\bP{\mathbb{P}}

\newcommand\bS{\mathbb{S}}
\newcommand\bT{\mathbb{T}}

\newcommand\C{\mathcal{C}}

\newcommand\cA{\mathcal{A}}
\newcommand\cB{\mathcal{B}}

\newcommand\cD{\mathcal{D}}
\newcommand\cE{\mathcal{E}}
\newcommand\cF{\mathcal{F}}

\newcommand\cH{\mathcal{H}}

\newcommand\cM{\mathcal{M}}
\newcommand\cN{\mathcal{N}}

\newcommand\cO{\mathcal{O}}
\newcommand\cP{\mathcal{P}}

\newcommand\cS{\mathcal{S}}
\newcommand\cT{\mathcal{T}}
\newcommand\cU{\mathcal{U}}

\newcommand\cW{\mathcal{W}}

\newcommand\cZ{\mathcal{Z}}

\newcommand\A{\mathscr{A}}

\newcommand\D{\mathcal{D}}

\newcommand\F{\mathscr{F}}

\renewcommand\O{\mathscr{O}}

\newcommand\sA{\mathscr{A}}
\newcommand\sB{\mathscr{B}}
\newcommand\sC{\mathscr{C}}

\newcommand\sE{\mathscr{E}}
\newcommand\sF{\mathscr{F}}
\newcommand\sG{\mathscr{G}}
\newcommand\sH{\mathscr{H}}

\newcommand\sK{\mathscr{K}}
\newcommand\sL{\mathscr{L}}
\newcommand\sM{\mathscr{M}}

\newcommand\sT{\mathscr{T}}

\newcommand\sV{\mathscr{V}}

\newcommand\Def{\mathfrak{Def}}
\newcommand\Del{\mathfrak{Del}}

\newcommand\OBJ{\mathfrak{OBJ}}

\newcommand\fB{\mathfrak{B}}

\newcommand\fM{\mathfrak{M}}

\newcommand\fR{\mathfrak{R}}

\newcommand\fX{\mathfrak{X}}

\renewcommand\L{\Lambda}

\newcommand\m{\mathfrak{m}}

\newcommand\g{\mathfrak{g}}

\newcommand\fh{\mathfrak{h}}

\newcommand\fk{\mathfrak{k}}

\newcommand\fu{\mathfrak{u}}

\newcommand\cHom{\mathcal{H}\!\mathit{om}}
\newcommand\cDer{\mathcal{D}\!\mathit{er}}

\newcommand\der{\mathscr{D}\!\mathit{er}}

\newcommand\bEnd{\mathbb{E}\mathrm{nd}}

\newcommand\Ho{\mathrm{Ho}}

\newcommand\Alg{\mathrm{Alg}}

\newcommand\Mod{\mathrm{Mod}}

\newcommand\Hom{\mathrm{Hom}}
\newcommand\Map{\mathrm{Map}}

\newcommand\HHom{\underline{\mathrm{Hom}}}

\newcommand\Ext{\mathrm{Ext}}
\newcommand\EExt{\mathbb{E}\mathrm{xt}}
\newcommand\End{\mathrm{End}}

\newcommand\Der{\mathrm{Der}}

\newcommand\Aut{\mathrm{Aut}}

\newcommand\VHS{\mathrm{VHS}}

\newcommand\Iso{\mathrm{Iso}}

\newcommand\cone{\mathrm{cone}}

\newcommand\Gal{\mathrm{Gal}}

\newcommand\coker{\mathrm{coker\,}}

\newcommand\im{\mathrm{Im\,}}

\newcommand\Ob{\mathrm{Ob}\,}

\newcommand\CoLie{\mathrm{CoLie}}

\newcommand\Top{\mathrm{Top}}
\newcommand\Gp{\mathrm{Gp}}
\newcommand\agp{\mathrm{AGp}}

\newcommand\mal{\mathrm{Mal}}

\newcommand\Spec{\mathrm{Spec}\,}
\newcommand\Dec{\mathrm{Dec}\,}

\newcommand\Spf{\mathrm{Spf}\,}

\newcommand\Cat{\mathrm{Cat}}

\newcommand\Aff{\mathrm{Aff}}

\newcommand\Sp{\mathrm{Sp}}

\newcommand\Sing{\mathrm{Sing}}

\newcommand\ad{\mathrm{ad}}
\newcommand\norm{\mathrm{norm}}

\newcommand\Lim{\varprojlim}
\newcommand\LLim{\varinjlim}
\DeclareMathOperator*{\holim}{holim}
\newcommand\ho{\mathrm{ho}\!}
\newcommand\into{\hookrightarrow}
\newcommand\onto{\twoheadrightarrow}
\newcommand\abuts{\implies}
\newcommand\xra{\xrightarrow}
\newcommand\xla{\xleftarrow}
\newcommand\pr{\mathrm{pr}}

\newcommand\alg{\mathrm{alg}}

\newcommand\dmd{\diamond}
\newcommand\bt{\bullet}
\newcommand\by{\times}

\newcommand\mc{\mathrm{MC}}

\newcommand\Gg{\mathrm{Gg}}

\newcommand\Res{\mathrm{Res}}

\newcommand\Symm{\mathrm{Symm}}

\newcommand\SL{\mathrm{SL}}

\newcommand\GL{\mathrm{GL}}

\newcommand\SU{\mathrm{SU}}

\newcommand\gl{\mathfrak{gl}}

\newcommand\Mat{\mathrm{Mat}}
\newcommand\et{\acute{\mathrm{e}}\mathrm{t}}

\newcommand\an{\mathrm{an}}
\newcommand\hol{\mathrm{hol}}

\newcommand\Tot{\mathrm{Tot}\,}
\newcommand\diag{\mathrm{diag}\,}

\newcommand\ind{\mathrm{ind}}
\newcommand\pro{\mathrm{pro}}

\newcommand\pd{\partial}
\newcommand\dc{d^{\mathrm{c}}}

\newcommand\tD{\tilde{D}}
\newcommand\tDc{\tilde{D^c}}

\newcommand\half{\frac{1}{2}}

\newcommand\cris{\mathrm{cris}}

\newcommand\st{\mathrm{st}}

\newcommand\MHS{\mathrm{MHS}}
\newcommand\MTS{\mathrm{MTS}}
\newcommand\SHS{\mathrm{SHS}}
\newcommand\STS{\mathrm{STS}}

\newcommand\gr{\mathrm{gr}}
\newcommand\ugr{\underline{\mathrm{gr}}}
\newcommand\ab{\mathrm{ab}}
\newcommand\cts{\mathrm{cts}}
\newcommand\gp{\mathrm{Gp}}

\newcommand\Zar{\mathrm{Zar}}

\newcommand\Fil{\mathrm{Fil}}

\renewcommand\alg{\mathrm{alg}}
\newcommand\red{\mathrm{red}}
\newcommand\Fr{\mathrm{Fr}}
\newcommand\Lie{\mathrm{Lie}}

\newcommand\row{\mathrm{row}}

\newcommand\dR{\mathrm{dR}}

\newcommand\op{\mathrm{opp}}

\newcommand\co{\colon\thinspace}

\newcommand\oR{\mathbf{R}}

\newcommand\oL{\mathbf{L}}

\newcommand\oSpec{\mathbf{Spec}\,}

\newcommand\uleft\underleftarrow
\newcommand\uline\underline
\newcommand\uright\underrightarrow

\newcommand\HOM{\mathrm{HOM}}

\begin{document}
\title[Non-abelian Hodge structures for  quasi-projective varieties]{Real non-abelian mixed Hodge structures for quasi-projective varieties: formality and splitting}
\author{J.P.Pridham}
\thanks{This work was supported by Trinity College, Cambridge; and by the Engineering and Physical Sciences Research Council [grant numbers  EP/F043570/1 and EP/I004130/2].}

\begin{abstract}
We define and construct mixed Hodge structures on real schematic homotopy types of complex quasi-projective varieties, giving mixed Hodge structures on their homotopy groups and pro-algebraic fundamental groups. We also show that these split on tensoring with the ring $\R[x]$ equipped with the Hodge filtration given by powers of $(x-i)$, giving new results even for simply connected varieties. The mixed Hodge structures can thus be recovered from the Gysin spectral sequence of cohomology groups of local systems, together with the monodromy action at the Archimedean place. As the basepoint varies, these structures all become real variations of mixed Hodge structure.
\end{abstract}

\maketitle

\section*{Introduction}

The main aims of this paper are to construct  mixed Hodge structures on the real relative Malcev homotopy types of  complex varieties, and to investigate how far these can be recovered from the structures on cohomology groups of local systems, and in particular from the Gysin spectral sequence. 

In \cite{Morgan}, Morgan established the existence of natural mixed Hodge structures on the minimal model of the rational homotopy type of a   smooth variety $X$, and used this to define natural mixed Hodge structures on the  rational homotopy groups $\pi_*(X \ten \Q)$ of $X$. This construction was extended to singular varieties by Hain in \cite{Hainhodge}, using an alternative approach based on Chen's reduced bar construction.

 When $X$ is  also projective, \cite{DGMS} showed that its rational homotopy type is formal; in particular, this means that the rational homotopy groups can be recovered from the cohomology ring $\H^*(X,\Q)$. However, in \cite{mhsnonsplit}, examples were given to show that the mixed Hodge structure on homotopy groups could not be recovered from that on integral cohomology. In this article, we will describe how formality interacts with the mixed Hodge structure, showing the extent to which the mixed Hodge structure on $\pi_*(X \ten \R,x_0)$ can be recovered from the pure Hodge structure on $\H^*(X,\R)$. 

This problem  was suggested to the author by Carlos Simpson, who asked what happens when we vary the formality quasi-isomorphism. \cite{DGMS} proved formality by using the $d\dc$ Lemma (giving real quasi-isomorphisms), while most subsequent work has used the $\pd\bar{\pd}$  Lemma (giving Hodge-filtered quasi-isomorphisms). 
The answer (Theorem \ref{morgansplitsa}) is that, if we define the ring $\cS:= \R[x]$ to be pure of weight $0$, with the Hodge filtration on $\cS\ten_{\R}\Cx$ given by powers of $(x-i)$, then there is an $\cS$-linear isomorphism
$$
\pi_*(X \ten \R,x_0)\ten_{\R}\cS \cong \pi_*(\H^*(X,\R))\ten_{\R} \cS,
$$
preserving the Hodge and weight filtrations, where the homotopy groups $\pi_*(\H^*(X,\R))$ are given the  Hodge structure coming from  the Hodge structure on the cohomology ring  $\H^*(X,\R)$, regarded as a  real homotopy type. 

This is proved by replacing $\dc$ with $\dc+xd$ in the proof of \cite{DGMS}, so $x \in \cS$ is the parameter for varying formality quasi-isomorphisms. In several respects, $\cS\ten_{\R}\Cx$ behaves like Fontaine's ring $B_{\st}$ of semi-stable periods, and the MHS can be recovered from a pro-nilpotent  operator on the real homotopy type $\H^*(X,\R)$, which we regard as monodromy at the Archimedean place. The isomorphism above says that the MHS on $\pi_*(X \ten \R,x_0)$ has an  $\cS$-splitting, and by Proposition \ref{cSubiq}, this is true for all mixed Hodge structures. However, the special feature here is that the splitting is canonical, so preserves the additional structure (such as Whitehead brackets).

For non-nilpotent topological spaces, the rational homotopy type is too crude an invariant to recover much information, so schematic homotopy types were introduced in \cite{chaff}, based on ideas from \cite{pursuingstacks}. \cite{htpy} showed how to recover  the groups  $\pi_n(X)\ten_{\Z}\R$ from schematic homotopy types for very general topological spaces, and also introduced the intermediate notion of relative Malcev homotopy type, simultaneously generalising both rational and schematic homotopy types, and giving the higher homotopical generalisation of Hain's relative Malcev fundamental groups from \cite{hainrelative}. In  Corollary \ref{unipfibs} we will see how relative Malcev homotopy types govern the variation of rational homotopy types in a fibration.

Since their inception, one of the main goals of schematic homotopy types has been to define and construct mixed Hodge structures. This programme was initiated in \cite{KPS}, and continued in \cite{KTP}. Although the structures in \cite{KTP} have important consequences, such as proving that the image of the Hurewicz map is a sub-Hodge structure, they are too weak to give rise to mixed Hodge structures on the homotopy groups, and disagree with the weight filtration on rational homotopy groups defined in \cite{Morgan} (see Remark \ref{ktpwgt}). 

In this paper, we take an alternative approach, giving a new notion of mixed Hodge structures on  schematic (and relative Malcev) homotopy types which is compatible with \cite{Hainhodge} and \cite{Morgan}. 
These often  yield mixed Hodge structures on the full  homotopy groups $\pi_n(X,x_0)$  (rather than just on rational homotopy groups). In 
Corollaries \ref{kmhspin} and \ref{kmhspinanal} we show not only that the homotopy types of compact K\"ahler manifolds naturally carry such mixed Hodge structures, but also that they also split and become formal on tensoring with $\cS$. 
The structure in \cite{KTP} can then be understood as an invariant of the $\cS$-splitting, rather than of the MHS itself (Remark \ref{ktpcf}). Corollary \ref{kvmhspin} shows that these MHS become variations of mixed Hodge structure as the basepoint varies.

When studying topological invariants of smooth varieties, it is often too much to expect a mixed Hodge structure to appear. For this reason, Simpson introduced mixed twistor structures in \cite{MTS} as a weaker notion. Essentially, these arise because not every semisimple local system underlies a variation of Hodge structures. We use a slight refinement of Simpson's notion (Definition \ref{realMTSdef}), utilising real rather than complex structures.

We then 
 construct mixed Hodge and mixed twistor structures for relative Malcev homotopy types of  quasi-projective varieties (Theorems \ref{qmts} and \ref{qmhs}),  but only when the monodromy around the divisor is trivial. Theorem  \ref{unitmhs} addresses a more general case, allowing  unitary monodromy around the divisor.

Whereas the $\cS$-splittings for projective varieties  
are realised concretely using the principle of two types, the last part of the paper 
establishes abstract existence  results for $\cS$-splittings of general mixed Hodge and mixed twistor structures (Corollary \ref{allsplit}). These latter results are then used  to construct  mixed Hodge and mixed twistor structures on relative Malcev homotopy groups of quasi-projective varieties (Corollaries \ref{mhspin2} and \ref{unitsingmtspin}). 

\subsection*{Structure of the paper}

Section \ref{real1} introduces the main idea of the paper, as specialised to real homotopy types of compact K\"ahler manifolds $X$. The key observation (Proposition \ref{gl2prop}) is that  the  principle of two types (or the $d\dc$-lemma) holds for any pair $ud+v\dc, xd+y\dc$ of operators, provided  $\left(\begin{smallmatrix} u &  v \\ x & y\end{smallmatrix}\right) \in \GL_2$. For the real Sullivan homotopy type $A^{\bt}(X, \R)$ and the ring $\cS= \R[x]$ above, Theorem \ref{formalityS}  uses the principle of two types to show that $A^{\bt}(X) \ten_{\R} \cS$ is quasi-isomorphic to  $\H^*(X,\R)\ten_{\R} \cS$, compatibly with the CDGA structure and the Hodge and weight filtrations. Transferring these quasi-isomorphisms  to Quillen's real homotopy type via the reduced bar construction, Theorem  \ref{morgansplitsa} then shows that the mixed Hodge structure on the real homotopy groups splits canonically on tensoring with $\cS$. 
Hodge structures on rational homotopy types and groups. 

In Section \ref{nonabfil}, we introduce our non-abelian notions of algebraic mixed Hodge and twistor structures.
If we define $C^* = (\prod_{\Cx/\R} \bA^1)-\{0\} \cong \bA^2_{\R}-\{0\}$ and the Deligne torus $S= \prod_{\Cx/\R} \bG_m$ by Weil restriction of scalars, then our first major observation (Corollary \ref{flathfil}) is that real vector spaces $V$ equipped with filtrations $F$ on $V\ten \Cx$ correspond to flat quasi-coherent modules on the stack $[C^*/S]$, via a Rees module construction, with $V$ being the pullback along $1 \in C^*$. This motivates us to define an algebraic Hodge filtration on a real object $Z$ as an extension of $Z$ over the base stack $[ C^*/ S]$.  This is similar to the approach taken by Kapranov to define  mixed Hodge structures in \cite{kapranovMHS}; see Remark \ref{kapranov} for details. There is an $S$-equivariant morphism $\row_1\co\SL_2\to C^*$ given by projection of the first row, which 
corresponds to the Hodge filtration on the ring $\cS$ above, and has important universal properties (Lemma \ref{sluniv}). 

Similarly, filtered vector spaces correspond to flat quasi-coherent modules on the stack $[\bA^1/\bG_m]$, so we define an algebraic mixed Hodge structure on $Z$ to consist of an extension $Z_{\MHS}$ over $[\bA^1/\bG_m]\by [C^*/S]$, with additional data corresponding to an opposedness condition (Definition \ref{mhsdef}). This gives rise to non-abelian mixed Hodge structures in the sense of \cite{KPS}, as explained in Remark \ref{cfkps}. In some cases, a mixed Hodge structure is too much to expect, and we then give an extension over  $[\bA^1/\bG_m]\by [C^*/\bG_m]$: an algebraic  mixed twistor structure. For vector bundles, algebraic mixed Hodge and twistor structures coincide with the classical definitions (Propositions \ref{flatmhs} and \ref{flatmts}). 
\S \ref{real} reprises the results for real homotopy types in this setting. All of the structures split on pulling back along $\row_1$, and these pullbacks can be recovered from cohomology of local systems. The pullback along  $\row_1$ corresponds to tensoring with the algebra $\cS$ described above. 

In Section \ref{cohosn}, the cohomology groups associated to the various non-abelian structures are considered. \S\S \ref{Bei1} and \ref{Bei2} show how algebraic Hodge filtrations and algebraic mixed Hodge structures determine Beilinson's weak Hodge and absolute Hodge cohomology, respectively. For real homotopy types, Proposition \ref{cfdeligne}, Corollary \ref{consani} and Proposition \ref{cfdeninger} show how the algebraic mixed Hodge structure recovers real Deligne cohomology, Consani's 
Archimedean cohomology and Deninger's $\Gamma$-factor of $X$ at the Archimedean place. Proposition \ref{steenkey} shows how the pullback to $\SL_2$ can be regarded as an analogue of a limit mixed Hodge structure.

Section \ref{relmalsect} is essentially a review of the relative Malcev homotopy types introduced in \cite{htpy}, generalising both schematic and real homotopy types.  However, much of the material is adapted to pointed homotopy types, and a greater emphasis is placed on Narkawicz' variant of the reduced bar construction from \cite{narkawicz}, together with some new material  in \S \ref{relhtpy} on families of homotopy types. Major new results are Theorem \ref{fibrations} and Corollary \ref{unipfibs}, which show how relative Malcev homotopy types arise naturally in the study of fibrations. Propositions \ref{eqhtpy} and  \ref{propforms} recall the most accessible manifestations of these homotopy types, in the form of equivariant cochains and the equivariant de Rham complex. Propositions \ref{barGRprop} and \ref{repbar} show how  Narkawicz' bar construction can be used to relate these explicitly to the group-based formulations of relative Malcev homotopy types, and
Theorem \ref{bigequiv} adapts the main comparison result of \cite{htpy} to the case of fixed basepoints.

In Section \ref{malstr}, the constructions of Section \ref{nonabfil} are then extended to homotopy types.  The main result is
  Theorem \ref{mhspin}, showing how  non-abelian algebraic mixed Hodge and twistor structures on relative Malcev homotopy types give rise to such structures on homotopy groups.
  
In the next two  sections, we establish the existence of algebraic  mixed Hodge structures on various relative Malcev homotopy types of compact K\"ahler manifolds, giving more information than rational homotopy types, especially when $X$ is not nilpotent (Corollaries \ref{kmhspin} and \ref{kmhspinanal}). The starting point is the Hodge structure defined on the reductive complex pro-algebraic fundamental group $\varpi_1(X,x_0)^{\red}_{\Cx}$ in  \cite{Simpson}, in the form of a discrete $\Cx^{\by}$-action. We only make use of the induced action of $S^1 \subset \Cx^{\by}$, since this preserves the real form $\varpi_1(X,x_0)^{\red}_{\R}$ and respects the harmonic metric.
We regard this as a kind of pure weight zero Hodge structure on $\varpi_1(X,x_0)^{\red}_{\R}$, since a 
pure weight zero Hodge structure is the same as an algebraic $S^1$-action. 
We then extend this to a mixed Hodge structure on the schematic (or relative Malcev) homotopy type (Theorem \ref{mhsmal} and Proposition \ref{redenrich}).

In some contexts, the unitary action is incompatible with the relative Malcev representation. In these cases, we instead only have mixed twistor structures (as defined in \cite{MTS}) on the homotopy type (Theorem \ref{mtsmal}) and homotopy groups (Corollary \ref{kmtspin}). 

Section \ref{vmhssn} shows how representations of  $\varpi_1(X,x_0)$ in the category of mixed Hodge structures correspond to variations of mixed Hodge structure (VMHS) on $X$ (Theorem \ref{kvmhsrep}), with similar results for mixed twistor structures. This implies (Corollary \ref{kvmtspin})   that the relative Malcev homotopy groups become VMHS as the basepoint varies. Taking the case of $\pi_1$, this proves \cite{arapurapi1} Conjecture 5.5 (see Remarks \ref{arapuraconj} and \ref{arapuraconj2} for details).
\S \ref{hodgeHTsn} then introduces absolute Hodge and absolute twistor homotopy types, giving an explicit description (Lemma \ref{hodgeHTlemma}) for compact K\"ahler manifolds. Local systems for these homotopy types are just VMHS and VMTS respectively (Proposition \ref{pi1Hprop}),  and these can be described in terms of $\cS$-splittings (Remark \ref{cSubiq2}).

Section \ref{archmon} is dedicated to describing the mixed Hodge structure on homotopy types in terms of a pro-nilpotent derivation on the split Hodge structure over $\SL_2$. It provides an explicit description of this derivation in terms of Green's operators on the complex of $\C^{\infty}$ forms on $X$, and in particular shows that the real Hodge structure on $\pi_3(X)\ten \R$ is split whenever $X$ is simply connected (Examples \ref{archmonegs}.2). 

In Section \ref{singsn}, we extend the results of Sections \ref{snkmhs} and \ref{snkmts} to simplicial compact K\"ahler manifolds, and hence to singular proper complex varieties.

Section \ref{quprojsn} then deals with the Malcev homotopy type $(Y,y)^{\rho, \mal}$ of a quasi-projective variety $Y=X-D$ with respect to a reductive  Zariski-dense representation $\rho\co \pi_1(X,y) \to R(\R)$. 
When $Y$ is smooth, Theorem \ref{qmts} establishes a non-positively weighted MTS on $(Y,y)^{\rho, \mal}$, with the associated graded object $\gr^W(Y,y)^{\rho, \mal}$ corresponding to the $R$-equivariant CDGA 
\[
 (\bigoplus_{a,b}\H^{a-b} ( X, \oR^b j_*j^{-1}O(\Bu_{\rho}))[-a], d_2),   
\]
 where  $O(\Bu_{\rho})$ is the local system given by $\rho(\pi_1(X,x))$ acting by multiplication on the structure sheaf $O(R)$, while  $d_2\co \H^{a-b} (X, \oR^b j_*j^{-1}O(\Bu_{\rho})) \to \H^{a-b+2}( X, \oR^{b-1} j_*j^{-1}O(\Bu_{\rho}))$ is  the  differential on the $E_2$ sheet of the Leray spectral sequence for $j: Y \to X$, and $\H^{a-b} ( X, \oR^b j_*j^{-1}O(\Bu_{\rho}))$ has weight $a+b$. 
Theorem \ref{qmhs} shows that if $R$-representations underlie variations of Hodge structure, then the MTS above  extends to   a non-positively weighted MHS on $(Y,y)^{\rho, \mal}$. 
Theorem \ref{singmhsmal2} gives the corresponding results for singular quasi-projective varieties $Y$, with  $\gr^W(Y,y)^{\rho, \mal}$ now characterised in terms of cohomology of a smooth simplicial resolution of $Y$.

In Section \ref{nontrivsn}, these results are extended to Zariski-dense representations $\rho\co \pi_1(Y,y) \to R(\R)$ with unitary monodromy around local components of the divisor. The construction of MHS and MTS in these cases is much trickier than for trivial monodromy. The idea behind Theorem  \ref{unitmhs}, inspired by \cite{Morgan}, is to construct the Hodge filtration on the complexified homotopy type, and then to use  homotopy limits of diagrams to glue this to give a  real form. When $R$-representations underlie variations of Hodge structure on $Y$, this gives a  non-positively weighted MHS on $(Y,y)^{\rho, \mal}$, with  $\gr^W(Y,y)^{\rho, \mal}$ corresponding to the $R$-equivariant CDGA
\[
 (\bigoplus_{a,b}\H^{a-b} ( X, \oR^b j_*O(\Bu_{\rho}))[-a], d_2),   
\]
regarded as a Hodge structure via the VHS structure on $O(\Bu_{\rho})$. For more general $R$, Theorem \ref{unitmts} gives a  non-positively weighted MTS, with the construction based on homotopy gluing over an affine cover of the analytic space $\bP^1(\Cx)$. Simplicial resolutions then extend these results to singular varieties in  
Theorems \ref{singunitmts} and
\ref{singunitmhs}. \S \ref{generalmonodromy} discusses possible extensions to more general monodromy.  

Section \ref{splitsn} is concerned with splittings of MHS and MTS on finite-dimensional vector spaces. Every mixed Hodge structure $V$ splits on tensoring with the ring $\cS$ defined above, giving an $\cS$-linear isomorphism $V\ten \cS\cong (\gr^WV)\ten \cS$ preserving the Hodge filtration $F$. Differentiating with respect to $V$, this gives a map $\beta\co (\gr^WV)\to  (\gr^WV)\ten \Omega(\cS/\R)$ from which $V$ can be recovered.  Theorem \ref{SHSequiv} shows that the $\cS$-splitting can be chosen canonically, corresponding to imposing certain restrictions on $\beta$, and this gives an equivalence of categories. In Remark \ref{cfRMHS}, $\beta$ is explicitly related to Deligne's complex splitting of \cite{RMHS}. Theorem \ref{STSequiv} then gives the corresponding results for mixed twistor structures. \S \ref{hodgeHTsn2} adapts Goncharov's construction to give reduced versions of absolute Hodge and twistor homotopy types, with consequences for canonical  $\cS$-splittings and VMHS/VMTS. 

The main result in Section \ref{nonabsplitsn} is Theorem \ref{strictmodel}, which shows that every  non-positively weighted MHS or MTS on a real relative Malcev homotopy type admits a strictification, in the sense that it is represented by an $R$-equivariant CDGA in ind-MHS or ind-MTS. Corollary \ref{allsplit} then applies the results of Section \ref{splitsn} to give canonical $\cS$-splittings for such MHS or MTS, while  Corollary \ref{qformalthing} shows that the splittings give equivalences $(Y,y)^{\rho, \mal} \simeq\gr^W(Y,y)^{\rho, \mal}$. Corollary \ref{mhspin2} shows that they give rise to  MHS or MTS on homotopy groups, and this is applied to quasi-projective varieties in  Corollary \ref{unitsingmtspin}. There are various consequences for  deformations of representations (Proposition \ref{defprop}). Finally, 
Theorem \ref{cfksplit} shows that for projective varieties, the canonical $\cS$-splittings coincide with the explicit Green's operator $\cS$-splittings established in Theorems \ref{mhsmal} and \ref{mtsmal}.

\subsubsection*{Acknowledgements}

I would like to thank Carlos Simpson for drawing my attention to the questions addressed in this paper, and for helpful discussions, especially concerning the likely generality for the results in \S \ref{nontrivsn}. I would also like to thank Jack Morava for suggesting that non-abelian mixed Hodge structures should be related to Archimedean $\Gamma$-factors, and
Tony Pantev for alerting me to  \cite{RMHS}. Finally, I would like to thank the anonymous referee who greatly simplified key proofs by  suggesting I use the reduced bar construction, and made many other  suggestions greatly improving the exposition.

\subsection*{Notation}

For any affine scheme $Y$, write $O(Y):= \Gamma(Y, \O_Y)$. Given a group $G$ acting on sets $X ,Y$, we will adopt the convention of writing $X\by^GY$ for the quotient
\[
 X\by^GY:= (X \by Y)/G;
\]
we use a superscript instead of the more usual subscript to avoid possible confusion with fibre products.


\tableofcontents

\section{Splittings for MHS on real homotopy types}\label{real1}

In \cite[Theorem 9.1]{Morgan} and \cite[Theorem 1]{Hainhodge},
a   mixed  Hodge structure was given on the rational  homotopy groups of a smooth complex variety $X$. Here, we study the consequences of formality quasi-isomorphisms for this mixed Hodge structure when  $X$ is a connected compact K\"ahler manifold.

Let $A^{\bt}(X)$ be the differential graded algebra of real $\C^{\infty}$ forms on $X$. As in \cite{DGMS}, this is the real (nilpotent) homotopy type of $X$. If we write $J$ for the operator on $A^{\bt}(X)$ coming from the complex structure on the cotangent bundle, then there is  a differential $\dc:=J^{-1}dJ$ on the underlying graded algebra  $A^*(X)$. Note that $d\dc+\dc d=0$.

\begin{definition}
 Define a cochain  algebra (or CDGA) to be  a  cochain complex $A=\bigoplus_{i \in \Z} A^i$ over $k$, equipped with a graded-commutative associative  product $A^i \by A^j\ra A^{i+j}$, and unit $1 \in A^0$; it is said to be non-negatively graded when $A^i=0$ for all $i<0$.
\end{definition}

Thus $A^{\bt}(X)$ is a real non-negatively graded cochain algebra, and hence a real homotopy type in the sense of \cite{sullivan}.

\begin{definition}\label{quasiMHS}
Define a (real) quasi-MHS to be a real vector space $V$, equipped with  an exhaustive  (i.e. $V= \bigcup_n W_nV$) increasing   filtration $W$ on $V$, and  an exhaustive decreasing filtration $F$ on $V\ten \Cx$.
 
We adopt the convention that a (real) MHS is a finite-dimensional  quasi-MHS on which $W$ is Hausdorff (i.e. $\bigcap_n W_nV=0$), satisfying the opposedness condition
$$
\gr_F^i\gr_{\bar{F}}^j\gr^W_n(V\ten \Cx) =0
$$
for $i+j\ne n$.

Define a (real)  ind-MHS to be a filtered colimit
 of MHS. Equivalently, an ind-MHS  is a quasi-MHS $V$ which is equal to the union of all MHS $U \subset V$.  Say that an ind-MHS $V$ is bounded below if $W_NV=0$ for $N \ll 0$. 
\end{definition}

A quasi-MHS on a CDGA $A$ is a quasi-MHS on the underlying cochain complex, in such a way that the multiplication map $A\ten A \to A$ is a morphism of quasi-MHS. As in \cite{Morgan}, there is a natural quasi-MHS on $A^{\bt}(X)$. The Hodge filtration on $ A^{\bt}(X)\ten_{\R} \Cx$ is $F^p(A^{\bt}(X)\ten_{\R} \Cx)= \bigoplus_q\bigoplus_{p'\ge p} A^{p'q}(X, \Cx)$, and the  weight filtration is given by the good truncation $W_iA^{\bt}(X) =\tau^{\le i}A^{\bt}(X)$. Explicitly,
\[
 (\tau^{\le i}V)^j= \left\{\begin{matrix} V^j & j<i, \\ \z^j(V) & j=i, \\ 0 & j>i. \end{matrix}\right.
\]

\subsection{The family of formality quasi-isomorphisms}

\begin{lemma}\label{gl2lemma}
Given a graded module $V^*$ over a ring $B$, equipped with operators $d_1, d_2$ of degree
$1$ such that $[d_1,d_2]=(d_1)^2=(d_2)^2=0$, then for 
$
\left(\begin{smallmatrix} u &  v \\ x & y\end{smallmatrix}\right) \in \GL_2(B),
$
\begin{eqnarray*}
  \ker d_1 \cap \ker d_2 &=&\ker (ud_1 + vd_2) \cap \ker (xd_1+yd_2) ,\\
 \im ( ud_1 + vd_2)+ \im(xd_1+yd_2)&=& \im d_1 + \im d_2,\\
\im ( ud_1 + vd_2)(xd_1+yd_2)&=&\im d_1d_2.
\end{eqnarray*}
\end{lemma}
\begin{proof}
Observe that if we take any matrix, the corresponding inequalities   (with $\le$ replacing $=$) all hold. For invertible matrices, we may express $d_1,d_2$ in terms of $( ud_1 + vd_2),(xd_1+yd_2)$ to give the reverse inequalities.
\end{proof}

\begin{proposition}\label{gl2prop}
If the pair $(d_1,d_2)$ of Lemma \ref{gl2lemma} satisfies the principle of two types, then so does $( ud_1 + vd_2),(xd_1+yd_2)$ whenever $
\left(\begin{smallmatrix} u &  v \\ x & y\end{smallmatrix}\right) \in \GL_2(B),
$.
\end{proposition}
\begin{proof}
The principle of two types states that
$$
\ker d_1 \cap \ker d_2 \cap (\im d_1 + \im d_2)=\im d_1d_2,
$$
so Lemma \ref{gl2lemma} gives precisely the result we require.
\end{proof}

\begin{definition}\label{cSdef}
 Define $\cS$ to be the ring 
\[
 \cS:=\R[x],
\]
with filtration $F^p (\cS\ten \Cx):= (x-i)^p\Cx[x]$ on $\cS\ten_{\R}\Cx$. We also give it a weight  filtration $W$, setting $W_0\cS=\cS$, $W_{-1}\cS=0$. 
\end{definition}
Observe that $\gr^W_0\cS=\cS$, and that $F^1(\cS\ten \Cx)+ \bar{F}^1(\cS\ten \Cx)= \cS\ten \Cx$, which together with multiplicative properties of $F$ imply that 
\[
 \gr_F^p\gr_{\bar{F}}^q(\cS\ten \Cx)=
\gr_F^p\gr_{\bar{F}}^q\gr^W_0(\cS\ten \Cx)=0
\]
 for all $p,q$.
Thus $\cS$ is a quasi-MHS satisfying the opposedness condition, but cannot be an  ind-MHS, since  $\gr^W(\cS\ten \Cx) $ is not isomorphic to $\gr_F\gr_{\bar{F}}\gr^W(\cS\ten \Cx)$.

\begin{definition}
 A morphism $(A,W) \to (B,W)$ of filtered CDGAs is said to be a filtered quasi-isomorphism if it induces quasi-isomorphisms $W_nA \to W_nB$ for all $n$. A morphism $(A,W,F) \to (B,W,F)$ of bifiltered CDGAs is said to be a bifiltered quasi-isomorphism if it induces quasi-isomorphisms $W_nF^pA \to W_nF^pB$ for all $n,p$.
\end{definition}

We now look at the CDGA $A^{\bt}(X) \ten_{\R} \cS $, with  quasi-MHS structure given by extending $W,F$ to tensor products by the usual conventions. These induce filtrations, and hence quasi-MHS structures, on cohomology with respect to any differential respecting the structures. In particular,  there is an isomorphism $\H^*(A^{\bt}(X) \ten_{\R} \cS) \cong \H^*(X,\R)\ten_{\R} \cS$ of quasi-MHS, where 
 $ \H^*(X,\R)$ is endowed with its usual Hodge structure.

The principle of two types now gives us a family of quasi-isomorphisms:
\begin{theorem}\label{formalityS}
We have the following  $W$-filtered quasi-isomorphisms of CDGAs 
$$
A^{\bt}(X) \ten_{\R} \cS \xla{i'} \ker(\dc+ xd) \xra{p'}  \H^*_{\dc+xd}(A^*(X) \ten \cS) \cong \H^*(X,\R)\ten_{\R} \cS,
$$
where $\ker(\dc+xd):=\ker(\dc+xd)\cap (A^{\bt}(X) \ten \cS)$, with differential $d$. 
Moreover, on tensoring with $\Cx$, these become $(W,F)$-bifiltered quasi-isomorphisms. 
\end{theorem}
\begin{proof}
Since $W$ is defined as good truncation, any quasi-isomorphism is automatically $W$-filtered, so it suffices to show that we have 
$F$-filtered quasi-isomorphisms. 

We  introduce a formal variable $w$, and write $\xi(V,F):= \bigoplus_p F^pV w^{-p} \subset V[w, w^{-1}]$ for the Rees module of a filtered complex, noting that this has the structure of a $\Z[w]$-module. Because this is a direct sum of the filtered pieces, filtered quasi-isomorphisms are those morphisms inducing quasi-isomorphisms of Rees modules. 

Writing $\cS_{\Cx}:=\cS\ten_{\R}\Cx$, $A^{\bt}(X, \cS_{\Cx}):=A^{\bt}(X)\ten_{\R}\cS_{\Cx}$ and $\H^*(X,\cS_{\Cx})\cong \H^*(X,\Cx)\ten_{\Cx} \cS_{\Cx}$, it will suffice to show that we have quasi-isomorphisms
$$
 \xi(A^{\bt}(X, \cS_{\Cx}),F) \xla{i'} \xi(\ker(\dc+ xd)_{\Cx},F) \xra{p'}  \xi(\H^*_{\dc+xd}(A^*(X,\cS_{\Cx}),F) \cong  \xi(\H^*(X,\cS_{\Cx}),F)
$$
of $\Cx[w]$-CDGAs.

In order to be able to apply Proposition \ref{gl2prop}, we note that $\xi(\cS_{\Cx},F)= \Cx[w^{-1}(x-i),w]$ and that
\[
\xi(A^n(X, \cS_{\Cx}),F)= \bigoplus_{p+q=n} w^{-p}A^{pq}(X, \Cx)[w^{-1}(x-i),w], 
\]
so multiplying $A^{pq}$ by $w^p$ gives an isomorphism
\[
 \uline{w}\co \xi(A^*(X, \cS_{\Cx}),F)\to A^*(X,\Cx)[w^{-1}(x-i),w]
\]
which leaves $\bar{\pd}$ unchanged and sends $\pd$ to $w\pd$. Writing $d_w:= w\pd+ \bar{\pd}$, we thus have an isomorphism
\[
 \uline{w}\co \xi(A^{\bt}(X, \cS_{\Cx}),F)\to (A^*(X,\Cx)[w^{-1}(x-i),w], d_w)
\]
of $\Cx[w^{-1}(x-i),w]$-CDGAs.

Now, the key observation to make is that $\dc+xd= (x+i)\pd + (x-i)\bar{\pd}$, with
\begin{eqnarray*}
 (x+i)\pd \co (F^a A^n(X, \Cx))\ten_{\Cx}(F^b \cS_{\Cx})&\to&  (F^{a+1} A^{n+1}(X, \Cx))\ten_{\Cx}F^{b}\cS_{\Cx}\\
(x-i)\bar{\pd}\co (F^a A^n(X, \Cx))\ten_{\Cx}(F^b \cS_{\Cx})&\to& (F^a A^{n+1}(X, \Cx))\ten_{\Cx}F^{b+1} \cS_{\Cx},
\end{eqnarray*}
and hence
\[
\dc+xd\co F^pA^n(X, \cS_{\Cx}) \to F^{p+1}A^{n+1}(X, \cS_{\Cx}),
\]
which induces a differential
\[
 w^{-1}(\dc+xd)\co \xi(A^n(X, \cS_{\Cx}),F)\to \xi(A^{n+1}(X, \cS_{\Cx}),F).
\]
On applying $\uline{w}$, this gives a differential $\dc_{w,x}:=(x+i)\pd + w^{-1}(x-i)\bar{\pd}$ on $A^*(X,\Cx)[w^{-1}(x-i),w]$

We now appeal to Proposition \ref{gl2prop}, taking $B= \Cx[w^{-1}(x-i),w]$, $V^*= A^*(X,\Cx)$, $d_1= \pd$ and $d_2=\bar{\pd}$. Since the matrix $\left(\begin{smallmatrix} w, & 1 \\ (x+i), & w^{-1} (x-i) \end{smallmatrix}\right)$ is invertible in $B$, the differentials  $d_w=w\pd+ \bar{\pd}$ and  $\dc_{w,x}=(x+i)\pd + w^{-1}(x-i)\bar{\pd} $ satisfy the principle of two types. 

In particular, this means that inclusion and projection give 
quasi-isomorphisms
\[
 (A^*(X,\Cx)[w^{-1}(x-i),w], d_w) \xla{i'} (\ker \dc_{w,x}, d_w) \xra{p'} \H^*_{\dc_{w,x}}( A^*(X,\Cx)[w^{-1}(x-i),w]),
\]
with differential $0$ on the final complex. It also means that the isomorphism $\H^*A^{\bt}(X,\Cx)\cong \frac{\ker \pd \cap \ker \bar{\pd}}{\im \pd\bar{\pd}}$ then gives us an isomorphism
\[
 \H^*_{\dc_{w,x}}(A^*(X,\Cx)[w^{-1}(x-i),w]) \cong \H^*(X, \Cx)[w^{-1}(x-i),w].
\]

On applying the isomorphism $\uline{w}^{-1}$, these quasi-isomorphisms yield quasi-isomorphisms
\[
  \ker(\dc+xd) \cap F^pA^{\bt}(X, \cS_{\Cx}) \xra{i'}F^pA^{\bt}(X, \cS_{\Cx}) 
\]
for all $p$, and
\[
 \ker(\dc+xd) \cap F^pA^{\bt}(X, \cS_{\Cx})\xra{p'} \frac{\ker(\dc+xd) \cap F^pA^{\bt}(X, \cS_{\Cx}) }{ (\dc+xd)F^{p-1}A^{\bt}(X, \cS_{\Cx})},
\]
together with an isomorphism of the last group with $F^p\H^*(X, \cS_{\Cx})$. 
\end{proof}

\begin{remarks}
Note that we cannot deduce Theorem \ref{formalityS} directly from Lemma \ref{gl2lemma} for the pair $d, \dc+xd$, since that would only establish that $i',p'$ are quasi-isomorphisms preserving the filtrations, rather than filtered quasi-isomorphisms.

Evaluating at $x=0$ recovers the real formality quasi-isomorphism of \cite{DGMS}, while taking $x=i$ gives the complex filtered quasi-isomorphism used in  \cite{Morgan}. 

If the proof of Theorem \ref{formalityS} looks like a string of coincidences, that is because it is the translation of a conceptual geometric proof into purely algebraic language. The original proof will be given in Theorem \ref{formalitysl}, where it arises from the geometry governing Hodge filtrations. 
\end{remarks}

Since $A^*(X,\R)\ten \cS$ is not a mixed Hodge structure in the classical sense (as  $F$ is not bounded on $\cS$), we cannot now apply the theory of mixed Hodge structures on real homotopy types from \cite{Morgan} to infer consequences for Hodge structures on homotopy groups. Instead, we follow \cite{Hainhodge} and use the reduced bar construction.

\subsection{The reduced bar construction}

\begin{definition}\label{CCdef}
Given a  flat dg associative algebra $A$ over a commutative ring $k$ and an $A$-bimodule $P$ in cochain complexes of flat $k$-modules,
recall that the Hochschild homology complex 
\[
 \uline{\CCC}_{\bt}(A/k, P)      
\]
(a semi-simplicial diagram of cochain complexes) is given by 
\[
 \uline{\CCC}_n(A/k,P):=   A^{\ten_k n} \ten_kP,   
\]
with face maps
\[
 \pd_i(a_1\ten\ldots a_n \ten p)= \left\{ \begin{matrix}   a_2\ten\ldots a_n \ten (pa_1) & i=0 \\  a_1\ten\ldots a_{i-1} \ten (a_ia_{i+1}) \ten a_{i+2} \ten \ldots\ten a_n \ten p & 0<i <n \\
  a_1\ten\ldots a_{n-1} \ten (a_np) & i=n.                                              
                                          \end{matrix}\right.
 \]

This has an associated chain cochain complex on the same bigraded module, with chain differential $\sum (-1)^i\pd_i$, and we write  $\CCC_{\bt}(A/k, P)$ for the direct sum total complex.
\end{definition}

\begin{definition}\label{barAdef}
 Given a  non-negatively graded  DGA  $A$ over a commutative ring $k$ with $\H^0A=k$, define the complex $\bar{A}^{\bt}$ by
\[
 \bar{A}^n := \left\{ \begin{matrix} A^n & n>1, \\ A^1/dA^0 & n=1, \\ 0 & n \le 0. \end{matrix} \right.  
\]
\end{definition}

\begin{definition}\label{barBdef}
Assume we have a commutative ring $k$, a non-negatively graded  $k$-DGA $A$, a right $A$-module $M$ and a left $A$-module $N$, with $\bar{A},M,N$ all flat over $k$. We then follow \cite[Definition 1.2.1]{Hainhodge} in defining the reduced  bar complex
\[
 \bar{B}_k(M,A,N)
\]
to be the quotient of the total homological Hochschild  complex $\CCC_{\bt}(A/k, N\ten_kM)$   by the subcomplex generated by the kernel of $\CCC_{\bt}(A, N\ten_kM)\to  \CCC_{\bt}(A^+, M\ten_kN)$, where $A^+=A^{>0}$. Equivalently, we kill all subspaces 
\[
 M \ten A^{\ten r} \ten A^0 \ten A^{\ten s} \ten N, \quad d(M \ten A^{\ten r} \ten A^0 \ten A^{\ten s} \ten N ).
\]
Note that Hain writes $T(M,A,N)$ for the complex  $\CCC_{\bt}(A^+, M\ten_kN)$.
\end{definition}

\begin{lemma}\label{barQIM}
The functor $\bar{B}_k$ preserves quasi-isomorphisms.
\end{lemma}
\begin{proof}
When $k$ is a field, this is proved in \cite{chenReducedBar} and  \cite[Corollary 1.2.3]{Hainhodge}, and the same proof carries over. We look at the bar filtration of $\bar{B}_k(M,A,N)$ by tensor powers of $A$, and obtain a convergent spectral sequence (the Eilenberg--Moore spectral sequence)
\[
 E_0^{-s,t}= (M\ten_k \bar{A}^{\ten_k s} \ten_k N)^t \abuts \bar{B}_k(M,A,N)^{t-s}.
\]
 
Since $M, \bar{A},N$ are flat over $k$, any quasi-isomorphism $(M,A,N) \to (M',A',N')$ induces a quasi-isomorphism on $E_0$ and hence on $\bar{B}_k$.
\end{proof}

Given compatible filtrations $F$ on $M,A,N$, we will write $F$ for the filtration induced on $\bar{B}_k(M,A,N)$ by the usual convention
\[
 F^p(M\ten A^{\ten s}\ten N):= \bigoplus_{a+ b_1 + \ldots +b_s +c=p} F^aM \ten F^{b_1}A\ten \ldots F^{b_s}A\ten F^cN.
\]

When $A$, $M$ and $N$  are all CDGAs, the shuffle product gives   $\bar{B}_k(M,A,N)$  a natural CDGA structure, by \cite[1.2.4]{Hainhodge}. Given a $k$-algebra homomorphism $A^0 \to k$,  \cite[1.2.5]{Hainhodge} defines a coassociative  comultiplication
\[
\Delta\co  \bar{B}_k(k,A,k) \to \bar{B}_k(k,A,k)\ten_k \bar{B}_k(k,A,k),
\]
together with co-unit and antipode making $\bar{B}_k(k,A,k)$ into a DG Hopf algebra. 

\begin{definition}\label{barHopfdef}
For any augmentation $A^0\to k$,  write $\bar{B}_k(A \to k)$ for the DG  algebra $\bar{B}_k(k,A,k)$ equipped with its  DG Hopf algebra structure.
\end{definition}

\begin{definition}\label{barg0}
Given a  flat non-negatively graded  DGA $A$ over a commutative ring $k$ with $\H^0A=k$ and an augmentation $A \to k$, define
 $\bar{G}^{\vee}_k(A\to k)$ to be  the ind-conilpotent DG Lie coalgebra given by the cotangent space
\[
 \bar{G}^{\vee}_k(A\to k):= \cot_1\bar{B}_k(A \to k)
\]
at $1$  of the DG Hopf algebra $\bar{B}_k(A \to k)$, where 
\[
 \cot_1\bar{B}_k(A \to k)= \bar{B}_k(A \to k)^+/(\bar{B}_k(A \to k)^+)^2
\]
for $\bar{B}_k(A \to k)^+$ the kernel of the counit  $\bar{B}_k(A \to k)\to k$. 

The dual $\bar{G}^k(A\to k) $ of $\bar{G}^{\vee}_k(A\to k) $ is then an inverse limit of nilpotent finite-dimensional dg Lie algebras, and we can recover $\bar{B}_k(A \to k)$ as the continuous dual of the universal pro-finite-dimensional enveloping algebra of  $\bar{G}(A\to k)$.
\end{definition}

When $A^0=k$, there is a unique augmentation $A \to k$ and we write $B_k(A):=\bar{B}_k(A \to k)$ and $G^k(A):= \bar{G}^k(A \to k)$. Note that $B_k(A)$ is the cofree ind-conilpotent graded coalgebra
\[
 B_k(A) \cong \bigoplus_{n\ge 0} (A^+[1])^{\ten_k n} 
\]
on cogenerators $A^+[1]$, with differential $d$ defined on cogenerators by $(A^+\ten_kA^+)[1] \oplus A^+ \xra{(\mu,d_A)} A^+[1]$, where $\mu$ denotes multiplication. Taking the quotient by shuffle permutations then gives  $\bar{G}^{\vee}_k(A) $, which is  cofree as an ind-conilpotent graded Lie coalgebra. Poincar\'e--Birkhoff--Witt gives a ring isomorphism
\[
  B_k(A) \cong \Symm_k(G_k^{\vee}(A)).
\]

By \cite[Corollary 1.23]{Hainhodge}, the reduced bar construction preserves quasi-isomorphisms, so the same is true of $\bar{G}$.

\subsection{The mixed Hodge structure on Quillen's real homotopy type and on homotopy groups}

\begin{definition}
We define the real homotopy groups $\pi_*(X\ten \R,x_0) $ of $X$  to be the pro-finite-dimensional real  vector spaces
\[
 \pi_n(X\ten \R,x_0):= \H_{n-1}\bar{G}^{\R}(A(X) \xra{x_0^*} \R).
\]
Note that for $n>1$, these agree with Sullivan's real homotopy groups by \cite[Remark \ref{htpy-deligne}]{htpy}, while $\pi_1(X\ten \R,x_0)$ is the Lie algebra of the real pro-unipotent completion of $\pi_1(X,x_0)$.
When $X$ is nilpotent, we have canonical isomorphisms
\[
 \pi_*(X\ten \R,x_0)\cong \pi_*(X,x_0)\ten_{\Z} \R,
\]
by \cite[Theorem 6.2]{Hainhodge}, where the graded vector space $\pi_*(X\ten \R,x_0) $ is denoted by $\g_{\bt}(X,x_0)\ten_{\Q}\R$.

Similarly, we write 
\[
 \pi_n(\H^*(X,\R)):= \H_{n-1}G^{\R}(\H*(X,\R))
\]
for the real homotopy groups of the formal homotopy type $\H^*(X,\R)$.
\end{definition}

As in \cite[Theorem 4.2.1]{Hainhodge}, the reduced bar construction transfers the real quasi-MHS on $A^{\bt}(X)$ to a real quasi-MHS on $G^{\vee}(x_0^*\co A(X) \to \R)$ and makes  $\pi_n(X\ten \R,x_0)$ a real pro-MHS (i.e. an inverse limit of finite-dimensional MHS). Similarly $\pi_n(H^*(X, \R))$ is an inverse system of real Hodge structures.


\begin{theorem}\label{morgansplitsa}
For $x_0 \in X$, and for all $n$, there are $\cS$-linear isomorphisms
$$
\pi_*(X\ten \R,x_0)\ten_{\R}\cS \cong \pi_*(H^*(X, \R))\ten_{\R}\cS,
$$
of inverse systems of quasi-MHS, compatible with Whitehead brackets and Hurewicz maps. The associated graded map from the weight filtration is just the pullback of the standard isomorphism $\gr_W\pi_*(X\ten \R,x_0)\cong\pi_*(H^*(X, \R))$ coming from degeneration of the 
spectral sequence associated to the weight filtration $W$ on $\bar{G}(x_0^*\co A(X) \to \R)$.
\end{theorem}
\begin{proof}
Observe that $A^n(X,\cS)= A^n(X,\R)\ten_{\R}\cS$, so is flat over $\cS$. Moreover, 
\[
 A^1(X,\cS)/dA^0(X,\cS)= (A^1(X,\R)/dA^0(X,\R))\ten_{\R}\cS,
\]
which is also flat over $\cS$.

We may therefore apply the reduced bar construction $\bar{B}_{\cS}$ to the augmented CDGA $x_0^* \co A^{\bt}(X,\cS)\to \cS$,  giving a dg Hopf algebra $\bar{B}_{\cS}(x_0^*\co A^{\bt}(X,\cS)\to\cS)= \bar{B}_{\R}(x_0^*\co A^{\bt}(X)\to \R)\ten_{\R}\cS$ over $\cS$.
 By \cite[Corollary 1.23]{Hainhodge}, the reduced bar construction preserves quasi-isomorphisms, so Theorem \ref{formalityS} gives quasi-isomorphisms
\[
 \bar{B}_{\R}( A^{\bt}(X)\xra{x_0^*} \R)\ten_{\R}\cS \xla{i'} \bar{B}_{\cS}(\ker(\dc+ xd) \xra{x_0^*} \cS)\xra{p'}  B_{\cS}(\H^*_{\dc+xd}(A^*(X) \ten \cS)
\]
 and an isomorphism between the last and $B_{\R}(\H^*(X,\R))\ten_{\R} \cS$.

The filtration $F$ on a flat augmented  $\cS_{\Cx}$-CDGA $A\to \cS_{\Cx}$ extends naturally to $\bar{B}_{\cS_{\Cx}}(A\to \cS_{\Cx}) $ by setting
\[
 F^p( A^{\ten_{\cS_{\Cx}} n}):= \sum_{p_1 + \ldots +p_n=p} (F^{p_1} A)\ten_{\cS_{\Cx}} \ldots \ten_{\cS_{\Cx}} (F^{p_n} A).
\]
 By the results of \cite[\S 3.2]{Hainhodge}, the reduced bar construction preserves filtered quasi-isomorphisms, so the maps of dg Hopf algebras above become $F$-filtered quasi-isomorphisms on tensoring with $\Cx$. 

Alternatively, we could observe that the description of $\xi(A^n(X, \cS_{\Cx}),F)$ from the proof of Theorem \ref{formalityS} shows it is flat over $\xi(\cS_{\Cx},F)= \Cx[w^{-1}(x-i),w]$, as is
\begin{align*}
& \xi(A^1(X, \cS_{\Cx}),F)/d\xi(A^0(X, \cS_{\Cx}),F)=  \\ 
&w^{-1}A^{10}(X, \Cx)[w^{-1}(x-i)]\oplus (A^1(X, \Cx)/dA^0(X,\Cx))[w^{-1}(x-i),w],
\end{align*}
the morphism $A^{10}(X, \Cx)\to A^1(X, \Cx)/dA^0(X,\Cx)$ being injective.
Thus we may apply the bar construction
\[
 \xi(\bar{B}_{\cS}(x_0^*\co A^{\bt}(X,\cS)\to\cS)_{\Cx}, F)= \bar{B}_{\xi(\cS_{\Cx},F)}(x_0^* \co \xi(A^{\bt}(X,\cS_{\Cx}),F)\to \xi(\cS_{\Cx},F)),
\]
and use the quasi-isomorphisms featuring in the proof of Theorem \ref{formalityS}.

The filtration $W$ on a flat augmented  $\cS$-CDGA $A\to \cS$ extends naturally to $\bar{B}_{\cS}(A\to \cS) $ by setting
\[
 W_s( A^{\ten_{\cS} n}):= \sum_{s_1 + \ldots +s_n=s} (W_{s_1} A)\ten_{\cS} \ldots \ten_{\cS} (W_{s_n} A),
\]
and the same argument shows that the morphisms of dg Hopf algebras above are $W$-filtered quasi-isomorphisms, or $(W,F)$-bifiltered quasi-isomorphisms on tensoring with $\Cx$. 

On taking cotangent spaces, we then have  quasi-isomorphisms of quasi-MHS
\[
 \bar{G}_{\R}^{\vee}( A^{\bt}(X)\xra{x_0^*} \R)\ten_{\R}\cS \xla{i'} \bar{G}^{\vee}_{\cS}(\ker(\dc+ xd) \xra{x^*} \cS)\xra{p'}  
G^{\vee}_{\R}(\H^*(X,\R))\ten_{\R} \cS,
\]
giving the required isomorphisms on $\H_*$.
\end{proof}

In \S \ref{vmhssn}, we will see how these MHS become variations of mixed Hodge structure as the basepoint $x_0 \in X$ varies.

\begin{remark}\label{trueweightrmk0}
 Beware that in \cite{Hainhodge}, the weight filtration on $A^{\bt}$ is defined to be the unshifted weight filtration, which is trivial in this case, and from which our weight filtration $W$ is obtained by d\'ecalage. Our convention is required to allow comparison with the weight filtration on cohomology, and by deeper considerations which we will discuss in Remark \ref{trueweightrmk}. 

One immediate simplification following from our convention is that our weight filtration  on the reduced bar complex is given by the natural extension to tensor powers. By contrast, the weight filtration from \cite{Hainhodge} in this case is defined to be the bar filtration $\fB$ instead. Note that inclusion gives quasi-isomorphisms $W \to \Dec \fB$, so  our weight filtration on the bar complex is again quasi-isomorphic to the  d\'ecalage of Hain's weight filtration. 
\end{remark}

\begin{definition}\label{gammadef}
Given a quasi-MHS $V$, define the decreasing filtration $\gamma^*$ on $V$ by $\gamma^p V= V\cap F^p(V\ten \Cx)$.
\end{definition}

\begin{remark}\label{gadata1}
Given the Hodge structure on the cohomology ring $\H^*(X,\R)$, Theorem \ref{morgansplitsa} leads us to ask what additional data are required to describe the mixed Hodge structure on real homotopy groups. Writing $N \co \cS \to \Omega(\cS/\R)= \cS dx$ for the canonical derivation, observe that $N$ is surjective, so gives a resolution of $\R$. Moreover,  $\cS dx $ carries a natural quasi-MHS inherited from $\cS$, with $dx = d(x-i) \in F^1$.

The derivation $N$ naturally extends to an $(\cS, N)$-derivation $N_{\pi_*}\co \pi_*(X\ten \R)^{\vee}\ten_{\R}\cS\to \pi_*(X\ten \R)^{\vee}\ten_{\R}\cS dx$, with $ \pi_*(X\ten \R)^{\vee}=\ker N_{\pi_*}$.
In order to recover the Hodge structure on $\pi_*(X\ten \R)$, it therefore suffices to determine the corresponding $(\cS, N)$-derivation of $\pi_*( H^*(X, \R))^{\vee}\ten_{\R}\cS$, or equivalently its restriction to generators. Since the derivation must be trivial on $ \gr_W\pi_*(X\ten \R)^{\vee}\ten_{\R}\cS$, this gives us an element of  
\begin{align*}
& W_{-1}\gamma^0\Hom_{\R}(\pi_*( H^*(X, \R))^{\vee},\pi_*( H^*(X, \R))^{\vee}\ten_{\R}\cS dt )) \\
&\cong W_{-1}\gamma^{-1}\Hom_{\R}(\pi_*( H^*(X, \R))^{\vee},\pi_*( H^*(X, \R))^{\vee}\ten_{\R}\cS)), 
\end{align*}
which is the datum we require to recover the mixed Hodge structure on $\pi_*(X\ten \R)$.  

We can  also ask what additional data are required to describe the mixed Hodge structure on the rational homotopy type. The proof of Theorem  \ref{morgansplitsa} gives a quasi-isomorphism
\[
 \bar{G}^{\R}(x_0^*\co A^{\bt}(X,\R)\to \R)\ten_{\R}\cS \simeq G^{\R}( H^*(X, \R))\ten_{\R}\cS
\]
of pro-nilpotent dg Lie $\cS$-algebras in quasi-MHS, so applying $N$ on the left gives rise to a homotopy class of Lie 
derivations
\[
 N_A\co G^{\R}( H^*(X, \R))\ten_{\R}\cS \to G^{\R}( H^*(X, \R))\ten_{\R}\cS dt.
\]

Since the algebraic mixed Hodge structure is formally defined from the algebraic Hodge filtration, we need no further data, but there is an additional restriction --- the derivation must be $0$ on $\gr^W$, corresponding to the  isomorphism $\gr^WA^{\bt}(X) \simeq \H^*(X,\R)$. Thus mixed Hodge structures correspond to choices of 
$$
N_A \in \H^0(W_{-1}\gamma^{-1}(\Der_{\R}(G^{\R}\H^*(X,\R), G^{\R}\H^*(X,\R)) \ten_{\R}\cS)).
$$

In \S \ref{archmon} we will show how to calculate $N_A$, and hence the MHS, explicitly   from the formality quasi-isomorphisms.
\end{remark}

Lacking a suitable reference, we now verify that our mixed Hodge structure on homotopy groups agrees with the mixed Hodge structure given in \cite{Morgan}.

\begin{proposition}\label{morganhodge}
The mixed Hodge structures on homotopy groups given in Theorem \ref{morgansplitsa} and \cite[Theorem 9.1]{Morgan} agree.
\end{proposition}
\begin{proof}
In \cite[\S 6]{Morgan}, a minimal model $\cM$ was constructed for $\A^{\bt}(X,\Cx)$, equipped with a bigrading (i.e. a $\bG_m\by \bG_m$-action). The associated quasi-isomorphism $\psi: \cM \to \A^{\bt}(X,\Cx)$ satisfies
$\psi(\cM^{pq}) \subset \tau^{\le p+q} F^p\A^{\bt}(X,\Cx) $. Thus $\psi$ is a map of bifiltered CDGAs. It is also a quasi-isomorphism of CDGAs, but we need to show that it is a quasi-isomorphism of bifiltered CDGAs. By \cite[Lemma 6.2b]{Morgan}, $\psi$  maps $\H^*(\cM^{pq})$ isomorphically to $\H^{pq}(X, \Cx)$, so the associated Rees algebras are quasi-isomorphic.

The bar construction then gives a $(W,F)$-filtered quasi-isomorphism
\[
 G^{\Cx}(\cM) \to \bar{G}^{\Cx}( x_0^*\co \A^{\bt}(X,\Cx) \to \Cx).
\]
Under the equivalences of \cite[Theorem \ref{htpy-bigequiv}]{htpy}, $\H_{n-1}(G^{\Cx}(\cM))^{\vee}= \H^n(\bL^{\cM/\R}\ten^{\oL}_{\cM}\R)$ where $\bL$ denotes the cotangent complex. Since $\cM$ is cofibrant, this is just $\H^n(\Omega(\cM/\R)\ten_{\cM}\R)$.  Finally, $\cM$ is minimal, so the complex $\Omega(\cM/\R)\ten_{\cM}\R$ is isomorphic to the  indecomposables $I$ of $\cM$, with trivial differential. This means that $\H_{n-1}(G^{\Cx}(\cM))^{\vee}\cong I^n$, and
$$
\xi(\varpi_n(X\ten \Cx,x)^{\vee}; W,F)= p^*\xi(\varpi_n(X\ten \R, x)^{\vee}, \MHS) \cong \xi(I^n;W,F),
$$
so the Hodge and weight filtrations from Theorem \ref{morgansplitsa} and \cite{Morgan} agree.
\end{proof}

\section{Non-abelian structures}\label{nonabfil}

\subsection{Hodge filtrations}\label{hodgefil}
In this section, we will define algebraic Hodge filtrations on real affine schemes. This construction is essentially that of \cite[\S 5]{Simfil}, with the difference that we are working over $\R$ rather than $\Cx$.

\subsubsection{Abelian Hodge filtrations revisited}

\begin{lemma}\label{flatfiltrn}
The category of flat quasi-coherent $\bG_m$-equivariant sheaves on $\bA^1$ is equivalent to the category of  exhaustive (i.e. $V= \bigcup_n F_nV$) filtered vector spaces, where $\bG_m$ acts on $\bA^1$ via the standard embedding $\bG_m \into \bA^1$. 
\end{lemma}
\begin{proof}
Let $t$ be the co-ordinate on $\bA^1$, and $M$ the space of global sections of a $\bG_m$-equivariant  sheaf on $\bA^1$.
Since $M$ is flat, $0 \to M \xra{t} M \to M\ten_{k[t],0}k \to 0$ is exact, so $t$ is an injective endomorphism. The $\bG_m$-action is equivalent to giving  a decomposition $M = \bigoplus M_n$, and  we have $t:M_n \into M_{n+1}$. Thus the images of $\{M_n\}_{n \in \Z}$ give  an exhaustive filtration on $M \ten_{k[t],1}k$.

Conversely, set $M$ to be the Rees module $\xi(V,F):=\bigoplus F_nV$, with $\bG_m$-action given by setting $F_nV$ to be of weight $n$, and the $k[t]$-module structure determined by letting $t$ be the inclusion $F_nV \into F_{n+1}V$. If $I$ is a $k[t]$-ideal, then $I=(f)$, for some $f \in k[t]$, since $k[t]$ is a principal ideal domain. The map $M\ten I \to M$ is thus isomorphic to $f:M\to M$. Writing $f= \sum a_n t^n$, we see that it is injective on $M=\bigoplus F_nV$. Thus $M\ten I \to M$ is injective, so  $M$ is flat by \cite[Theorem 7.7]{Mat}.
\end{proof}

Note that a  quasi-coherent $\bG_m$-equivariant sheaf on $\bA^1$ is the same as a quasi-coherent sheaf on the stack $[\bA^1/\bG_m]$.

\begin{remark}\label{nonflatfil}
We might also ask what happens if we relax the condition that the  filtration be flat, since non-flat structures might sometimes arise as quotients. 

An arbitrary quasi-coherent $\bG_m$-equivariant sheaf $M$ on $\bA^1$ with $M \ten_{k[t],1}k=V$   corresponds to  a system  $M_r$ of vector spaces with (not necessarily injective) linear maps $t:M_r \to M_{r+1}$, such that $\varinjlim_{r \to \infty} W_r \cong V$.
\end{remark}

\begin{definition}\label{Cdef}
Define $C$ to be the real affine scheme $\prod_{\Cx/\R}\bA^1$ obtained from $\bA^1_{\Cx}$ by restriction of scalars, so for any real commutative algebra $A$, $C(A)= \bA^1_{\Cx}(A\ten_{\R}\Cx)\cong A\ten_{\R}\Cx$. Choosing $i \in \Cx$ gives an isomorphism $C \cong \bA^2_{\R}$, and we let $C^*$ be the quasi-affine scheme $C - \{0\}$.

Define $S$ to be the Deligne torus. This is the real algebraic group  $S=\prod_{\Cx/\R} \bG_m$ obtained as in \cite[2.1.2]{Hodge2}  from $\bG_{m,\Cx}$ by restriction of scalars. Explicitly, for any real commutative  algebra $A$, $S(A)= (A\ten_{\R}\Cx)^{\by}$.

There is a canonical inclusion $\bG_m \into S$ given by $A^{\by} \into (A\ten_{\R}\Cx)^{\by}$, and that   $S$ acts on $C$ and $C^*$
by inverse multiplication, i.e.
\begin{eqnarray*}
S(A) \by C(A) &\to& C(A)\\
(\lambda, c) &\mapsto& (\lambda^{-1}c).
\end{eqnarray*}
 \end{definition}

\begin{remarks}\label{Sactionrmks}
By \cite[Definition 2.1.4]{Hodge2}, a real Hodge structure is a finite-dimensional real vector space $V$ equipped with an algebraic action of $S$.  This is equivalent to giving a bigraded decomposition $V\ten_{\R}\Cx = \bigoplus_{pq} V^{pq}$ with $\overline{V^{pq}}= V^{qp}$. For $\lambda \in S(\R) \cong \Cx^{\by}$, the action on $V^{pq}$ is multiplication by $\lambda^p \bar{\lambda}^q$. Note that the equivalence extends to infinite-dimensional $S$-representations.

A more standard $S$-action on $C$ is given by the inclusion $S \into  C $ coming from restriction of scalars applied to $\bG_{m,\Cx} \into \bA^1_{\Cx}$. However, we wish  $C$ to be of weight $-1$ rather than $+1$.
\end{remarks}

\begin{remark}\label{Ccoords}
Fix a choice of $i \in \Cx$, and thus define an isomorphism $(u,v) \co C \to \bA^2_{\R}$ given by $u+iv \mapsto (u,v)$ for $u+iv \in A\ten_{\R}\Cx$. Thus the algebra $O(C)$ of functions on  $C$ is the polynomial ring $\R[u,v]$.  Note that $(u,v)$ induces an isomorphism $(u,v) \co C^* \to \bA^2_{\R}-\{0\}$.
Similarly, define an isomorphism $(x,y) \co S \to \bA^2_{\R}-\{(x,y)\,:\, x^2+y^2=0\} $ by $x+iy \mapsto (x,y)$ for $x+iy \in (A\ten_{\R}\Cx)^{\by}$.

On $C_{\Cx}$, we have alternative co-ordinates $w=u+iv$ and $\bar{w}=u-iv$. On $S_{\Cx}$, the alternative co-ordinates $z=x+iy$ and $\bar{z}=x-iy$  give the   isomorphism $(z, \bar{z})\co S_{\Cx}\to \bG_{m,\Cx} \by \bG_{m,\Cx} $ of \cite[\S 2.1]{Hodge2}.
Given  a real Hodge structure $V$, the action of $S$ on $V^{pq}$ is thus given by $z^p\bar{z}^q$. 
 
Note that for the $S$-action of Definition \ref{Cdef}, the composition $S \by C \to C \xra{(w, \bar{w})} \bA^2_{\Cx}$ is given by $(z^{-1}w, \bar{z}^{-1}\bar{w})$. Thus the co-ordinates  $w$ and $\bar{w}$
on $C$  are of types $(-1,0)$ and  $(0,-1)$ respectively. 
\end{remark}

\begin{corollary}\label{flathfil}
The category of flat quasi-coherent $S$-equivariant sheaves on $C^*$  is equivalent to the category of pairs $(V, F)$, where $V$ is  a real vector space and $F$ an exhaustive  decreasing filtration on $V\ten_{\R}\Cx$. 
\end{corollary}
\begin{proof}
Take a  flat quasi-coherent $S$-equivariant sheaf $\sF$ on $C$. Regard $\sF_{\Cx}$ as a sheaf on $\bA^2_{\Cx}-\{0\}$, via the isomorphism $(w, \bar{w}) \co C^* \to \bA^2_{\Cx}-\{0\}$. Writing $\bA^2_{\Cx}-\{0\}= (\bA^1\by \bG_m) \cup (\bG_m \by \bA^1)$, it follows from Lemma \ref{flatfiltrn} that $\sF$ gives two exhaustively filtered complex vector spaces $(V_{\Cx},F)$ and $(V'_{\Cx},F')$, with
\[
 \Gamma( \bA^1\by \bG_m,\sF) \cong (\bigoplus_n F^nV w^{-n})[\bar{w}, \bar{w}^{-1}], \quad  \Gamma( \bG_m \by \bA^1,\sF) \cong  (\bigoplus_n F_n'V' \bar{w}^{-n})[w, w^{-1}].
\]
The gluing data then give an isomorphism $V_{\Cx} \cong V'_{\Cx}$, and the descent datum from $C_{\Cx}^*$ to $C^*$ gives a real vector space $V$ with $V\ten_{\R}\Cx \cong V_{\Cx}$, such that $\bar{F}=F'$.

For the inverse construction,  write $\xi(V\ten \Cx; F, \bar{F})$ for the Rees module
\[
 \bigoplus_{p,q} F^p \bar{F}^qV_{\Cx}w^{-p}\bar{w}^{-q},
\]
regarded as a complex $S$-equivariant $O(C)$-module.
Letting $M\subset \xi(V\ten \Cx; F, \bar{F})$  be the real elements gives a real $S$-equivariant $O(C)$-module, and we define $\xi(V, \bF)$ to be the pullback of $M$ to $C^*$.
\end{proof}

\begin{remark}\label{kapranov}
Although the direct image functor from $C^*$ to $C$ gives an equivalence of the categories of flat quasi-coherent sheaves, we do not follow \cite{kapranovMHS} in working over $C$, since the derived category of $C$ has too few quasi-isomorphisms.
A morphism $U \to V$ of complexes with Hodge filtrations gives rise to a quasi-isomorphism $\xi(U,\bF) \to \xi(V,\bF)$ of sheaves on $C^*$ if and only if the maps $F^pU_{\Cx} \to F^pV_{\Cx}$ are all quasi-isomorphisms. However, to get quasi-isomorphism on $C$ we would need the maps $F^p\bar{F}^qU_{\Cx} \to F^p\bar{F}^qV_{\Cx} $ to be quasi-isomorphisms.

The motivating example comes from the embedding $\cH^* \to A^{\bt}$ of real  harmonic forms into the real de Rham algebra of a compact K\"ahler manifold. This gives a quasi-isomorphism of the associated complexes on $C^*$, since the maps $F^p(\cH^*\ten \Cx) \to F^p(A^{\bt}\ten \Cx)$ are quasi-isomorphisms. However, the associated map on $C$ is not a quasi-isomorphism, as this would force the maps $\cH^{pq} \to A^{pq}$ to be isomorphisms.

However, our approach has the disadvantage that we cannot simply describe the bigraded vector space $\gr_F\gr_{\bar{F}}V$, which would otherwise be given by pulling back along $0 \to C$. 
\end{remark}

\begin{remark}\label{nonflathfil}
We might also ask what happens if we relax the condition that the Hodge filtration be flat.

An arbitrary algebraic  Hodge filtration on a real vector space $V$ is a system  $F^p$ of complex vector spaces with (not necessarily injective) linear maps $s:F^{p} \to F^{p-1}$, such that $\varinjlim_{p \to -\infty} F^{p} \cong V\ten \Cx$.
\end{remark}

\subsubsection{Non-abelian Hodge filtrations}

\begin{definition}\label{hfildef}
Given an affine scheme $X$ over $\R$, we define an algebraic  Hodge filtration $X_{\bF}$ on $X$ to consist of the following data:
\begin{enumerate}
\item an $S$-equivariant  affine morphism $ X_{\bF} \to   C^*$,
\item an isomorphism $X \cong X_{\bF, 1}:= X_{\bF}\by_{C^*, 1}\Spec \R$, where $1 \co \Spec \R \to C$ denotes inclusion of the point $u=1, v=0$ in the co-ordinates of Remark \ref{Ccoords}.
\end{enumerate}
\end{definition}

\begin{definition}\label{splithfil}
A real splitting of the  Hodge filtration $X_{\bF}$ consists of an $S$-action on $X$, and an $S$-equivariant isomorphism
$$
X \by C^* \cong X_{\bF}
$$
 over $ C^* $.
\end{definition}

\begin{remark}\label{KPSrmk}
Note that giving $X_{\bF}$ as above is equivalent to giving the affine morphism $[ X_{\bF}/ S] \to [ C^*/ S]$ of stacks.
This fits in with the idea in \cite{KPS} that if $\OBJ$ is  an $\infty$-stack parametrising some $\infty$-groupoid of objects, then the groupoid of  non-abelian filtrations of this object is  $\cHom([\bA^1/\bG_m], \OBJ)$, since by analogy the groupoid of non-abelian Hodge filtrations would then be  $\cHom([C^*/S], \OBJ)$, replacing Lemma \ref{flatfiltrn} with Corollary \ref{flathfil}.
\end{remark}

\begin{definition}\label{tildeC} 
Let $\widetilde{C^*}\to C^*$ be the \'etale covering of $C^*$ given by cutting out the divisor $\{\bar{w}=0\}$ from $C^*\ten_{\R}\Cx$, for co-ordinate $\bar{w}$ as in Remark \ref{Ccoords}. Explicitly, 
\[
 \widetilde{C^*}= \Spec \Cx[u,v, (u-iv)^{-1}],
\]
with the morphism $\widetilde{C^*} \to C$ given on functions by the inclusion $\R[u,v]\to \Cx[u,v, (u-iv)^{-1}]$.
\end{definition}

\begin{lemma}\label{tildemods}
There is an equivalence of categories between flat $S$-equivariant quasi-coherent sheaves on $\widetilde{C^*}$, and exhaustive   filtrations on complex vector spaces.
\end{lemma}
\begin{proof}
The isomorphism $(w, \bar{w}) \co C_{\Cx} \to \bA^2_{\Cx}$ restricts to an isomorphism $\widetilde{C^*} \cong \bA^1_{\Cx} \by \bG_{m, \Cx}$. The $S$-action on $X$ from Definition \ref{Cdef} induces an action on $\widetilde{C^*}$, and for  the isomorphism $(z,\bar{z}) \co S_{\Cx} \to \bG_{m, \Cx}\by \bG_{m, \Cx}$ of Remark \ref{Ccoords}, this action is given by $(w, \bar{w}) \mapsto (z^{-1}w,\bar{z}^{-1}\bar{w})$.

 Thus $S$-equivariant quasi-coherent sheaves on $\widetilde{C^*}$ are equivalent to $\bG_{m, \Cx} \by 1$-equivariant quasi-coherent sheaves on the scheme $\bA^1_{\Cx}\subset  \widetilde{C^*}$ given by $\bar{w}=1$. Now apply Lemma \ref{flatfiltrn}.
\end{proof}

\subsubsection{$\SL_2$}\label{slfirst}
 
\begin{definition}
Define maps $\row_1, \row_2 \co \GL_2 \to \bA^2$ by projecting onto the first and second rows, respectively. If we make the identification $C=\bA^2$ of Definition \ref{Cdef}, then these are equivariant with respect to the right $S$-action $\GL_2 \by S \to \GL_2$, given by 
$$(A, x+iy) \mapsto A\left(
\begin{matrix} x & y  \\ -y & x \end{matrix} \right)^{-1},$$
for co-ordinates $x,y$ as in Definition \ref{Ccoords}. 

This follows  because
\[
 \left(\begin{matrix} a & b \end{matrix}\right ) \left( \begin{matrix} x & y  \\ -y & x \end{matrix} \right)= \left(\begin{matrix} ax -by & bx+ay \end{matrix}\right ),
\]
and $(ax -by)+i(bx+ay)= (a+ib)(x+iy)$.
\end{definition}

\begin{definition}\label{rowdef}
Define an $S$-action on  $\SL_2$   by    
$$
(\lambda, A) \mapsto\left(
\begin{matrix}  1 & 0  \\ 0 & \lambda\bar{\lambda}  \end{matrix} \right)A  \left(
\begin{matrix} \Re  \lambda & \Im\lambda  \\ -\Im \lambda & \Re \lambda \end{matrix} \right)^{-1}=
 \left(
\begin{matrix}  \lambda\bar{\lambda} & 0  \\ 0 & 1  \end{matrix} \right)^{-1}A  \left(
\begin{matrix} \Re  \lambda & -\Im\lambda  \\ \Im \lambda & \Re \lambda \end{matrix} \right).
$$ 
Let $\row_1 :\SL_2 \to C^*$ be the  map given by projection onto the first row. Observe that this map is  $S$-equivariant because the formulae for the $S$-actions on $\SL_2$ and on $\GL_2$ only differ in the second row.
\end{definition}

Without further comment, we will use the map $\row_1$ to regard $\SL_2$ as a scheme over $C^*$.

\begin{remark}\label{sltrivia}
Observe that, as an $S$-equivariant scheme over $C^*$, we may decompose $\GL_2$  as $\GL_2=  \left(\begin{smallmatrix} 1 & 0  \\ 0 & \bG_m \end{smallmatrix} \right)\by \SL_2 $, where the $S$-action on $\bG_m$ has $\lambda$ acting as multiplication by $(\lambda\bar{\lambda})^{-1}$. 

We may also write $C^*=[\SL_2/\bG_a]$, where $\bG_a$ acts on $\SL_2$ as left multiplication by $ \left(\begin{smallmatrix} 1 & 0  \\ \bG_a & 1 \end{smallmatrix} \right)$, where the $S$-action on $\bG_a$ has $\lambda$ acting as multiplication by $\lambda\bar{\lambda}$.
\end{remark}

\begin{lemma}\label{sluniv}
The morphism $\row_1:\SL_2 \to C^*$ is weakly final in the category of $S$-equivariant affine schemes over $C^*$.
\end{lemma}
\begin{proof}
We need to show that for any affine scheme $U$ equipped with an $S$-equivariant morphism $f:U \to C^*$, there exists a (not necessarily unique) $S$-equivariant morphism $g:U \to \SL_2$ such that $f= \row_1 \circ g$.

If $U=\Spec A$, then $A$ is an $O(C)=\R[u,v]$-algebra, with the ideal $(u,v)_A$ equalling $A$, so there exist $a,b \in A$ with $ua-vb=1$. Thus the map factors through $\row_1: \SL_2 \to C^*$. Complexifying gives an expression $\alpha w +\beta \bar{w}=1$, for $w, \bar{w}$ as in Remark \ref{Ccoords} and suitable $\alpha, \beta \in A\ten_{\R}\Cx$.

 Now splitting $\alpha, \beta$ into types $\alpha= \sum_{pq} \alpha^{pq}, \beta= \sum_{pq} \beta^{pq}$ as in Remarks \ref{Sactionrmks}, we have $\alpha^{10}w+ \beta^{01}\bar{w}=1$, since $1$ is $S$-invariant. On conjugating and averaging, this gives 
\[
 \half(\alpha^{10}+ \overline{\beta^{01}})w+ \half(\beta^{01}+ \overline{\alpha^{10}})\bar{w}=1.
\]
 Write this as $\alpha'w +\beta'\bar{w}=1$. Finally, note that $y:=\alpha'+\beta', -x:=i\alpha'-i\beta'$ are both real, giving $uy-vx=1$, with $x,y$ having the appropriate $S$-action to regard $A$ as an $O(\SL_2)$-algebra when $\SL_2$ has co-ordinates $\left(\begin{smallmatrix} u & v  \\ x & y \end{smallmatrix} \right) $.
\end{proof}

\begin{remark}
Observe that for our action of $\bG_m\subset S$ (corresponding to left multiplication by diagonal matrices) on $\SL_2$, the stack $[\SL_2/\bG_m]$ is just the affine  scheme $\bP^1\by \bP^1- \Delta(\bP^1)$. Here, $\Delta$ is the diagonal embedding, and the projections to $\bP^1$ correspond to the maps  $\row_1,\row_2: [\SL_2/\bG_m] \to [(\bA^2-\{0\})/\bG_m]$ (noting that for $\row_2$ this means taking the inverse of our usual $\bG_m$-action on $C^*$). Lemma \ref{sluniv} can then be reformulated to say that $\bP^1\by \bP^1- \Delta(\bP^1) $ is weakly final in the category of $S^1$-equivariant  affine schemes over $\bP^1$.
\end{remark}

\begin{lemma}\label{slhodge}
The affine scheme $ \SL_2 \xra{\row_1} C^*$ is a flat algebraic   Hodge filtration, corresponding to the algebra $\cS$ of Definition \ref{cSdef}.
\end{lemma}
\begin{proof}
Since $\row_1$ is flat and equivariant for the inverse right $S$-action, 
we know by Corollary \ref{flathfil} that we have a filtration on the ring of functions of  $\SL_2\by_{\row_1, C^*, 1}\Spec \R$. This scheme  consists of invertible matrices 
$\left(
\begin{smallmatrix} 1 & 0 \\ x & 1\end{smallmatrix} \right),
$
so the ring of functions is just $\R[x]$.

To describe the filtration, we use Lemma \ref{tildemods}, considering the pullback of $\row_1$ along  $\widetilde{C^*}\to C^*$. The scheme $\widetilde{\SL_2}:=\SL_2\by_{\row_1, C}\widetilde{C^*} $ is isomorphic to $\widetilde{C^*}\by \bA^1$, with projection onto $\bA^1_{\Cx}$ given by 
$\left(\begin{smallmatrix} u &  v \\ x & y\end{smallmatrix}\right) \mapsto x-iy$. This isomorphism is moreover $S_{\Cx}$-equivariant over $\widetilde{C^*}$, when we set the co-ordinates of $\bA^1$ to be of type $(1,0)$.

The filtration $F$ on $\cS\ten \Cx$  then just comes from the decomposition on $\Cx[x-iy] $ associated to the action of $\bG_{m,\Cx} \by \{1\} \subset S_{\Cx}$, giving
$$
F^p \Cx[x-iy]= \bigoplus_{p' \ge p}(x-iy)^{p'}\Cx.
$$
The filtration on $\R[x]\ten \Cx$ is then given by evaluating this at $y=1$, giving $F^p (\R[x]\ten \Cx)= (x-i)^p\Cx[x]$, as required.

For an explicit inverse construction, the complex  Rees module $\bigoplus_{p,q \in \Z} w^{-p}\bar{w}^{-q} F^p\bar{F}^q\cS$ associated to $\cS$ is the $\Cx[w,\bar{w}]$-subalgebra of $(\cS \ten \Cx)[w,w^{-1}, \bar{w}, \bar{w}^{-1}]$ generated by $\bar{z}:= w^{-1}(x-i)$ and $z:= \bar{w}^{-1}(x+i)$. These satisfy the sole relation $w\bar{z}-\bar{w}z=-2i$, giving $\left(\begin{smallmatrix} u &  v \\ \xi & \eta\end{smallmatrix}\right) \in \SL_2$, where $z=\xi+i\eta, \, \bar{z}=\xi -i\eta$.
\end{proof}

\begin{remark}
We may now reinterpret Lemma \ref{sluniv} in terms of Hodge filtrations. An $S$-equivariant affine scheme, flat over $C^*$, is equivalent to a real commutative algebra $A$, equipped with an exhaustive decreasing filtration $F$ on $A\ten_{\R}\Cx$, such that $\gr_F\gr_{\bar{F}}(  A\ten_{\R}\Cx)=0$. This last condition is equivalent to saying that $1 \in F^1+\bar{F}^1$, or even that there exists $\alpha \in F^1(A\ten_{\R}\Cx)$ with $\Re \alpha =1$. We then define a homomorphism $f:\cS \to A$ by setting $f(x) = \Im \alpha$, noting that $f(1+ix)= \alpha \in F^1(A\ten_{\R}\Cx)$,  so $f$ respects the Hodge filtration.
\end{remark}

\begin{proposition}\label{cSubiq}
Every (finite-dimensional abelian) MHS $V$ admits an $\cS$-splitting, i.e. an $\cS$-linear isomorphism  
$$
V\ten \cS \cong (\gr^W V)\ten \cS,
$$
of quasi-MHS, inducing the identity on the grading associated to $W$.  The set of such splittings is a torsor for the group $\id + W_{-1}\gamma^0\End( (\gr^W V)\ten \cS)$, for $\gamma$ as in Definition \ref{gammadef}
\end{proposition}
\begin{proof}
We proceed by induction on the weight filtration. $\cS$-linear extensions $0 \to W_{n-1}V\ten \cS \to W_nV \ten \cS \to \gr^W_nV \ten \cS \to 0$ of quasi-MHS are parametrised by 
$$
\Ext^1_{\bA^1\by \SL_2}( \gr^W V \ten O(\bA^1) \ten O(\SL_2), \xi(W_{n-1}V, W,F, \bar{F}) \ten_{O(C)} O(\SL_2) )^{\bG_m \by S},  
$$
since $\bG_m \by S$ is (linearly) reductive. Now, $\gr^W V \ten O(\bA^1) \ten O(\SL_2)$ is a projective $O(\bA^1) \ten O(\SL_2)$-module, so its higher $\Ext$s are all $0$, and all $\cS$-linear quasi-MHS extensions of $\gr^W_nV \ten \cS $ by $W_{n-1}V\ten \cS$ are isomorphic, so $W_nV\ten \cS \cong W_{n-1}V\ten \cS \oplus \gr^W_nV \ten \cS$.

Finally, observe that any two splittings differ by a unique automorphism of $(\gr^W V)\ten \cS$, preserving the  quasi-MHS structure, and inducing the identity on taking $\gr^W$. This group is just $\id + W_{-1}\gamma^0\End( (\gr^W V)\ten \cS)$, as required. 
\end{proof}

We may make use of the covering $\row_1:\SL_2 \to C^*$ to give an explicit description of the derived direct image $\oR j_*\O_{C^*}$ as a commutative DG algebra on $C$, for $j:C^* \to C$,  as follows. 

\begin{definition}\label{Ndef}
The $\bG_a$-action on $\SL_2$ of Remark \ref{sltrivia} gives rise to an action of the associated Lie algebra $\g_a \cong \R$ on $O(\SL_2)$. Explicitly, define the standard generator $N \in \g_a$ to act as the derivation with $Nx=u, Ny=v, Nu=Nv=0$, for co-ordinates  $\left(\begin{smallmatrix} u &  v \\ x & y\end{smallmatrix}\right)$ on $\SL_2$. 

This is equivalent to the $O(\SL_2)$-linear isomorphism $\Omega(\SL_2/C) \to O(\SL_2)$ given by $dx\mapsto u$, $dy \mapsto v$. This is not $S$-equivariant, but has type $(-1,-1)$, so we write $\Omega(\SL_2/C) \cong O(\SL_2)(-1)$. Note that the inverse map  $O(\SL_2)(-1)\to \Omega(\SL_2/C)$ is multiplication by $ydx-xdy$. 
\end{definition}

The DG algebra $O(\SL_2) \xra{N\eps} O(\SL_2)(-1)\eps$, for $\eps$ of degree $1$, is an algebra over $O(C)=\R[u,v]$, so we may consider the DG algebra $j^{-1}O(\SL_2) \xra{N\eps} j^{-1}O(\SL_2)(-1)\eps$ on $C^*$, for $j:C^* \to C$. This is an acyclic resolution of the structure sheaf $\O_{C^*}$, so 
$$
\oR j_*\O_{C^*} \simeq j_*(j^{-1}O(\SL_2) \xra{N\eps} j^{-1}O(\SL_2)(-1)\eps)= (O(\SL_2) \xra{N\eps} O(\SL_2)(-1)\eps),
$$
regarded as an $O(C)$-algebra. This construction is moreover  $S$-equivariant.

\begin{definition}\label{rjc}
From now on, we will  denote the DG algebra $O(\SL_2) \xra{N\eps} O(\SL_2)(-1)\eps$  by $\oR O(C^*)$, thereby making a canonical choice of representative in the equivalence class $\oR j_*\O_{C^*}$.
 \end{definition}

\subsection{Twistor filtrations}

There are some topological invariants which are known not to carry Hodge filtrations and mixed Hodge structures. However, as explained in \cite{MTS}, in these cases we should expect twistor filtrations and mixed twistor structures instead. We now develop these in the generality we need, the essential idea being to replace the Deligne torus $S$ with the natural copy of $\bG_m$ inside it. 

\begin{definition}\label{tfildef}
Given an affine scheme $X$ over $\R$, we define an algebraic (real)  twistor filtration $X_{\bT}$  on $X$ to consist of the following data:
\begin{enumerate}

\item a $\bG_m$-equivariant affine morphism $\bT: X_{\bT} \to  C^*$,

\item an isomorphism $X \cong X_{\bT, 1}:= X_{\bT}\by_{C^*, 1}\Spec \R$.

\end{enumerate}
\end{definition}

\begin{remark}\label{tfilqcohrmk}
 Note that the category of quasi-coherent  sheaves $\sF$ on a stack $\fX$ is contravariantly equivalent to the category of  affine cogroups  over $\fX$, sending $\sF$ to the cogroup $\mathbf{\Spec}(\O_{\fX} \oplus \sF)$, where the multiplication on $\sF$ is zero. Uncoiling the definition above, it then follows that an algebraic twistor filtration on a real vector space $V$ consists of a $\bG_m$-equivariant quasi-coherent sheaf $\sV_{\bT}$ on $C^*$, equipped with an isomorphism $1^*\sV_{\bT} \cong V$.
\end{remark}

\begin{definition}\label{splittfil}
A real splitting of the  twistor filtration $X_{\bT}$ consists of a $\bG_m$-action on $X$, and an $\bG_m$-equivariant isomorphism
$$
X \by C^* \cong X_{\bT}
$$
 over $ C^* $.
\end{definition}

\begin{definition}\label{realMTSdef}
Adapting \cite[\S 1]{MTS}  from complex to real structures, say that a twistor structure on a real vector space $V$  consists of a  vector bundle $\sE$ on $\bP^1_{\R}$, 
with
an isomorphism $V \cong \sE_1$, the fibre of $\sE$ over $1 \in \bP^1$.
\end{definition}

\begin{proposition}\label{flattfil}
The category of finite flat algebraic  twistor filtrations on real vector spaces is equivalent to the category of  twistor structures.
\end{proposition}
\begin{proof}
The flat algebraic twistor filtration is a flat $\bG_m$-equivariant quasi-coherent sheaf $M$ on $C^*$, with $M|_{1}=V$. Taking the quotient by the right $\bG_m$-action, 
$M$ corresponds to a flat quasi-coherent sheaf $M_{\bG_m}$ on $[C^*/\bG_m]$. Now, $[C^*/\bG_m] \cong [(\bA^2-\{0\})/\bG_m] = \bP^1$, so Lemma \ref{flatfiltrn} implies that  $M_{\bG_m}$ corresponds to a flat quasi-coherent sheaf $\sE$ on $\bP^1$. Note that $\sE_1= (M|_{\bG_m})_{\bG_m} \cong M_1 \cong V$, as required.
\end{proof}

\begin{definition}\label{U1def}
Define the real algebraic group $S^1$ to be the circle group, whose $A$-valued points are given by $\{(a,b) \in  A^2\,:\, a^2+b^2=1\}$. Note that $S^1 \into S$, and that $S/\bG_m \cong S^1$. This latter $S$-action gives $S^1$ a split Hodge filtration.
\end{definition}

\begin{lemma}\label{tfilen}
There is an equivalence of categories between algebraic twistor filtrations $X_{\bT}$ on a real  affine scheme $X$, and morphisms $\breve{X} \to S^1$ with $X=\breve{X}\by_{S^1, 1} \Spec \R$, equipped with  algebraic Hodge filtrations $\breve{X}_{\bF}$ compatible with the standard Hodge filtration on $S^1$.
\end{lemma}
\begin{proof}
Given an algebraic Hodge filtration $\breve{X}_{\bF}$ over $ S^1 \by C^*$, take 
$$
X_{\bT}:=\breve{X}_{\bF}\by_{S^1, 1}\Spec \R,
$$ 
and observe that this satisfies the axioms of an algebraic twistor filtration. Conversely, given an algebraic twistor filtration  $X_{\bT}$ (over $C^*$),  set 
$$
\breve{X}_{\bF}=(X_{\bT}\by S^1)/(-1,-1),
$$ 
with projection $\pi(x,t)= (\pr(x)t^{-1}, t^2) \in  C^*\by S^1$.
\end{proof}

\begin{corollary}\label{twistormeans}
A flat  algebraic twistor filtration $\sV_{\bT}$ on a real vector space $V$ is equivalent to the data of a flat $O(S^1)$-module $\tilde{V}^{{S^1}}$ with $\tilde{V}^{{S^1}}\ten_{O(S^1)}\R=V$, together with 
an exhaustive  decreasing filtration $F$ on $(\tilde{V}^{{S^1}})\ten \Cx$, with the  morphism $O(S^1)\ten_{\R}\tilde{V}^{{S^1}} \to \tilde{V}^{{S^1}}$ respecting the filtrations (for the standard Hodge filtration on $O(S^1)\ten \Cx$). In particular, the filtration is given by  $F^p(\tilde{V}^{S^1}\ten \Cx)= (a+ib)^pF^0(\tilde{V}^{S^1}\ten \Cx)$.
\end{corollary}
\begin{proof}
Passing between quasi-coherent sheaves and affine cogroups as in Remark \ref{tfilqcohrmk}, we set $X_{\bT}:= \oSpec_{C^*}(\O_{C^*} \oplus\sV_{\bT}) $
and $\breve{X}:= \Spec (O(S^1) \oplus \tilde{V}^{{S^1}})$. Then combining Lemmas \ref{flathfil} and \ref{tfilen},  the correspondence  is given explicitly by 
\[
 \sV_{\bT}=  \xi(\tilde{V}^{{S^1}} , \bF)\ten_{O(S^1)}\R, \quad  \xi(\tilde{V}^{{S^1}} , \bF)= (\sV_{\bT}\ten O(S^1))^{(-1,-1)}.
\]
\end{proof}

\begin{definition}\label{grt}
Given a flat algebraic twistor filtration on a real vector space $V$ as above, define $\gr_{\bF}\tilde{V}^{S^1}$ to be the real part of $\gr_F\gr_{\bar{F}}(\tilde{V}^{{S^1}}\ten \Cx)$. Note that this is an $O(S^1)$-module,  and define $\gr_{\bT}V:= (\gr_{\bF}\tilde{V}^{{S^1}})\ten_{O(S^1)}\R$.
\end{definition}

These results have the following trivial converse.
\begin{lemma}\label{tfilenrich}
An algebraic Hodge filtration $X_{\bF}\to C^*$ on $X$ is equivalent to an algebraic twistor filtration $\bT:X_{\bT}\to C^*$ on $X$, together with a $S^1$-action on $X_{\bT}$ with
the properties that
\begin{enumerate}
 \item the $S^1$-action and $\bG_m$-actions  on $X_{\bT}$ commute,
\item 
$\bT$ is $S^1$-equivariant, and
\item  $-1 \in S^1$ acts as $-1 \in \bG_m$.
\end{enumerate}
\end{lemma}
\begin{proof}
The subgroups $S^1$ and $\bG_m$ of $S$ satisfy $(\bG_m \by S^1)/(-1,-1) \cong S$.
\end{proof}

\subsection{Mixed Hodge structures}\label{nmhs} 

We now define algebraic mixed Hodge structures on real affine schemes. 

\begin{definition}\label{mhsdef}
Given an affine scheme $X$ over $\R$, we define an algebraic mixed Hodge structure  $X_{\MHS}$ on $X$ to consist of the following data:
\begin{enumerate}
\item a $\bG_m \by S$-equivariant  affine morphism $ X_{\MHS} \to   \bA^1 \by C^*$,
\item a real affine scheme $\ugr X_{\MHS}$ equipped with an $S$-action,
\item an isomorphism $X \cong  X_{\MHS}\by_{(\bA^1\by C^*), (1,1)}\Spec \R$,
\item a $\bG_m \by S$-equivariant  isomorphism $\ugr X_{\MHS} \by C^* \cong  X_{\MHS}\by_{\bA^1, 0}\Spec \R$, where  $\bG_m$ acts on $\ugr X_{\MHS}$ via the inclusion $\bG_m \into S$. This is called the opposedness isomorphism.
\end{enumerate}
\end{definition}

\begin{definition}
Given an algebraic mixed Hodge structure $X_{\MHS}$ on $X$, define $\gr^WX_{\MHS}:= X_{\MHS}\by_{\bA^1, 0}\Spec \R$, noting that this is isomorphic to $\ugr X_{\MHS} \by C^*$. We also define $X_{\bF}:= X_{\MHS}\by_{\bA^1, 1}\Spec \R$, noting that this is a Hodge filtration on $X$. 
\end{definition}

\begin{definition}\label{splitmhs}
A real splitting of the mixed  Hodge structure $X_{\MHS}$ is a $\bG_m \by S$-equivariant isomorphism
$$
\bA^1 \by \ugr X_{\MHS} \by C^* \cong X_{\MHS},
$$
giving  the opposedness isomorphism on pulling back along $\{0\} \to \bA^1$.
\end{definition}

\begin{remarks}\label{cfkps}
\begin{enumerate}
\item
Note that giving $X_{\MHS}$ as above is equivalent to giving the affine morphisms $[ X_{\MHS}/ \bG_m \by S] \to [\bA^1/\bG_m] \by [ C^*/ S]$ and $\ugr X_{\MHS} \to BS$ of stacks, satisfying an opposedness condition. 

\item To compare this with the non-abelian mixed Hodge structures postulated in \cite{KPS}, note that pulling back along the morphism $\widetilde{C^*}\to C^*$ gives an object over $[\bA^1/\bG_m]\by [\widetilde{C^*}/S_{\Cx}]\cong [\bA^1/\bG_m]\by [\bA^1/\bG_m]_{\Cx}$; this is essentially the stack $X_{\dR}$ of \cite{KPS}. The stack $X_{B,\R}$ of \cite{KPS} corresponds to pulling back along $1: \Spec \R \to C^*$. Thus our algebraic mixed Hodge structures give rise to pre-non-abelian mixed Hodge structures (pre-NAMHS) in the sense of \cite{KPS}. Our treatment of the opposedness condition is also similar to the linearisation condition for a pre-NAMHS, since both introduce additional data corresponding to the associated graded object. 
\end{enumerate}
\end{remarks}

 As for Hodge filtrations, this gives us a notion of an algebraic  mixed Hodge structure on a real vector space. We now show how this is equivalent to the standard definition.

\begin{proposition}\label{flatmhs}
The category of flat $\bG_m \by S$-equivariant quasi-coherent sheaves $M$ on $\bA^1 \by C^*$ is equivalent to the category of quasi-MHS.

Under this equivalence,  bounded below ind-MHS $(V,  W, F)$ correspond to flat algebraic mixed Hodge structures $M$ on  $V$
 whose weights with respect to the $\bG_m \by 1$-action are bounded below.

A real splitting of the Hodge filtration is equivalent to giving a (real) Hodge structure on $V$ (i.e. an $S$-action).
\end{proposition}
\begin{proof}
 Adapting Corollary \ref{flathfil}, we see that a flat $\bG_m \by S$-equivariant module $M$ on $\bA^1 \by C^*$   corresponds to giving exhaustive  filtrations $W$ on $V=M|_{(1,1)}$ and $F$ on $V \ten \Cx$, i.e. a quasi-MHS on $V$. Write $\xi(V,\MHS)$ for the $\bG_m \by S$-equivariant  quasi-coherent sheaf on $\bA^1 \by C^*$ associated to a quasi-MHS $(V,W,F)$.

A flat algebraic mixed Hodge structure is a flat $\bG_m \by S$-equivariant module $M$ on $\bA^1 \by C^*$, with $M|_{(1,1)}=V$, together with a $\bG_m \by S$-equivariant splitting of the algebraic Hodge filtration  $M|_{\{0\}\by C^*}$. Under the equivalence above, this gives a quasi-MHS $(V,  W, F)$, with $W$ bounded below, satisfying the 
split opposedness condition 
$$
(\gr^W_nV)\ten \Cx= \bigoplus_{p+q=n} F^p(\gr^W_nV\ten \Cx) \cap \bar{F}^q(\gr^W_nV\ten \Cx).
$$

When the weights of $M$ are bounded below, we need to express this as a filtered direct limit of MHS. Since $W$ is exhaustive, it will suffice to prove  that each $W_rV$ is an ind-MHS. Now  $W_{N}V=0$ for some $N\ll 0$, so split opposedness means that $W_{N+1}V$ is a direct sum of pure Hodge structures (i.e. an $S$-representation), hence an ind-MHS. Assume inductively that $W_{r-1}V$ is an ind-MHS, and consider the exact sequence
$$
0 \to W_{r-1}V \to W_rV \to \gr^W_rV \to 0.
$$   
of quasi-MHS. Again, split opposedness shows that $\gr^W_rV$ is an ind-MHS, so we may express it as $\gr^W_rV= \LLim_{\alpha} U_{\alpha}$, with each $U_{\alpha}$ a MHS. Thus $W_rV= \LLim_{\alpha} W_rV\by_{\gr^W_rV}U_{\alpha}$, so we may assume that $\gr^W_rV$ is finite-dimensional (replacing $W_r$ with $ W_rV\by_{\gr^W_rV}U_{\alpha}$). Then quasi-MHS extensions of $ \gr^W_rV$  by $ W_{r-1}V$ are parametrised by
$$
\Ext^1_{\bA^1 \by C^*}(\xi(\gr^W_rV, \MHS), \xi(W_{r-1}V, \MHS))^{\bG_m \by S}.
$$
Express $W_{r-1}V$ as a filtered direct limit $\LLim_{\beta} T_{\beta}$ of MHS, and note that 
\begin{align*}
 &\Ext^1_{\bA^1 \by C^*}(\xi(\gr^W_rV, \MHS), \xi(W_{r-1}V, \MHS))^{\bG_m \by S}\\
= \LLim_{\beta}&\Ext^1_{\bA^1 \by C^*}(\xi(\gr^W_rV, \MHS), \xi(T_{\beta}, \MHS))^{\bG_m \by S},
\end{align*}
since $\xi(\gr^W_rV, \MHS)$ is finite and locally free. Thus the extension $W_rV \to \gr^W_rV$ is a pushout of an extension 
$$
0 \to T_{\beta} \to E \to  \gr^W_rV \to 0
$$
for some $\beta$, so $W_rV$
can be expressed as  the ind-MHS
$
W_rV= \varinjlim_{\beta' >\beta}E\oplus_{T_{\beta}}T_{\beta'}.
$

Conversely, any MHS  $V$ satisfies the split opposedness condition by \cite[Proposition 1.2.5]{Hodge2}, so the same holds for any ind-MHS. Thus every ind-MHS corresponds to a flat algebraic MHS under the equivalence above.

Finally, note that the split opposedness condition determines the data of any real splitting.
\end{proof}

\begin{remark}
Note that  the proof of \cite[Proposition 1.2.5]{Hodge2} does not adapt  to infinite filtrations. For instance, the quasi-MHS $\cS$ of Definition \ref{cSdef} satisfies the opposedness condition, but does not give an ind-MHS. Geometrically, this is because the fibre over $\{0\} \in C$ is empty. Algebraically, it is because the Hodge filtration on the ring $\cS=\gr^W_0\cS$ is not split, but $\gr_F\gr_{\bar{F}}(\cS \ten \Cx)=0$, which is a pure Hodge structure of weight $0$. 
\end{remark}

\subsection{Mixed twistor structures}\label{nmts}

\begin{definition}\label{mtsdef}
Given an affine scheme $X$ over $\R$, we define an algebraic mixed twistor structure  $X_{\MTS}$ on $X$ to consist of the following data:
\begin{enumerate}
\item a $\bG_m \by \bG_m$-equivariant  affine morphism $ X_{\MTS} \to   \bA^1 \by C^*$,
\item a real affine scheme $\ugr X_{\MTS}$ equipped with a $\bG_m$-action,
\item an isomorphism $X \cong  X_{\MTS}\by_{(\bA^1\by C^*), (1,1)}\Spec \R$,
\item a $\bG_m \by \bG_m$-equivariant  isomorphism $\ugr X_{\MTS} \by C^* \cong  X_{\MTS}\by_{\bA^1, 0}\Spec \R$.
This is called the opposedness isomorphism.
\end{enumerate}
\end{definition}

\begin{definition}
Given an algebraic mixed twistor structure $X_{\MTS}$ on $X$, define $\gr^WX_{\MTS}:= X_{\MTS}\by_{\bA^1, 0}\Spec \R$, noting that this is isomorphic to $\ugr X_{\MTS} \by C^*$. We also define $X_{\bT}:= X_{\MTS}\by_{\bA^1, 1}\Spec \R$, noting that this is a twistor filtration on $X$. 
\end{definition}

\begin{definition}\label{splitmts}
A real splitting of the mixed  twistor structure $X_{\MTS}$ is a $\bG_m \by \bG_m$-equivariant isomorphism
$$
\bA^1 \by \ugr X_{\MTS} \by C^* \cong X_{\MTS},
$$
giving  the opposedness isomorphism on pulling back along $\{0\} \to \bA^1$.
\end{definition}

\begin{remark}
Note that giving $X_{\MTS}$ as above is equivalent to giving the affine morphism $[ X_{\MTS}/ \bG_m \by \bG_m] \to [\bA^1/\bG_m] \by [ C^*/ \bG_m]$ of stacks, satisfying a split opposedness condition. 
\end{remark}

\begin{definition}\label{abmtsdef}
Adapting \cite[\S 1]{MTS}  from complex to real structures, say that a (real) mixed twistor structure (real MTS) on a real vector space $V$  consists of a vector bundle $\sE$ on $\bP^1_{\R}$, equipped with an exhaustive Hausdorff increasing filtration by   sub-bundles $W_i\sE$, such that for all $i$ the graded bundle $\gr^W_i\sE$ is semistable of slope $i$ (i.e. a direct sum of copies of $\O_{\bP^1}(i)$). We also require an isomorphism $V \cong \sE_1$, the fibre of $\sE$ over $1 \in \bP^1$.

Define  a quasi-MTS on $V$ to be a flat  quasi-coherent sheaf $\sE$ on $\bP^1_{\R}$, equipped with an exhaustive increasing filtration by quasi-coherent subsheaves  $W_i\sE$, together with an isomorphism $V \cong \sE_1$.  Define an ind-MTS to be a filtered direct limit of real MTS, and say that an ind-MTS $\sE$ on $V$ is bounded below if $W_N\sE=0$ for $N \ll 0$.
\end{definition}

Applying Corollary \ref{twistormeans} gives the following result.
\begin{lemma}\label{mtsdef2}
A flat algebraic   mixed twistor structure on a real vector space $V$ is equivalent to giving an $O(S^1)$-module $V'$, equipped with a  mixed Hodge structure (compatible with the weight $0$ real Hodge structure on $O(S^1)$), together with an isomorphism $V'\ten_{O(S^1)}\R\cong V$.
\end{lemma}

\begin{proposition}\label{flatmts}
The category of flat $\bG_m \by \bG_m$-equivariant quasi-coherent sheaves on $\bA^1 \by C^*$ is equivalent to the category of quasi-MTS.

Under this equivalence,  bounded below ind-MTS on $V$ correspond to flat algebraic mixed twistor structures $\xi(V, \MTS)$ on  $V$
 whose weights with respect to the $\bG_m \by 1$-action are bounded below.
\end{proposition}
\begin{proof}
The first statement follows by combining Lemma \ref{flattfil} with Lemma \ref{flatfiltrn}.

Now, given a flat algebraic mixed twistor structure $\xi(V, \MTS)$ on  $V$
 whose weights with respect to the $\bG_m \by 1$-action are bounded below, the proof of Proposition \ref{flatmhs} adapts (replacing $S$ with $\bG_m$) to show that $\xi(V, \MTS)$ is a filtered direct limit of finite flat algebraic mixed twistor structures.
 It therefore suffices to show that finite flat algebraic mixed twistor structures correspond to MTS.

A finite flat algebraic mixed twistor structure is a  finite locally free $\bG_m \by \bG_m$-equivariant module $M$ on $\bA^1 \by C^*$, with $M|_{(1,1)}=V$, together with a $\bG_m \by \bG_m$-equivariant splitting of the algebraic twistor filtration  $M|_{\{0\}\by C^*}$. 
Taking the quotient by the right $\bG_m$-action, 
$M$ corresponds to a  finite locally free $\bG_m$-equivariant module $M_{\bG_m}$ on $\bA^1 \by [C^*/\bG_m]$. Note that $[C^*/\bG_m] \cong [(\bA^2-\{0\})/\bG_m] = \bP^1$, so Lemma \ref{flatfiltrn} implies that  $M_{\bG_m}$ corresponds to a  finite locally free  module on $\sE$ on $\bP^1$, equipped with a finite  filtration $W$. 

Now, $\ugr X_{\MTS}$ corresponds to a  $\bG_m$-representation $V$, or equivalently a  
graded vector space $V=\bigoplus V^n$. If $\pi$ denotes the projection $\pi: C^* \to \bP^1$, then the opposedness isomorphism is equivalent to  a $\bG_m$-equivariant isomorphism
$$
\gr^W\sE \cong V\ten^{\bG_m}(\pi_*\O_{C^*})=  \bigoplus_n V^n \ten_{\R} \O_{\bP^1}(n),
$$
so $\gr^W_n\sE \cong   V^n \ten_{\R} \O_{\bP^1}(n)$, as required.
\end{proof}

\begin{remark}\label{mhstomts}
Note that every MHS $(V,W,F)$ has an underlying MTS $\sE$ on $V$, given by forming the $S$-equivariant  Rees module $\xi(V, \bF)$ on $C^*$ as in Corollary \ref{flathfil},  and setting $\sE$ to be the quotient $\xi(V, \bF)_{\bG_m}$ by the action of $\bG_m \subset S$. 

Beware that if $\sE$ is the MTS underlying $V$, then the sheaf $\sE\ten_{\O_{\bP^1}}\O_{\bP^1}(-2n)$ is the MTS underlying the MHS $V(n)$. 
\end{remark}

\subsection{Real homotopy types revisited}\label{real}

We now illustrate how non-abelian mixed Hodge structures and $\SL_2$-splittings arise for real homotopy types.   Fix a compact connected  K\"ahler manifold $X$ as in \S \ref{real1}.

\subsubsection{The mixed Hodge structure}

\begin{definition}\label{Atildedef}
Define the CDGA $\tilde{A}^{\bt}(X)$ on $C$ by
$$
\tilde{A}^{\bt}(X)= (A^*(X) \ten_{\R} O(C), ud + v\dc),
$$
for co-ordinates $u,v$ as in Remark \ref{Ccoords}. We denote the differential by $\tilde{d}:=ud + v\dc$.
Note that $\tilde{d}$ is indeed flat:
$$
\tilde{d}^2= u^2d^2+ uv(d\dc+\dc d)+v^2(\dc)^2=0.
$$
\end{definition}

\begin{definition}\label{dmd}
There is an action of $S$ on $A^*(X)$, which we will denote by $a \mapsto \lambda \dmd a$, for $\lambda \in \Cx^{\by} = S(\R)$. For $a \in (A^*(X)\ten \Cx)^{pq}$, it is given by
$$
\lambda \dmd a := \lambda^p\bar{\lambda}^qa.
$$
\end{definition}

\begin{lemma}
There is a natural algebraic $S$-action on $\tilde{A}^{\bt}(X)$ over $C$.
\end{lemma}
\begin{proof}
For $\lambda \in S(\R)= \Cx^{\by}$, this action is given on $A^*(X)$ by $a \mapsto \lambda \dmd a$, extending to $\tilde{A}^{\bt}(X)$ by tensoring with the action on $C$ from Definition \ref{Cdef}. We need to verify that this action respects the differential $\tilde{d}$.

 Taking the co-ordinates $(u,v)$ on $C$ from Remark \ref{Ccoords}, we will consider the co-ordinates $w=u+iv, \bar{w}=u-iv$ on $C_{\Cx}$. Now, we may decompose $d$ and $\dc$ into types (over $\Cx$) as $d=\pd +\bar{\pd}$ and $\dc= i\pd -i\bar{\pd}$. Thus $\tilde{d}= w\pd +\bar{w}\bar{\pd}$, so
$$
\tilde{d}: (A^*(X)\ten \Cx)^{p,q}\to w (A^*(X)\ten \Cx)^{p+1,q}\oplus \bar{w} (A^*(X)\ten \Cx)^{p,q+1}.
$$
Because $w$ is of type $(-1,0)$, this map is equivariant under the $S$-action given, with $\lambda$ acting as multiplication by $\lambda^{p}\bar{\lambda}^{q}$ on both sides. 
\end{proof}

\begin{lemma}\label{hodgenice}
The $S$-equivariant $C^*$-bundle $j^*\tilde{A}^{\bt}(X)$ corresponds under Corollary \ref{flathfil} to the  Hodge filtration on $A^{\bt}(X,\Cx)$.
\end{lemma}
\begin{proof}
 We just need to verify that $\tilde{A}^{\bt}(X)\ten \Cx$ is isomorphic to the Rees algebra $\xi(A^{\bt}(X), F, \bar{F})$  (for $F$ the Hodge filtration), with the same complex conjugation.

Now,
$$
\xi(A^{\bt}(X), F, \bar{F})= \bigoplus_{pq} F^p\cap \bar{F}^q,
$$
with $\lambda \in S(\R)\cong \Cx^{\by}$ acting as $\lambda^{p}\bar{\lambda}^{q}$ on $F^p\cap \bar{F}^q$, and inclusion $F^p \to F^{p-1}$ corresponding to multiplication by $w=u+iv$. We therefore define an $O(C)$-linear map $f:\tilde{A}^*(X)\to \xi(A^{\bt}(X), F, \bar{F})$ by mapping $(A(X)\ten \Cx)^{pq}$ to $F^p\cap \bar{F}^q$. It only remains to check that this respects the differentials.

For $a \in (A(X)\ten \Cx)^{pq}$,
$$
f(\tilde{d} a)= f(w\pd a +\bar{w}\bar{\pd}a)= w(\pd a) + \bar{w}(\bar{\pd}a) \in w(F^{p+1}\cap \bar{F}^q)+ \bar{w}(F^{p}\cap \bar{F}^{q+1}).
$$
But $w(F^{p+1}\cap \bar{F}^q)=F^{p}\cap \bar{F}^q = \bar{w}(F^{p}\cap \bar{F}^{q+1})$, so this is just 
$\pd a +\bar{\pd}a = da$  in $F^p\cap \bar{F}^q$, which is just $df(a)$, as required.
\end{proof}

Combining this with the weight filtration means that  the bundle  $\xi(A^{\bt}(X),\MHS  )$ on $[\bA^1/\bG_m]\by [C^*/S]$  associated to the quasi-MHS $A^{\bt}(X)$ is just  the Rees algebra $\xi(j^*\tilde{A}^{\bt}(X),W)$, regarded as a $\bG_m \by S$-equivariant $\bA^1 \by C^*$-bundle.

\subsubsection{The family of formality quasi-isomorphisms}

We will now give a more conceptual reformulation of Theorem \ref{morgansplitsa}.

\begin{definition}
Define a differential $\tilde{d}^c$ on $A^*(X)\ten O(\SL_2)$ by $\tilde{d}^c:= xd+y\dc$.
\end{definition}

The principle of two types now gives us a family of quasi-isomorphisms:
\begin{theorem}\label{formalitysl}
We have the following $S$-equivariant quasi-isomorphisms of CDGAs over $\SL_2$, with notation from Definition \ref{rowdef}:
$$
\row_1^*j^*\tilde{A}^{\bt}(X) \xla{i} \ker(\tilde{d}^c) \xra{p}  \row_2^* \H^*(j^*\tilde{A}^{\bt}(X)) \cong \H^*(A^{\bt}(X))\ten_{\R} O(\SL_2),
$$
where $\ker(\tilde{d}^c)$ means $\ker(\tilde{d}^c)\cap \row_1^*j^*\tilde{A}^{\bt}(X)$, with differential $\tilde{d}$.
\end{theorem}
\begin{proof}
Because $\left(\begin{smallmatrix}  u & v \\ x &y  \end{smallmatrix}\right) \in \GL_2(O(\SL_2))$, the operators satisfy the principle of two types by Proposition \ref{gl2prop}. Thus $i$ is a quasi-isomorphism, and  we may define $p$ as projection onto $\H^*_{\tilde{d}^c}(A^*(X)\ten O(\SL_2) )$, on which the differential $\tilde{d}$ is $0$. 
The final isomorphism now follows from the description
$\H^*(A^{\bt}(X))\cong \frac{\ker d \cap \ker \dc}{\im d\dc}$, which clearly maps to $\H^*(j^*\tilde{A}^{\bt}(X))$, the principle of two types showing it to be an isomorphism.
\end{proof}

Since the weight filtration is just defined in terms of good truncation, this also implies that
$$
\xi(\row_1^*j^*\tilde{A}^{\bt}(X), W) \simeq \xi(\H^*(X, \R), W)\ten \O_{\SL_2}
$$ 
as $\bG_m \by S$-equivariant dg algebras over $\bA^1\by \SL_2$.

\begin{remark}
Under the equivalence of  Lemma \ref{slhodge}, $\Spec \cS$ corresponds to  $\left(\begin{smallmatrix} 1 & 0 \\ \bA^1 & 1\end{smallmatrix} \right) \subset \SL_2,$ equipped with a Hodge filtration. Thus
Theorem \ref{formalitysl} is equivalent to Theorem \ref{formalityS}.
\end{remark}

\section{Structures on cohomology}\label{cohosn}

As we saw in Remark \ref{kapranov}, our definitions are chosen so that they give the desired quasi-isomorphisms. We now investigate the various cohomology groups associated to our structures, and relate them to existing constructions. These will include Beilinson's weak and absolute Hodge cohomologies, real Deligne cohomology and Consani's Archimedean cohomology.

\subsection{Cohomology of Hodge filtrations}\label{Bei1} 

Given a complex $\sF^{\bt}$  of  algebraic Hodge filtrations, we now show how to calculate hypercohomology $\bH^*([C^*/S], \sF^{\bt})$, and compare this with Beilinson's weak Hodge cohomology.

For the \'etale cover  $\widetilde{C^*} \to C^*$ Definition \ref{tildeC}, we have an \'etale pushout     $ C^*= \widetilde{C^*}\cup_{S_{\Cx}}S$ of affine schemes. Thus $\oR \Gamma(C^*, \sF^{\bt})$  is the cone of the morphism
$$
  \oR \Gamma(\widetilde{C^*}, \sF^{\bt}) \oplus  \oR \Gamma(S, \sF^{\bt})  \to \oR \Gamma(S_{\Cx}, \sF^{\bt}).
$$

If $\sF^{\bt}$ is a flat complex, it corresponds under Corollary \ref{flathfil} to a complex $V^{\bt}$ of real vector spaces, equipped with an exhaustive filtration $F$ of $V^{\bt}_{\Cx}:= V^{\bt}\ten \Cx$. The expression above then becomes
$$
(\bigoplus_{n \in \Z}  F^{n}(V^{\bt}_{\Cx})w^{-n})[\bar{w}, \bar{w}^{-1}] \oplus    V^{\bt}_{\R}[u,v, (u^2+v^2)^{-1}] \to V^{\bt}_{\Cx}[w,w^{-1}, \bar{w}, \bar{w}^{-1}],
$$
for co-ordinates $u,v$ and $w, \bar{w}$ on $C^*$ as in Remark \ref{Ccoords}.

Since $S$ is a reductive group, taking $S$-invariants is an exact functor, so $\oR \Gamma([C^*/S], \sF^{\bt})$  is the cone of the morphism
$$
\oR \Gamma(\widetilde{C^*}, \sF^{\bt})^S \oplus  \oR \Gamma(S, \sF^{\bt})^S  \to \oR \Gamma(S_{\Cx}, \sF^{\bt})^S
$$
which is just
$$
F^0(V^{\bt}_{\Cx}) \oplus V^{\bt}_{\R} \to V^{\bt}_{\Cx},
$$
which is just the functor $\oR \Gamma_{\cH w}$ from \cite{beilinson}.

Therefore
$$
\oR \Gamma([C^*/S], \sF^{\bt}) \simeq \oR \Gamma_{\cH w} (V^{\bt}),
$$

Likewise, if $\sE^{\bt}$ is another such complex, coming from a complex $U^{\bt}$ of real vector spaces with complex filtrations,  then
$$
\oR \Hom_{[C^*/S]}(\sE^{\bt}, \sF^{\bt}) \simeq \oR\Hom_{\cH w}(U^{\bt},V^{\bt}).
$$

\begin{remark}\label{BeiS}
For $\cS$ as in Lemma \ref{slhodge}, and a complex $V^{\bt}$ of $\cS$-modules, with compatible filtration $F$ on $V^{\bt}\ten \Cx$, let $\sF^{\bt}$ be the associated bundle on $[C^*/S]$. By Lemma \ref{slhodge}, this is a $\row_{1*}\O_{[\SL_2/S]}$-module, so $\sF^{\bt}= \row_{1*}\sE^{\bt}$, for some quasi-coherent complex $\sE^{\bt}$ on $[\SL_2/S]$, and 
\begin{eqnarray*}
\oR \Gamma([C^*/S], \sF^{\bt}) &=& \oR \Gamma([C^*/S], \row_{1*}\sE^{\bt})\\  
&\simeq& \oR \Gamma([\SL_2/S], \sE^{\bt})\\
&=& \Gamma([\SL_2/S], \sE^{\bt})\\
&=& \Gamma([C^*/S, \sF^{\bt}),
\end{eqnarray*}
since $\SL_2$ and $\row_1$ are both affine. 

In other words,
$$
\oR \Gamma_{\cH w} (V^{\bt}) \simeq \gamma^0V^{\bt},
$$
for $\gamma$ as in Definition \ref{gammadef},
which is equivalent to saying that $V \oplus F^0(V\ten \Cx) \to V\ten\Cx$ is necessarily surjective for all $\cS$-modules $V$.
\end{remark}


\subsection{Cohomology of MHS}\label{Bei2}

Given a complex $\sF^{\bt}$  of  algebraic MHS, we now show how to calculate hypercohomology $\bH^*([C^*/S]\by [\bA^1/\bG_m], \sF^{\bt})$, and compare this with Beilinson's absolute Hodge cohomology. By Proposition \ref{flatmhs}, $\sF^{\bt}$ gives rise  to a complex $V^{\bt}$ of quasi-MHS. 

Since $\bA^1$ is affine and $\bG_m$ reductive,  $\oR \pr_*=\pr_*$ for the  projection $\pr: [C^*/S]\by [\bA^1/\bG_m] \to [C^*/S]$. Thus
$$
\oR \Gamma([C^*/S]\by [\bA^1/\bG_m], \sF^{\bt}) \simeq \oR\Gamma([C^*/S], \pr_* \sF^{\bt}),
$$
and $\pr_*\sF^{\bt}$ just corresponds under Corollary \ref{flathfil} to the complex $W_0 V^{\bt}_{\R}$ with filtration $F$ on $W_0 V^{\bt}_{\Cx}$. 

Hence \S \ref{Bei1} implies that  $\oR \Gamma([C^*/S]\by [\bA^1/\bG_m], \sF^{\bt})$ is just the cone of 
$$
W_0F^0(V^{\bt}_{\Cx}) \oplus W_0V^{\bt}_{\R} \to W_0V^{\bt}_{\Cx},
$$
which is just the absolute Hodge functor $\oR \Gamma_{\cH}$ from \cite{beilinson}.  

Therefore
$$
\oR \Gamma([C^*/S]\by [\bA^1/\bG_m], \sF^{\bt}) \simeq \oR \Gamma_{\cH} (V^{\bt}),
$$

Likewise, if $\sE^{\bt}$ is another such complex, coming from a complex $(U^{\bt},W,F)$,  then
$$
\oR \Hom_{[C^*/S]\by [\bA^1/\bG_m]}(\sE^{\bt}, \sF^{\bt}) \simeq \oR\Hom_{\cH}(U^{\bt},V^{\bt}).
$$

\subsection{Real Deligne cohomology}

Given a compact K\"ahler manifold $X$, the algebraic Hodge filtration on $A^{\bt}(X)$ is given by the complex $j^*\tilde{A}^{\bt}(X)$, where $\tilde{A}^{\bt}(X)$ is the $S$-equivariant complex of sheaves on $C$ from Definition \ref{Atildedef}, and $j \co C^* \to C$ is the inclusion of Definition \ref{Cdef}.
 
We now consider the derived direct image of $j^*\tilde{A}^{\bt}(X)$ under  the morphism $q:[C^*/S] \to [\bA^1/\bG_m]$ given by $u, v \mapsto u^2+v^2$. This is equivalent to $(\oR j_* j^*\tilde{A}^{\bt}(X))^{S^1}$, since $S^1$ is reductive, $\bG_m= S/S^1$ and $\bA^1= C^*/S^1$.

\begin{proposition}\label{cfdeligne}
There are canonical isomorphisms
$$
(\oR^m j_*j^*\tilde{A}^{\bt}(X) )^{S^1} \cong (\bigoplus_{a< 0} \H^m(X, \R)) \oplus  (\bigoplus_{a\ge 0} (2\pi i)^{-a} \H^m_{\cD}(X, \R(a))),
$$
where $a$ is the weight under the action of $S/S^1 \cong \bG_m$, and $\H^m_{\cD}(X, \R(a))$ is real Deligne cohomology.
\end{proposition}
\begin{proof}
The isomorphism $\bG_m= S/S^1$ allows us to regard $O(\bG_m)$ as an $S$-representation, and
$$
(\oR q_* j^*\tilde{A}^{\bt}(X))^{S^1} \simeq \oR \Gamma([C^*/S], j^*\tilde{A}^{\bt}(X)\ten O(\bG_m)).
$$ 
Now, $O(\bG_m)= \R[s,s^{-1}]$, with $s$ of type $(-1,-1)$, so $O(\bG_m)\cong \bigoplus_a (2\pi i)^{-a}\R(a)$, giving (by \S \ref{Bei1})
$$
(\oR q_* j^*\tilde{A}^{\bt}(X))^{S^1} \simeq \bigoplus_a (2\pi i)^{-a} \oR \Gamma_{\cH w} (A^{\bt}(X)(a)),
$$
which is just real Deligne cohomology by \cite{beilinson}.
\end{proof}

We may also compare these cohomology groups with the groups considered in \cite{deninger} and \cite{deninger2} for defining $\Gamma$-factors of smooth projective varieties at Archimedean places.

\begin{proposition}\label{cfdeninger}
The torsion-free quotient of the $\bG_m$-equivariant $\bA^1$-module $(\oR^m j_* j^*\tilde{A}^{\bt}(X))^{S^1}$ is the Rees module of $\H^m(X,\R)$ with respect to the filtration $\gamma$.
\end{proposition}
\begin{proof}
The results of \S \ref{Bei1} give  a long exact sequence
$$
\ldots 
\to (\oR^m j_*j^*\tilde{A}^{\bt}(X))^{S^1} \to \bigoplus_{a \in \Z} (F^a\H^m(X, \Cx)\oplus \H^m(X, \R)) \to \bigoplus_{a \in \Z}\H^m(X, \Cx) \to \ldots ,
$$
and hence
$$
0 \to \bigoplus_{a \in \Z}\frac{\H^{m-1}(X, \Cx)}{F^a\H^{m-1}(X, \Cx)+ \H^{m-1}(X, \R)} \to (\oR^m j_* j^*\tilde{A}^{\bt}(X))^{S^1} \to \bigoplus_{a \in \Z} \gamma^a\H^m(X, \R)\to 0.
$$
Since  multiplication by the standard co-ordinate of $\bA^1$ corresponds to the embedding $F^{a+1} \into F^a$, the left-hand module is torsion, giving the required result.
\end{proof}

\begin{remark}\label{realrk}
In \cite{deninger} and \cite{consani}, $\Gamma$-factors of real varieties were also considered. If we  let $\sigma$ denote the  de Rham conjugation of the associated complex variety, then we may replace $S$ throughout this paper  by $S \rtimes \langle \sigma \rangle$, with $\sigma$ acting on $S(\R)$ by $\lambda \mapsto \bar{\lambda}$, and on $\SL_2$ by $\left(\begin{smallmatrix} u & v  \\ x & y \end{smallmatrix} \right) \mapsto \left(\begin{smallmatrix} u & -v  \\ -x & y \end{smallmatrix} \right) $ (i.e. conjugation by  $\left(\begin{smallmatrix} 1 & 0  \\ 0 & -1 \end{smallmatrix} \right)$, noting that $\sigma(\dc)= -\dc$. In that case, the cohomology group considered in \cite{deninger} is the  torsion-free quotient of $(\oR^m j_* j^*\tilde{A}^{\bt}(X)  )^{S^1\rtimes \langle \sigma \rangle }$.
\end{remark}

\begin{lemma}\label{splitrjc}
There is a canonical $S$-equivariant quasi-isomorphism
$$
\oR j_* j^*\tilde{A}^{\bt}(X) \simeq \H^*(X,\R) \ten_{\R}\oR O(C^*)
$$
of $C$-modules, where $\H^*(X,\R)$ is equipped with its standard $S$-action (the real Hodge structure), and $\oR O(C^*)$ is from Definition \ref{rjc}.
\end{lemma}
\begin{proof}
The natural inclusion $\cH^*\ten O(C) \to \tilde{A}^{\bt}$ of real harmonic forms gives rise to a morphism 
$$
\cH^*\ten \O(C^*) \to j^*\tilde{A}^{\bt}
$$
of $S$-equivariant cochain complexes over $C^*$, which is a quasi-isomorphism by Lemma \ref{gl2lemma}, and hence
$$
\cH^* \ten \oR O(C^*)\simeq \oR j_*j^*\tilde{A}^{\bt},
$$
 as required.
\end{proof}

\begin{corollary}\label{splitdeligne}
As an $S$-representation, the summand of $\bH^n (C^*,j^*\tilde{A}^{\bt})\ten_{\R}\Cx$ of type $(p,q)$ is given by
$$
\bigoplus_{\substack{p' \ge p\\q'\ge q\\ p'+q'=n}} \cH^{p'q'} \oplus \bigoplus_{\substack{p' < p\\q'< q\\ p'+q'=n-1}}\cH^{p'q'}.
$$
In particular, this describes  Deligne cohomology by taking invariants under complex conjugation when $p=q$. 
 \end{corollary}
\begin{proof}
 This follows from Lemma \ref{splitrjc}, since  $\H^*(C, \O_{C^*}) \cong \bigoplus_n \H^*(\bP^1, \O_{\bP^1}(n))$.
\end{proof}

\subsection{Analogies with limit Hodge structures}\label{analogies}

If $\Delta$ is the open unit disc, and $f: X \to \Delta$ a proper surjective morphism of complex K\"ahler manifolds, smooth over the punctured disc $\Delta^*$, then Steenbrink (\cite{steenbrink}) defined a limit mixed Hodge structure at $0$. Take the universal covering space $\widetilde{\Delta^*}$ of $\Delta^*$, and let $\widetilde{X^*}:=X\by_{\Delta}\widetilde{\Delta^*}$. Then the limit Hodge structure is defined as a Hodge structure on 
$$
\lim_{\substack{t \to 0} }\H^*(X_t):= \H^*(\widetilde{X^*})
$$
\cite[(2.19)]{steenbrink}  gives an exact sequence
$$
\ldots \to \H^n(X^*) \to \H^n(\widetilde{X^*}) \xra{N} \H^n(\widetilde{X^*})(-1) \to \ldots, 
$$
where $N$ is the monodromy operator associated to the deck transformation of $ \widetilde{\Delta^*}$. 


Since we are working with quasi-coherent sheaves, connected affine schemes replace contractible topological spaces, and Lemma \ref{sluniv} implies that we may then regard $\SL_2$ as the universal cover of $C^*$, with deck transformations $\bG_a$. We then substitute $C$ for  $\Delta$, $C^*$ for $\Delta^*$  and $\SL_2$ for $\widetilde{\Delta^*}$. We also replace  $\O_{X^*}$ with $j^*\tilde{A}^{\bt}(X)$, so $\O_{\widetilde{X^*}}$ becomes $\row_1^*j^*\tilde{A}^{\bt}(X)$. This suggests that we should think of $\row_1^*j^*\tilde{A}^{\bt}(X)$ (with its natural $S$-action)  as the limit mixed Hodge structure at the Archimedean special fibre. 

The derivation $N$ of Definition \ref{Ndef} then acts as the monodromy transformation. Since $N$ is of type $(-1,-1)$ with respect to the $S$-action, the weight decomposition given by the action of $\bG_m \subset S$ splits the monodromy-weight filtration. The following result allows us to regard $\row_1^*j^*\tilde{A}^{\bt}(X)$ as the limit Hodge structure at the special fibre corresponding to the Archimedean place.

\begin{proposition}\label{steenkey}
The complex 
$
\oR \Gamma(C^*, j^*\tilde{A}^{\bt}(X)) 
$
is naturally isomorphic to the cone complex of the diagram $ \row_1^*j^*\tilde{A}^{\bt}(X) \xra{N} \row_1^*j^*\tilde{A}^{\bt}(X) (-1)$, where $N$ is the locally nilpotent derivation given by differentiating the $\bG_a$-action on $\SL_2$.
\end{proposition}
\begin{proof}
This follows immediately from the description of $\oR O(C^*)$ in \S \ref{slfirst}.
\end{proof}

\subsection{Archimedean cohomology}\label{arch}

As in \S \ref{analogies}, the  $S$-action gives a real (split)  Hodge structure on the cohomology groups $\H^q(\row_1^* j^*\tilde{A}^{\bt}(X))$. In order to avoid confusion with the weight filtration on $j^*\tilde{A}^{\bt}(X)$, we will denote the associated weight filtration by $M_*$.

\begin{corollary}\label{grmcoho}
There are canonical isomorphisms
$$
\gr^M_{q+r}\H^q( \row_1^*j^*\tilde{A}^{\bt}(X))\cong \H^q(X,\R )\ten \gr^M_rO(\SL_2)
$$
\end{corollary}
\begin{proof}
This is an immediate consequence of the splitting in Theorem \ref{formalitysl}.
\end{proof}

\begin{lemma}\label{kercokern}
For the derivation $N$ of Proposition \ref{steenkey}, the kernel and cokernel of the induced map
\[
 N \co \H^q(\row_1^*j^*\tilde{A}^{\bt}(X) \to \H^q(\row_1^*j^*\tilde{A}^{\bt}(X)(-1))
\]
on cohomology are given by
\[
 H^q(X,\R)\ten \R[u,v] \quad \text{ and }\quad  \H^q(X,\R)\ten \R[x,y](-1)
\]
respectively.
\end{lemma}
\begin{proof}
This is a direct consequence of Corollary \ref{grmcoho}, since $ \R[u,v]= \ker N|_{O(\SL_2)}$ and the map $ \R[x,y]\to \coker N|_{O(\SL_2)}$ is an isomorphism.
\end{proof}

\begin{corollary}\label{consani}
The $S^1$-invariant subspace $\H^q(\row_1^*j^*\tilde{A}^{\bt}(X))^{S^1}$ is canonically isomorphic as an $N$-representation to the Archimedean cohomology group $\H^q(\tilde{X}^*)$ defined in \cite{consani}.
\end{corollary}
\begin{proof}
First observe that $N$ acts on $O(\SL_2)$ as the derivation $\left(\begin{smallmatrix} 0 & 0  \\ 1 & 0 \end{smallmatrix} \right) \in {\mathfrak{sl}_2}$ acting on the left, and that differentiating the action of $\bG_m \subset S$ on $O(\SL_2)$ gives the derivation $\left(\begin{smallmatrix} -1 & 0  \\ 0 & 1 \end{smallmatrix} \right) \in {\mathfrak{sl}_2}$, also acting on the left. Therefore decomposition by the weights of the $\bG_m$-action gives a splitting of  the filtration $M$ associated to the locally nilpotent operator $N$.

By Proposition \ref{cfdeligne} and \cite[Proposition  4.1]{consani}, we know that Deligne cohomology arises as the cone of $N: H^* \to H^*(-1)$, for both cohomology theories $H^*$.  

It now follows from Corollary \ref{grmcoho} and Lemma \ref{kercokern} that  the graded $N$-module $\H^q(\row_1^*j^*\tilde{A}^{\bt}(X))^{S^1}$ shares all the properties  of \cite[ Corollary 4.4, Proposition 4.8 and Corollary 4.10]{consani}, which combined are sufficient to determine the  graded $N$-module $\H^q(\tilde{X}^*)$ up to isomorphism.
\end{proof}

Note that under the formality isomorphism of Theorem \ref{formalitysl}, this isomorphism becomes $\H^q(\tilde{X}^*)\cong \H^q(X,\R)\ten^{S^1} O(\SL_2)$.

\subsubsection{Archimedean periods}

We can construct the ring $\cS$ without choosing co-ordinates  as follows. Given any $\R$-algebra $A$, let $UA$ be the underlying $\R$-module. Then
$$
\cS= \R[U\Cx]\ten_{\R[U\R]}\R,
$$
 For an explicit comparison, write $U\Cx= \R x_1 \oplus \R x_i$, so the right-hand side is $\R[x_1, x_i]\ten_{\R[x_1]}\R= \R[x_i]$. 

The filtration $F$ is then given by powers of the augmentation ideal of the canonical map $\cS\ten_{\R}\Cx \to \Cx$, since the ideal is  $(x_i-i)$. The derivation $N$ (from Definition \ref{Ndef}) is differentiation $\cS \to \Omega(\cS/\R)$, and $\Omega(\cS/\R)\cong \cS \ten_{\R} (\Cx/\R)$. 
There is also an action of $\Gal(\Cx/\R)$ on $\cS$, determined by the action on $U\Cx$, which corresponds to the generator $\sigma \in \Gal(\Cx/\R)$  acting $\Cx$-linearly as $\sigma(x)=-x$.

For $K= \R, \Cx$, we can then define $B(K):= \cS \ten_{\R} \Cx $, with Frobenius $\phi$ acting as complex conjugation, and the Hodge filtration, $\Gal(\Cx/K)$-action and $N$ defined as above. However,  
beware that $B(K)$  differs from the ring $B_{\mathrm{ar}}$ from \cite{deninger}.

We think of  $B(K)$  as analogous to the ring $B_{\st}$ of semi-stable periods (see e.g. \cite{Bcrisillusie}) used in crystalline cohomology. For a $p$-adic field $K$, recall that $B_{\st}(K)$ is a $\Q_p$-algebra equipped with a $\Gal(\bar{K}/K)$-action, a Frobenius-linear automorphism $\phi$, a decreasing filtration $F^i B_{\st}(K)$, and a nilpotent derivation
$$
N: B_{\st}(K) \to B_{\st}(K)(-1).
$$

Thus we think of $X\ten \R$ as being of semi-stable reduction at $\infty$, with nilpotent monodromy operator $N$ on the Archimedean fibre $X\ten \cS$. The comparison with $B_{\st}$ is further justified by comparison with \cite{weiln}, where the crystalline comparison from \cite{olssonhodge} is used to show that for a variety of good reduction, the $p$-adic \'etale homotopy type $(X_{\et}\ten \Q_p)\ten_{\Q_p} B_{\st}^{\phi}$ is formal as Galois representation in homotopy types, and that formality preserves $N$ (since good reduction means that $N$ acts trivially on $(X_{\et}\ten \Q_p)$, while $B_{\st}^{\phi} \cap \ker N= B_{\cris}^{\phi}$). 

In our case, $ B(K)^{\phi}=\cS$, so $(X\ten \R)\ten_{\R} B(K)^{\phi}$ is formal as a $\Gal(\bar{K}/K)$-representation in non-abelian MHS. However, formality does not preserve $N$ since we only have semi-stable, not good, reduction at $\infty$. 

In keeping with the philosophy of Arakelov theory, there should be a norm $\langle -,- \rangle$ on $B(K)$ to compensate for the finiteness of $\Gal(\Cx/K)$. In order to ensure that $\tilde{d}^*= -[\L,\tilde{d}^c]$, we define a semilinear involution  $*$ on $O(\SL_2)\ten \Cx$ by $u^*=y, v^*=-x$. This corresponds to the involution $A \mapsto (A^{\dagger})^{-1}$ on $\SL_2(\Cx)$, so the most natural metric on $O(\SL_2)$ comes via Haar measure on $\SU_2(\Cx)$ (the unit quaternions). However, the ring homomorphism $O(\SL_2) \to B(K)$ (corresponding to $\left(\begin{smallmatrix} 1 &  0 \\ \bG_a & 1\end{smallmatrix}\right) \le \SL_2$) is not then bounded for any possible norm on $B(K)$, suggesting that we should think of $\SL_2$ and even $\SU_2$ as being more fundamental than $\cS$. See \cite[\S \ref{mhsDS-splittingtwistor}]{mhsDS} for related phenomena.

\begin{remark}
If we wanted to work with $k$-MHS for a subfield $k \subset \R$, we could replace $\cS$ with the ring 
$$
\cS_k= k[U_k\Cx]\ten_{k[U_kk]}k,
$$
where $U_kB$ is the $k$-module underlying a $k$-algebra $B$. The results of \S\S \ref{Bei1}, \ref{Bei2} then carry over, including Remark \ref{BeiS}. 

We can use this to find the analogue of Definitions \ref{hfildef} and \ref{mhsdef} for $k$-MHS. First, note that $S$-equivariant $\SL_2$-modules are quasi-coherent sheaves on $[\SL_2/S]$, and that $[\SL_2/S] = [\mathrm{SSym}_2/\bG_m]$, where $\mathrm{SSym}_2 \subset \SL_2$ consist of symmetric matrices, and the identification $\SL_2/S^1= \mathrm{SSym}_2$ is given by $A \mapsto AA^t$ (noting that $S^1$ acts on $\SL_2$ as right multiplication by $\OO_2$). The action of $\sigma \in \Gal(\Cx/\R)$ on $\mathrm{SSym}_2 $ is conjugation by $\left(\begin{smallmatrix} 1 & 0  \\ 0 & -1 \end{smallmatrix} \right)$, while the involution $*$ is given by $B \mapsto \bar{B}^{-1}$.

Note that $\mathrm{SSym}_2= \Spec \xi(\cS,\gamma)$, for $\xi(\cS,\gamma)$ the Rees algebra with respect to the filtration $\gamma$.
%

Now,  algebraic  Hodge filtrations on real complexes  correspond to $S$-equivariant $\oR O(C^*)$-complexes (for $\oR O(C^*)$ as in Definition \ref{rjc}). The identifications above and  Remark \ref{BeiS} ensure that these are equivalent to $\bG_m$-equivariant $\oR O(C^*)^{S^1}$-complexes, where $\oR O(C^*)^{S^1}$ is the cone of  $\xi(\cS,\gamma) \xra{N} \xi(\Omega(\cS/\R),\gamma)$. 

Therefore we could define algebraic  Hodge filtrations on $k$-complexes to be $\bG_m$-equivariant $\xi(\Omega^{\bt}(\cS_k/k), \gamma)$-complexes, where $\gamma^pV= V\cap F^p(V\ten_k\Cx)$.  
\end{remark}

To complete the analogy with \'etale and crystalline homotopy types, there should be something like a graded  homotopy type  $\xi(X_{\st}, \gamma)$  over the generalised ring (in the sense of \cite{haran}) $\xi(B, \gamma)^{\langle -, -\rangle, \Gal(\Cx/K)}$ of norm $1$ Galois-invariant  elements in the Rees algebra $\xi(B, \gamma)= O(\mathrm{SSym}_2)$, equipped with a monodromy operator $N$ and complex conjugation $\phi$. 

The generalised tensor product $\xi(X_{\st}, \gamma)\ten_{\xi(B, \gamma)^{\langle -,- \rangle}} \xi(B, \gamma)$ should then be equivalent to $\xi((X\ten \R)\ten_{\R}B, \gamma)$, and then we should recover the rational homotopy type as the subalgebra 
$$
(X\ten \R)= \xi( X_{\st}\ten B, \gamma)^{\bG_m, \phi=1,N=0}= F^0(X_{\st}\ten B)^{\phi=1,N=0}.
$$
 By comparison with \'etale cohomology, the existence of a Hodge filtration on $X\ten \R$ seems anomalous, but it survives this process because (unlike the crystalline case) $F^0B=B$. 

See \cite[\S \ref{mhsDS-splittingtwistor}]{mhsDS} for related phenomena, showing that  the MHS on a homotopy type in terms of an $S^1$-equivariant, Galois-equivariant sheaf of $\C^{\infty}$ DG Fr\'echet algebras over $\bP^1(\Cx)$, equipped with a flat $\bar{\pd}$-connection $N$.

\section{Relative Malcev homotopy types}\label{relmalsect}

Rational homotopy types have the disadvantage that for non-nilpotent spaces, they destroy a lot of information, leaving no possibility of recovering the homotopy groups.  For instance, say we have a fibration $F \to X \to Y$ of topological spaces. Unless the action of $\pi_1(Y)$ on $\H^*(F, \Q)$ is nilpotent, we will not be able to relate the rational homotopy types of $F$, $X$ and $Y$, even if all three spaces are nilpotent. 

In \cite{hainrelative}, Hain addressed this problem on the level of fundamental groups by introducing relative Malcev fundamental groups, which are certain completions of fundamental groups with respect to reductive representations. A far more abstract approach to this problem was given by To\"en's schematic homotopy types in \cite{chaff}.

In \cite{htpy}, the notion of relative Malcev homotopy types was introduced, extending Hain's relative Malcev fundamental groups from \cite{hainrelative}. These generalise of Sullivan's and Quillen's rational homotopy theories in a way that gives more information  for non-nilpotent spaces, for instance recovering the cohomology of  finite-dimensional local systems. The schematic homotopy types of \cite{chaff} also arise as a special case of a relative Malcev homotopy type.

Given a reductive pro-algebraic group $R$, a topological space  $X$, and a Zariski-dense  morphism $\rho:\pi_1(X,x) \to R(k)$, there is a Malcev homotopy type $X^{\rho, \mal}$ of $X$ relative to $\rho$. If $R=1$ and $k=\Q$ (resp. $k=\R$),  then this is just the rational (resp. real) homotopy type of $X$. If $R$ is the reductive pro-algebraic fundamental group of $X$, then $X^{\rho, \mal}$ is the schematic homotopy type of $X$.

There is in fact a whole hierarchy of relative Malcev homotopy types associated to $X$, each corresponding to a tensor category of semisimple local systems. At one extreme is the rational homotopy type and  at the other extreme is To\"en's schematic homotopy type, but in between are often tractable homotopy types which are just large enough to detect information killed by the rational homotopy type.

\subsection{Review of pro-algebraic homotopy types}\label{review}
Here we give a summary of the results from \cite{htpy} which will be needed in this paper. Fix a field $k$ of characteristic zero.

\begin{definition}
Given a pro-algebraic group $G$ (i.e. an affine group scheme over $k$), define the reductive quotient $G^{\red}$ of $G$ by 
$$
G^{\red}=G/\Ru(G),
$$
where $\Ru(G)$ is the pro-unipotent radical of $G$.  Observe that $G^{\red}$ is then a reductive pro-algebraic group, and that representations of $G^{\red}$ correspond to semisimple representations of $G$.
\end{definition}

\begin{proposition}
For any pro-algebraic group $G$, there is a Levi decomposition $G=G^{\red} \ltimes \Ru(G)$, unique up to conjugation by $\Ru(G)$.
\end{proposition}
\begin{proof}
See \cite{Levi}.
\end{proof}

\subsubsection{The pointed pro-algebraic homotopy type of a topological space}
We now recall the results from \cite[\S \ref{htpy-ptd}]{htpy}. 

\begin{definition}
Let $\bS$ be the category of simplicial sets, let $\bS_0$ be the category of reduced simplicial sets, i.e. simplicial sets with one vertex, and let $s\Gp$ be the category of simplicial groups. Let $\Top_0$ denote the category of pointed connected compactly generated Hausdorff topological spaces.
\end{definition}

Note that there is a functor from $\Top_0$ to $\bS_0$ which sends $(X,x)$ to the simplicial set
$$
\Sing(X,x)_n:=\{ f \in \Hom_{\Top}(|\Delta^n|, X) : f(v)=x \quad \forall v \in \Delta^n_0\}.
$$
this is a right Quillen equivalence, the corresponding left equivalence being geometric realisation. For the rest of this section, we will therefore restrict our attention to reduced simplicial sets.

As in \cite[Ch. V]{sht}, there is a classifying space  functor $\bar{W}\co s\Gp \to \bS_0$. This has a left adjoint $G\co \bS_0 \to s\Gp$, Kan's loop group functor (\cite{loopgp}), and these give a Quillen equivalence of model categories. In particular, $\pi_i(G(X))=\pi_{i+1}(X)$. This allows us to study simplicial groups instead of pointed topological spaces.

\begin{definition}
Given a simplicial object $G_{\bullet}$ in the category of pro-algebraic groups, define $\pi_0(G_{\bullet})$ to be the coequaliser 
$$
\xymatrix@1{G_1 \ar@<1ex>[r]^{\pd_1} \ar@<-1ex>[r]_{\pd_0}& G_0 \ar[r] &\pi_0(G) }
$$ 
in the category of pro-algebraic groups.
\end{definition}

\begin{definition}
Define a pro-algebraic simplicial  group  to consist of a simplicial diagram $G_{\bullet}$ of pro-algebraic groups, such that the maps $G_n \to \pi_0(G)$ are pro-unipotent extensions of pro-algebraic groups, i.e. $\ker(G_n \to \pi_0(G))$ is pro-unipotent.  We denote the category of pro-algebraic simplicial groups by $s\agp$. 
\end{definition}

There is a forgetful functor $(k)\co s\agp \to s\gp$, given by sending $G_{\bt}$ to $G_{\bt}(k)$. This functor clearly commutes with all limits, so has a left adjoint $G_{\bt} \mapsto (G_{\bt})^{\alg}$. We can describe $(G_{\bt})^{\alg}$ explicitly. First let $(\pi_0(G))^{\alg}$ be the pro-algebraic completion of the abstract group $\pi_0(G)$, then let $(G^{\alg})_n$ be the relative Malcev completion  (in the sense of \cite{hainrelative}) of the morphism
$$
G_n \to (\pi_0(G))^{\alg}.
$$
In other words, $G_n \to (G^{\alg})_n \xra{f} (\pi_0(G))^{\alg}$ is the universal diagram with $f$ a pro-unipotent extension.

\begin{examples}\label{proalgegs}
The pro-algebraic completion $G^{\alg}$ of an abstract group $G$ has the property that its linear representations correspond to the finite-dimensional representations of $G$, and this is often used to define $G^{\alg}$ via Tannaka duality. It is rare for $G^{\alg}$ to be finite-dimensional --- when this happens, $G$ is called super rigid (cf. \cite{BassLubotzkyMagidMozes}). Examples of super rigid groups include $\SL_n(\Z)$ for $n \ge 3$ --- in these cases, $ \SL_n(\Z)^{\alg}= \SL_n$.

In general, the pro-algebraic completion is large and unwieldy. For commutative groups $G$, the affine group scheme $G^{\alg}$ has a fairly simple description, at least when $k$ is algebraically closed. In this case, $G^{\alg}= \Spec k[\Hom_{\gp}(G,k^{\by})]$, so for instance $\Z^{\alg}= \Spec k[k^{\by}]$. Given a $k$-algebra $A$, an $A$-valued point of $G^{\alg}$ is then a group homomorphism $\Hom_{\gp}(G,k^{\by}) \to A^{\by}$. In particular, $\Z^{\alg}(A)$ consists of group homomorphisms $k^{\by}\to A^{\by}$.  
\end{examples}

\begin{proposition}\label{detectweak}
 For a morphism $f \co G \to K$ in $s\agp$, the following conditions are equivalent:
\begin{enumerate}
\item the maps $\pi_n(f)\co \pi_n(G) \to \pi_n(K)$ are isomorphisms for all $n$;

 \item the pro-unipotent radicals satisfy $f(\Ru(G))\le \Ru(K)$,  with the quotient map
$$
G^{\red} \to K^{\red}
$$
 an isomorphism on pro-reductive quotients, and for all (finite-dimensional) irreducible $\pi_0K$-representations $V$, the maps
$$
\H^i(f)\co \H^i(K,V)\to \H^i(G,f^*V)  
$$
are isomorphisms for all $i>0$.

\item the maps $\H^iO(f)\co \H^iO(K) \to \H^iO(G)$ are isomorphisms for all $i$. 
\end{enumerate}

\end{proposition}
\begin{proof}
 This combines  \cite[Lemma \ref{htpy-OG} and Corollary \ref{htpy-detectweak}]{htpy}.
\end{proof}

\begin{definition}
A map satisfying the conditions of Proposition \ref{detectweak} is referred to as a \emph{weak equivalence} or \emph{quasi-isomorphism}, and we denote the localisation of $s\agp$ at weak equivalences by $\Ho(s\agp)$.
\end{definition}

\begin{proposition}\label{algq}
The functors $(k)$ and ${}^{\alg}$ give rise to a pair of adjoint functors 
$$
\xymatrix@1{\Ho(s\gp) \ar@<1ex>[r]^{\oL^{\alg}} & \Ho(s\agp) \ar@<1ex>[l]^{(k)}_{\bot} }
$$
on the homotopy categories, with $\oL^{\alg}G(X)=G(X)^{\alg}$, for any $X \in \bS_0$.
\end{proposition} 
\begin{proof}
As in \cite[Proposition \ref{htpy-algq}]{htpy}, it suffices to observe that $(k)$ preserves fibrations and trivial fibrations, so is a right Quillen functor. 
\end{proof}

\begin{definition}\label{varpidef}
Given a reduced simplicial set (or equivalently a pointed, connected topological space) $(X,x)$, define the pro-algebraic homotopy type $(X,x)^{\alg}$ of $(X,x)$ over $k$ to be the  object
$$
G(X,x)^{\alg}
$$ 
in $s\agp$. Given a pointed, connected topological space $(X,x)$, set
\[
 G(X,x)^{\alg}:= G(\Sing(X,x))^{\alg}.
\]

Define the pro-algebraic fundamental group by $\varpi_1(X,x):=\pi_0(G(X,x)^{\alg})$. Note that  $\pi_0(G^{\alg})$ is the pro-algebraic completion of the  group $\pi_0(G)$.

We then define the higher  pro-algebraic homotopy groups $\varpi_n(X,x)$ by 
$$
\varpi_n(X,x):=\pi_{n-1}(G(X,x)^{\alg}).
$$
\end{definition}

By \cite[Corollary \ref{htpy-eqtoen}]{htpy}, pro-algebraic homotopy types are equivalent to To\"en's schematic homotopy types from \cite{chaff}.
When $X$ is a nilpotent space and $k=\Q$, note that the pro-algebraic homotopy type is just Quillen's rational homotopy type from \cite{QRat}.

\subsubsection{Pointed relative Malcev homotopy types}\label{malcev}

\begin{definition}\label{malcevdef}
Assume we have an abstract group $G$, a reductive pro-algebraic group $R$, and a representation $\rho:G \to R(k)$ which is  Zariski-dense on morphisms. Define the Malcev completion $(G,\rho)^{\mal}$ (or $G^{\rho, \mal}$, or $G^{R, \mal}$)   of $G$ relative to $\rho$  to be the universal diagram
$$
G \to (G,\rho)^{\mal} \xra{p} R,
$$
with $p$  a pro-unipotent extension, and the composition equal to $\rho$. 
\end{definition}

Note that finite-dimensional representations of $(G,\rho)^{\mal}$ correspond to $G$-representations which are Artinian extensions of $R$-representations.

\begin{definition}
Given a reduced simplicial set $(X,x)$ and  a Zariski-dense morphism $\rho:\pi_1(X,x) \to R(k)$, let the Malcev completion $G(X,x)^{\rho,\mal}$ (or $G(X,x)^{R, \mal}$) of $(X,x)$ relative to $\rho$ be the pro-algebraic simplicial group $(G(X,x), \rho)^{\mal}$.   

Let $\varpi_1(X^{\rho,\mal},x)=\pi_0(G(X,x),\rho)^{\mal}$ and $\varpi_n(X^{\rho,\mal},x)=\pi_{n-1}(G(X,x),\rho)^{\mal}$. Observe that $\varpi_1(X^{\rho,\mal},x)$ is just the relative Malcev completion $\varpi_1(X,x)^{\rho,\mal}$ of $\rho:\pi_1(X,x) \to R(k)$.  
\end{definition} 

We therefore have a whole family of homotopy types associated to $X$, each corresponding to a quotient group scheme of $(\pi_1(X,x))^{\red}$. At one extreme, taking $R=1$ gives the rational homotopy type $G(X,x)^{1, \mal}$. At the other extreme, observe that the Malcev completion of $(X,x)$ relative to $(\pi_1(X,x))^{\red}:=(\pi_1(X,x)^{\alg})^{\red} $ is just the pro-algebraic homotopy type $G(X,x)^{\alg}$. This diversity of homotopy types has the advantage that we can often choose $R$ to be just large enough to detect information killed by the rational homotopy type, without having to involve the huge objects which tend to arise from pro-algebraic completion.

Note that for any cosimplicial  $G(X,x)^{R, \mal}$-representation (i.e. $O(  G(X,x)^{R, \mal})$-comodule, and in particular any $\varpi_1(X^{\rho,\mal},x)$-representation) $V$, the canonical map $\bH^*(G(X,x)^{\rho,\mal}, V)\to \bH^*(X, V) $ from group cohomology to singular cohomology is an isomorphism.

\begin{definition}
Recall that the ring $O(R)$ of algebraic functions on $R$ has the natural structure of an $R\by R$-representation,  with the $R$-actions given by left and right multiplication. We will usually just regard $O(R)$ as an $R$-representation via the right action.
\end{definition}

Note that for any $R$-representation $V$, we  have
\[
 \Hom_R(V, O(R)) \cong V^{\vee},
\]
with the $R$-action on $V^{\vee}$ then coming from the left $R$-action on $O(R)$.

\begin{remark}\label{irreprmk}
If 
$\{V_{\alpha}\}_{\alpha}$ a set of representatives of the isomorphism classes of irreducible $R$-representations, 
then we have an isomorphism
\[
O(R) \cong \bigoplus_{\alpha} V_{\alpha}^{\vee}\ten_{\End_R(V_{\alpha})}V_{\alpha},
\]
with the left and right $R$-actions coming from the actions on $V_{\alpha}^{\vee}$ and $V_{\alpha}$.
When $k$ is algebraically closed, note that the division rings $\End_R(V_{\alpha}) $ are all isomorphic to $k$.
\end{remark}

\begin{theorem}\label{fibrations}
Take a fibration $f:(X,x) \to (Y,y)$ (of pointed connected  topological spaces) with connected fibres, and set $F:= f^{-1}(y)$. Take a Zariski-dense representation $\rho: \pi_1(X,x) \to R(k)$ to a  reductive pro-algebraic group $R$, let $K$ be the closure of $\rho(\pi_1(F,x))$, and $T:= R/K$. If the monodromy action of $\pi_1(Y,y)$ on $\H^*(F, V)$ factors through $\varpi_1(Y,y)^{T, \mal}$ for all $K$-representations $V$, then $G(F,x)^{K,\mal}$ is the homotopy fibre of $ G(X,x)^{R, \mal} \to G(Y,y)^{T, \mal}$.

In particular, there is a long exact sequence
\begin{eqnarray*}
\ldots \to \varpi_n(F,x)^{K,\mal} \to \varpi_n(X,x)^{R,\mal}\to \varpi_n(Y,y)^{T,\mal} \to \varpi_{n-1}(F,x)^{K,\mal}\to \\
\ldots \to \varpi_1(F,x)^{K,\mal} \to \varpi_1(X,x)^{R,\mal}\to \varpi_1(Y,y)^{T,\mal} \to 1.
\end{eqnarray*}
\end{theorem}
\begin{proof}
First observe that  $\rho(\pi_1(F,x))$ is normal in $\pi_1(X,x)$, so $K$ is normal in $R$, and $T$ is therefore a reductive pro-algebraic group, so $(Y,y)^{T, \mal}$ is well-defined. Next, observe that since $K$ is normal in $R$, 
 $\Ru(K)$ is also normal in $R$, and is therefore $1$, ensuring that $K$ is reductive, so $(F,x)^{K,\mal}$ is also well-defined. 

Consider the complex $O(R\by_TG(Y)^{T, \mal})=O(R)\ten_{O(T)}O(G(Y)^{T, \mal})$ of $ G(X,x)^{R, \mal}$-representations, regarded as a complex of sheaves on $X$. The Leray spectral sequence for $f$ with coefficients in this complex is 
 $$
E_2^{i,j}=\bH^i(Y, \H^j(F, O(R))\ten_{O(T)}O(G(Y)^{T, \mal}))\abuts \bH^{i+j}(X,O(R\by_TG(Y)^{T, \mal}) ).
$$
 
Regarding $O(R)$ as a $K$-representation,   $\H^*(F, O(R))$ is a $\varpi_1(Y,y)^{T, \mal}$-representation by hypothesis. Hence
$
 \H^*(F, O(R))\ten_{O(T)}O(G(Y)^{T, \mal})
$
is a cosimplicial $G(Y)^{T, \mal}$-representation, so
\begin{align*}
&\bH^i(Y, \H^j(F, O(R))\ten_{O(T)}O(G(Y)^{T, \mal}))\\
&\cong \bH^i(G(Y)^{T, \mal},  \H^j(F, O(R))\ten_{O(T)}O(G(Y)^{T, \mal})).
\end{align*}

Now, $\H^*(F, O(R))\ten_{O(T)}O(G(Y)^{T, \mal})$ is a fibrant cosimplicial $G(Y)^{T, \mal}$-representation, so
\begin{eqnarray*}
&&\bH^i(G(Y)^{T, \mal},  \H^j(F, O(R))\ten_{O(T)}O(G(Y)^{T, \mal}))\\
&&\cong  \H^i\Gamma(G(Y)^{T, \mal},  \H^j(F, O(R))\ten_{O(T)}O(G(Y)^{T, \mal}))\\
&&=  \left\{ \begin{matrix} \H^j(F, O(R))\ten_{O(T)}k = \H^j(F, O(K))& i=0 \\ 0 & i \ne 0, \end{matrix} \right.\\
\end{eqnarray*}
so
$$
\bH^j(X, O(R\by_TG(Y)^{T, \mal}))\cong \H^j(F, O(K)).
$$

Now, let $\cF$ be the homotopy fibre of $ G(X,x)^{R, \mal} \to G(Y,y)^{T, \mal}$, noting that there is a natural map $G(F,x)^{K,\mal} \to \cF$. We have 
$$
\bH^j(X, O(R\by_TG(Y)^{T, \mal}))=\bH^j(G(X)^{R, \mal}, O(R\by_TG(Y)^{T, \mal})), 
$$
and a Leray-Serre spectral sequence
$$
\bH^i(G(Y)^{T, \mal}, \H^j(\cF, O(R))\ten_{O(T)}O(G(Y)^{T, \mal})) \abuts \bH^{i+j}(G(X)^{R, \mal}, O(R\by_TG(Y)^{T, \mal})).
$$
The reasoning above adapts to show that this spectral sequence also collapses, yielding
$$
\H^j(\cF, O(K)) = \bH^j(X, O(R\by_TG(Y)^{T, \mal})).
$$

We have therefore shown that the map $G(F,x)^{K,\mal} \to \cF $ gives an isomorphism
$$
\H^*(\cF, O(K)) \to \H^*(G(F,x)^{K,\mal}, O(K)),
$$
and hence isomorphisms $\H^*(\cF, V) \to \H^*(G(F,x)^{K,\mal}, V) $ for all $K$-representations $V$.
Since this is a morphism of simplicial pro-unipotent extensions of $K$, Proposition \ref{detectweak} implies that $G(F,x)^{K,\mal} \to \cF$ is a weak equivalence.
\end{proof}

A special case of Theorem \ref{fibrations} has appeared in \cite[Proposition 4.20]{schematicv2}, when $F$ is simply connected and of finite type, and $T= \varpi_1(Y,y)^{\red}$. Note that the theorem allows us to study rational homotopy types of homotopy fibres:

\begin{corollary}\label{unipfibs}
Given a fibration $f:(X,x)\to (Y,y)$  with connected fibres, assume that the fibre $F:=f^{-1}(y)$ has finite-dimensional cohomology groups $\H^i(F,k)$ and  let $R$ be the reductive quotient of the Zariski closure of the homomorphism $\pi_1(Y,y) \to \prod_i \GL(H^i(F, k))$. Then the Malcev homotopy type $(F\ten k, x)$ is the homotopy fibre of
$$
G(X,x)^{R, \mal} \to G(Y,y)^{R, \mal}.
$$ 
\end{corollary}
\begin{proof}
This is just Theorem \ref{fibrations}, with $R=T$ and $K=1$.
\end{proof}

Note that for a morphism $f:(X,x)\to (Y,y)$ which is not a fibration, we can apply Theorem \ref{fibrations} to a weakly equivalent fibration, replacing $F$ with the homotopy fibre of $f$ over $y$.

\begin{remark}
Beware that even when $Y$ is a $K(\pi,1)$, the relative completion $Y^{R, \mal}$ need not be so. For instance, \cite{haincontemp} and \cite{haintorelli} are concerned with studying the exact sequence $1 \to T_g \to \Gamma_g \to \Sp_g(\Z)\to 1$, where $\Gamma_g$ is the mapping class group and $T_g$ the Torelli group. Taking $R =\Sp_{g,\Q}$, we get $\H^1(\Sp_g(\Z), O(R))=0$, but $\H^2(\Sp_g(\Z), O(R))\cong \Q$. Therefore $\varpi_1(B\Sp_g(\Z))^{R, \mal}=R$, but the Hurewicz theorem gives $\varpi_2(B\Sp_g(\Z))^{R, \mal} =\Q$. Thus the long exact sequence for homotopy has final terms
$$
\Q \to T_g\ten \Q \to \Gamma_g^{R, \mal}\to \Sp_g(\Z)^{R, \mal} \to 1.
$$
This is consistent with \cite[ Proposition 7.1]{haincontemp} and \cite[Theorem 3.4]{haintorelli}, which show that  $T_g\ten \Q \to \Gamma_g^{R, \mal}$ is a central extension by $\Q$.
  \end{remark}

\begin{definition}\label{gooddef}
Define a group $\Gamma$ to be good with respect to a Zariski-dense representation $\rho: \Gamma \to R(k)$ to a reductive pro-algebraic group if the homotopy groups $\varpi_n(B\Gamma)^{R, \mal}$ are $0$ for all $n\ge 2$.
\end{definition}

 The fundamental group of a compact Riemann surface is  good with respect to all representations, as are finite groups, free groups and finitely generated nilpotent groups --- see \cite[Examples \ref{htpy-goodexamples}]{htpy}.

\begin{lemma}\label{goodchar}
A group $\Gamma$ is good relative to $\rho: \Gamma \to R(k)$ if and only if the map
$$
\H^n(\Gamma^{\rho, \mal}, V) \to \H^n(\Gamma, V)
$$  
is an isomorphism for all $n$ and all finite-dimensional $R$-representations $V$.
\end{lemma}
\begin{proof}
This follows by looking at the map $f:G(B\Gamma)^{R, \mal} \to \Gamma^{R, \mal}$ of simplicial pro-algebraic groups, which is a weak equivalence if and only if $\varpi_n(B\Gamma)^{R, \mal}=0$ for all $n \ge 2$. By Proposition \ref{detectweak}, $f$ is a weak equivalence if and only if the  morphisms
$$
\H^*(\Gamma^{R, \mal}, V) \to \H^*(G(B\Gamma)^{R, \mal}, V)
$$
are isomorphisms for all $R$-representations $V$. Since $\H^*(G(B\Gamma)^{R, \mal}, V)  = \H^*(B\Gamma, V)= \H^*(\Gamma, V)$, the result follows.
\end{proof}

\begin{lemma}\label{goodh}
%
A group  $\Gamma$ is good with relative to $\rho: \Gamma \to R(k)$ if and only if for any finite-dimensional $\Gamma^{\rho, \mal}$-representation $V$, and any $\alpha \in \H^n(\Gamma, V)$, there exists an injection $f:V \to W_{\alpha}$ of finite-dimensional $\Gamma^{\rho, \mal}$-representations, with $f(\alpha)=0 \in \H^n(\Gamma, W_{\alpha})$.
\end{lemma}
\begin{proof}
The proof of \cite[Lemma 4.15]{schematicv2} adapts to this generality --- see for instance \cite[\S \ref{heid-relgoodsn}]{heid}.
\end{proof}

\begin{definition}\label{relgood}
Say that a group $\Gamma$ is  $n$-\emph{good} with respect to a Zariski-dense representation $\rho\co  \Gamma \to R(k)$ to a reductive pro-algebraic group if for all finite-dimensional $\Gamma^{\rho, \mal}$-representations $V$, the map
$$
\H^i(\Gamma^{\rho, \mal}, V) \to \H^i(\Gamma, V)
$$  
is an isomorphism for all $i\le n$ and an inclusion for $i=n+1$. 
\end{definition}

The following is \cite[Theorem \ref{weiln-classicalpimal}]{weiln}, which strengthens \cite[Theorem \ref{htpy-classicalpimal}]{htpy}: 
\begin{theorem}\label{classicalpimal}
If $(X,x)$ is a pointed connected topological space with fundamental group $\Gamma$, equipped with a Zariski-dense representation $\rho\co  \Gamma \to R(k)$ to a reductive pro-algebraic group for which: 
\begin{enumerate}
\item $\Gamma$ is  $(N+1)$-good with respect to $\rho$,
\item $\pi_n(X,x)$ is of finite rank for all $1<n \le N$, and
\item  the $\Gamma$-representation  $\pi_n(X,x)\ten_{\Z}k$ is an extension of $R$-representations (i.e. a $\Gamma^{\rho, \mal}$-representation) for all $1<n \le N$,
\end{enumerate}
then the canonical map
$$
  \pi_n(X,x)\ten_{\Z} k \to \varpi_{n}(X^{\rho, \mal},x)(k) 
$$
is an isomorphism for all $1<n \le N$.
\end{theorem}

To see how to compare homotopy groups when the goodness hypotheses are not satisfied, apply Theorem \ref{fibrations} to the fibration $(X,x) \to B\Gamma$.

\subsection{Cosimplicial and DG Hopf algebras}

A relative Malcev homotopy type is a simplicial pro-algebraic group, or equivalently a cosimplicial Hopf algebra. Cosimplicial and simplicial objects are very difficult to compute with, because non-constant objects have generators in infinitely many levels. We now show how to relate them to DG Hopf algebras and to simplicial and chain Lie algebras.

\subsubsection{Definitions and basic properties}

\begin{definition}
Define $\cE(R)$ to be the full subcategory of pro-algebraic groups over $R$  consisting of those morphisms $\rho:G\to  R$ of  pro-algebraic groups which are pro-unipotent extensions. Similarly, define $s\cE(R)$ to consist of the pro-unipotent extensions in $s\agp\da R$, and $\Ho_*(s\cE(R))$ to be full subcategory of $\Ho(s\agp)$ on objects $s\cE(R)$.

Note that Malcev completions $G(X,x)^{R, \mal}$ relative to $R$ are objects of $s\cE(R)$. Also note that $\cE(R)$ is opposite to the category of Hopf algebras $H$ equipped with an injective map $O(R) \to H$ of Hopf algebras such that the comultiplication on the Hopf algebra $H\ten_{O(R)}k$ is ind-conilpotent.
\end{definition}

\begin{definition}\label{N^s}
Given a simplicial diagram $V_{\bt}$ in an abelian category, recall that the normalised chain complex $N^s(V)_{\bt}$ is given by  $N^s(V)_n:= V_n/\sum_i\sigma_i V_{n-1})$, with differential $\sum_j (-i)^j\pd_j$. The simplicial Dold--Kan correspondence says that $N^s$ gives an equivalence of categories between simplicial diagrams and non-negatively graded chain complexes in any abelian category. 
\end{definition}

\begin{definition}
Let $\cN(R)$ be the category of  $R$-representations in finite-dimensional nilpotent non-negatively graded Lie algebras over $k$. Write $\hat{\cN}(R)$ for the category $\pro(\cN(R))$ of pro-objects in $\cN(R)$, and $s\hat{\cN}(R)$ for the category of simplicial diagrams in $\hat{\cN}(R)$. We also write $s\cN(R)$ for the category consisting of those simplicial diagrams in $\cN(R)$ with finite-dimensional Dold--Kan normalisation.
\end{definition}

\begin{lemma}\label{expscNlemma}
The functor $R\ltimes \exp \co s\hat{\cN}(R) \to s\cE(R)$ is essentially surjective on objects. On morphisms, we have
\[
 \Hom_{s\cE(R)}(R\ltimes \exp(\g),R\ltimes \exp(\fh)) \cong \exp(\fh_0)\by^{\exp(\fh^R_0)}\Hom_{s\hat{\cN}(R)}(\g, \fh),
\]
where $\fh_0^R$ (the Lie subalgebra of $R$-invariants in $\fh_0$)  acts by conjugation on the set of homomorphisms.
Composition of morphisms is given by $(u,f) \circ (v,g)= (u\circ f(v),f\circ g)$.
\end{lemma}
\begin{proof}
 The functor $R \ltimes \exp :s\hat{\cN}(R) \to s\cE(R) $ maps $\g$ to the simplicial pro-algebraic group given in level $n$ by $R \ltimes \exp(\g_n)$. The existence of Levi decompositions ensures that for any $G \in s\cE(R)$, the map $G_0 \to R$ admits a section, so  $R \ltimes \exp(-)$ is essentially surjective. Since the section of $G_0 \to R$ is unique up to inner automorphism, it follows that   every morphism in $s\cE(R)$ is the composite of an inner automorphism and a morphism preserving Levi decompositions. Thus any  homomorphism
\[
 R\ltimes \exp(\g)\to R\ltimes \exp(\fh)
\]
factorises as  $\ad_u \circ (R\ltimes \exp(f))$, for $u \in \exp(\fh_0)$ and $f\co \g \to \fh$. The choice of $u$ is unique up to $R$-invariants $\exp(\fh_0)^R$, giving the description required.
\end{proof}

\begin{definition}\label{grouplike}
Given $A \in DG\Alg(R)$, define the category of $R$-equivariant  dg pro-algebraic   groups $G_{\bt}$ over $A$ to be opposite to the category of $R$-equivariant DG Hopf algebras over $A$. Explicitly, this consists of objects $Q \in DG_{\Z}\Alg_A(R)$ equipped with morphisms  $Q \to Q\ten_A Q$, 
(comultiplication), $Q \to A$ (coidentity) and $Q \to Q$ (coinverse), satisfying the usual axioms. Write $O(G)$ for the DG Hopf algebra associated to a dg pro-algebraic   group $G_{\bt}$, and $\Spec Q$ for the  dg pro-algebraic   group  associated to a DG Hopf algebra $Q$.

Given a scheme $Y$, a dg pro-algebraic   group $G_{\bt}$ over $Y$ is a sheaf of dg pro-algebraic   groups over the sheaf $\O_Y$, such that the associated sheaf $O(G)$ of DG Hopf $\O_Y$-algebras is quasi-coherent.

A morphism $f:G_{\bt} \to K_{\bt}$ of dg pro-algebraic   groups is said to be a quasi-isomorphism if it induces an isomorphism $\H^*O(K) \to \H^*O(G)$ on cohomology of the associated DG Hopf algebras. Say that $G_{\bt}$ in concentrated in non-negative degrees if the same is true of the DG Hopf algebra $O(G)$.

Say that a dg pro-algebraic   group $G_{\bt}$ is pro-unipotent if the comultiplication on the DG Hopf algebra $O(G)$ is ind-conilpotent.
\end{definition}

\begin{definition}\label{dgcEdef}
Define $dg\cE(R)$ to consist of dg pro-algebraic   groups $G_{\bt}$ over $k$ concentrated in non-negative degrees, and equipped with a surjective map $G_0 \to R$ such that the kernel 
\[
 G\by_R1 :=\Spec ( O(G)\ten_{O(R)}k)
\]
is pro-nilpotent. Write $\Ru G \in dg\cE(1) $ for the object  $G\by_R1 $.
  
Define $ \Ho_*(dg\cE(R))$ to be the category obtained by formally inverting quasi-isomorphisms of Hopf algebras in $dg\cE(R)$.
\end{definition}

\begin{definition}\label{semitendef}
Given an affine group scheme $G$ acting on a DG coalgebra $C$, we define their semidirect tensor product $G\ltimes C$ to be the DG coalgebra  with underlying cochain complex $O(G) \ten C$, with counit
\[
 \vareps \ten \vareps \co O(G) \ten C \to k,
\]
and comultiplication
\begin{align*}
 O(G) \ten C \xra{\Delta \ten \Delta} O(G)\ten O(G) \ten C\ten C \xra{\id \ten \id \ten \nu \ten \id} O(G)\ten O(G) \ten O(G)\ten C\ten C  \\
\xra{\id \ten m\ten \id \ten \id} O(G) \ten O(G)\ten C\ten C \xra{\tau_{23}}  O(G)\ten C\ten  O(G)\ten C,
\end{align*}
where $\nu \co C \to O(G) \ten C$ is the action of $G$ on $C$, $m \co O(G) \ten O(G)$ denotes multiplication, and $\tau_{23}$ transposes the second and third factors.

When $C$ is a DG Hopf algebra,  $G\ltimes C$ becomes a DG Hopf algebra with underlying ring  $O(G) \ten C $ and antipode 
\[
 O(G) \ten C \xra{\id \ten \nu} O(G)\ten O(G)\ten  C  \xra{m \ten \id} O(G)\ten C \xra{\iota \ten \iota}O(G)\ten C.
\]
Note that when $C$ is concentrated in degree $0$, $\Spec C$ is a group scheme and $\Spec (G\ltimes C) = G \ltimes (\Spec C)$.
\end{definition}

\begin{definition}
Define $dg\cN(R)$ to be the category of $R$-representations in finite-dimensional nilpotent non-negatively graded chain Lie algebras. Let $dg\hat{\cN}(R)$ be the category of pro-objects in the Artinian category $dg\cN(R)$.

Here, a chain Lie algebra is  a  chain complex $\g=\bigoplus_{i \in \N_0} \g_i$ over $k$, equipped with a bilinear Lie bracket $[,]\co \g_i \by \g_j\ra \g_{i+j}$,  satisfying:

\begin{enumerate}
\item $[a,b]+(-1)^{\bar{a}\bar{b}}[b,a]=0$,

\item $(-1)^{\bar{c}\bar{a}}[a,[b,c]]+ (-1)^{\bar{a}\bar{b}}[b,[c,a]]+ (-1)^{\bar{b}\bar{c}}[c,[a,b]]=0$,

\item $d[a,b] = [da,b] +(-1)^{\bar{a}}[a,db]$,
\end{enumerate}
where $\bar{a}$ denotes the degree of $a$, for $a$ homogeneous. 
\end{definition}

\begin{lemma}\label{expdgcNlemma}
If  $\cU$ denotes the  pro-finite-dimensional universal enveloping dg algebra functor and $(-)^{\vee}$ the continuous dual, then 
the functor $R \ltimes \exp \co dg\hat{\cN}(R) \to dg\cE(R) $ given by 
\[
 O(R \ltimes \exp(\g)):= R \ltimes \cU(\g)^{\vee}, 
\]
 is essentially surjective on objects. On morphisms, we have
\[
 \Hom_{dg\cE(R)}(R\ltimes \exp(\g),R\ltimes \exp(\fh)) \cong \exp(\fh_0)\by^{\exp(\fh^R_0)}\Hom_{dg\hat{\cN}(R)}(\g, \fh).
\]
\end{lemma}
\begin{proof}
 Given $G \in dg\cE(R)$, the existence of a Levi decomposition for $G_0:= \Spec O(G)^0$ ensures that the functor $R \ltimes \exp$ is essentially surjective.
The proof of Lemma \ref{expscNlemma} adapts to give the other results.
\end{proof}

Note that  $ \cU(\g)^{\vee}\cong \R[ \g^{\vee}]$, with comultiplication dual to the Campbell--Baker--Hausdorff formula. 

Given $G \in dg\cE(R)$ with corresponding DG Hopf algebra $O(G)$, we may regard $G$ as a functor from CDGAs to groups, by sending $A$ to $\Hom_{\mathrm{CDGA}}(O(G),A)$, which the coalgebra structures on $O(G)$ make into a group. The functor corresponding to  $(R \ltimes \exp(\g))$ is given   by
\[
 (R \ltimes \exp(\g))(A):= R(A^0) \ltimes \exp(\Hom_{\mathrm{DG}}(\g^{\vee},A)),
\]
for CDGAs $A$, because  the ind-conilpotent Hopf algebra $O(\exp(\g))$ representing $A \mapsto \exp(\Hom_{\mathrm{DG}}(\g^{\vee},A))$ is isomorphic to $\cU(\g)^{\vee}$, as in \cite[Theorem B.4.5]{QRat}. 

\begin{remark}\label{baseptslike}
It is sometimes natural to consider multiple basepoints on a space, so the results of  \cite{htpy} were formulated for groupoids rather than for groups. 

We can adapt Definition \ref{grouplike} in the same spirit  by defining an $R$-equivariant  dg pro-algebraic   groupoid $G$ over $A$ to consist of a set $\Ob G$ of objects, together with $O(G)(x,y) \in DG\Alg_A(R)$ for all $x,y \in \Ob $, equipped with morphisms  $O(G)(x,z) \to O(G)(x,y)\ten_A O(G)(y,z)$ 
(comultiplication), $O(G)(x,x) \to A$ (coidentity) and $O(G)(x,y)\to O(G)(y,x)$ (coinverse), satisfying the usual axioms.

Given a reductive pro-algebraic groupoid $R$ with an $S$-action, and  $\g \in dg\hat{\cN}(R \rtimes S)$, we then define the $S$-equivariant dg pro-algebraic group  $R \ltimes \exp(\g)$ to have objects $\Ob R$, with
$$
(R \ltimes \exp(\g))(x,y)= R(x,y) \by \exp(\g(y)),
$$
and multiplication as in \cite[Definition \ref{htpy-semidirect}]{htpy}.
\end{remark}

\subsubsection{Representability}

\begin{lemma}\label{nSch}
 A set-valued functor $F$ on  $s\cE(R)$ (resp. $dg\cE(R)$)   is representable  if and only if it satisfies the following conditions:
\begin{enumerate}
 \item $F(R)$ is a one-point set;
\item the map $F(G\by_HK) \to F(G)\by_{F(H)}F(K)$ is an isomorphism for all surjections $G \to H$ and all maps $K \to H$ in $s\cE(R)$ (resp. $dg\cE(R)$);
\item $F$ preserves filtered limits. 
\end{enumerate}
\end{lemma}
\begin{proof}
 For a functor of the form $F= \Hom(E,-)$ , these conditions are automatically satisfied, $R$ being the final object in both categories. 

For the converse, we first observe that the categories $s\cN(R)$ and $dg\cN(R)$ are Artinian, since all subobjects are of lower dimension as vector spaces. 
The proof of  \cite[Proposition \ref{ddt1-cSp}]{ddt1} (which dealt with commutative rings rather than Lie algebras) then shows that 
the functor $U \co  \pro(s\cN(R)) \to s\hat{\cN}(R)$ given by $(U\{\g(\alpha)\}_{\alpha})_n := \{\g(\alpha)_n\}_{\alpha}$ is an equivalence of categories.

If we now write $s\cE^{\flat}(R)\subset s\cE(R)$ (resp. $dg\cE^{\flat}(R)\subset dg\cE(R)$) for the essential image of $ s\cN(R)$  (resp. $dg\cN(R)$) under the functor $R \ltimes \exp$, then Lemmas \ref{expscNlemma} and \ref{expdgcNlemma} imply that the categories $s\cE^{\flat}(R)$ and $dg\cE^{\flat}(R)$ are Artinian, with equivalences 
\[
 U\co \pro(s\cE^{\flat}(R)) \to s\cE(R)\quad  U\co \pro(dg\cE^{\flat}(R)) \to dg\cE(R)
\]
 of categories.

Take a functor $F$ satisfying the conditions above; since $F$ preserves filtered limits, it is uniquely determined by its restriction to $s\cE^{\flat}(R) $ (resp. $dg\cE^{\flat}(R)$).  If $F$ is known to preserve all finite limits, then the result follows from \cite[Proposition A.3.1]{descent}. With our weaker hypotheses, the proof of \cite[Theorem \ref{higgs-nSch}]{higgs} (itself motivated by \cite[Theorem 2.11]{Sch}) adapts from $\cN$ to $
s\cE^{\flat}(R) $ and $dg\cE^{\flat}(R)$ to give the required result.
\end{proof}

\subsection{Equivalent formulations}\label{equivformsn}

We now introduce various equivalent models for pointed relative Malcev homotopy types, allowing us to interchange between pro-algebraic simplicial  groups (which generalise Quillen's rational homotopy types) and augmented equivariant CDGAs (which generalise Sullivan's rational homotopy types), as well as dg pro-algebraic groups and equivariant cosimplicial algebras.

In  Propositions \ref{eqhtpy} and \ref{propforms} we use these to give explicit models for the homotopy types of topological spaces an manifolds, in terms of equivariant cochains and the equivariant de Rham complex.

\subsubsection{Maurer--Cartan and reduced loops}

\begin{definition}
 Let $c\Alg(R)_*$ be the category of of $R$-representations in cosimplicial $k$-algebras, equipped with an augmentation to the structure sheaf $O(R)$ of $R$. A weak equivalence in $c\Alg(R)_*$ is a map which induces isomorphisms on cohomology groups.  Let $c\Alg(R)_{0*}$ be the full subcategory of $c\Alg(R)_*$ whose objects $A$ satisfy $\H^0(A)=k$.
 Denote the respective  opposite categories by $s\Aff(R)_*= R \da s\Aff(R)$ and  $c\Aff(R)_{0*}$.
\end{definition}

The main motivating example of an object of $c\Alg(R)_*$ will be given in Example \ref{Budef}. Roughly speaking, it is given by the cochains on a topological space with coefficients in a suitable local system of  $R$-equivariant $k$-algebras. Such local systems arise from Zariski-dense representations $\pi_1(X,x) \to R(k)$.

\begin{definition}
Define $DG\Alg(R)_*$ to be the category of $R$-representations in  non-negatively graded cochain $k$-algebras, equipped with an augmentation to $O(R)$. A weak equivalence in $DG\Alg(R)_*$ is a map which induces isomorphisms on cohomology groups. Let $DG\Alg(R)_{0*}$ be the full subcategory of $DG\Alg(R)_*$ whose objects $A$ satisfy $\H^0(A)=k$.
\end{definition}

The main motivating example of an object of $DA\Alg(R)_*$ will be given in Proposition \ref{propforms}, as the de Rham  complex on a smooth manifold with coefficients in a suitable local system of  $R$-equivariant $k$-algebras. 

\begin{definition}\label{mcDG}
Given a  DG Lie algebra $L^{\bt}$, define the Maurer--Cartan space $\mc(L)$ by
\[
 \mc(L):= \{\omega \in L^1 \,|\,d\omega+\half[\omega,\omega]=0\}.
\]

In particular, for a cochain algebra $A \in DG\Alg(R)$, and a chain Lie algebra $\g \in dg\hat{\cN}(R)$, we have 
$$
\mc(A\hat{\ten}^R\g):=\{\omega \in \prod_n A^{n+1}\hat{\ten}^R \g_n \,|\,d\omega+\half[\omega,\omega]=0\},
$$
where, for an inverse system $\{V_i\}$, $\{V_i\}\hat{\ten}A:= \Lim (V_i\ten A)$, and $ \{V_i\}\hat{\ten}^RA$ consists of $R$-invariants in this.
\end{definition}

Note that this is essentially the same as the functor of twisting cochains from \cite{QRat}.

\begin{definition}\label{dgdef}\label{dgdefgauge}
Given $A \in DG\Alg(R)$ and $\g \in dg\hat{\cN}(R)$, we define the gauge group by
$$
\Gg(A\hat{\ten}^R\g):= \exp(\prod_n A^n\hat{\ten}^R\g_n).
$$ 
Define a gauge action of $\Gg(A\hat{\ten}^R\g)$ on $\mc(A\hat{\ten}^R\g)$ by 
$$
g(\omega):= g\cdot \omega \cdot g^{-1} -(dg)\cdot g^{-1}.
$$
Here, $a\cdot b$ denotes multiplication in the universal enveloping algebra $\cU(A\hat{\ten}^R\g)$ of the differential graded Lie algebra (DGLA) $A\hat{\ten}^R\g$. That $g(\omega)$ lies in $ \mc(A\hat{\ten}^R\g)$ is a standard calculation (see \cite{Kon} or \cite{Man}).
\end{definition}

\begin{definition}\label{mcc}
Given a cosimplicial algebra $A \in c\Alg(R)$, and a simplicial Lie algebra $\g \in s\hat{\cN}(R)$, define the Maurer--Cartan space 
\[
 \mc(A\hat{\ten}^R\g) \subset \prod_{n \ge 0}\exp(A^{n+1}\hten^R\g_n)
\]
to consist of those
$\{\omega_n\}_{n\ge 0}$ 
satisfying
\begin{eqnarray*}
\pd_i\omega_n &=& \left\{\begin{matrix} \pd^{i+1}\omega_{n-1}  & i>0 \\ (\pd^1\omega_{n-1})\cdot(\pd^0\omega_{n-1})^{-1} & i=0,\end{matrix} \right.\\
\sigma_i\omega_n &=& \sigma^{i+1}\omega_{n+1},\\
\sigma^0\omega_n&=& 1.
\end{eqnarray*}
 \end{definition}

\begin{definition}\label{gauge}
Given a cosimplicial algebra $A \in c\Alg(R)$, and a simplicial Lie algebra $\g \in s\hat{\cN}(R)$, define the gauge group $\Gg(A\hat{\ten}^R\g )$ to  be the subgroup of  $\prod_n \exp(A^n\hat{\ten}^R\g_n)$ consisting of those $g$ satisfying
\begin{eqnarray*}
\pd_ig_n &=& \pd^{i}g_{n-1}  \quad \forall i>0, \\
\sigma_ig_n &=& \sigma^{i}g_{n+1} \quad \forall i.
\end{eqnarray*}
Note that $\exp(A^0\hten^R\g_0)$ can be regarded as a subgroup of $\Gg( A\hat{\ten}^R\g)$, setting $g_n= (\pd^1)^n(\sigma_0)^ng$, for $g\in \exp(A^0\hten^R\g_0)$.

The gauge group acts on the Maurer--Cartan space, with the action
 given by 
$$
g(\omega)_n= (\pd_0g_{n+1}) \cdot \omega_n \cdot (\pd^0g_n^{-1}) .
$$ 
\end{definition}

\begin{proposition}\label{barGRprop}
 Given $\phi \co A \to O(R)$   in $DG\Alg(R)_{0*}$ (resp. $c\Alg(R)_{0*}$), there is a functor 
\[
 R \ltimes \exp(\fu) \mapsto \exp(\fu_0)\by^{\exp(A^0\hten^R\fu_0)}\mc(A\hten^R\fu) 
\]
on $dg\cE(R)$ (resp. $s\cE(R)$) which is representable. We denote the representing object by $\bar{G}_R(A)$ or $\bar{G}_R(\Spec A)$.
\end{proposition}
\begin{proof}
First, we use Lemmas \ref{expdgcNlemma} and \ref{expscNlemma} to establish functoriality. Given $ u \in \exp(\fu_0)$, we send the automorphism $\ad_u $  of $R \ltimes \exp(\fu)$ to the map on  $\exp(\fu_0)\by^{\exp(A^0\hten^R\fu_0)}\mc(A\hten^R\fu)$ given by left multiplication by $u$. 

To see that this is well--defined, take $z \in \exp(\fu_0^R)$, and observe that it sends a pair $(g, \omega)$ to 
\[
 (zg, \omega)= ((\ad_zg) z, \omega).
\]
Now, the map $k \to A^0$ gives an embedding $\exp(\fu_0^R) \into \exp(A^0\ten^R\fu_0) $, and so regard $z$ as an element of the latter. Thus
\[
 (zg, \omega)= (\ad_zg, z(\omega)) \in \exp(\fu_0)\by^{\exp(A^0\ten^R\fu_0)}\mc(A\hten^R\fu). 
\]
Now, the definition of the gauge action ensures that $z(\omega) = \ad_z(\omega)$ for $ z \in \exp(\H^0A\ten^R\fu_0)$, so $z$ acts as the morphism $\ad_z \co \fu \to \fu$, and we have defined a functor. 

We need to show that  this functor satisfies the conditions of Lemma \ref{nSch}. The only non-trivial condition is compatibility with fibre products, 
so take morphisms $\g \onto \fh \la \fk$ in $dg\cN(R)$ (resp. $s\cN(R)$), and look at 
\[
 \exp((\g\by_{\fh}\fk)_0)\by^{\exp(A^0\ten^R(\g\by_{\fh}\fk)_0)}\mc(A\ten^R(\g\by_{\fh}\fk)),
\]
which is the quotient of
\[
 \exp(\g_0)\by_{\exp(\fh_0)}\exp(\fk_0) \by \mc(A\ten^R\g_0) \by_{\mc(A\ten^R\fh)}\mc(A\hten^R\fk)
\]
by the action of $\exp(A^0\ten^R\g_0)\by_{\exp(A^0\ten^R\fh)_0} \exp(A^0\ten^R\fk)_0) $. 

The key observations to make are that surjectivity of $\g \to \fh$ implies surjectivity of $ \exp(A^0\ten^R\g_0)\to \exp(A^0\ten^R\fh_0)$, and that the action is faithful because the map $(d, \phi) \co A^0 \to A^1 \by O(R)$ has trivial kernel. The proof of \cite[Corollary \ref{paper1-keyhgs}]{paper1} then adapts to give the required result.
\end{proof}

\subsubsection{Evaluation maps and equivariant cochains}


For a group $\Gamma$, let $\bS(\Gamma)$ denote the category of $\Gamma$-representations in  simplicial sets.

\begin{definition}
Given a  commutative $k$-algebra $A$ and $X \in \bS( R(A))$, define $\CC^{\bt}(X,\bO(R)\ten A)$  to be the cosimplicial $R$-equivariant commutative $A$-algebra given by 
$$
\CC^n(X,\bO(R)\ten A):= \Hom_{R(A)}(X_n, A \ten O(R)).
$$
\end{definition}

\begin{lemma}\label{evalata}
Given a  commutative $k$-algebra $A$, the functor $  s\Aff(R)\da \Spec A \to \bS(R(A))$ given by $Y \mapsto Y(A)$ is right Quillen, with left adjoint
$
X \mapsto \Spec \CC^{\bt}(X,\bO(R)\ten A).
$
\end{lemma}
\begin{proof}
This is essentially the same as \cite[Lemma 3.52]{htpy}, which takes the case $A=k$.
\end{proof}

Recall from \cite[Lemma  VI.4.6]{sht} that there is a right Quillen equivalence  $\mathrm{ho}\!\varinjlim_{{R(A)}}\co \bS(R(A)\to\bS \da BR(A)$, with left adjoint given by the covering system functor $X \mapsto \widetilde{X}$.

\begin{definition}\label{Budef}
Given $f:X\to BR(A)$, define
$$
\CC^{\bt}(X,O(\Bu_f)):=\CC^{\bt}(\widetilde{X},\bO(R)\ten A).
$$
\end {definition}

\begin{lemma}\label{relevalata}
Given a $k$-algebra $A$, the functor $s\Aff(R)\da \Spec A \to \bS\da BR(A)$  given by $Y \mapsto \mathrm{ho}\!\varinjlim_{{ R(A)}} Y(A)$ is right Quillen,  with left adjoint
$$
(X\xra{f} BR(A)) \mapsto \Spec \CC^{\bt}(X,O(\Bu_f)).
$$
\end{lemma}
\begin{proof}
The
 functor $s\Aff(R)\da \Spec A \to \bS(R(A))$ given by $Y \mapsto Y(A)$ is right Quillen, with left adjoint as in Lemma \ref{evalata}.  Composing this  right Quillen functor
  with $\mathrm{ho}\!\varinjlim_{{ R(A)}}$ gives the right Quillen functor required.
\end{proof}

\begin{proposition}\label{eqhtpy}
The  Malcev homotopy type $G(X,x)^{R, \mal} $ of a  pointed, connected topological space $(X,x)$ relative to a Zariski-dense morphism $\rho:\pi_1(X,x) \to R(k)$ is canonically isomorphic to
\[
 \bar{G}_R(\CC^{\bt}( \Sing(X,x),O(\Bu_{\rho}))).
\]
\end{proposition}
\begin{proof}
Because $\Sing(X,x)_0=\{x\}$, we have $\CC^0( \Sing(X,x),O(\Bu_{\rho})= O(R)$, so
\begin{align*}
 &\exp(\fu_0)\by^{\exp(\CC^0( \Sing(X,x),O(\Bu_{\rho}))\hten^R\fu_0)}\mc( \CC^{\bt}( \Sing(X,x),O(\Bu_{\rho})) \hten^R\fu)\\
&= \exp(\fu_0)\by^{\exp(\fu_0)}\mc( \CC^{\bt}( \Sing(X,x),O(\Bu_{\rho})) \hten^R\fu)\\
&=\mc( \CC^{\bt}( \Sing(X,x),O(\Bu_{\rho})) \hten^R\fu).
\end{align*}

The proof of  \cite[Lemma \ref{htpy-eqhtpy}]{htpy}  uses the Quillen adjunctions above to give an isomorphism
\[
\Hom_{s\gp\da R(k))}(G(\Sing(X,x)), (R\ltimes\exp(\fu))(k))\cong \mc( \CC^{\bt}( \Sing(X,x),O(\Bu_{\rho})) \hten^R\fu),
\]
and we then have
\[
 G(X,x)^{R, \mal} \cong \bar{G}_R(\CC^{\bt}( \Sing(X,x),O(\Bu_{\rho}))),
\]
since both objects represent the same functor.
 \end{proof}
 
Note that because $\Sing(X,x) \to \Sing(X)$ is a weak equivalence, and $\bar{G}_R$ preserves weak equivalences, this induces a weak equivalence
\[
 G(X,x)^{R, \mal}\to \bar{G}_R( x^*\co \CC^{\bt}(X,O(\Bu_{\rho}))\to O(R)).
\]

\subsubsection{The Dold--Kan correspondence}

\begin{definition}\label{DKD}
Recall that the cosimplicial Dold--Kan correspondence gives a denormalisation functor $D$ from cochain complexes to cosimplicial abelian groups by setting 
$$
D^n(V)=\bigoplus_{\begin{smallmatrix} m+s=n \\ 1 \le j_1 < \ldots < j_s \le n \end{smallmatrix}} \pd^{j_s}\ldots\pd^{j_1}V^m,
$$
where we define the $\pd^j$ and $\sigma^i$ using the simplicial identities, subject to the conditions that  for all $v \in V^n$,\, $dv = \sum_{i=0}^{n+1}(-1)^i \pd^i v$ and  $\sigma^i v =0$.
This functor is quasi-inverse to the cosimplicial normalisation functor
$$
N_c^n(V):= \{v \in V^n \,:\, \sigma^iv =0 \quad \forall i\},\quad d = \sum_{i=0}^{n+1}(-1)^i \pd^i.
$$
\end{definition}
In particular, note that the cosimplicial Dold--Kan correspondence  can be interpreted as the simplicial Dold--Kan correspondence of Definition \ref{N^s} applied to the opposite category to abelian groups.

\begin{definition}
Given a bicosimplicial abelian group $V$, the Eilenberg--Mac Lane shuffle product $\nabla\co N_c(\diag V) \to \Tot (N_cV) $ is a quasi-isomorphism of cochain complexes given by summing:
$$
\nabla^{pq}=\sum_{ (\mu, \nu) \in \mathrm{Sh}(p,q)}(-1)^{(\mu,\nu)} \sigma^{\nu_1}_h\ldots \sigma^{\nu_q}_h\sigma_v^{\mu_1}\ldots \sigma_v^{\mu_p}\co  N_c(V^{p+q,p+q}) \to (N_cV)^{pq}
$$
This is associative and commutative.

The Dold--Kan correspondence then gives, for any bi-cochain complex $W$, a quasi-isomorphism
$$
\nabla:=D\nabla N_c\co \diag(DW) \to D(\Tot W)
$$
of cosimplicial abelian groups.
\end{definition}

\begin{definition}
Given a bicosimplicial abelian group $V$, the Alexander--Whitney cup  product $\smile \co \Tot (N_cV)\to N_c(\diag V)  $, from the total complex of the binormalisation to the normalisation of the diagonal, is a quasi-isomorphism of cochain complexes given by summing:
\[
 (\pd^{p+1}_h)^q (\pd^0_v)^p \co (N_cV)^{pq} \to V^{p+q,p+q}.
\]
This is associative but not commutative, and is a right inverse to $\nabla$.

The Dold--Kan correspondence then gives, for any bi-cochain complex $W$, a quasi-isomorphism
$$
\smile:=D\smile N_c\co  D(\Tot W)\to \diag(DW)
$$
of cosimplicial complexes.
\end{definition}

\begin{definition}
 Denormalisation gives a functor $D \co DG\Alg(R) \to c\Alg(R)$, sending $A$ to the cosimplicial $R$-representation $DA \to O(R)$, with the multiplication on $DA$ given by the composition
\[
 \diag(DA \ten DA) \xra{\nabla}  D\Tot(A \ten A) \to DA
\]
of the shuffle product with multiplication on $A$.
\end{definition}

\begin{definition}\label{Daffdef}
 Let $\Spec D \co dg\Aff(R)_* \to s\Aff(R)_*$ be the functor given by  $(\Spec D)(Y):= \Spec( DO(Y))$. We also define $\Spec D \co dg\cE(R) \to s\cE(R)$ by the same formula, with the comultiplication on $DO(G)$ given as the composition
\[
 DO(G) \to D\Tot(O(G)\ten O(G)) \xra{\smile} \diag(DO(G)\ten D(O(G)))
\]
of the  Alexander--Whitney cup product with the comultiplication on $O(G)$.
\end{definition}

\begin{remark}\label{D*}
 The denormalisation functor $D$ on CDGAs  has a left adjoint $D^*$. In general, this is difficult to describe, but if $C$ is of the form $C=B[V]$, for $B$ a commutative $k$-algebra and $V$ a cosimplicial diagram of $k$-vector spaces, then $D^*C= B[N_cV]$.
\end{remark}

The functor $D$ is a right Quillen equivalence, so a homotopy inverse is given by composing $D^*$ with a cofibrant replacement functor. In general, however, there is a more convenient choice of homotopy inverse:

\begin{definition}\label{Th}
Recall that the Thom--Sullivan (or Thom--Whitney) functor $\Th$ from cosimplicial algebras to DG algebras is defined as follows. Let $\Omega(|\Delta^n|)$ be the DG algebra of rational polynomial forms on the $n$-simplex, so
$$
\Omega(|\Delta^n|)=\Q[t_0, \ldots, t_n, dt_0, \ldots, dt_n]/(1-\sum_i t_i),
$$ for $t_i$ of degree $0$. The usual face and degeneracy maps for simplices yield $\pd_i:   \Omega(|\Delta^n|) \to \Omega(|\Delta^{n-1}|)$ and  $\sigma_i:   \Omega(|\Delta^n|) \to \Omega(|\Delta^{n-1}|)$, giving a simplicial  CDGA. Given a cosimplicial $\Q$-algebra $A$, we then set
$$
\Th(A):= \{a \in  \prod_n A^n\ten_{\Q}\Omega(|\Delta^n|)\,:\, \pd^i_A a_n = \pd_ia_{n+1},\, \sigma^j_Aa_n= \sigma_ja_{n-1}\, \forall i,j\}.
$$

This functor is denoted by $D$ in \cite[\S 5.2]{Hainhodge}.
\end{definition}

By \cite[4.1]{HinSch} and the proof of \cite[Proposition \ref{stacks2-Dequiv}]{stacks2}, integration gives a natural quasi-isomorphism 
\[
 \int \co D\Th(A) \to A
\]
for all commutative cosimplicial $\Q$-algebras $A$. Since $\oL D^*$ is a homotopy inverse to $D$, it follows that $\Th$ is a model for $\oL D^*$.

 The Eilenberg--Zilber shuffle product and Alexander--Whitney cup product give maps
\[
 \nabla \co \Tot N^s W \to N^s\diag W \quad \smile \co  N^s\diag W \to \Tot N^s W
\]
for any bisimplicial diagram $W$ in an abelian category.

Combining Proposition \ref{eqhtpy}, with this construction leads us to make the following definition:
\begin{definition}\label{XRmaldef}
 Given a reduced simplicial set $(X,x)$ and  a Zariski-dense morphism $\rho:\pi_1(X,x) \to R(k)$, define the relative Malcev homotopy type $(X,x)^{\rho, \mal}$ of $(X,x)$ relative to $\rho$ by
\[
 (X,x)^{\rho, \mal}:= ( R \xra{x} \Spec \Th\CC^{\bt}( \Sing(X),O(\Bu_{\rho}))) \in  dg\Aff(R)_{0*}.
\]
\end{definition}
Note that $DO(X,x)^{\rho, \mal}$ is weakly equivalent to $x^* \co \CC^{\bt}( \Sing(X),O(\Bu_{\rho}))\to O(R)$, so Proposition \ref{eqhtpy} gives
\[
 G(X,x)^{R, \mal} \simeq \bar{G}_R(DO(X,x)^{\rho, \mal}).
\]

\begin{definition}\label{N^sg}
The simplicial  Dold--Kan normalisation of Definition \ref{N^s} gives a functor 
\[
N^s \co s\cN(R) \to dg\cN(R), 
\]
where the Lie bracket on $N^s\g$ comes from composing $\nabla$ with the bracket on $\g$. On passing to pro-categories, this gives a functor
\[
 N^s \co s\hat{\cN}(R) \to dg\hat{\cN}(R).
\]
\end{definition}

\begin{definition}
Given $G \in s\agp$, define the  dg pro-algebraic  group $N^sG$ over $k$ by setting $O(N^sG)= D^*O(G)$, for $D^*$ as in Remark \ref{D*}. The comultiplication on $O(N^sG)$ is then defined using the fact that $D^*$ preserves coproducts, so $D^*(O(G)\ten O(G))= O(N^sG)\ten O(N^sG)$, where $ (O(G)\ten O(G))^n= O(G)^n\ten O(G)^n$, but    $(O(N^sG)\ten O(N^sG))^n= \bigoplus_{i+j=n}O(N^sG)^i\ten O(N^sG)^j$.
\end{definition}

\begin{definition}
 We define $N^s \co s\cE(R) \to dg\cE(R)$ to be the functor $N^s G= \Spec (D^* O(G))$. 
\end{definition}

Every object of $s\cE(R) $ is of the form $G= R \ltimes \exp(\g) $ for $\g \in s\hat{\cN}(R)$, so we have an isomorphism $O(G) \cong O(R)[\g^{\vee}]$ of cosimplicial algebras. Observe that  as in Remark \ref{D*}, we then  have $D^*O(G)\cong O(R)[(N^s\g)^{\vee}]$. Analysis of the comultiplication then gives
\[
 D^*O( R \ltimes \exp(\g)) \cong R \ltimes \exp(N^s\g),
\]
which explains the notation.

\begin{lemma}\label{DequivcE}
 The functor $N^s \co s\cE(R) \to dg\cE(R)$  is right adjoint to $\Spec D$, and the unit and counit of the adjunction are both quasi-isomorphisms. 
\end{lemma}
\begin{proof}
 Since $D^* \dashv D$ form a pair of equivalences between cochain and cosimplicial commutative algebras, we know that the unit   $A \to DD^*A$ is a quasi-isomorphism whenever $A$ is cofibrant, and that the co-unit $D^*DB \to B$ is a quasi-isomorphism whenever $DB$ is cofibrant. The description of Lemma \ref{expscNlemma} ensures that for  all objects $G$ of $s\cE(R)$, the cosimplicial Hopf algebra $O(G)$ is cofibrant as a cosimplicial commutative algebra. 
\end{proof}

\subsubsection{Loops and cochains}

\begin{proposition}\label{DGRcommute}
There is a commutative diagram
\[
 \begin{CD}
  dg\Aff(R)_* @>{\Spec D}>> s\Aff(R)_*\\
@V{\bar{G}_R}VV @VV{\bar{G}_R}V\\
dg\cE(R) @>{\Spec D}>> s\cE(R) 
 \end{CD}
\]
of functors. 
\end{proposition}
\begin{proof}
By \cite[Theorem \ref{monad-cfexp}]{monad}, we have canonical isomorphisms
\[
 \exp(\fu_0)\by^{\exp(A^0\hten^R\fu_0)}\mc(A\hten^RN^s\fu)\cong \exp(\fu_0)\by^{\exp(A^0\hten^R\fu_0)}\mc(DA\hten^R\fu)
\]
for any $A \to O(R)$ in $DG\Alg(R)_*$ and any $\fu \in s\hat{\cN}(R)$. Thus
\[
 \Hom_{dg\cE(R)}(\bar{G}_R(A \to O(R)), N^s(R \ltimes \exp(\fu))) \cong  \Hom_{s\cE(R)}(\bar{G}_R(DA \to O(R)), R \ltimes \exp(\fu)),
\]
so by adjunction
\[
 \Hom_{s\cE(R)}( (\Spec D)\bar{G}_R(A \to O(R)), R \ltimes \exp(\fu)) \cong \Hom_{s\cE(R)}(\bar{G}_R(DA \to O(R)), R \ltimes \exp(\fu)),
\]
and we have shown that
\[
 (\Spec D)\bar{G}_R(A \to O(R)) \cong \bar{G}_R(DA \to O(R)).
\]
\end{proof}

\begin{remark}\label{GXRmalrmk}
 For $(X,x)^{\rho,\mal}$ as in Definition \ref{XRmaldef}, we have seen  that $\bar{G}_R(DO(X,x)^{\rho,\mal}) \simeq G(X,x)^{R,\mal}$, and Proposition \ref{DGRcommute} then gives 
\[
(\Spec D) \bar{G}_R(O(X,x)^{\rho,\mal})\simeq G(X,x)^{R,\mal}.
\]
\end{remark}

\begin{definition}
Given a manifold $X$, denote the sheaf  of real $\C^{\infty}$ $n$-forms on $X$ by $\sA^n$. Given a real sheaf $\sF$ on $X$, write
$$
A^n(X,\sF):=\Gamma(X,\sF\ten_{\R} \sA^n).
$$ 
\end{definition}

\begin{proposition}\label{propforms}
For any pointed, connected  manifold $(X,x)$, there is a canonical chain of quasi-isomorphisms between $G(X,x)^{R, \mal} $ and
\[
 (\Spec D)\bar{G}_R(x^* \co A^{\bt}(X, O(\Bu_{\rho}))\to O(R)),
\]
where $O(\Bu_{\rho}) $ is the local system on $X$ corresponding to the $\rho(\pi_1(X,x))$-representation $O(R)$.
\end{proposition}
\begin{proof}
By Proposition \ref{eqhtpy}, we have a quasi-isomorphism
\[
 G(X,x)^{R, \mal}\to \bar{G}_R( x^*\co \CC^{\bt}(X,O(\Bu_{\rho}))\to O(R)),
\]
and we now observe that as in the proof of \cite[Proposition \ref{htpy-propforms}]{htpy}, we have quasi-isomorphisms
$$
\CC^{\bt}(X,\bO(R))\to \diag \CC^{\bt}(X,\bO(R)\ten_{\R}D\sA^{\bt}) \la \Gamma(X,\bO(R)\ten_{\R}D\sA^{\bt})=DA^{\bt}(X, \bO(R))
$$
in $c\Alg(R)_*$. Applying the functor $\bar{G}_R$ and using the substitution of Proposition \ref{DGRcommute} then gives quasi-isomorphisms
\begin{align*}
& \bar{G}_R( x^*\co \CC^{\bt}(X,O(\Bu_{\rho}))\to O(R))\\
& \to \bar{G}_R( x^*\co \CC^{\bt}(X,\bO(R)\ten_{\R}D\sA^{\bt}) \to O(R))\\
& \la (\Spec D)\bar{G}_R(x^*\co A^{\bt}(X, O(\Bu_{\rho}))\to O(R)).
\end{align*}
\end{proof}

\subsection{The reduced bar construction}\label{redbarsn}

We now show that the functor $\bar{G}_R$ on $R$-equivariant CDGAs is just given by the bar construction (Proposition \ref{repbar} below).

\begin{definition}\label{barBRdef}
Given $A \in DG\Alg(R)_{0*}$,  we form  the reduced bar complex
\[
\bar{B}_R(A \to O(R)):= \bar{B}_k(k, A, O(R))
\]
as in Definition \ref{barBdef},
where the $A$-module structure on $O(R)$ is given by the morphism $\phi\co A \to O(R)$, and the $A$-module structure on $k$ is given by combining this with the co-unit $O(R) \to k$. The complex $\bar{B}_R(A \to O(R))$ has the natural structure of a CDGA, by \cite[1.2.4]{Hainhodge}, and in fact $\bar{B}_R(A \to O(R))$ is a DG Hopf algebra, by \cite[Theorem 7.5]{narkawicz}. 

To understand the coalgebra structure, observe that as a graded algebra,
$\bar{B}_k(k, A, O(R))$ is defined as a quotient of $C \ten O(R)$, where $C$ is  the free graded coalgebra on cogenerators $A^+[1]$. 
Noting that $C \ten O(R)$ is cogenerated as a coalgebra by $A^+[1]\oplus O(R)$,  the Hopf algebra  $\bar{B}_R(A \to O(R))$ is then defined to be the quotient of the semi-direct tensor product $C \rtimes R$ from Definition \ref{semitendef} cogenerated by
\[
 \coker((d,\phi) \co A^0 \to A^1\oplus O(R)).
\]
\end{definition}

\begin{definition}\label{barwg}
Define a  functor $\bar{W}: dg\hat{\cN}(R) \to dg\Aff(R)_{*}$ by 
$
O(\bar{W}\g):= \Symm(\g^{\vee}[-1])
$
the graded polynomial ring on generators $\g^{\vee}[-1]$, with derivation defined on generators by $d_{\g} +\Delta$, for $\Delta$ the Lie cobracket on $\g^{\vee}$. Since $O(\bar{W}\g)^0=k$, there is a unique augmentation to $O(R)$.

The functor $\bar{W}$ has a left adjoint $G$, given  by writing $\sigma A^{\vee}[1]$ for the brutal truncation (in non-negative degrees) of $A^{\vee}[1]$, and setting
$$
G(A)= \Lie(\sigma A^{\vee}[1]),
$$
the free graded Lie algebra, with differential similarly defined on generators by $d_A +\Delta$, with $\Delta$ here being the coproduct on $A^{\vee}$. 
\end{definition}

Note that
$$
\Hom_{dg\Aff(R)_*}(\Spec A, \bar{W}\g) \cong \mc(A\hat{\ten}^R\g).
$$

\begin{remark}
 When  $A^0=k$, observe that $G(A)$ is just the dg Lie algebra $\bar{G}^k(A \to k)$ of Definition \ref{barg0}, given by the tangent space of the reduced bar construction, together with its induced $R$-action. 
\end{remark}

\begin{lemma}\label{cfbarlemma}
 For $A \in DG\Alg(R)$ with $A^0=k$, there is a canonical isomorphism
\[
 \Spec \bar{B}_R (A) \cong \exp(G(A)) \rtimes R 
\]
in $dg\cE(R)$.
\end{lemma}
\begin{proof}
Because $A^0=k$, the morphism $A \to O(R)$ factors as $A \to k \to O(R)$ so   we have an isomorphism 
\[
 \bar{B}_R(A \to O(R))= \bar{B}_k(A)\ten_k O(R)
\]
of DG algebras. The coalgebra structure on the reduced bar construction is given by semidirect tensor product, so we have a DG Hopf algebra isomorphism
\[
 \bar{B}_R(A \to O(R))= \bar{B}_k(A)\rtimes R.
\]

Comparing Definitions \ref{barg0} and \ref{barwg} then gives an $R$-equivariant isomorphism
\[
 O(\exp(G(A)) \cong  \bar{B}_k(A)
\]
of DG Hopf algebras, as required.
\end{proof}

\begin{proposition}\label{repbar}
There is a canonical natural isomorphism
\[
 \bar{G}_R \cong \Spec \bar{B}_R
\]
of functors from $dg\Aff_{0*}$ to $dg\cE(R)$.
\end{proposition}
\begin{proof}
   When $A^0=k$,  Lemma \ref{cfbarlemma} and Lemma \ref{expdgcNlemma} combine with the adjunction $G \dashv \bar{W}$ to imply that for $\g \in dg\cN(R)$ we have 
\begin{eqnarray*}
  \Hom_{dg\cE(R)}( \bar{G}_R (A), \exp(\g)\rtimes R)&\cong& \exp(\g_0)\by^{\exp(\g_0^R)}\Hom_{dg\hat{\cN}(R)}(G(A), \g) ,\\
&\cong& \exp(\g_0)\by^{\exp(\g_0^R)}\Hom_{dg\Aff(R)_*}(\Spec A, \bar{W}\g) ,\\
&\cong& \exp(\g_0)\by^{\exp(\g_0^R)}\mc(A\ten^R\g), \\
&\cong& \Hom_{dg\cE(R)}(G(A)\rtimes R, \exp(\g)\rtimes R),\\
&\cong& \Hom_{dg\cE(R)}(\Spec \bar{B}_R(A), \exp(\g)\rtimes R)
\end{eqnarray*}

For general $\phi \co A\to O(R)$  in $DG\Alg(R)_{0*}$, choose a decomposition $A^1= dA^0\oplus B^1$, and set $B=k$ with $B^n=A^n$ for all $n>1$. Then $B \to A$ is a quasi-isomorphism, so gives isomorphisms $\bar{B}_R(B) \to \bar{B}(\phi\co A \to O(R))$. The map $B \to A$ factors through $k \oplus A^+[1]$, so the map
\[
 \Hom_{dg\cE(R)}(\Spec \bar{B}_R(k \oplus A^+), \exp(\g)\rtimes R) \to \Hom_{dg\cE(R)}(\Spec \bar{B}_R(A), \exp(\g)\rtimes R)
\]
admits a section and is hence surjective.

Now, Lemma \ref{expdgcNlemma} adapts to show that the set of graded Hopf algebra morphisms $ O(R\rtimes \exp(\g)) \to \bar{B}_R(k \oplus A^+)$ is given by
\[
 \exp(\g_0)\by^{\exp(\g_0^R)}(A\ten^R\g)^1.
\]
Since  $\bar{B}_R(\phi \co A \to O(R))$ is the quotient of  $\bar{B}_R(k \oplus A^+)$ cogenerated by 
\[
 \coker( (\phi,d)\co   A^0 \to O(R)\oplus A^+[1]),
\]
it follows that graded Hopf algebra morphisms $ O(R\rtimes \exp(\g)) \to \bar{B}_R(\phi \co A \to O(R))$
are given by the quotient
\[
 \exp(\g_0)\by^{\exp(A^0\ten^R\g_0)}(A\ten^R\g)^1.
\]

Thus $\Hom_{dg\cE(R)}(\Spec \bar{B}_R(A), \exp(\g)\rtimes R) $ is the image of
\[
 \Hom_{dg\cE(R)}(\Spec \bar{B}_R(k \oplus A^+), \exp(\g)\rtimes R)\to \exp(\g_0)\by^{\exp(A^0\ten^R\g_0)}(A\ten^R\g)^1;
\]
in other words, the image of
\[
 \exp(\g_0)\by^{\exp(\g_0^R)}\mc(A\ten^R\g) \to \exp(\g_0)\by^{\exp(A^0\ten^R\g_0)}(A\ten^R\g)^1.
\]
Since $(\phi,d)\co   A^0 \to O(R)\oplus A^+[1])$ is injective, the action of $ \exp(A^0\ten^R\g_0)$ on $(A\ten^R\g)^1 $ is faithful, so this image is just
\[
 \exp(\g_0)\by^{\exp(A^0\ten^R\g_0}\mc(A\ten^R\g),
\]
and we have shown that $\Spec \bar{B}_R (\phi \co A \to O(R))$ represents the functor of Proposition \ref{barGRprop}, so is canonically isomorphic to  $\bar{G}_R(\phi \co A \to O(R))$.
\end{proof}

\begin{remark}\label{barfiltrn}
The bar filtration on $ \bar{B}_k(k, A, O(R))$ is given by setting $\sB_r\bar{B}_R(A \to O(R))$ to be the image of $\bigoplus_{n \le r} (A^+[1])^{\ten n}\ten O(R)$ under the quotient map above, and then we have
\[
 \gr^{\sB}_r \bar{B}_R(A \to O(R)) \cong (\bar{A}[1])^{\ten r} \ten O(R),
\]
which gives rise to the Eilenberg--Moore spectral sequence in the proof of Lemma \ref{barQIM}. 

In fact, we can recover the bar filtration from the coalgebra structure on $\bar{B}:=\bar{B}_R(A \to O(R))$ by noting that 
\[
 \sB_r\bar{B}= \Delta_r^{-1}( \sum_{i+j=r-1} \bar{B}^{\ten i} \ten O(R) \ten \bar{B}^{\ten j} \subset \bar{B}^{\ten r}),
\]
where $\Delta_r$ is $r$-fold comultiplication.
In particular, $\bar{A}[1] \cong \gr^{\sB}_1 \bar{B}_R(A \to O(R))$ is just the abelianisation of the pro-unipotent radical 
\[
 \Ru \bar{G}_R( A \to O(R)).
\]
\end{remark}

\begin{example}\label{kformalex}
When $X$ is a compact K\"ahler manifold, the formality isomorphism of \cite[Theorem \ref{higgs-formal}]{higgs} (or see Remark \ref{cfhiggs} below) gives a quasi-isomorphism
\[
 A^{\bt}(X, O(\Bu_{\rho})) \simeq \H^*(X, O(\Bu_{\rho})).
\]
Because $\H^0(X, O(\Bu_{\rho}))=\R $, we then have
\[
 \bar{G}_R( \H^*(X, O(\Bu_{\rho})))\cong R \ltimes \exp G \H^*(X, O(\Bu_{\rho})),
\]
combining Lemma \ref{cfbarlemma} and Proposition \ref{repbar}.

Thus Proposition \ref{propforms} shows that the relative Malcev homotopy type of  $X$ is given by the bar construction
\[
G(X,x)^{R, \mal} \simeq   R \ltimes (\Spec  D\bar{B}_{\R}\H^*(X, O(\Bu_{\rho}))).
\]

Much of this paper will involve understanding how this quasi-isomorphism interacts with the mixed Hodge structures.
\end{example}

\subsection{Equivalences of homotopy categories}\label{equivsubsn}

We will now combine our equivalences so far and establish further equivalences on the level of homotopy categories, culminating in Theorem \ref{bigequiv} below, an unpointed version of which appeared as \cite[Theorem \ref{htpy-bigequiv}]{htpy}. 

\begin{proposition}\label{calchom}
For $A \in DG\Alg(R)_*$ and $\g \in dg\hat{\cN}(R)$,
$$
\Hom_{\Ho(dg\Aff(R)_*)}(\Spec A,  \bar{W}\g) \cong \exp(\H_0 \g)\by^{\Gg(A\hat{\ten}^R\g)}\mc(A\hat{\ten}^R \g) ,
$$
where $\bar{W}\g \in  dg\Aff_* $ is the composition $R \to\Spec k\to \bar{W}\g$, and the morphism $\Gg(A\hat{\ten}^R\g) \to \exp(\H_0 \g)$ factors through $\Gg(O(R)\hat{\ten}^R \g)= \g_0$.
\end{proposition}
\begin{proof}
The derived $\Hom$ space
$
\oR \HHom_{dg\Aff(R)_*}(\Spec A,  \bar{W}\g) 
$
is the homotopy fibre of 
$$
  \oR\HHom_{dg\Aff(R)}(\Spec A,  \bar{W}\g)\to { \oR\HHom_{dg\Aff(R)}(R,  \bar{W}\g)} ,
$$
 over  the unique element $0$  of $\mc(O(R)\hat{\ten}^R\g)$. For a morphism $f:X \to Y$ of simplicial sets (or topological spaces), path components $\pi_0F$  of the homotopy fibre over $0 \in Y$ are given by pairs $(x, \gamma)$, for $x \in X$ and $\gamma$ a homotopy class of paths from $0$ to $fx$, modulo the equivalence relation $(x,\gamma) \sim (x', \gamma')$ if there exists a path $\delta: x \to x'$ in $X$ with $\gamma * f\delta= \gamma'$. If $Y$ has a unique vertex $0$, this reduces to pairs $(x, \gamma)$, for $x \in X$ and $\gamma \in \pi_1(Y,0)$, with $\delta$ acting as before.

Now, we can define an object $V\g \in dg\Aff(R)$ by 
$$
\Hom_{dg\Aff(R)}(\Spec A, V\g) \cong \Gg(A\hat{\ten}^R\g), 
$$
and by \cite[Lemma \ref{htpy-defworks}]{htpy}, $V\g \by\bar{W}\g$ is  a path object  for $\bar{W}\g$ in  $dg\Aff(R)$  via the maps
$$
\xymatrix@1{ \bar{W}\g \ar[r]^-{(\id,1)} &\bar{W}\g \by V\g  \ar@<0.5ex>[r]^-{\pr_1} \ar@<-0.5ex>[r]_-{\phi} & \bar{W}\g,}
$$
where $\phi$ is the gauge action.

Thus the loop object $\Omega(\bar{W}\g,0)$ for $0 \in \mc(A\hat{\ten}^R\g)$) is given by 
$$
\Hom_{dg\Aff(R)}(\Spec A, \Omega(\bar{W}\g,0)) = \{g \in \Gg(A\hat{\ten}^R\g)\,:\, g(0)=0\} = \exp (\ker d \cap \prod_n A^n\hat{\ten}^R\g_n)
$$
Hence
$$
\pi_i\oR\HHom_{dg\Aff(R)}(\Spec A, \Omega(\bar{W}\g,0)) \cong \H^{-i}(\prod_n A^n\hat{\ten}^R\g_n),
$$
and in particular,
$$
\pi_1(\oR\HHom_{dg\Aff(R)}(R,  \bar{W}\g),0)= \pi_0\oR\HHom_{dg\Aff(R)}(\Spec A, \Omega(\bar{W}\g,0)) \cong \exp(\H_0\g).
$$

This gives us a description of
$$
\Hom_{\Ho( dg\Aff(R)_*) }(\Spec A,  \bar{W}\g)= \pi_0\oR \HHom_{dg\Aff(R)_*}(\Spec A,  \bar{W}\g)
$$
as consisting of pairs $(x, \gamma)$ for $x \in \mc(A\hat{\ten}^R \g)$ and $\gamma \in \exp(\H_0\g)$, modulo the equivalence $(x,\gamma) \sim (\delta(x), \delta*\gamma)$ for $\delta\in \Gg(A\hat{\ten}^R\g)$. In other words,
$$
\Hom_{\Ho( dg\Aff(R)_*}(\Spec A,  \bar{W}\g) \cong \mc(A\hat{\ten}^R \g) \by^{\Gg(A\hat{\ten}^R\g)}\exp(\H_0 \g),
$$
as required.
\end{proof}

\begin{corollary}\label{WRcor}
There is a  functor $\bar{W}_R:\Ho_*(dg\cE(R)) \to \Ho(dg\Aff(R)_*)$ given on objects by 
\[
 \bar{W}_R(R \ltimes \exp(\g))= \bar{W}\g.
\]
\end{corollary}
\begin{proof}
Given a morphism  $f:\g \to \fh$ in $dg\hat{\cN}(R)$  and $h \in \H_0\fh$, we can use Proposition \ref{calchom} to define an element $\bar{W}_R(h,f)$ of
$$
\Hom_{\Ho(dg\Aff(R)_*)}(\bar{W}\g,  \bar{W}\fh)
$$ 
by $[(\exp(h),\bar{W}f)] \in \exp(\H_0 \g)\by^{\Gg(A\hat{\ten}^R\g)}\Hom_{dg\Aff(R) }(\bar{W}\g,  \bar{W}\fh)$. Lemma \ref{expdgcNlemma} then ensures that this defines a functor
\[
 \bar{W}_R\co dg\cE(R) \to \Ho(dg\Aff(R)_*).
\]

If $f$ is a weak equivalence then  $\bar{W}(h,f)$ is a weak equivalence in $dg\Aff(R)_*$, which implies that $\bar{W}_R$ must descend to a functor
$$
\bar{W}: \Ho_*(dg\cE(R)) \to \Ho(dg\Aff(R)_*),
$$
since $\bar{W}_R(h,f)$ depends only on the homotopy class of $f$. 
\end{proof}

\begin{theorem}\label{bigequiv}
We have the following commutative diagram of equivalences of categories:
$$
\xymatrix{
\Ho(dg\Aff(R)_*)_0 \ar@<1ex>[r]^{\Spec D}  \ar@<1ex>[d]^{\bar{G}_R} & \Ho(s\Aff(R)_*)_0 \ar@<1ex>[l]^{\Spec \Th} \ar@<1ex>[d]^{\bar{G}_R}  \\
\Ho_*(dg\cE(R)) \ar@<1ex>[u]^{\bar{W}_R}\ar@<1ex>[r]^{\Spec D}  &  \Ho_*(s\cE(R)) \ar@<1ex>[l]^{N^s}. 
}
$$
\end{theorem}
\begin{proof}
First, \cite[Propositions  \ref{htpy-affequiv} and \ref{htpy-nequiv}]{htpy} ensure that $\Spec D$ and $N^s$ are equivalences, and  \cite[4.1]{HinSch}  shows that $D$ and $\Th$ are homotopy inverses. We now adapt the proof of \cite[Corollary \ref{htpy-bigequiv}]{htpy}.

By \cite[Proposition \ref{htpy-wequiv}]{htpy}, there is a canonical quasi-isomorphism $\bar{W}_R \bar{G}_R(X) \to X$, for all $X \in dg\Aff(R)$ with $X_0=\Spec k$. As in the proof of Proposition \ref{repbar}, for any $A \in DG\Alg(R)_{0*}$, there exists a  quasi-isomorphism $B \to A$ with $B^0 =k$, which means that 
$$
\bar{W}_R: \Ho_*(dg\cE(R)) \to \Ho(dg\Aff(R)_*)_0
$$
is essentially surjective, with $\bar{G}_R(Y)$ in the essential pre-image of $Y$.

To establish that $\bar{W}_R$ is full and faithful, it will suffice to show that for all $A \in DG\Alg(R)$ with $A^0=k$, the transformation
$$
\Hom_{\Ho_*(dg\cE(R)) }(\bar{G}_R(A), R \ltimes \exp(\fh)) \to \Hom_{\Ho(dg\Aff(R)_*)}(\Spec A,  \bar{W}\fh)
$$
is an isomorphism. For $A=k$, this is certainly true, since in both cases we get $\exp(\H_0\fh)/\exp(\H_0\fh)^R$ for both $\Hom$-sets (using Proposition \ref{calchom}). The morphism $k \to A$ gives surjective maps from both $\Hom$-sets above to $\exp(\H_0\fh)/\exp(\H_0\fh)^R$, and by  Proposition \ref{calchom}, the map on any fibre is just
\begin{align*}
& \Hom_{\Ho(dg\hat{\cN}(R))}(G(A), \fh)/\exp(\ker(\fh_0^R \to \H_0\fh^R)) \\
& \xra{\theta}\mc(A\hat{\ten}^R \fh)/\ker(\Gg(A\hat{\ten}^R \fh) \to \exp(\H_0 \fh^R)).
\end{align*}

Now, $G(A)$ is a hull for both functors on $dg\cN(R)$ (in the sense of \cite[Proposition \ref{htpy-schrep}]{htpy}), so by the argument of \cite[Proposition \ref{htpy-qwell}]{htpy}, it suffices to show that $\theta$ is an isomorphism whenever $\fh \in \cN(R)$ (i.e. whenever $\fh_i=0$ for all $i>0$). In that case,
$$
\Hom_{\Ho(dg\hat{\cN}(R))}(G(A), \fh)= \Hom_{dg\hat{\cN}(R)}(G(A), \fh)= \mc(A\hat{\ten}^R\fh),
$$
 and 
$$
\Gg(A\hat{\ten}^R\fh)= \exp(A^0\hat{\ten}^R \fh)= \exp(k \hat{\ten}^R\fh)= \exp(\fh^R)= \exp(\fh_0^R), 
$$
so $\theta$ is indeed an isomorphism. Hence  $\bar{W}_R$ is an equivalence, with quasi-inverse $\bar{G}_R$.

Commutativity of the diagram is then given by Proposition \ref{DGRcommute}.
\end{proof}

\begin{remark}\label{basepts}
If we write $|X|$ for the set of points of $X$,  take a subset $T\subset |X|$, then the groupoid 
\[
 \Gamma:=T\by_{|X|}\pi_fX
\]
 has objects $T$, with morphisms $\Gamma(x,y)$ corresponding to homotopy classes of paths from $x$ to $y$ in $X$. If $T=\{x\}$, note that $\Gamma$ is just $\pi_1(X,x)$. 

Take  a reductive pro-algebraic groupoid $R$ (as in \cite[\S \ref{htpy-unptd}]{htpy}) on objects $T$, and a Zariski-dense morphism $\rho: \Gamma \to R$ preserving $T$.  The relative Malcev completion $G(X; T)^{\rho, \mal}$ is then a pro-unipotent extension of $R$ (as a simplicial pro-algebraic groupoid --- see \cite[\S \ref{htpy-sagpdsn}]{htpy}). Then, for  
\[
 \varpi_n(X;T)^{\rho, \mal}:= \pi_{n-1}G(X; T)^{\rho, \mal},
\]
it follows that 
$\varpi_1(X;T)^{\rho, \mal}$ is a groupoid on objects $T$. Moreover  $ \varpi_n(X;T)^{\rho, \mal}$ is a $(\varpi_1(X;T)^{\rho, \mal})^2$-representation under left and right multiplication, with  
\[
  \varpi_n(X;T)^{\rho, \mal}(x,x)= \varpi_1(X,x)^{\rho_x, \mal}.
\]
 Here, $\rho_x: \pi_1(X,x)\to R(x,x)$ is defined by restricting $\rho$ to $x \in T$.

If we set $dg\Aff(R)_*:= (\coprod_{x \in T} R(x,-)) \da dg\Aff(R)$ and $s\Aff(R)_*:=(\coprod_{x \in T} R(x,-))\da s\Aff(R)$, where $R(x,-)$ is the $R$-representation $y \mapsto R(x,y)$, then 
Theorem \ref{bigequiv}  adapts to this setting, \emph{mutatis mutandis}, using the constructions of Remark \ref{basepts}.

Proposition \ref{eqhtpy} also  adapts to say that the relative Malcev homotopy type $G(X;T)^{\rho, \mal}$ corresponds to the complex
$$
(\CC^{\bt}(X,O(\Bu_{\rho})) \xra{\prod_{x \in T} x^*} \prod_{x \in T} O(R)(x,-)) \in c\Alg(R)_{0*},
$$
and Proposition \ref{propforms} adapts to show that $(X;T)^{\rho,\mal}$ is given  by 
$$
(A^{\bt}(X, O(\Bu_{\rho})) \xra{\prod_{x \in T} x^*} \prod_{x \in T} O(R)(x,-)) \in DG\Alg(R)_*. 
$$

For $A \to \prod_{x\in T} O(R)(x,-)$ in $DG\Alg(R)_{0*}$, we can also define a DG Hopf algebroid $\bar{B}_R(A\to \prod_{x\in T} O(R)(x,-))$ to have objects $T$, with 
\[
 \bar{B}_R(A \to \prod_{x\in T} O(R)(x,-))(x,y):= \bar{B}_k(k,A_y, O(R)(x,y)), 
\]
where the morphism $A_y \to k$ is given by composing $A_y \to O(R)(y,y)$ with the co-unit of the Hopf algebra $O(R)(y,y)$.

The comultiplication 
\[
 \bar{B}_k(k,A_y, O(R)(x,z))\to \bar{B}_k(k,A_y, O(R)(x,y))\ten \bar{B}_k(k,A_y, O(R)(y,z)),
\]
antipode
\[
 \bar{B}_k(k,A_y, O(R)(x,y))\to \bar{B}_k(k,A_y, O(R)(y,x)) 
\]
 and co-units $\bar{B}_k(k,A_y, O(R)(x,x))\to k$
are then defined by analogous formulae to those in \cite[\S 7.3]{narkawicz}. 
  When $T=\{x\}$, this recovers Definition \ref{barBRdef}.
\end{remark}

\subsection{Families of homotopy types}\label{relhtpy}

We will often need to consider families of homotopy types parametrised by affine schemes, so we now introduce the necessary concepts.

\begin{lemma}\label{modmodel}
For  an $R$-representation $A$ in commutative DG algebras, there is a cofibrantly generated model structure on the category  $DG_{\Z}\Mod_A(R)$ of $R$-representations in  $\Z$-graded cochain $A$-modules, in which fibrations are surjections, and weak equivalences are isomorphisms on cohomology. 
\end{lemma}
\begin{proof}
Let $S(n)$ denote the cochain complex    $A[-n]$. Let $D(n)$ be the cone complex of $\id:A[1-n]\to A[1-n]$, so the underlying graded vector space is just $A[1-n] \oplus A[-n]$. 

Define $I$ to be the set of canonical maps $S(n)\ten V \to D(n)\ten V $, for $n \in \Z$ and $V$ ranging over all finite-dimensional $R$-representations. Define $J$ to be the set of morphisms $0 \to D(n)\ten V $, for $n \in \Z$ and $V$ ranging over all finite-dimensional $R$-representations.  Then we have a cofibrantly generated model structure, with $I$ the generating cofibrations and $J$ the generating trivial cofibrations, by verifying the conditions of \cite[Theorem 2.1.19]{Hovey}.
\end{proof}

\begin{definition}
Let $DG_{\Z}\Alg(R)$ be the category of $R$-representations in  $\Z$-graded   cochain $k$-algebras. 
For  an $R$-representation $A$ in commutative algebras,   we define   $DG_{\Z}\Alg_A(R)$ to be the comma category  $A\da DG_{\Z}\Alg(R)$. Denote the opposite category  by $dg_{\Z}\Aff_A(R)$. We will also sometimes write this as  $dg_{\Z}\Aff_{\Spec A}(R)$.
\end{definition}

\begin{lemma}\label{algmodel}
There is a cofibrantly generated model structure on $DG_{\Z}\Alg_A(R)$, in which fibrations are surjections, and weak equivalences are quasi-isomorphisms.
\end{lemma}
\begin{proof}
This follows by applying \cite[Theorem 11.3.2]{Hirschhorn} to the forgetful functor $DG_{\Z}\Alg_A(R) \to DG_{\Z}\Mod_{\Q}(R)$, since the left adjoint preserves quasi-isomorphisms between cofibrant objects.
\end{proof}

\subsubsection{Derived pullbacks and base change}

\begin{definition}
Given a morphism $f:X \to Y$ in $dg\Aff(R)$, the pullback functor $f^*:  DG_{\Z}\Alg_Y(R) \to DG_{\Z}\Alg_X(R)$ 
is left Quillen, with right adjoint $f_*$. Denote the derived left Quillen functor by $\oL f^*: \Ho(DG_{\Z}\Alg_Y(R)) \to \Ho(DG_{\Z}\Alg_X(R))$. 
Observe that $f_*$ preserves weak equivalences, so the derived right Quillen functor is just $\oR f_*=f_*$.
Denote the functor opposite to $\oL f^*$ by $\by_{Y}^{\oR}X: \Ho(dg_{\Z}\Aff_Y(R))\to \Ho(dg_{\Z}\Aff_X(R))$.
\end{definition}

\begin{lemma}\label{flat}
If $f:\Spec B \to \Spec A$ is a flat morphism in $\Aff(R)$, then $\oL f^*=f^*$.
\end{lemma}
\begin{proof}
This is just the observation that flat pullback preserves weak equivalences. $\oL f^* C$ is defined to be $f^*\tilde{C}$, for $\tilde{C} \to C$ a cofibrant replacement, but we then have $f^*\tilde{C} \to f^*C$ a weak equivalence, so $\oL f^* C=f^*C$.
\end{proof}

\begin{proposition}\label{nicepullback}
If $S \in DG_{\Z}\Alg_A(R)$, and $f:A \to B$ is any morphism in  $DG\Alg(R)$, then cohomology of $\oL f^*S$ is given by the hypertor groups
$$
\H^i( \oL f^* S) = \mathbf{Tor}_{-i}^A(S,B).
$$ 
\end{proposition}
\begin{proof}

Take   a cofibrant replacement $C\to S$, so $\oL f^*S \cong f^*C$. Thus $A \to C$ is a retraction of a transfinite composition of pushouts of generating cofibrations. The generating cofibrations are  filtered direct limits of projective  bounded complexes, so $C$ is  a retraction of a filtered direct limit of projective bounded cochain complexes.  Since cohomology and hypertor both commute with filtered direct limits (the latter following since we may choose a Cartan-Eilenberg resolution of the colimit in such a way that it is a colimit of Cartan-Eilenberg resolutions of the direct system), we may apply \cite[Application 5.7.8]{W} to see that $C$ is a  resolution computing the hypertor groups of $S$. 
\end{proof}

\begin{proposition}\label{pullbackflatobj}
If $S \in DG_{\Z}\Alg_A(R)$ is flat, and $f:A \to B$ is any morphism in  $\Alg(R)$, with either $S$ bounded or $f$ of finite flat dimension, then
$$
\oL f^*S \simeq f^*S.
$$
\end{proposition}
\begin{proof}
If $S$ is bounded, then $\oL f^*S \simeq S\ten^{\oL}_AB$, which is just $S\ten_AB$ when $S$ is also flat. If instead $f$ is of finite flat dimension, then \cite[Corollary 10.5.11]{W} implies that $\H^*(S\ten_AB) = \mathbf{Tor}_{-*}^A(S,B)$, as required.
\end{proof}

\begin{definition}\label{hoc*}
Given an $R$-representation $Y$ in schemes,   define  $DG_{\Z}\Alg_{Y}(R)$ to be the category of $R$-equivariant quasi-coherent $\Z$-graded cochain algebras on $Y$. Define a weak equivalence in this category to be a map giving isomorphisms on cohomology sheaves (over $Y$), and define    $\Ho(DG_{\Z}\Alg_{Y}(R))$ to be the homotopy category obtained by localising at weak equivalences. Define the categories $dg_{\Z}\Aff_Y(R), \Ho(dg_{\Z}\Aff_Y(R))$ to be the respective opposite categories. Let $ DG\Alg_{Y}(R)\subset DG_{\Z}\Alg_{Y}(R) $ consist of cochain algebras $A$ concentrated in non-negative cochain degrees.
\end{definition}

\begin{definition}
Given a quasi-compact, quasi-affine scheme $X$, let $j:X \to \bar{X}$ be the open immersion $X \to \Spec \Gamma(X,\O_X)$. Take a resolution $\O_X \to \sC_X^{\bt}$ of $\O_X$ in $DG_{\Z}\Alg_{X}(R)$, flabby with respect to Zariski cohomology (for instance by applying the Thom-Sullivan functor $\Th$ to the cosimplicial algebra $\check{\sC}^{\bt}(\O_X)$ arising from a \v Cech resolution). Define $\oR j_*\O_X$ to be $j_*\sC_X^{\bt} \in  DG_{\Z}\Alg_{\bar{X}}(R)$. 
\end{definition}

\begin{proposition}\label{quaffworks1}\label{othermodelc*}
 The functor $ j^*: DG_{\Z}\Alg_{\oR j_*\O_X}(R) \to DG_{\Z}\Alg_{X}(R)$ induces an equivalence $\Ho(DG_{\Z}\Alg_{\oR j_*\O_X}(R)) \to \Ho(DG_{\Z}\Alg_{X}(R))$.

For any $R$-representation $B$ in algebras, this extends to an equivalence
$\Ho(DG_{\Z}\Alg_{\oR j_*\O_X}(R)\da \oR j_*\O_X\ten B) \to \Ho(DG_{\Z}\Alg_{X}(R)\da \O_X\ten B)$.
\end{proposition}
\begin{proof}
Since $j$ is flat, $j^*$ preserves quasi-isomorphisms, so $j^*$ descends to a morphism of homotopy categories. If $\sC_X^{\bt} = \Th\check{\sC}^{\bt}(\O_X)$, then a 
 quasi-inverse functor will be given by $\sA \mapsto j_*\Th\check{\sC}^{\bt}(\sA)$. The inclusion $\sA \to \Th\check{\sC}^{\bt}(\sA)$  is a quasi-isomorphism, as is the map $j^*j_*\Th\check{\sC}^{\bt}(\sA) \to \Th\check{\sC}^{\bt}(\sA)$, 
since
$$
\sH^i(j^*j_*\Th\check{\sC}^{\bt}(\sA))= j^*\oR^ij_*(\sA)= \sH^i(\sA),
$$
as $j^*\oR^ij_*\sF=0$ for $i>0$ and $\sF$ a quasi-coherent sheaf (concentrated in degree $0$), $X$ being quasi-affine. 

Now, the composite morphism 
$$
 \oR j_*\O_X\to j_*j^*(\oR j_*\O_X) \to j_*\Th\check{\sC}^{\bt}(j^*(\oR j_*\O_X))
$$
is a quasi-isomorphism, since $j^*(\oR j_*\O_X)\to \O_X$ is a quasi-isomorphism. Cofibrant objects $\sM \in DG_{\Z}\Mod_{\oR j_*\O_X}(R)$ are retracts of $I$-cells, which admit (ordinal-indexed) filtrations whose graded pieces are copies of $(\oR j_*\O_X)[i]$, so we deduce that for cofibrant modules $\sM$, the map
$$
\sM \to j_*\Th\check{\sC}^{\bt}(j^*\sM)
$$
is a quasi-isomorphism. Since cofibrant algebras are a fortiori cofibrant modules, $\sB \to j_*\Th\check{\sC}^{\bt}(j^*\sB)$ is a quasi-isomorphism for all cofibrant $\sB \in DG_{\Z}\Alg_{\oR j_*\O_X}(R)$, which completes the proof in the case when $\sC_X^{\bt} = \Th\check{\sC}^{\bt}(\O_X)$.

For the general case, note that we have quasi-isomorphisms $\Th\check{\sC}^{\bt}(\O_X) \to \Th\check{\sC}^{\bt}(\sC_X^{\bt}) \la \sC_X^{\bt}$, giving quasi-isomorphisms $j_*\Th\check{\sC}^{\bt}(\O_X) \to j_*\Th\check{\sC}^{\bt}(\sC_X^{\bt}) \la j_*\sC_X^{\bt}$, and hence right Quillen equivalences
$$
 DG_{\Z}\Alg_{j_*\Th\check{\sC}^{\bt}(\O_X)}(R) \la DG_{\Z}\Alg_{j_*\Th\check{\sC}^{\bt}(\sC_X^{\bt})}(R)  \to  DG_{\Z}\Alg_{j_*\sC_X^{\bt}}(R).  
$$
\end{proof}

\begin{lemma}\label{stackqcoh}
Let $G$ be an affine group scheme, with a reductive subgroup scheme $H$ acting on a reductive pro-algebraic group $R$. Then the  model categories $dg_{\Z}\Aff_{G}(R \rtimes H)$ and $dg_{\Z}\Aff_{G/H}(R)$ are equivalent.
\end{lemma}
\begin{proof}
This is essentially the observation that $H$-equivariant quasi-coherent sheaves on $G$ are equivalent to quasi-coherent sheaves on $G/H$. Explicitly, define $U:dg_{\Z}\Aff_{G/H}(R) \to dg_{\Z}\Aff_{G}(R \rtimes H)$ by $U(Z)= Z\by_{G/H}G$. This has a left adjoint $F(Y)=Y/H$. We need to show that the unit and co-unit of this adjunction are isomorphisms.

The co-unit is given on $Z \in dg_{\Z}\Aff_{G/H}(R)$ by
$$
Z \la FU(Z)= (Z\by_{G/H}G)/H \cong Z\by_{G/H}(G/H) \cong Z,
$$
so is an isomorphism.

The unit is
$
Y \to UF(Y)= (Y/H)\by_{G/H}G,
$
for $Y \in dg_{\Z}\Aff_{G}(R \rtimes H)$.
Now, there is an isomorphism $Y\by_{G/H}G \cong Y \by H$, given by $(y, \pi(y)\cdot h^{-1}) \mapsfrom (y,h)$, for $\pi:Y \to G$. This map is $H$-equivariant for the left $H$-action on  $Y\by_{G/H}G$, and the diagonal $H$-action on $Y \by H$.
Thus
$$
UF(Y)= (Y\by_{G/H}G)/(H \by 1) \cong (Y \by H)/H \cong Y,
$$
with the final isomorphism given by $(y,h) \mapsto y\cdot h^{-1}$.
\end{proof}

\subsubsection{Extensions}

\begin{definition}
Given $B \in DG_{\Z}\Alg_{A}(R)$, define the   cotangent complex  
$$
\bL_{B/A}^{\bt} \in  \Ho(DG_{\Z}\Mod_{B}(R))
$$
   by taking a factorisation  $A \to C \to B$, with $A \to C$ a cofibration and $C \to B$ a trivial fibration.  Then set $ \bL_{B/A}^{\bt}:= \Omega^{\bt}_{C/A}\ten_C B =I/I^2$, where $I = \ker(C\ten_AB \to B)$. Note that $ \bL_{B/A}^{\bt}$ is independent of the choices made, as it can be characterised as the evaluation at $B$ of the derived left adjoint to the functor $M \mapsto B\oplus M$ from DG $B$-modules to $B$-augmented DG algebras over $A$.  
\end{definition}

\begin{lemma}\label{extchar}
Given a surjection $A \to B$ in  $DG_{\Z}\Alg(R)$, with square-zero kernel $I$, and a morphism $f:T \to C$ in $DG_{\Z}\Alg_A(R)$, the hyperext group
$$
\EExt^1_{T,R}(\bL_{T/A}^{\bt}, T\ten^{\oL}_AI \xra{f} C\ten^{\oL}_AI )
$$
of the cone complex
is naturally isomorphic to the set of
 weak equivalence classes of triples $(\theta, f', \gamma)$, where $\theta: T'\ten^{\oL}_AB \to T\ten^{\oL}_AB$ is a weak equivalence, $f':T' \to C$ a morphism,  and  $\gamma$ a homotopy between the morphisms $(f\ten_AB) \circ \theta, \,(f'\ten_AB):  T'\ten^{\oL}_AB \to C\ten^{\oL}_AB$.
\end{lemma}
\begin{proof}
This is a slight generalisation of a standard result, and we now sketch a proof. Assume that $A \to T$ is a cofibration, and that $T \to C$ is a fibration (i.e. surjective). We first consider  the case $\gamma=0$, considering   objects $T'$ (flat over $A$) such that $\theta:T'\ten_AB \to T\ten_AB$ is an isomorphism and  $(f\ten_AB) \circ \theta =(f'\ten_AB)$. 

Since $T$ is cofibrant over $A$,  the underlying graded ring $UT$ is a  retract of a polynomial ring, so $UT'\cong UT$. The problem thus reduces to deforming the differential $d$ on $T$. If we denote the differential of $T'$ by $d'$, then fixing an identification $UT=UT'$ gives $d'=d+\alpha$, for $\alpha: UT \to UT\ten_AI[1]$ a derivation with $d\alpha+\alpha d =0$. In order for $f:T' \to C$ to be a chain map, we also need $f\alpha=0$. Thus
$$
\alpha \in \z^1\HOM_{T,R}(\Omega_{T/A}, \ker(f)\ten_AI),
$$   
where $\HOM(U,V)$ is the $\Z$-graded cochain complex given by setting $\HOM(U,V)^n$ to be the space of graded morphisms $U \to V[n]$ (not necessarily respecting the differential).

Another choice of isomorphism $UT\cong UT'$ (fixing $T\ten_AB$) amounts to giving a derivation $\beta: UT \to UT\ten_AI$, with $\id +\beta$ the corresponding automorphism of $UT$. In order to respect the augmentation $f$, we need $f\beta=0$. This new choice of isomorphism sends $\alpha$ to $\alpha+d\beta$, so the isomorphism class is
$$
[\alpha] \in \EExt^1_{T,R}(\Omega_{T/A}, \ker(f)\ten_AI).
$$

Since $A\to T$ is a cofibration and $f$ a fibration, this is just hyperext
$$
\EExt^1_{T,R}(\bL_{T/A}^{\bt}, T\ten^{\oL}_AI \xra{f} C\ten^{\oL}_AI )
$$
of the cone complex.
Since this expression is invariant under weak equivalences, it follows that it gives the weak equivalence class required. 
\end{proof}

\subsubsection{The reduced bar construction}

We now extend Definition \ref{barBRdef} to more general bases:
\begin{definition}\label{barBRYdef}
Take a $k$-scheme $Y$ and  $A \in DG\Alg_Y(R)$  with $\H^0A= \O_Y$, together with an augmentation $\phi\co A^0 \to O(R)\ten_k\O_Y$, and assume that $\bar{A}$ is flat over $Y$, so $A^n$ is flat for all $n>1$, and that $A^1/dA^0$ is also flat.

We then form  the sheaf 
\[
\bar{B}_{R,Y}(A \to O(R)):= \bar{B}_{\O_Y}(\O_Y, A, O(R)\ten_k\O_Y)
\]
of reduced bar complexes as in Definition \ref{barBdef},
where the $A$-module structure on $O(R)$ is given by the morphism $\phi$, and the $A$-module structure on $k$ is given by combining this with the co-unit $O(R) \to k$. 
\end{definition}

Again, The complex $\bar{B}_{R,Y}(A \to O(R))$ has the natural structure of a sheaf of CDGAs on $Y$, by \cite[1.2.4]{Hainhodge}, and  $\bar{B}_{R,Y}(A \to O(R))$ is a DG Hopf $\O_Y$-algebra by the formulae of \cite[\S 7.3]{narkawicz}. 
By Lemma \ref{barQIM}, this construction preserves quasi-isomorphisms.

We now consider a partial analogue of Proposition \ref{repbar} when working over an affine $k$-scheme $Y$. Take an affine group scheme $G$ over $Y$, together with a surjection $G \to R\by Y$ with pro-(smooth unipotent) kernel $U$. For the lower central series $\{[U]_n\}_n$ of $U$ given by $[U]_1:=U$ and $[U]_{n+1}:= [U, [U_n]]$, this means that the quotients $[U]_n/[U]_{n+1}$ are filtered inverse limits of vector bundles on $Y$.

Because $Y$ is affine and $[U]_n/[U]_{n+1}$ is dual to a projective $\O_Y$-module, the argument of \cite[Proposition \ref{htpy-leviprop}]{htpy} adapts to give a  section $\sigma_G \co R \by Y \to G$, and hence a decomposition $G\cong (R \by Y) \ltimes U$ of group schemes over $Y$. If  a reductive group scheme $S'$ acts on $G$, then $\sigma_G$ can also be chosen to be $S'$-equivariant. Since $U$ is pro-unipotent, it takes the form $U= \exp( \Hom_{\O_Y}(\fu^{\vee}, - ))$, for an ind-conilpotent  Lie $\O_Y$-coalgebra $\fu^{\vee}$, projective as an $\O_Y$-module. 

\begin{definition}\label{defDel}
Given a pro-nilpotent DG Lie algebra $L^{\bt}$ in non-negative cochain degrees, define the Deligne groupoid $\Del(L)$ to have objects $\mc(L) \subset L^1$ (see Definition \ref{mcDG}), with morphisms  $\omega \to \omega'$ consisting of $g \in \Gg(L)=\exp(L^0)$ with $g* \omega = \omega'$, for the gauge action of Definition \ref{dgdefgauge}.
\end{definition}

\begin{proposition}\label{repbarY}
 Take an affine $k$-scheme $Y$, an augmented CDGA $A\xra{\phi} O(R)\ten_k\O_Y $ as in Definition \ref{barBRYdef}  and an $R$-equivariant pro-unipotent group scheme $U= \exp(\fu)$ over $Y$. If $A$ is quasi-isomorphic to a flat  CDGA $A'$ with $(A')^0=\O_Y$, then there is a natural isomorphism
 \[
  \Hom_{\gp\Aff_Y\da (R \by Y)}(\Spec \H^0\bar{B}_{R,Y}(A), (R \by Y) \ltimes U) \cong U(O(Y))\by^{U(A^0)^R}\mc( \Hom_{\O_Y}(\fu^{\vee}, A)^R), 
 \]
where $\gp\Aff_Y$ is the category of affine group schemes over $Y$.
 \end{proposition}
\begin{proof}
 If $A^0=\O_Y$, then the proof of Lemma \ref{cfbarlemma} adapts to give $\Spec \bar{B}_{R,Y} (A) \cong  (R \by Y)\ltimes (\Spec \bar{B}_{Y} (A) )$, and the proof of Proposition \ref{repbar} adapts to give the isomorphism above, noting that a morphism $O(G) \to \bar{B}_{R,Y}(A) \to G$ of DG Hopf algebras necessarily factors through $\H^0\bar{B}_{R,Y}(A)$ when $O(G)$ is concentrated in degree $0$.
 
 In general, we thus know that the isomorphism holds for $A'$, and (since $A' \to A$ is a quasi-isomorphism) that $\H^0\bar{B}_{R,Y}(A')\cong \H^0\bar{B}_{R,Y}(A)$. It thus suffices to show that 
\[
 F(\fu,A):= U(O(Y))\by^{U(A^0)^R}\mc( \Hom_{\O_Y}(\fu^{\vee}, A)^R)
\]
is invariant under quasi-isomorphisms in $A$. 

Now observe that $F(\fu,A)$ is 
isomorphic to the fibre of the morphism
$$
\Del( \Hom_{\O_Y}(\fu^{\vee}, A)^R ) \xra{\phi}
\Del(\Hom_{\O_Y}(\fu^{\vee}, \O_Y\ten_kO(R))^R)= \Del(\Hom_{\O_Y}(\fu^{\vee}, \O_Y )) 
$$
of groupoids.  Since $\fu^{\vee}$ is projective, the map $\Hom_{\O_Y}(\fu^{\vee}, A') \to \Hom_{\O_Y}(\fu^{\vee}, A)$ is a quasi-isomorphism of pro-nilpotent  DGLAs, so 
\[
 \Del( \Hom_{\O_Y}(\fu^{\vee}, A')^R )\to \Del(\Hom_{\O_Y}(\fu^{\vee}, A))
\]
is an equivalence of groupoids by \cite[Theorem 2.4]{GM}, and hence $F(\fu,A')\to  F(\fu,A)$ is an isomorphism.
\end{proof}

\section{Structures  on relative Malcev homotopy types}\label{malstr}

Now, fix a real reductive pro-algebraic group $R$, a pointed connected topological space  $(X,x)$, and a Zariski-dense  morphism $\rho:\pi_1(X,x) \to R(\R)$.

\begin{definition}
Given a pro-algebraic group $K$ acting on $R$ and on a scheme $Y$, define $dg_{\Z}\Aff_{Y}(R)_*(K)$ to be the category $(Y \by R) \da dg_{\Z}\Aff_{Y}(R\rtimes K)$ of objects under $Y \by R$. Beware that this is not the same as
$
dg_{\Z}\Aff_{Y}(R\rtimes K)_* = (Y \by R \rtimes K) \da dg_{\Z}\Aff_{Y}(R\rtimes K).
$
\end{definition}

\subsection{Homotopy types}

Motivated by Definitions \ref{hfildef}, \ref{tfildef}, \ref{mhsdef} and \ref{mtsdef}, we make the following definitions:

\begin{definition}\label{alghfildef}
An algebraic  Hodge filtration on a pointed Malcev homotopy type $(X,x)^{\rho, \mal} \in $ consists of the following data:
\begin{enumerate}
\item an algebraic action of $S^1$ on $R$,
\item an object $(X,x)_{\bF}^{\rho, \mal} \in \Ho( dg_{\Z}\Aff_{C^*}(R)_*(S))$, where the $S$-action on $R$ is defined via the isomorphism $S/\bG_m \cong S^1$, while the $R\rtimes S$-action on $R$ combines multiplication by $R$ with conjugation by $S$.
\item an isomorphism $(X,x)^{\rho, \mal} \cong  (X,x)_{\bF}\by_{C^*, 1}^{\oR}\Spec \R \in \Ho(dg_{\Z}\Aff(R)_*)$.
\end{enumerate}
\end{definition}

Thus $(X,x)_{\bF}^{\rho, \mal}$ consists of an $ R\rtimes S$-equivariant dg scheme $X_{\bF}^{\rho, \mal}$ over $C^*$, together with an  $ R\rtimes S$-equivariant map $x \co R \by C^* \to X_{\bF}^{\rho, \mal} $ over $C^*$, where $R$ acts on itself by multiplication.

Note that under the equivalence $dg_{\Z}\Aff(R)\simeq dg_{\Z}\Aff_S(R \rtimes S)$ of Lemma \ref{stackqcoh}, $(X,x)^{\rho, \mal}$ corresponds to the flat pullback $(X,x)_{\bF}\by_{C^*}S$.

\begin{definition}
An algebraic  twistor filtration on a pointed Malcev homotopy type $(X,x)^{\rho, \mal}$ consists of the following data:
\begin{enumerate}
\item an object $(X,x)_{\bT}^{\rho, \mal} \in \Ho( dg_{\Z}\Aff_{C^*}(R)_*(\bG_m))$,
\item an isomorphism $(X,x)^{\rho, \mal} \cong  (X,x)_{\bT}^{\rho, \mal}\by_{C^*, 1}^{\oR}\Spec \R \in \Ho(dg_{\Z}\Aff(R)_*)$.
\end{enumerate}
\end{definition}

Thus $(X,x)_{\bT}^{\rho, \mal}$ consists of an $ R\by \bG_m$-equivariant dg scheme $X_{\bT}^{\rho, \mal}$ over $C^*$, together with an  $ R\by \bG_m$-equivariant map $x \co R \by C^* \to X_{\bT}^{\rho, \mal} $ over $C^*$.

Note that under the equivalence $dg_{\Z}\Aff(R)\simeq dg_{\Z}\Aff_{\bG_m}(R \by \bG_m)$ of Lemma \ref{stackqcoh}, $(X,x)^{\rho, \mal}$ corresponds to 
the  derived pullback $(X,x)_{\bT}^{\rho, \mal}\by_{C^*}^{\oR}\bG_m$. 

\begin{definition}\label{algmhsdef}
Given a pointed relative Malcev homotopy type $(X,x)^{\rho, \mal}$ as in Definition \ref{XRmaldef},
an algebraic  mixed Hodge structure $(X,x)_{\MHS}^{\rho, \mal}$ on  $(X,x)^{\rho, \mal}$  consists of the following data:
\begin{enumerate}
\item an algebraic action of $S^1$ on $R$,
\item an object 
$$
(X,x)_{\MHS}^{\rho, \mal} \in \Ho(dg_{\Z}\Aff_{\bA^1 \by C^*}(R)_*(\bG_m \by S)),
$$
where $S$ acts on $R$ via the $S^1$-action, using the canonical isomorphism $S^1 \cong S/\bG_m$, 
\item an object 
$$
\ugr (X,x)_{\MHS}^{\rho, \mal} \in \Ho(dg_{\Z}\Aff(R)_*(S)),
$$
\item an isomorphism $(X,x)^{\rho, \mal} \cong  (X,x)_{\MHS}^{\rho, \mal} \by_{(\bA^1\by C^*), (1,1)}^{\oR}\Spec \R \in \Ho(dg_{\Z}\Aff(R)_*)$,
\item an isomorphism (called the opposedness isomorphism)
$$
\theta^{\sharp}(\ugr (X,x)_{\MHS}^{\rho, \mal}) \by C^* \cong  (X,x)_{\MHS}^{\rho, \mal} \by_{\bA^1, 0}^{\oR}\Spec \R \in \Ho( dg_{\Z}\Aff_{C^*}( R)_*(\bG_m \by S)),
$$
 for the canonical map $\theta: \bG_m \by S \to S$ given by combining the inclusion $\bG_m \into S$ with the identity on $S$. 
\end{enumerate}
\end{definition}

Thus $(X,x)_{\MHS}^{\rho, \mal}$ consists of an $\bG_m \by R\rtimes S$-equivariant dg scheme $X_{\MHS}^{\rho, \mal}$ over $\bA^1 \by C^*$, together with a  $\bG_m \by R\rtimes S$-equivariant map $x \co \bA^1 \by R \by C^* \to X_{\MHS}^{\rho, \mal} $ over $\bA^1 \by C^*$, while $\ugr (X,x)_{\MHS}^{\rho, \mal}$ is an $R\rtimes S$-equivariant dg scheme $\ugr X_{\MHS}^{\rho, \mal}$ equipped with  an $R\rtimes S$-equivariant map $x \co R \to \ugr X_{\MHS}^{\rho, \mal}$.

\begin{definition}
Given an algebraic mixed Hodge structure $(X,x)_{\MHS}^{\rho, \mal}$ on $(X,x)^{\rho, \mal}$, define
\[
 \gr^W(X,x)_{\MHS}^{\rho, \mal}:= (X,x)_{\MHS}^{\rho, \mal}\by_{\bA^1, 0}^{\oR}\Spec \R\in \Ho(dg_{\Z}\Aff_{C^*}(\bG_m \by R)_*(S)),
\]
 noting that this is isomorphic to $\theta^{\sharp}(\ugr (X,x)_{\MHS}^{\rho, \mal}) \by C^*$. We also define $(X,x)_{\bF}^{\rho, \mal}:= (X,x)_{\MHS}^{\rho, \mal}\by_{\bA^1, 1}^{\oR}\Spec \R$, noting that this is an algebraic  Hodge filtration on $(X,x)^{\rho, \mal}$. 
\end{definition}

\begin{definition}\label{splitmhsmal}
A real splitting of the mixed  Hodge structure $(X,x)_{\MHS}^{\rho, \mal}$ is a $\bG_m \by S$-equivariant isomorphism
$$
\bA^1 \by \ugr (X,x)^{\rho, \mal}_{\MHS} \by C^* \cong (X,x)_{\MHS}^{\rho, \mal},
$$
in $\Ho( dg_{\Z}\Aff_{\bA^1 \by C^*}(R)_*(\bG_m \by S))$, giving  the opposedness isomorphism on pulling back along $\{0\} \to \bA^1$.
\end{definition}

\begin{definition}
An algebraic  mixed twistor structure $(X,x)_{\MTS}^{\rho, \mal}$ on a pointed Malcev homotopy type $(X,x)^{\rho, \mal}$ consists of the following data:
\begin{enumerate}
\item an object 
$$
(X,x)_{\MTS}^{\rho, \mal} \in \Ho(  dg_{\Z}\Aff_{\bA^1 \by C^*}(R)_*( \bG_m \by \bG_m)),
$$
\item an object $\ugr (X,x)_{\MTS}^{\rho, \mal} \in \Ho( dg_{\Z}\Aff(R)_*(\bG_m ))$,
\item an isomorphism $(X,x)^{\rho, \mal} \cong  (X,x)_{\MTS}^{\rho, \mal} \by_{(\bA^1\by C^*), (1,1)}^{\oR}\Spec \R \in \Ho( dg_{\Z}\Aff(R)_*)$,
\item an isomorphism (called the opposedness isomorphism)
$$
\theta^{\sharp}(\ugr (X,x)_{\MTS}^{\rho, \mal}) \by C^* \cong  (X,x)_{\MTS}^{\rho, \mal} \by_{\bA^1, 0}^{\oR}\Spec \R \in \Ho( dg_{\Z}\Aff_{C^*}( R)_* (\bG_m \by\bG_m)),
$$
 for the canonical diagonal map $\theta: \bG_m \by \bG_m \to \bG_m$. 
\end{enumerate}
\end{definition}

Thus $(X,x)_{\MTS}^{\rho, \mal}$ consists of a $\bG_m \by R\by \bG_m$-equivariant dg scheme $X_{\MTS}^{\rho, \mal}$ over $\bA^1 \by C^*$, together with a  $\bG_m \by R\by \bG_m$-equivariant map $x \co \bA^1 \by R \by C^* \to X_{\MTS}^{\rho, \mal} $ over $\bA^1 \by C^*$, while $\ugr (X,x)_{\MTS}^{\rho, \mal}$ is an $R\by \bG_m$-equivariant dg scheme $\ugr X_{\MTS}^{\rho, \mal}$ equipped with  an $R\by \bG_m$-equivariant map $x \co R \to \ugr X_{\MHS}^{\rho, \mal}$.

\begin{definition}
Given an  algebraic mixed twistor structure $(X,x)_{\MTS}^{\rho, \mal}$ on $(X,x)^{\rho, \mal}$, define 
\[
 \gr^W(X,x)_{\MTS}^{\rho, \mal}:= (X,x)_{\MTS}^{\rho, \mal}\by_{\bA^1, 0}^{\oR}\Spec \R\in \Ho( dg_{\Z}\Aff_{C^*}(R)_*(\bG_m \by \bG_m))
\]
 noting that this is isomorphic to $\theta^{\sharp}(\ugr (X,x)_{\MTS}^{\rho, \mal}) \by C^*$. We also define $(X,x)_{\bT}^{\rho, \mal}:= (X,x)_{\MTS}^{\rho, \mal}\by^{\oR}_{\bA^1, 1}\Spec \R$, noting that this is an algebraic  twistor filtration on $(X,x)^{\rho, \mal}$. 
\end{definition}

\begin{remark}\label{baseptsmhs}
As in Remark \ref{basepts}, we might want to consider many basepoints. The definitions above  then have analogues $(X;T)^{\rho, \mal}_{\bF},(X;T)^{\rho, \mal}_{\bT},(X;T)^{\rho, \mal}_{\MHS},(X;T)^{\rho, \mal}_{\MTS}$, given by replacing the $R$-representation $R$ with the representation $\coprod_{x \in T} R(x,-)$. 
\end{remark}

\subsection{Grouplike structures}

\begin{definition}\label{gplmhs}
Define a grouplike mixed Hodge structure on 
 a pointed relative Malcev homotopy type $G(X,x)^{\rho, \mal}$ to consist of the following data:
\begin{enumerate}

\item an algebraic action of $S^1$ on $R$,

\item a  flat $\bG_m \by S$-equivariant   dg pro-algebraic  group $G(X,x)^{\rho, \mal}_{\MHS}$ over $\bA^1 \by C^*$, equipped with an $S$-equivariant  pro-unipotent surjection $G(X,x)^{\rho, \mal}_{\MHS} \to  \bA^1\by R \by C^*$ of dg pro-algebraic groups over $\bA^1\by C^*$, where $S$ acts on $R$ via the $S^1$-action of Definition \ref{U1def}.

\item a weak equivalence $N^sG(X,x)^{\rho, \mal} \simeq  G(X,x)_{\MHS}^{\rho, \mal} \by_{(\bA^1\by C^*), (1, 1)}\Spec \R$ of  dg pro-algebraic groups  on $\Spec \R$, respecting the $R$-augmentations.

\item a   flat $ S$-equivariant   dg pro-algebraic  group $ \ugr G(X,x)^{\rho, \mal}_{\MHS}$ over $\R$, equipped with an $S$-equivariant  pro-unipotent surjection $\ugr G(X,x)^{\rho, \mal}_{\MHS} \to  R$.

\item a $(\bG_m \by S)$-equivariant weak equivalence
$$
\theta^{\sharp}(\ugr G(X,x)^{\rho, \mal}_{\MHS} ) \by C^* \simeq G(X,x)_{\MHS}^{\rho, \mal} \by_{\bA^1, 0}\Spec \R 
$$
of $R \by C^*$-augmented dg pro-algebraic groups on $C^*$, 
 for the canonical map $\theta: \bG_m \by S \to S$ given by combining the inclusion $\bG_m \into S$ with the identity on $S$. 
\end{enumerate}
\end{definition}

\begin{definition}\label{gplmts}
Define a  grouplike mixed twistor structure similarly, dispensing with the $S^1$-action on $R$, and replacing $S$ with $\bG_m$.
\end{definition}

\begin{proposition}\label{gplike}
 Take a MHS  $(X,x)_{\MHS}^{\rho, \mal}$ (resp. a MTS $(X,x)_{\MTS}^{\rho, \mal}$)  on a relative Malcev homotopy type in the sense of Definition \ref{algmhsdef} (resp. \ref{algmtsdef}), and assume that the sheaf $A$ of DGAs on $\bA^1 \by C^* $ given by $O((X,x)_{\MHS}^{\rho, \mal})$ (resp. $O((X,x)_{\MTS}^{\rho, \mal})$) satisfies the conditions of Definition \ref{barBRYdef}. Explicitly, $A$ is concentrated in non-negative degrees, with $A^n$  flat for all $n>1$, and  $A^1/dA^0$  also flat.

Then there is a canonical grouplike MHS (resp. grouplike MTS) on  $(X,x)^{\rho, \mal}$.

Moreover, the induced pro-MHS  $\Ru(G(X,x)^{\rho,\mal}_{\MHS})^{\ab}$ (resp. pro-MTS $\Ru(G(X,x)^{\rho,\mal}_{\MTS})^{\ab}$) on the abelianisation of the pro-unipotent radical of $G(X,x)^{\rho,\mal}$ is quasi-isomorphic to the dual of the complex given by the cokernel of
$$
\O_{\bA^1 \by C^*}[1] \to A[1]. 
$$
\end{proposition}
\begin{proof}
We deal with the case of mixed Hodge structures; for mixed twistor structures just replace $S$ with $\bG_m$ throughout.
Applying the bar  construction $\bar{B}_{R,  \bA^1\by C^*}$ of Definition \ref{barBRYdef} to the map $A \to O(R)\ten_{\R}\O_{\bA^1}\ten_{\R}\O_{C^*}$ gives a $\bG_m \by S$-equivariant non-negatively graded DG Hopf algebra on $\bA^1 \by C^*$. We then set
\[
 O(G(X,x)^{\rho, \mal}_{\MHS}) := \bar{B}_{R,  \bA^1\by C^*}(A \to O(R)\ten_{\R}\O_{\bA^1}\ten_{\R}\O_{C^*}).
\]

Since $\bar{B}_{R,Y}$ is compatible with pullbacks in $Y$, we then have
\begin{eqnarray*}
 G(X,x)_{\MHS}^{\rho, \mal} \by_{(\bA^1\by C^*), (1, 1)}\Spec \R &=&  G(X,x)^{\rho, \mal}_{\MHS}\by_{\bA^1 \by C^*, (1,1)}\Spec \R \\
&=& \Spec  \bar{B}_R( (1, 1)^*A \to O(R)).
\end{eqnarray*}

Since $\bar{B}_{R}$ preserves quasi-isomorphisms, we then have a quasi-isomorphism
\[
  G(X,x)_{\MHS}^{\rho, \mal} \by_{(\bA^1\by \Spec \oR O(C^*)), (1, 1)}\Spec \R \simeq \Spec  \bar{B}_R(O(X,x)^{\rho, \mal})
\]
coming from the definition of an algebraic mixed Hodge structure. By Proposition \ref{repbar}, we have
\[
 \Spec  \bar{B}_R(O(X,x)^{\rho, \mal})= \bar{G}_R(O(X,x)^{\rho, \mal}),
\]
and Remark \ref{GXRmalrmk} gives a weak equivalence
\[
 D\bar{G}_R((X,x)^{\rho, \mal})\simeq G(X,x)^{\rho, \mal}.
\]
Applying simplicial normalisation $N^s$ and using the equivalences of Theorem \ref{bigequiv}, we then have the required weak equivalence
\[
 \bar{G}_R(O(X,x)^{\rho, \mal})\simeq N^sG(X,x)^{\rho, \mal}.
\]

We then define
\[
 \ugr G(X,x)^{\rho, \mal}_{\MHS}:= \bar{G}_R(\ugr (X,x)_{\MHS}^{\rho, \mal}), 
\]
noting that 
\[
 \ugr G(X,x)^{\rho, \mal}_{\MHS}  \by C^*= \Spec \bar{B}_{R,C^*}( O(\ugr (X,x)_{\MHS}^{\rho, \mal}) \ten \O_{C^*}),
\]
so the opposedness quasi-isomorphism gives
$$
\psi \co \theta^{\sharp}(\ugr G(X,x)^{\rho, \mal}_{\MHS} ) \by C^* \simeq G(X,x)_{\MHS}^{\rho, \mal} \by_{\bA^1, 0}\Spec \R. 
$$

The final statement just follows from from the bar filtration as in  Remark \ref{barfiltrn}, which shows  that $\Ru(G(X,x)^{\rho,\mal}_{\MHS})^{\ab}$ is dual to $\bar{A}[1]$, the map $\cone(\O_{\bA^1 \by C^*} \to A) \to \bar{A}$ being a quasi-isomorphism because $\H^0A= \O_{\bA^1 \by C^*}$. 
\end{proof}

\begin{remark}\label{baseptsgplmhs}
Rather than working with a single basepoint, we can make use of Remarks \ref{baseptslike} and \ref{basepts} to consider  a set $T$ of points on $X$. Definitions \ref{gplmhs} and \ref{gplmts} then adapt to multipointed Malcev homotopy types $(X;T)^{\rho, \mal}$, replacing dg pro-algebraic groups with dg pro-algebraic   groupoids on objects $T$, noting that $\Ob R =T$.
\end{remark}

\subsection{Splittings over $\cS$}\label{cSsplits} 

We now work with the  $S$-equivariant map $\row_1:\SL_2 \to C^*$ as defined in \S \ref{slfirst}.

\begin{definition}\label{splitsl2}
An $\cS$-splitting (or $\SL_2$-splitting) of a mixed  Hodge structure $(X,x)_{\MHS}^{\rho, \mal}$ on a relative Malcev homotopy type  is a $\bG_m \by S$-equivariant isomorphism
$$
\bA^1 \by \ugr (X,x)^{\rho, \mal}_{\MHS} \by \SL_2 \cong \row_1^*(X,x)_{\MHS}^{\rho, \mal},
$$
in $\Ho(  dg_{\Z}\Aff_{\bA^1 \by \SL_2}(R)_*(\bG_m \by S))$, giving  $\row_1^*$ of the opposedness isomorphism on pulling back along $\{0\} \to \bA^1$. 

An $\cS$-splitting (or $\SL_2$-splitting) of a mixed  twistor structure $(X,x)_{\MTS}^{\rho, \mal}$ on a relative Malcev homotopy type  is a $\bG_m \by \bG_m$-equivariant isomorphism
$$
\bA^1 \by \ugr (X,x)^{\rho, \mal}_{\MTS} \by \SL_2 \cong \row_1^*(X,x)_{\MTS}^{\rho, \mal},
$$
in $\Ho( dg_{\Z}\Aff_{\bA^1 \by \SL_2}(R)_*(\bG_m \by \bG_m))$, giving  $\row_1^*$ of the opposedness isomorphism on pulling back along $\{0\} \to \bA^1$.
\end{definition}

\begin{lemma}\label{keysplit}
Let $S'$ be $S$ or $\bG_m$.
Take  flat fibrant objects 
$$
Y \in dg_{\Z}\Aff_{\bA^1 \by \SL_2}(R)_*(\bG_m \by S') \text{ and } Z \in dg_{\Z}\Aff(R)_*(\bG_m \by S'),
$$ 
together with a surjective quasi-isomorphism $\phi^{\sharp}:0^*\O_Y \to  \O_Z\ten  \O_{\SL_2} $ in $ dg_{\Z}\Aff_{\SL_2}( R)_*(\bG_m \by S')$.
Then the set of  weak equivalence classes of objects $X \in   dg_{\Z}\Aff_{\bA^1 \by C^*}(R)_*(\bG_m \by S')$ equipped with weak equivalences $f:\row_1^*X \to Y$ and $g:0^*X \to Z\by C^*$ with $\phi \circ \row_1^*g= 0^*f$ is either $\emptyset$ or a principal homogeneous space for the group
$$
\EExt^0(L^{\bt}(1),  \ker(\phi^{\sharp}: \O_Y \to \O_Z\ten \O_{\SL_2})\to (W_{-1}\O_{\bA^1}) \ten (y_*O(R))\ten \O_{\SL_2} )^{\bG_m \by R \rtimes S'},
$$
where $L^{\bt}$ is the cotangent complex of $Y\cup^{\bL}_{(Z\by \SL_2)}(Z \by C^*)$ over $(\bA^1\by \SL_2)\cup_{(\{0\}\by \SL_2)}(\{0\}\by C^*)$, and $\Ext$ is taken over $_{Y\cup^{\bL}_{(Z\by \SL_2)}(Z \by C^*) }$.
\end{lemma}
\begin{proof}
The data $Y,Z,\phi$ determine the pullback of $X$ to 
$$
(\bA^1\by \SL_2)\cup_{(\{0\}\by \SL_2)}(\{0\}\by C^*). 
$$
Since $\phi^{\sharp}$ is surjective, we may  define
$$
\O_Y\by_{\phi^{\sharp}, (\O_Z\ten \O_{\SL_2})}(\O_Z \ten \oR O(C^*)) \to O(R)\ten ((\O_{\bA^1}\ten \O_{\SL_2})\by_{\O_{\SL_2}}\oR O(C^*))
$$
over 
$$
(O(\bA^1)\ten O(\SL_2))\by_{O(\SL_2)}\oR O(C^*),
$$
which we wish to lift to $\oR O(C^*)$, 
 making use of Proposition \ref{quaffworks1}.

Now, the morphism $\oR O(C^*) \to (O(\bA^1)\ten O(\SL_2))\by_{O(\SL_2)}\oR O(C^*)$ is surjective, with square-zero  kernel $(W_{-1}O(\bA^1))\ten O(\SL_2)(-1)[-1]$,  where $W_{-1} O(\bA^1)= \ker(O(\bA^1) \xra{0^*}\R)$,    so Proposition \ref{extchar} gives the required result.
\end{proof}

\begin{corollary}\label{keysplit2}
The set of weak equivalence classes of   $\cS$-split mixed Hodge structures $(X,x)_{\MHS}^{\rho, \mal}$  with $\ugr (X,x)_{\MHS}^{\rho, \mal}= (R \xra{z} Z)$ is  canonically isomorphic to
$$
\EExt^0_{Z }(\bL_{Z}^{\bt},  (W_{-1} O(\bA^1))\ten (\O_Z \to z_*O(R)) \ten O(\SL_2)(-1))^{\bG_m \by R \rtimes S}.
$$

The set of weak equivalence classes of   $\cS$-split mixed twistor structures $(X,x)_{\MTS}^{\rho, \mal}$  with $\ugr (X,x)_{\MTS}^{\rho, \mal}= (R \xra{z} Z)$ is  canonically isomorphic to
$$
\EExt^0_{Z }(\bL_{Z}^{\bt},  (W_{-1} O(\bA^1))\ten (\O_Z \to z_*O(R)) \ten O(\SL_2)(-1))^{\bG_m \by R \by \bG_m}.
$$
\end{corollary}
\begin{proof}
Set $Y= \bA^1 \by Z \by \SL_2$ in Lemma \ref{keysplit}, and note that the cone of $O(\bA^1) \xra{0^*} \R$ is quasi-isomorphic to $ W_{-1} O(\bA^1)$. The set of  of possible extensions is non-empty, since $\bA^1 \by Z \by C^*$ is one possibility for $(X,x)_{\MHS}^{\rho, \mal}$ (resp. $(X,x)_{\MTS}^{\rho, \mal}$). This gives a canonical basepoint for the principal homogeneous space, and hence the canonical isomorphism.   
\end{proof}

\subsection{Mixed Hodge structures on homotopy groups}
Take an $\cS$-split MHS   $(X,x)_{\MHS}^{\rho, \mal}$ (resp. an $\cS$-split  MTS $(X,x)_{\MTS}^{\rho, \mal}$)  on a relative Malcev homotopy type, and assume that the sheaf $A$ of DGAs on $\bA^1 \by C^* $ given by $O((X,x)_{\MHS}^{\rho, \mal})$ (resp. $O((X,x)_{\MTS}^{\rho, \mal})$)  is concentrated in non-negative degrees, with $A^n$  flat for all $n>1$, and  $A^1/dA^0$  also flat.

Assume moreover that the $S$-action (resp. the $\bG_m$-action) on $\H^*O(\ugr X^{\rho, \mal})$ is of non-negative weights.

\begin{theorem}\label{mhspin}\label{mtspin}
 Under the conditions above, the grouplike MHS (resp. grouplike MTS) of Proposition \ref{gplike} gives rise to  ind-MHS (resp. ind-MTS) on the duals $(\varpi_n(X,x)^{\rho, \mal})^{\vee}$ of the  relative Malcev homotopy groups for $n\ge 2$, and on the Hopf algebra $O(\varpi_1(X,x)^{\rho, \mal})$, independent of the choice of $\cS$-splitting.

These structures are compatible with the action of $\varpi_1$ on $\varpi_n$, the Whitehead bracket and the Hurewicz maps $\varpi_n(X^{\rho, \mal})\to \H^n(X, O(\Bu_{\rho}))^{\vee}$ ($n \ge 2$) and $\Ru\varpi_1(X^{\rho, \mal}) \to\H^1(X, O(\Bu_{\rho}))$, for $\Bu_{\rho}$ as in Definition \ref{Budef}. 
\end{theorem}
\begin{proof}
As for Proposition \ref{gplike}, we give  the proof for MHS only, since the MTS case follows by replacing $S$ with $\bG_m$ and Proposition \ref{flatmhs} with Proposition \ref{flatmts}.
For the grouplike MHS $G(X,x)^{\rho,\mal}_{\MHS}$ of Proposition \ref{gplike}, define 
$$
\varpi_1(X,x)_{\MHS}^{\rho,\mal}:= \oSpec \sH^0(O(G(X,x)^{\rho,\mal}_{\MHS}) ),
$$
which is an affine group object over $\bA^1 \by C^*$. Set 
\[
 \ugr \varpi_1(X,x)_{\MHS}^{\rho,\mal}:= \Spec \H^0(O(\ugr G(X,x)^{\rho,\mal}_{\MHS}) ).
\]

In order to show that these define mixed Hodge structures, we will have to satisfy  Proposition \ref{flatmhs}, so we need to establish flatness and boundedness, which is where the $\cS$-splitting and non-negativity of the weights will feature.
Since $\row_1: \SL_2\to C^*$ is flat, we have
$$
\row_1^*\varpi_1(X,x)_{\MHS}^{\rho,\mal}= \oSpec \sH^0(\row_1^*O(G(X,x)^{\rho,\mal}_{\MHS}) ).
$$ 

Choose a representative $Z$ for $\ugr (X,x)_{\MHS}^{\rho, \mal}$ with $O(Z)^0=\R$, and $O(Z)$ of non-negative weights. [To see that this is possible, take a minimal model $\m$ for $\bar{G}(\ugr (X,x)_{\MHS}^{\rho, \mal})$ as in \cite[Proposition \ref{htpy-dgminimal}]{htpy}, and note that $\m/[\m,\m]\cong \H^*O(\ugr X^{\rho, \mal})^{\vee}$ is of non-positive weights, so $\m$ is of non-negative weights, and therefore $O(\bar{W}\m)$ is of non-negative weights, so $\bar{W}\m$ is a possible choice for $Z$.] Setting  $\g:= G(Z)$, we then have
\[
 \ugr G(X,x)_{\MHS}^{\rho,\mal}\simeq \bar{G}_R(Z)= R \ltimes \exp(\g).
\]

Moreover, the choice of $\cS$-splitting gives $\row_1^*O(X,x)^{\rho, \mal}_{\MHS} \simeq  O(\bA^1) \ten O(Z)\ten O(\SL_2)$, so applying $\bar{G}_{R, \bA^1\by \SL_2}$ gives
\[
 \row_1^*G(X,x)^{\rho,\mal}_{\MHS} \simeq \bA^1 \by (R \ltimes \exp(\g))\by \SL_2,
\]
and hence
\[
 \row_1^*\varpi_1(X,x)_{\MHS}^{\rho,\mal}\cong  \bA^1 \by (R \ltimes \exp(\H_0\g))\by \SL_2.
\]
Since $\row_1$ is faithfully flat, this implies that $O(\varpi_1(X,x)_{\MHS}^{\rho,\mal})$ is  flat over $\bA^1 \by C^*$, with non-negative weights. 

Now the choice of $\cS$-splitting combines with the equivalence $Z \simeq \ugr (X,x)_{\MHS}^{\rho, \mal}$ to give
$$
\chi:\row_1^*(G(X,x)^{\rho,\mal}_{\MHS})\simeq \bA^1 \by (R \ltimes \exp(\g))\by \SL_2,
$$
 whose structure sheaf is flat on $\bA^1\by \SL_2$, and has non-negative weights with respect to the $\bG_m \by 1$-action. Lemma \ref{sluniv} then implies that the structure sheaf of $ \varpi_1(X,x)_{\MHS}^{\rho,\mal}$ is flat over $\bA^1 \by C^*$, with non-negative weights.

Set $\ugr \varpi_1(X,x)_{\MHS}^{\rho,\mal}:= (R \ltimes \exp(\H_0\g))$, and write $\phi$ for the pullback of $\chi$ to $0 \in \bA^1$. Combined with the morphism $\psi$ from the proof of Proposition \ref{gplike}, this induces an $S$-equivariant isomorphism
$$
\Spec \R\by_{\bA^1,0} \varpi_1(X,x)_{\MHS}^{\rho,\mal}\cong \ugr \varpi_1(X,x)_{\MHS}^{\rho,\mal} \by C^*,
$$
and an isomorphism 
$$
 \varpi_1(X,x)_{\MHS}^{\rho,\mal}\by_{\bA^1\by C^*,(1,1)}\Spec \R\cong  \varpi_1(X,x)^{\rho,\mal} ,
$$
giving the data of  a flat algebraic MHS on $O(\varpi_1(X,x)^{\rho,\mal})$, of non-negative weights. In particular, the weights are bounded below, so  by Proposition \ref{flatmhs}, this is the same as an ind-MHS of non-negative weights. 

Next, we consider the dg Lie coalgebra over $\bA^1 \by C*$ given by the cotangent space
$$
\C(G(X,x)^{\rho,\mal}_{\MHS}):= \Omega(G(X,x)^{\rho,\mal}_{\MHS}/C^*)\ten_{O(G(X,x)^{\rho,\mal}_{\MHS}), 1}C^*,
$$
so the cohomology sheaves $\sH^*(j^{-1}\C(G(X,x)^{\rho,\mal}_{\MHS}))$ form a graded Lie coalgebra over $O(\bA^1)\ten \O_{C^*}$.

The isomorphism $\chi$ above implies that these sheaves are flat over $\bA^1 \by C^*$, and therefore that $\sH^0(j^{-1}\C(G(X,x)^{\rho,\mal}_{\MHS}))$ is just the Lie coalgebra of $\varpi_1(X,x)_{\MHS}^{\rho,\mal}$. 
For $n \ge 2$, we set
$$
\varpi_n(X,x)^{\rho,\mal}_{\MHS})^{\vee}:= \sH^{n-1}(\row_{1*}\C(G(X,x)^{\rho,\mal}_{\MHS})),
$$
noting that these have a conjugation action by $\varpi_1(X,x)_{\MHS}^{\rho,\mal}$ and a natural Lie bracket. 

Setting $\ugr  \varpi_n(X,x)_{\MHS}^{\rho,\mal} := (\H_{n-1}\g)^{\vee} $, the isomorphisms $\phi$ and $\psi$ 
induce  $ S$-equivariant isomorphisms
$$
\Spec \R\by_{\bA^1,0} \varpi_n(X,x)_{\MHS}^{\rho,\mal}\cong \ugr \varpi_n(X,x)_{\MHS}^{\rho,\mal} \by C^*,
$$
and  isomorphisms 
$$
 \varpi_n(X,x)_{\MHS}^{\rho,\mal}\by_{\bA^1\by C^*,(1,1)}\Spec \R\cong  \varpi_n(X,x)^{\rho,\mal} ,
$$
so Proposition \ref{flatmhs} gives the data of an non-negatively weighted ind-MHS on $(\varpi_n(X,x)^{\rho,\mal})^{\vee}$, compatible with the $\varpi_1$-action and Whitehead bracket.

Finally, the Hurewicz map comes from 
$$
\Ru G(X,x)^{\rho,\mal} \to (\Ru G(X,x)^{\rho,\mal})^{\ab} \simeq \coker(\R \to O(X^{\rho,\mal}))[-1]^{\vee},
$$
which is compatible with the ind-MHS, by the final part of Proposition \ref{gplike}. Thus the Hurewicz maps
$$
\varpi_n(X,x)^{\rho,\mal} \to \H^n(X, O(\Bu_{\rho}))^{\vee}\quad \Ru\varpi_1(X^{\rho, \mal}) \to\H^1(X, O(\Bu_{\rho}))
$$
preserve the  ind-MHS.
\end{proof}

In 
Theorem \ref{mhspin}, the only r\^ole of the $\cS$-splitting is to ensure that the algebraic MHS is flat. 
We now show how a choice of $\cS$-splitting gives additional data.
\begin{theorem}\label{pinsplit}
A choice of $\cS$-splitting for $(X,x)_{\MHS}^{\rho, \mal}$ (resp. $(X,x)_{\MTS}^{\rho, \mal}$) gives an isomorphism 
$$
O(\varpi_1(X,x)^{\rho, \mal})\ten \cS \cong \gr^WO(\varpi_1(X,x)^{\rho, \mal})\ten \cS
$$
of  (real) quasi-MHS (resp. quasi-MTS) in Hopf algebras, and isomorphisms 
$$
 (\varpi_n(X,x)^{\rho, \mal})^{\vee}\ten \cS  \cong \gr^W(\varpi_n(X,x)^{\rho, \mal})^{\vee}\ten \cS
$$
of (real) quasi-MHS (resp. quasi-MTS), inducing the identity on $\gr^W$, and compatible with the Whitehead bracket.
\end{theorem}
\begin{proof}
Applying the functor $\bar{G}_{R, \bA^1\by \SL_2}$ to the splitting quasi-isomorphism gives a quasi-isomorphism
\[
 \row_1^*G(X,x)^{\rho,\mal}_{\MHS}\simeq \bA^1 \by (\ugr G(X,x)^{\rho,\mal}_{\MHS}) \by \SL_2,
\]
for the grouplike MHS $ G(X,x)^{\rho,\mal}_{\MHS}$ of Proposition \ref{gplike}.
 The isomorphisms now follow from Lemma \ref{slhodge} and the constructions of  Theorem \ref{mhspin}.
\end{proof}

\begin{remark}\label{gadata}
This leads us to ask what additional data are required to describe the ind-MHS on  homotopy groups in terms of the Hodge structure $\ugr (X,x)^{\rho,\mal}_{\MHS}$. If we set $\g= \bar{G}(\ugr (X,x)^{\rho,\mal}_{\MHS})$, then $\ugr G(X,x)^{\rho,\mal}_{\MHS} \simeq  R \rtimes \exp(\g)$ and  we can let $\cD^{\bt}:= \der_{\R}(R \rtimes \exp(\g), R \rtimes \exp(\g))$ be the complex of DG Hopf algebra derivations on $O(R \rtimes \exp(\g))$. 

The derivation $N$ of Definition \ref{Ndef}  induces a $\bG_m\by S$-equivariant DG Hopf algebra derivation
\[
 N \co \row_1^* O(G(X,x)^{\rho,\mal})_{\MHS} \to \row_1^* O(G(X,x)^{\rho,\mal})_{\MHS}(-1),
\]
$\bA^1$-linear and $N$-linear in the sense that $\beta(af)= (Na)f + a\beta(f)$ for $a \in O(\SL_2)$. 

The $\cS$-splitting quasi-isomorphism
\[
 \row_1^*G(X,x)^{\rho,\mal}_{\MHS}\simeq \bA^1 \by (\ugr G(X,x)^{\rho,\mal}_{\MHS}) \by \SL_2
\]
then allows us to transfer $N$  to  a derivation  $N+\beta: O(R \rtimes \exp(\g))\ten \cS \to O(R \rtimes \exp(\g))\ten \cS(-1)$, unique up to homotopy.   As in \S \ref{analogies}, we think of $N +\beta$ as the  monodromy operator at the Archimedean place. This will be constructed explicitly in \S \ref{archmon}.

Moreover, for any $\cS$-split MHS $V$ arising as an invariant of $O(G(X,x)) $ (including the homotopy groups), the induced  map $N+\bar{\beta}: (\gr^WV)\ten \cS \to (\gr^WV)\ten \cS(-1)$ just comes from  conjugating the surjective map $\id \ten N: V\ten \cS \to V\ten \cS(-1)$ with respect to the splitting isomorphism $(\gr^WV)\ten \cS
\cong V\ten \cS$. Therefore $N+\bar{\beta}$ is surjective, and $V= \ker(N+\bar{\beta})$.

All these results have direct analogues  for mixed twistor structures.
\end{remark}

\begin{remark}\label{gadatab}
In Remark \ref{gadata}, note that $O(\bar{W}\g)$ is a cofibrant replacement for $O(\ugr (X,x)^{\rho,\mal}_{\MHS} )$, so its cotangent space is a model for the cotangent complex of $O(\ugr (X,x)^{\rho,\mal}_{\MHS})$, so Corollary \ref{keysplit2} then shows that the mixed Hodge structure $ (X,x)^{\rho,\mal}_{\MHS}$ is determined up to quasi-isomorphism by a derivation
\[
\nu \in \EExt^0_{O(\bar{W}\g) }(\Omega(O( \bar{W}\g)/\R), (W_{-1}O(\bA^1)) \ten (O(Z) \to  O(R))\ten O(\SL_2)(-1))^{\bG_m \by R \rtimes S}.
\]

Now,  $\Omega(O( \bar{W}\g)/\R) \cong \g^{\vee}[-1]$, so we may choose a representative
$$
(\alpha', \gamma'_x): \g^{\vee}[-1] \to (W_{-1}O(\bA^1)) \ten (O(Z)\by O(R)[-1])\ten O(\SL_2)(-1)
$$
for $\nu$, with $[d,\alpha']=0$, $[d,\gamma'_x]= z^*\alpha'$.

Studying the adjunction  $\bar{W} \vdash G$, we see that $\alpha'$ is equivalent to a $\bG_m \by R \ltimes S$-equivariant Lie coalgebra derivation $\alpha: \g^{\vee}  \to W_{-1}O(\bA^1)\ten \g^{\vee}\ten   O(\SL_2)(-1)$ with $[d, \alpha]=0$. This generates a derivation $\alpha :  O(R \ltimes \exp(\g)) \to (W_{-1}O(\bA^1))\ten O(R \ltimes \exp(\g))\ten   O(\SL_2)(-1)$. 
Meanwhile, $\gamma'_x$ corresponds to an element $\gamma_x\in (\g_0\hat{\ten}(W_{-1}O(\bA^1)\ten O(\SL_2)(-1)))^{\bG_m \by S}$, and conjugation by this gives another such derivation $[\gamma_x,-]$, 

It then follows from the properties of the bar construction that
\[
[\beta]:=[\alpha+[\gamma_x, -]] \in \H^0(W_{-1}\gamma^0(\cD^{\bt}\ten \cS(-1))),
\]
for $\gamma^0V= V \cap(F^0V_{\Cx})$ as in Definition \ref{gammadef}, noting that  $\gamma^0(\cD^{\bt}\ten \cS(-1))\simeq \oR \Gamma_{w\cH}(\cD^{\bt}\ten \cS(-1))$, by Remark \ref{BeiS}. 
\end{remark}

\begin{remark}\label{mhspinbasepts}
If we have a multipointed MHS (resp. MTS) $(X;T)^{\rho, \mal}_{\MHS}$ (resp. $(X;T)^{\rho, \mal}_{\MTS}$) as in Remark \ref{baseptsmhs}, then Proposition \ref{gplike} and Theorems \ref{mhspin} and \ref{pinsplit}  adapt to give 
multipointed grouplike MHS (resp. MTS)
as in Remark \ref{baseptslike}, together with 
ind-MHS (resp. ind-MTS) on the algebras $O(\varpi_1(X;x,y)^{\rho, \mal})$, compatible with the pro-algebraic groupoid structure. The  
 ind-MHS (resp. ind-MTS) on $(\varpi_n(X,x)^{\rho, \mal})^{\vee}$ are then compatible with the co-action
$$
(\varpi_n(X,x)^{\rho, \mal})^{\vee} \to O(\varpi_1(X;x,y)^{\rho, \mal})\ten  (\varpi_n(X,y)^{\rho, \mal})^{\vee}.
$$
\end{remark}

\section{MHS on relative Malcev homotopy types of compact K\"ahler manifolds}\label{snkmhs}

Fix a compact connected K\"ahler manifold $X$ and a point $x \in X$.

\subsection{Real homotopy types}\label{realmhs}

We now look again at the phenomena studied in Sections \ref{real1} and \ref{real}.

\begin{definition}\label{nilphfil}
Define  the Hodge filtration on the real homotopy type $(X\ten \R, x)$ by $(X\ten \R,x)_{\bF}:= (\Spec \R \by C^* \xra{x} \oSpec j^*\tilde{A}^{\bt}(X)) \in \Ho(C^* \da dg_{\Z}\Aff_{C^*}(S))$, for $j:C^* \to C$ and $\tilde{A}^{\bt}(X)$ as in Definition
\ref{Atildedef}.
\end{definition}

\begin{definition}\label{nilpmhs}
Define the algebraic mixed Hodge structure $(X\ten \R,x)_{\MHS}$ on $(X\ten \R, x)$ to be $\oSpec$ of the Rees algebra associated to the good truncation filtration $W_r =\tau^{\le r} j^*\tilde{A}^{\bt}(X) $, equipped with the augmentation $j^*\tilde{A}^{\bt}(X) \xra{x^*} \O_{C^*}$. 

Define  $(\ugr  (X\ten \R)_{\MHS},0)$ to be $\Spec H^*(X, \R)$ equipped with the  unique morphism from $\Spec \R$ determined by  the isomorphism  $H^0(X, \R)\cong \R$. Now 
$$
(X\ten \R, x)_{\MHS}\by_{\bA^1}^h\{0\}= (C^* \xra{x} \oSpec   \gr^Wj^*\tilde{A}^{\bt}(X) ),
$$ 
and there is  a canonical  quasi-isomorphism $\gr^Wj^*\tilde{A}^{\bt}(X) \to \sH^*(j^*\tilde{A}^{\bt}(X))$.  As in the proof of  Theorem \ref{formalitysl}, this  is $S$-equivariantly isomorphic to $\H^*(X,\R)\ten \O_{C^*}$, giving the opposedness quasi-isomorphism
$$
(X\ten \R, x)\by_{\bA^1}^h\{0\} \xla{\sim} (\ugr  (X\ten \R)_{\MHS},0)\by C^*.
$$
\end{definition}

\begin{proposition}\label{nilpsplits}
The algebraic MHS $(X\ten \R,x)_{\MHS}$ splits on pulling back along $\row_1: \SL_2 \to C^*$. Explicitly, there is an  isomorphism 
$$
(X\ten \R,x)_{\MHS}\by_{C^*, \row_1}^{\oR}\SL_2 \cong   \bA^1 \by (\ugr(X\ten \R)_{\MHS},0)\by C^*,
$$
in $\Ho(\bA^1\by \SL_2 \da dg_{\Z}\Aff_{ \bA^1\by \SL_2}( \bG_m \by S))$, whose pullback to $0 \in \bA^1$ is given by the opposedness isomorphism.
\end{proposition}
\begin{proof}
Theorem \ref{formalitysl} establishes the corresponding splitting for the Hodge filtration $(X\ten \R,x)_{\bF}$, and good truncation commutes with everything, giving the splitting for $(X\ten \R,x)_{\MHS}$. The proof of  Theorem \ref{formalitysl} ensures that pulling the $\cS$-splitting back  to $0 \in \bA^1$ gives $\row_1^*$ applied to the opposedness isomorphism.
\end{proof}

\begin{corollary}\label{morgansplits}
For $\cS$ as in Definition \ref{cSdef}, and for all $n\ge 1$, there are $\cS$-linear isomorphisms
$$
\pi_n(X\ten \R,x)^{\vee}\ten_{\R}\cS \cong \pi_n(H^*(X, \R))^{\vee}\ten_{\R}\cS,
$$
of quasi-MHS, compatible with Whitehead brackets and Hurewicz maps. The graded map  associated to the weight filtration is just the pullback of the standard isomorphism $\gr_W\pi_n(X\ten \R,x)\cong\pi_n(H^*(X, \R))$ (coming from the opposedness isomorphism).
\end{corollary}
\begin{proof}
The  $\cS$-splitting of Proposition \ref{nilpsplits} allows us to apply Theorem \ref{pinsplit}, giving isomorphisms
$$
\pi_n(X\ten \R,x)^{\vee}\ten_{\R}\cS \cong \varpi_n(\ugr  (X\ten \R)_{\MHS},0)^{\vee}\ten_{\R}\cS
$$
of quasi-MHS. 

The definition of $\ugr  (X\ten \R)_{\MHS}$  implies that $\varpi_n(\ugr  (X\ten \R)_{\MHS},0) =\pi_{n-1}\bar{G}(\H^*(X, \R))$, giving the required result.
\end{proof}

\subsection{Relative Malcev homotopy types}

\subsubsection{The reductive fundamental groupoid is pure of weight $0$}

\begin{lemma}\label{discreteact}
There is a canonical action of the discrete group $(S^1)^{\delta}$ on the real reductive pro-algebraic completion $\varpi_1(X,x)^{\red}$ of the fundamental group $\pi_1(X,x)$.
\end{lemma}
\begin{proof}
By Tannakian duality, this is equivalent to establishing a $(S^1)^{\delta}$-action on the category of real semisimple local systems on $X$.
This is just the unitary part of the $\Cx^{\by}$-action on complex local systems from \cite{Simpson}. Given a real $\C^{\infty}$ vector bundle $\sV$ with a flat connection $D$, there is an essentially unique pluriharmonic metric, giving a  unique decomposition  $D=d^++\vartheta$ of $D$ into antisymmetric and symmetric parts. In the notation of \cite{Simpson}, $d^+=\pd +\bar{\pd}$ and $\vartheta =\theta+\bar{\theta}$. Given $t \in (S^1)^{\delta}$, we define $t\circledast D$ by $d^++ t\dmd \vartheta=  \pd +\bar{\pd}+t\theta +t^{-1} \bar{\theta}$ (for $\dmd$ as in Definition \ref{dmd}), which preserves the metric.
\end{proof}

\subsubsection{Variations of Hodge structure}

The following results are taken from \cite[\S \ref{hodge-varhodgesn}]{hodge}.
\begin{definition}\label{algact}
Given  a discrete group $\Gamma$ acting on a pro-algebraic group $G$, 
define ${}^{\Gamma}\!G$ to be the maximal quotient of $G$ on which $\Gamma$ acts algebraically.  This is the inverse limit $\Lim_{\alpha} G_{\alpha}$ over those surjective maps
$$
G \to G_{\alpha},
$$
with $G_{\alpha}$ algebraic (i.e. of finite type), for which the $\Gamma$-action descends to $G_{\alpha}$. Equivalently, $O( {}^{\Gamma}\!G)$ is the sum of those finite-dimensional $\Gamma$-representations of $O(G)$ which are closed under  comultiplication.
\end{definition}

\begin{definition}
Define the quotient  group ${}^{\VHS}\!\varpi_1(X,x)$ of $\varpi_1(X,x)$ by 
$$
{}^{\VHS}\!\varpi_1(X,x) := {}^{(S^1)^{\delta}}\!\varpi_1(X,x)^{\red}.
$$
\end{definition}

\begin{remarks}
This notion is analogous to the definition given in \cite{weight1} of the maximal quotient of the $l$-adic pro-algebraic fundamental group on which Frobenius acts algebraically. In the same way that representations of that group corresponded to semisimple subsystems of local systems underlying Weil sheaves, representations of ${}^{\VHS}\!\varpi_1(X,x)$ will correspond to 
local systems underlying variations of Hodge structure (Proposition \ref{vhsequiv}). 
\end{remarks}

\begin{proposition}\label{vhsalg}
The action of $S^1$ on ${}^{\VHS}\!\varpi_1(X,x)$ is algebraic, in the sense that
$$
S^1 \by {}^{\VHS}\!\varpi_1(X,x)\to {}^{\VHS}\!\varpi_1(X,x)
$$
is a morphism of schemes.

It is also an inner action, coming from a morphism
$$
S^1  \to ( {}^{\VHS}\!\varpi_1(X,x))/Z({}^{\VHS}\!\varpi_1(X,x))
$$
of pro-algebraic groups, where $Z$ denotes the centre of the group.
\end{proposition}
\begin{proof}
In the notation of Definition \ref{algact}, write $\varpi_1(X,x) = \Lim_{\alpha} G_{\alpha}$. As in  \cite[Lemma 5.1]{Simpson}, the map
$$
\Aut(G_{\alpha}) \to \Hom(\pi_1(X,x), G_{\alpha})
$$
is a closed immersion of schemes, so the map
$$
(S^1)^{\delta} \to \Aut(G_{\alpha})
$$
is 
continuous. This means that it defines a one-parameter subgroup, so is algebraic.
Therefore the map
$$
S^1 \by  {}^{\VHS}\!\varpi_1(X,x)\to{}^{\VHS}\!\varpi_1(X,x) 
$$
is algebraic, as $ {}^{\VHS}\!\varpi_1(X,x)=\Lim G_{\alpha}$.

Since $\varpi_1(X,x)^{\red}$  is a  reductive pro-algebraic group, $G_{\alpha}$ is a reductive algebraic group.
This implies that the connected component $\Aut(G_{\alpha})^0$ of the identity in $\Aut(G_{\alpha})$ is given by
$$
\Aut(G_{\alpha})^0=  G_{\alpha}(x,x)/Z(G_{\alpha}).
$$

Since 
$$
{}^{\VHS}\!\varpi_1(X,x)/Z({}^{\VHS}\!\varpi_1(X,x))=\Lim G_{\alpha}/Z(G_{\alpha}), 
$$ 
we have an algebraic  map 
$$
S^1 \to{}^{\VHS}\!\varpi_1(X,x)/Z({}^{\VHS}\!\varpi_1(X,x)),
$$
as required.
\end{proof}

\begin{proposition}\label{vhsequiv}
The following conditions are equivalent:
\begin{enumerate}
\item $V$ is a representation of ${}^{\VHS}\!\varpi_1(X,x)$;
\item $V$ is a  representation of $\varpi_1(X,x)^{\red}$ such that $t \circledast V \cong V$ for all $t \in (S^1)^{\delta}$;
\item $V$ is a  representation of $\varpi_1(X,x)^{\red}$ such that $t \circledast V \cong V$ for some non-torsion  $t \in (S^1)^{\delta}$.
\end{enumerate}

Moreover, representations of ${}^{\VHS}\!\varpi_1(X,x) \rtimes S$  correspond to variations of Hodge structure on $X$. 
\end{proposition}
\begin{proof}
$ $

\begin{enumerate}
\item[1.$\implies$2.] If $V$ is a representation of ${}^{\VHS}\!\varpi_1(X,x)$, then it is a representation of $\varpi_1(X,x)^{\red}$, so is a semisimple representation of $\varpi_1(X,x)$. By Lemma \ref{vhsalg},  $t \in (S^1)^{\delta}$ is an inner automorphism of ${}^{\VHS}\!\varpi_1(X,x)$, coming from $g \in {}^{\VHS}\!\varpi_1(X,x)$, say. Then multiplication by $g$ gives the isomorphism $t \circledast V \cong V$.

\item[2.$\implies$3.] Trivial.

\item[3.$\implies$1.] Let $M$ be the monodromy group of $V$; this is a quotient of  $\varpi_1(X,x)^{\red}$. The isomorphism $t \circledast V \cong V$ gives an element $g \in \Aut(M)$, such that $g$ is the image of $t$  in $\Hom(\pi_1(X,x), M)$, using the standard embedding of  $\Aut(M)$ as a closed subscheme of  $\Hom(\pi_1(X,x), M)$. The same is true of $g^n,t^n$, so  the image of $S^1$ in $\Hom(\pi_1(X,x), M)$ is just the closure of $\{g^n\}_{n \in \Z}$, which is contained in $\Aut(M)$, as $\Aut(M)$ is closed. Thus  the $S^1$-action on $\varpi_1(X,x)^{\red}$ descends to $M$, so $M$ is a quotient of ${}^{\VHS}\!\varpi_1(X,x)$, as required.
 \end{enumerate}

Finally, a representation of ${}^{\VHS}\!\varpi_1(X,x) \rtimes S$ gives a semisimple local system $\vv= \ker(D:\sV\to \sV\ten_{\sA^0}\sA^1)$ (satisfying one of the equivalent conditions above), together with a coassociative coaction $\mu:(\sV,D) \to (\sV\ten O(S), t \circledast D)$ of ind-finite-dimensional local systems, for $t=a+ib \in O(S^1)\ten \Cx$, where we regard $O(S^1)$ as a subspace of $O(S)$ via the isomorphism $S/\bG_m=S^1$. 

This is equivalent to giving a decomposition $\sV\ten \Cx= \bigoplus_{p+q} \sV^{pq}$ with $\overline{\sV^{pq}}= \sV^{qp}$, and with the decomposition $D= \pd +\bar{\pd}+t\theta +t^{-1} \bar{\theta}$ (as in Lemma \ref{discreteact}) satisfying
$$
\pd: \sV^{pq} \to \sV^{pq}\ten\sA^{10}, \quad \bar{\theta}: \sV^{pq} \to \sV^{p+1,q-1}\ten\sA^{01},
$$ 
which is precisely the condition for $\vv$ to be a VHS. Note that if we had chosen $\vv$ not satisfying one of the equivalent conditions, then $(\sV\ten O(S), t \circledast D)$ would not yield an ind-finite-dimensional local system.
\end{proof}

\begin{lemma}
The obstruction $\varphi$ to a surjective  map $\alpha:\varpi_1(X,x)^{\red}  \to R$, for $R$ algebraic, factoring through  ${}^{\VHS}\!\varpi_1(X,x)$ lies in $\H^1(X, \ad \Bu_{\alpha})$, for $\ad \Bu_{\alpha}$ the vector bundle associated to the adjoint representation of $\alpha$ on the Lie algebra of $R$. Explicitly, $\varphi$ is given by $\varphi=[i\theta -i\bar{\theta}]$, for $\theta\in A^{10}(X, \ad \Bu_{\alpha})$ the Higgs form associated to $\alpha$.
\end{lemma}
\begin{proof}
We have a 
real analytic 
map
$$
S^1 \by \pi_1(X,x) \to R,
$$
and $\alpha$ will factor through ${}^{\VHS}\!\varpi_1(X,x)$ if and only if the induced map
$$
S^1 \xra{\phi} \Hom(\pi_1(X,x), R)/\Aut(R)
$$ 
is constant. Since $R$ is reductive and $S^1$ connected, it suffices to replace $\Aut(R)$ by the group of inner automorphisms. On tangent spaces, we then have a map
$$
i\R \xra{D_1\phi } \H^1(X, \ad \Bu_{\alpha});
$$ 
let $\varphi \in \H^1(X, \ad \Bu_{\alpha})$ be the image of $i$. The description $\varphi=[i\theta -i\bar{\theta}]$ comes from differentiating $e^{ir}\circledast D=   \pd +\bar{\pd}+e^{ir}\theta +e^{-ir} \bar{\theta}$ with respect to $r$. 

If $\phi$ is constant, then $\varphi=0$. Conversely, observe that for $t \in S^1(R)$, $D_t\phi= t(D_1\phi) t^{-1}$, making use of the  action of $(S^1)^{\delta}$ on $\Hom(\pi_1(X,x), G)$. If $\varphi=0$, this implies that $D_t\phi=0$ for all $t \in (S^1)^{\delta}$, so $\phi$ is constant, as required.
\end{proof}

\subsubsection{Mixed Hodge structures}

\begin{theorem}\label{mhsmal}
If  $R$ is any quotient of ${}^{\VHS}\!\varpi_1(X,x)^{\red}_{\R}$, then there is an algebraic mixed Hodge structure $(X,x)^{\rho,\mal}_{\MHS}$ on the relative Malcev homotopy type $(X,x)^{\rho, \mal}$, where $\rho$ denotes the quotient map to $R$.

There is also an $S$-equivariant  splitting
$$
\bA^1 \by (\ugr (X,x)^{\rho,\mal}_{\MHS},0) \by \SL_2 \simeq (X^{\rho, \mal},x)_{\MHS}\by_{C^*, \row_1}^{\oR}\SL_2
$$
 on pulling back along $\row_1:\SL_2 \to C^*$, whose pullback over $0 \in \bA^1$ is given by the opposedness isomorphism. 
\end{theorem}
\begin{proof}
By Proposition \ref{vhsalg}, we know that representations of $R$ all correspond to local systems underlying polarised variations of Hodge structure, and that the $(S^1)^{\delta}$-action  on ${}^{\VHS}\!\varpi_1(X,x)^{\red}_{\R}$ descends to an inner algebraic action on $R$, via $S^1 \to R/Z(R)$.
This allows us to consider the semi-direct products $R \rtimes S^1$ and $R\rtimes S$ of pro-algebraic groups, making use of the isomorphism $S^1 \cong S/\bG_m$.

The $R$-representation $O(\Bu_{\rho})=\Bu_{\rho}\by^RO(R)$ in local systems of $\R$-algebras on $X$  thus has an algebraic $S^1$-action, denoted by $(t,v) \mapsto t\circledast v$ for $t \in S^1, v \in O(\Bu_{\rho})$, and we define an  $S$-action  on the de Rham complex
$$
\sA^*(X, O(\Bu_{\rho}))=\sA^*(X, \R)\ten_{\R}O(\Bu_{\rho})
$$
by $\lambda \boxast (a\ten v) := (\lambda \dmd a) \ten (\frac{\bar{\lambda}}{\lambda} \circledast v)$, noting that the $\dmd$ and $\circledast$ actions commute. This gives an action on the global sections
$$
A^*(X, O(\Bu_{\rho})):=\Gamma(X,\sA^*(X, O(\Bu_{\rho}))).
$$

It follows from \cite[Theorem 1]{Simpson} that
there exists a  harmonic metric on every semisimple local system $\vv$, and hence on
$O(\Bu_{\rho})$. We then decompose the connection $D$ as $D= d^++\vartheta$ into  antisymmetric and symmetric  parts, and let $D^c:= i\dmd d^+ -i\dmd \vartheta$. To see that this is independent of the choice of metric, observe that for $-1 \in S^1$ acting on $O(\varpi_1(X,x)^{\red})$, antisymmetric and symmetric  parts are the $1$- and $-1$-eigenvectors.

Now, we define the CDGA $\tilde{A}^{\bt}(X, O(\Bu_{\rho}))$ on $C$ by
$$
\tilde{A}^{\bt}(X, O(\Bu_{\rho})):= (A^*(X,O(\Bu_{\rho})) \ten_{\R} O(C), uD + vD^c),
$$
and we denote the differential by $\tilde{D}:=uD + vD^c$. Note that the  $\boxast$ $S$-action makes this $S$-equivariant over $C$. Thus $\tilde{A}(X, O(\Bu_{\rho}))\in DG\Alg_{C}(R\rtimes S)$, and we define the Hodge filtration by 
$$
(X_{\bF}^{\rho, \mal},x):= (R \by C^* \xra{x} (\Spec \tilde{A}(X, O(\Bu_{\rho})))\by_CC^* )\in  dg_{\Z}\Aff_{C^*}(R)_*(S),
$$
making use of the isomorphism $O(\Bu_{\rho})_x \cong O(R)$.

We then define the mixed Hodge structure  $(X^{\rho, \mal}_{\MHS},x)$ by 
$$
(\bA^1 \by R \by C^* \xra{x}(\Spec \xi(\tilde{A}(X, O(\Bu_{\rho})), \tau))\by_CC^* )\in  dg_{\Z}\Aff_{\bA^1\by C^*}(R)_*(\bG_m \by S),
$$
with $(\ugr X^{\rho,\mal}_{\MHS},0)$ given by
$$
(R \to \Spec \H^*(X, O(\Bu_{\rho})))\in dg_{\Z}\Aff(R)_*(S).
$$

The rest of the proof is now the same as in \S \ref{realmhs}, using the principle of two types from \cite[Lemmas 2.1 and 2.2]{Simpson}. Theorem \ref{formalitysl} adapts to give the quasi-isomorphism
$$
(X_{\MHS}^{\rho, \mal},x)\by_{C^*, \row_1}^{\oR}\SL_2 \simeq (\ugr X^{\rho,\mal}_{\MHS},0)  \by \SL_2,
$$
which gives the splitting.
\end{proof}

Observe that this theorem easily adapts to multiple basepoints, as considered in Remark \ref{baseptsmhs}.

\begin{remark}\label{ktpwgt}
Note that the filtration $W$ here and later is not related to the weight tower $W^{*}F^0$ of \cite[\S 3]{KTP}, which does not agree with the weight filtration of \cite{Morgan}.  $W^{*}F^0$ corresponded to the lower central series filtration $\Gamma_n\g$ on  $\g:=\Ru(G(X)^{\alg})$, given by $\Gamma_1\g=\g$ and $\Gamma_n\g=[\Gamma_{n-1}\g, \g]$, by the formula $W^{i}F^0 =\g/\Gamma_{n+1}\g$. Since this is just the filtration $\bar{G}(\Fil)$ coming from the filtration $\Fil_{-1}A^{\bt}=0, \Fil_0A^{\bt}=\R,\, \Fil_1A^{\bt}=A^{\bt}$ on  $A^{\bt}$, it amounts to setting higher cohomology groups to be pure of weight $1$;  \cite[Proposition 3.2.6(4)]{KTP} follows from this observation, as the graded pieces $\gr_W^{i}F^0$ defined in \cite[Definition 3.2.3]{KTP} are just $\gr_{G(\Fil)}^{i+1}\g$.
\end{remark}

\begin{corollary}\label{kmhspin}
In the scenario of Theorem \ref{mhsmal}, the
 homotopy groups $\varpi_n(X^{\rho, \mal},x)$ for $n\ge 2$, and the Hopf algebra $O(\varpi_1(X^{\rho, \mal},x))$ carry natural ind-MHS, functorial in $(X,x)$, and compatible with the action of $\varpi_1$ on $\varpi_n$, the Whitehead bracket  and the Hurewicz maps $\varpi_n(X^{\rho, \mal},x) \to \H^n(X,O(\Bu_{\rho}))^{\vee}$. 

Moreover, there are canonical $\cS$-linear isomorphisms
\begin{eqnarray*}
\varpi_n(X^{\rho, \mal},x)^{\vee}\ten\cS &\cong& \pi_n(H^*(X, O(\Bu_{\rho})))^{\vee}\ten\cS\\
O(\varpi_1(X^{\rho, \mal},x)) \ten\cS &\cong& O(R \ltimes\pi_1(H^*(X, O(\Bu_{\rho}))))\ten\cS
\end{eqnarray*}
of quasi-MHS. The associated graded map from the weight filtration is just the pullback of the standard isomorphism $\gr_W\varpi_*(X^{\rho, \mal})\cong\pi_*(H^*(X, O(\Bu_{\rho})))$.

Here, $\pi_*(H^*(X, O(\Bu_{\rho})))$ are the  homotopy groups $\H_{*-1}\bar{G}(\H^*(X, O(\Bu_{\rho})))$  associated to the $R \rtimes S$-equivariant  CDGA $\H^*(X,O(\Bu_{\rho}))$ (as constructed in Definition \ref{barwg}), with the induced real Hodge structure.
\end{corollary}
\begin{proof}
Theorem \ref{mhsmal} provides the data required by Theorems \ref{mhspin} and \ref{pinsplit} to construct $\cS$-split ind-MHS on homotopy groups. 
\end{proof}

\begin{remark}\label{kmhspinbasepts}
If we have a set  $T$ of basepoints, then Remark \ref{mhspinbasepts} gives  $\cS$-split ind-MHS  on the algebras $O(\varpi_1(X;x,y)^{\rho, \mal})$, compatible with the pro-algebraic groupoid structure. The  $\cS$-split ind-MHS  on $(\varpi_n(X,x)^{\rho, \mal})^{\vee}$ are then compatible with the co-action
$$
(\varpi_n(X,x)^{\rho, \mal})^{\vee} \to O(\varpi_1(X;x,y)^{\rho, \mal})\ten  (\varpi_n(X,y)^{\rho, \mal})^{\vee}.
$$
\end{remark}

\begin{remark}\label{arapuraconj}
Corollary \ref{kmhspin} confirms the first part of \cite[Conjecture 5.5]{arapurapi1}. If $\vv$ is a $k$-variation of Hodge structure on $X$, for  a field $k \subset \R$, and $R$ is the Zariski closure of $\pi_1(X,x) \to \GL(\vv_x)$, the conjecture states that there is a natural ind-$k$-MHS on the $k$-Hopf algebra $O(\varpi_1(X^{\rho, \mal},x)^{\rho, \mal})$. Applying Corollary \ref{kmhspin} to the Zariski-dense real representation $\rho_{\R}:\pi_1(X,x) \to R(\R)$ gives a real ind-MHS on the real Hopf algebra $O(\varpi_1(X,x)^{\rho_{\R}, \mal})=O(\varpi_1(X,x)^{\rho, \mal})\ten_k\R$. The weight filtration is just given by the lower central series on the pro-unipotent radical, so descends to $k$, giving an ind-$k$-MHS on the $k$-Hopf algebra $O(\varpi_1(X,x)^{\rho, \mal})$.

If $\vv$ is a variation of Hodge structure on $X$, and $R= \GL(\vv_x)$, then Corollary \ref{kmhspin} recovers the ind-MHS on $O(\varpi_1(X,x)^{\rho, \mal})$ first described in \cite[Theorem 13.1]{hainrelative}. If $T$ is a set of basepoints, and $R$ is the algebraic groupoid $R(x,y)= \Iso(\vv_x,\vv_y)$ on objects $T$, then Remark \ref{kmhspinbasepts} recovers the ind-MHS on $\varpi_1(X;T)^{\rho, \mal}$  first described in \cite[Theorem 13.3]{hainrelative}.
\end{remark}

\begin{corollary}
If $\pi_1(X,x)$ is algebraically good with respect to $R$ and the homotopy groups $\pi_n(X,x)$ have finite rank for all $n\ge 2$, with each $\pi_1(X,x)$-representation  $\pi_n(X,x)\ten_{\Z} \R$ an extension of $R$-representations,  then  Theorem \ref{mhspin} gives mixed Hodge structures on $\pi_n(X,x)\ten\R$ for all $n \ge 2$, by Theorem \ref{classicalpimal}.  
\end{corollary}

Before stating the next proposition, we need to observe that for any morphism $f:X \to Y$ of compact K\"ahler manifolds, the induced map $\pi_1(X,x) \to \pi_1(Y,fx)$ gives rise to a map $\varpi_1(X,x)^{\red} \to \varpi_1(Y,fx)^{\red}$ of reductive pro-algebraic fundamental groups. This is not true for arbitrary topological spaces, but holds in this case because
semisimplicity is preserved by pullbacks between compact K\"ahler manifolds, since Higgs bundles pull back to Higgs bundles.

\begin{proposition}\label{kmhsfun}
If we have a morphism $f:X \to Y$  of compact K\"ahler manifolds, and a commutative diagram
$$
\begin{CD}
\pi_1(X,x) @>f>> \pi_1(Y,fx) \\
@V{\rho}VV @VV{\varrho}V\\
R @>{\theta}>> R'
\end{CD}
$$
of groups, with $R,R'$ real reductive pro-algebraic groups to which the $(S^1)^{\delta}$-actions descend and act algebraically, and $\rho, \varrho$ Zariski-dense, then the natural map $G(X,x)^{\rho, \mal} \to \theta^{\sharp} G(Y,fx)^{\varrho, \mal}=G(Y,fx)^{\varrho, \mal}\by_{R'}R$ extends to a natural map 
$$
(X^{\rho, \mal}_{\MHS},x) \to \theta^{\sharp} (Y^{\varrho, \mal}_{\MHS}, fx)
$$
of algebraic mixed Hodge structures.
\end{proposition}
\begin{proof}
This is really just the observation that the construction $\tilde{A}^{\bt}(X,\vv)$ is functorial in $X$.
\end{proof}

Note that, combined with  Theorem \ref{fibrations}, this gives canonical MHS on homotopy types of homotopy fibres.

\section{MTS on relative Malcev homotopy types of compact K\"ahler manifolds}\label{snkmts}

Let $X$ be a compact K\"ahler manifold.

\begin{theorem}\label{mtsmal}
If   $\rho:(\pi_1(X,x))^{\red}_{\R}\to R$ is any quotient, then there is an algebraic mixed twistor structure on the relative Malcev homotopy type $(X,x)^{\rho, \mal}$, functorial in $(X,x)$, which splits on pulling back along $\row_1:\SL_2 \to C^*$, with the pullback of the splitting over $0 \in \bA^1$  given by the opposedness isomorphism.
\end{theorem}
\begin{proof}
For $O(\Bu_{\rho})$ as in Definition \ref{Budef}, we define a  $\bG_m$-action  on the de Rham complex
$$
\sA^*(X, O(\Bu_{\rho}))= \sA^*(X, \R)\ten_{\R} O(\Bu_{\rho})
$$
by taking the $\dmd$-action of $\bG_m$ on $\sA^*(X, \R)$, acting trivially on $O(\Bu_{\rho})$.

There is an essentially unique  harmonic metric on $O(\Bu_{\rho})$, and we decompose the connection $D$ as $D= d^++\vartheta$ into antisymmetric and symmetric parts, and let $D^c:= i\dmd d^+ -i\dmd \vartheta$. Now, we define the CDGA $\tilde{A}(X, O(\Bu_{\rho}))$ on $C$ by
$$
\tilde{A}^{\bt}(X, O(\Bu_{\rho})):= (A^*(X,O(\Bu_{\rho})) \ten_{\R} O(C), uD + vD^c),
$$
and we denote the differential by $\tilde{D}:=uD + vD^c$. Note that the  $\dmd$-action of $\bG_m$ makes this $\bG_m$-equivariant over $C$. Thus $\tilde{A}(X, O(\Bu_{\rho}))\in DG\Alg_{C}(R \by \bG_m)$. The construction is now the same as in Theorem \ref{mhsmal}, except that we only have a $\bG_m$-action, rather than an $S$-action.
\end{proof}

Observe that this theorem easily adapts to multiple basepoints, as considered in Remark \ref{baseptsmhs}.

\begin{corollary}\label{kmtspin}
In the scenario of Theorem \ref{mtsmal}, 
the
 homotopy groups $\varpi_n(X^{\rho, \mal},x)$ for $n\ge 2$, and the Hopf algebra $O(\varpi_1(X^{\rho, \mal},x))$ carry natural ind-MTS, functorial in $(X,x)$, and compatible with the action of $\varpi_1$ on $\varpi_n$, the Whitehead bracket  and the Hurewicz maps $\varpi_n(X^{\rho, \mal},x) \to \H^n(X,O(\Bu_{\rho}))^{\vee}$. 

Moreover, there are $\cS$-linear isomorphisms
\begin{eqnarray*}
\varpi_n(X^{\rho, \mal,x})^{\vee}\ten\cS &\cong& \pi_n(H^*(X, O(\Bu_{\rho})))^{\vee}\ten\cS\\
O(\varpi_1(X^{\rho, \mal},x)) \ten\cS &\cong& O(R \ltimes\pi_1(H^*(X, O(\Bu_{\rho}))))\ten\cS
\end{eqnarray*}
of quasi-MTS. The associated graded map from the weight filtration is just the pullback of the standard isomorphism $\gr_W\varpi_*(X^{\rho, \mal})\cong\pi_*(H^*(X, O(\Bu_{\rho})))$.
\end{corollary}
\begin{proof}
Theorem \ref{mtsmal} provides the data required by Theorems \ref{mtspin} and \ref{pinsplit} to construct $\cS$-split ind-MTS on homotopy groups. 
\end{proof}

\begin{remark}\label{cfhiggs}
Pulling back the $\SL_2$-splitting  quasi-isomorphism 
\[
 \row_1^*\tilde{A}^{\bt}(X, O(\Bu_{\rho})) \simeq \H^*(X, O(\Bu_{\rho})) \ten O(\SL_2)
\]
at the identity matrix $I \in \SL_2(\R)$ gives a real quasi-isomorphism 
\[
 {A}^{\bt}(X, O(\Bu_{\rho})) \simeq \H^*(X, O(\Bu_{\rho}));
\]
this is precisely the formality quasi-isomorphism (given by the $dd^c$-lemma) of \cite[Theorem \ref{higgs-formal}]{higgs}. 

If instead we take the pullback at $\left(\begin{smallmatrix} 1 & 0 \\ -i &1  \end{smallmatrix}\right) \in \SL_2(\Cx)$, then we get the complex quasi-isomorphism 
\[
 {A}^{\bt}(X, O(\Bu_{\rho}))\ten \Cx \simeq \H^*_{\bar{\pd}}(X, O(\Bu_{\rho})\ten \Cx)
\]
\cite[Corollary 2.1.3]{KTP} given by the  $\pd\bar{\pd}$-lemma, because since $-id+\dc=-2i\bar{\pd}$.
\end{remark}

\subsection{Unitary actions}
Although we only have a mixed twistor structure (rather than a mixed Hodge structure) on general Malcev homotopy types,  $\varpi_1(X,x)^{\red}$ has a discrete unitary action, as in Lemma \ref{discreteact}. We will extend this to a discrete unitary action on the mixed  twistor structure. On some invariants, this action will become algebraic, and then we have a mixed Hodge structure as in Lemma \ref{tfilenrich}.

For the remainder of this section, assume that  $R$ is any quotient of $\varpi_1(X,x)^{\red}$  to which the action of the discrete group $(S^1)^{\delta}$ descends, but does not necessarily act algebraically, and let $\rho: \pi_1(X,x) \to R$ be the associated representation.

\begin{proposition}\label{redenrich}
The mixed twistor structure $(X^{\rho, \mal}_{\MTS},x)$ of Theorem \ref{mtsmal} is equipped with a $(S^1)^{\delta}$-action, satisfying the properties of Lemma \ref{tfilenrich} (except algebraicity of the action). Moreover, there is a $(S^1)^{\delta}$-action on $\ugr (X^{\rho, \mal}_{\MTS},0)$, such that the $\bG_m\by \bG_m$-equivariant  splitting
$$
\bA^1 \by \ugr (X^{\rho, \mal}_{\MTS},0)\by \SL_2 \cong (X^{\rho, \mal}_{\MTS},x)\by_{C^*, \row_1}^{\oR}\SL_2
$$
of Theorem \ref{mtsmal} is also $(S^1)^{\delta}$-equivariant.
\end{proposition}
\begin{proof}
Since $(S^1)^{\delta}$ acts on $R$, it acts on $O(\Bu_{\rho})$, and we denote this action by $v \mapsto t \circledast v$, for $t \in (S^1)^{\delta}$. 
We may now adapt the proof of Theorem \ref{mhsmal}, defining the $(S^1)^{\delta}$-action on 
$\sA^*(X, \R)\ten_{\R}O(\Bu_{\rho})$
by setting $t \boxast (a\ten v) := (t \dmd a) \ten (t^2 \circledast v)$ for $t \in (S^1)^{\delta}$.
\end{proof}

\begin{remark}\label{ktpcf}
Note that taking $R=(\pi_1(X,x))^{\red}_{\R}$ satisfies the conditions of the Proposition.   As explained in Remark \ref{cfhiggs},  Theorem \ref{mtsmal} implies the formality result of \cite{KTP}, which  relates $X^{\rho,\mal}$ to Dolbeault cohomology. Because $-id+\dc=-2i\bar{\pd}$, this is just the cohomology of $X^{\rho,\mal}_{\bT,(1,i)}$. 

Now, $(-i,1)$   is not a stable point for the
 $S$-action, but has stabiliser $1\by\bG_{m,\Cx}\subset S_{\Cx}$. In \cite{KTP}, it is effectively shown that this action  of $\bG_m(\Cx)\cong \Cx^{\by}$ lifts to a discrete action 
 on    $X^{\rho,\mal}_{\bT,(1,i)}$.  From our algebraic $\bG_m$-action and discrete $S^1$-action on $X_{\bT}^{\rho,\mal}$, we may recover the restriction of this  action   to $S^1\subset\Cx^{\by}$, with $t^2$ acting as the composition of $t \in \bG_m(\Cx)$ and $t\in S^1$.

Beware that the non-abelian ``Hodge decomposition'' of \cite{KTP} is just a discrete $ \Cx^{\by}$-action, so is far too weak to give decompositions  (or even Hodge filtrations) on abelian invariants.

Another type of Hodge structure defined on $X^{\rho, \mal}$ was the real Hodge structure (i.e. $S$-action) of \cite{hodge}. This corresponded to taking the fibre of the splitting over $\left(\begin{smallmatrix} 1 & 0 \\ 0 &1  \end{smallmatrix}\right)$, giving an isomorphism $ X^{\rho, \mal}\cong \ugr X^{\rho, \mal}_{\MTS}$, and then considering the $S$-action on the latter. However, that Hodge structure was not in general compatible with the Hodge filtration.
\end{remark}

Now, Proposition \ref{redenrich} implies that the mixed twistor structures on homotopy groups given in  Theorem \ref{mtspin} have discrete $S^1$-actions. By Lemma \ref{tfilenrich}, we know that this will give a mixed Hodge structure whenever the $S^1$-action is algebraic.

\subsubsection{Continuity}

\begin{definition}
 Let $(S^1)^{\cts}$ be the real affine scheme given by setting $O((S^1)^{\cts})$ to be the ring of real-valued continuous functions on the circle. 
\end{definition}

\begin{lemma}\label{analyticpi}
There is   a group homomorphism
$$
\sqrt h:\pi_1(X,x) \to R((S^1)^{\cts}),
$$
 invariant with respect to the $(S^1)^{\delta}$-action given by combining the actions on $R$ and $(S^1)^{\cts}$,
such that $1^*\sqrt h=\rho:\pi_1(X,x) \to R(\R) $, for $1: \Spec \R \to (S^1)^{\cts}$.
\end{lemma}
\begin{proof}
This is just the unitary action from Lemma \ref{discreteact}, given on connections by $\sqrt h (t)(d^+,  \vartheta) = (d^+, t\dmd\vartheta)$, for $t \in (S^1)$. By \cite[Theorem 7]{Simpson}, the map $(S^1) \by \pi_1(X,x) \to R(\R)$ is continuous, which is precisely the property we need. 
\end{proof}

Informally, this gives a continuity 
property of the discrete $S^1$-action, and we now wish to show a similar continuity 
property for the $(S^1)^{\delta}$-action on the mixed twistor structure $(X^{\rho, \mal},x)_{\MTS}$ of Proposition \ref{redenrich}. Recalling that $X_{\bT}= X_{\MTS}\by^{\oR}_{\bA^1}\{1\}$, we want   a continuous 
map
$$
(X,x) \by S^1 \to \holim_{\substack{\lra \\ R}} (X^{\rho, \mal},x)_{\bT}
$$
over $ C^*$,  where $\ho\LLim_R$ denotes the homotopy quotient by $R$. 

 The following is essentially \cite[\S \ref{hodge-fullaction}]{hodge}:
\begin{proposition}\label{mtsmalenrich}
For the $(S^1)^{\delta}$-actions on $\ugr X^{\rho,\mal}_{\MTS}$ of Proposition \ref{redenrich} and on $(S^1)^{\cts}$, there is  an $(S^1)^{\delta}$-invariant  map
$$
h \in\Hom_{\Ho(\bS_0 \da BR((S^1)^{\cts}))}(\Sing(X,x), \holim_{\substack{\lra \\ R((S^1)^{\cts})}} (X,x)_{\bT}^{\rho,\mal}((S^1)^{\cts})_{C^*})  ),
$$
extending the map  $ h:X \to BR((S^1)^{\cts})$  corresponding to the group homomorphism $h: \pi_1(X,x) \to R((S^1)^{\cts})$ given by  $h(t)= \sqrt h(t^2)$, for $\sqrt h$ as in  Lemma \ref{analyticpi} and $t \in S^1$. Here, $(X,x)^{\rho,\mal}_{\bT}((S^1)^{\cts})_{C^*}:= \Hom_{C^*}((S^1)^{\cts}, (X,x)^{\rho,\mal}_{\bT})$. 

Moreover, for  $1: \Spec \R \to (S^1)^{\cts}$, the map
$$
1^*h:\Sing(X,x)\to (\holim_{\substack{\lra \\ R(\R)}}  (X^{\rho,\mal},x)_{\bT}((S^1)^{\cts})_{C^*})\by_{BR((S^1)^{\cts})}BR(\R)
$$
in $\Ho(\bS_0 \da BR(\R))$ is just the canonical map
$$
\Sing(X,x)\to \holim_{\substack{\lra \\ R(\R)}} (X^{\rho, \mal}(\R),x).
$$
\end{proposition}
\begin{proof}
By Lemma \ref{relevalata}, this is equivalent to giving a $(S^1)^{\delta}$-equivariant morphism
$$
\Spec \CC^{\bt}(\Sing(X),O(\Bu_h)) \to  (X^{\rho,\mal},x)_{\bT}\by_{C^*}^{\oR}(S^1)^{\cts}
$$
in $\Ho((R\by (S^1)^{\cts}) \da s\Aff_{(S^1)^{\cts}}(R))$, noting that for the trivial map $f: \{x\} \to BR((S^1)^{\cts})$, we have $\CC^{\bt}(\{x\},O(\Bu_f))= O(R)\ten O((S^1)^{\cts})$, so $x \into X$ gives a map $R\by (S^1)^{\cts}\to \Spec \CC^{\bt}(\Sing(X),O(\Bu_h))$.

Now, the description of the $S^1$-action in Lemma \ref{analyticpi} shows that the local system $O(\Bu_h)$ on $X$ has a resolution given by
$$
(\sA^*(X, O(\Bu_{\rho}))\ten_{\R} O((S^1)^{\cts}), d^+ + t^{-2}\dmd\vartheta),
$$
for $t$ the complex co-ordinate on $S^1$, so $\CC^{\bt}(\Sing(X),O(\Bu_h)) \xra{x^*} O(R)\ten O((S^1)^{\cts}) $ is quasi-isomorphic to $  E^{\bt}\xra{x^*} O(R)\ten O((S^1)^{\cts})$, where 
$$
E^{\bt}:=D(A^*(X, O(\Bu_{\rho}))\ten_{\R} O((S^1)^{\cts}), d^+ + t^{-2}\dmd\vartheta),
$$
for $D$  the denormalisation functor.

Now, $O(S^1)$ is the quotient of $O(S)$ given by $\R[x,y]/(x^2 +y^2-1)$, where $t=x+iy$, and then
$$
uD +v D^c=   t\dmd d^+ +\bar{t}\dmd \vartheta = t\dmd(d^+ + t^{-2}\dmd\vartheta).
$$

Thus $t\dmd$ gives a $(S^1)^{\delta}$-equivariant quasi-isomorphism from $ R\by (S^1)^{\cts} \xra{x} \Spec E^{\bt}$ to $ (X_{\bT},x)^{\rho, \mal}\by_{C^*}(S^1)^{\cts}$, as required.
\end{proof}

\begin{corollary}\label{kmtspinen}
For all $n$, the map
$
\pi_n(X,x) \by S^1\to \varpi_n(X^{\rho,\mal},x)_{\bT},
$
given by composing the map $\pi_n (X,x) \to \varpi_n(X^{\rho,\mal},x)$  with the $(S^1)^{\delta}$-action on $(X^{\rho,\mal},x)_{\bT}$, is continuous. 
\end{corollary}
\begin{proof}
Proposition \ref{mtsmalenrich} gives a $(S^1)^{\delta}$-invariant map
$$
\pi_n(h):\pi_n(X,x)  \to\pi_n( \holim_{\substack{\lra \\ R((S^1)^{\cts})}} (X^{\rho,\mal},x)_{\bT}((S^1)^{\cts})_{C^*}).
$$
It therefore suffices to prove that 
$$
\pi_n( \holim_{\substack{\lra \\ R((S^1)^{\cts})}} (X^{\rho,\mal},x)_{\bT}((S^1)^{\cts})_{C^*})= \varpi_n(X^{\rho,\mal},x)_{\bT}((S^1)^{\cts})_{C^*}.
$$

Observe that the morphism $S \to C^*$ factors  through $\row_1:\SL_2 \to C^*$, via the map $S \to \SL_2$  given by  the $S$-action on the identity matrix. This gives us a factorisation of 
 $(S^1)^{\cts}\to C^*$ through $\SL_2$, using the maps $(S^1)^{\cts} \to S^1 \subset S$. It thus gives a morphism $(S^1)^{\cts} \to \Spec \oR O(C^*)$, so the $\SL_2$-splitting of Theorem \ref{mtsmal} gives an equivalence 
 $$
 (X^{\rho,\mal},x)_{\bT}\by^{\oR}_{C^*}(S^1)^{\cts} \simeq (\ugr X^{\rho,\mal}_{\MTS},0)\by (S^1)^{\cts}.
 $$
 
Similarly,  we may pull back the grouplike MTS $G(X,x)^{\rho,\mal}_{\MTS}$ from Proposition  \ref{gplike} to a dg pro-algebraic group over $(S^1)^{\cts}$, and the $\SL_2$-splitting then gives us an isomorphism 
$$
\varpi_n(X,x)^{\rho,\mal}_{\bT}\by_{C^*}(S^1)^{\cts} \cong \varpi_n(\ugr X^{\rho,\mal}_{\MTS},0)\by (S^1)^{\cts},
$$
compatible with the equivalence above. 

Thus it remains only to show that
$$
\pi_n( \holim_{\substack{\lra \\ R((S^1)^{\cts})}} (\ugr X^{\rho,\mal}_{\MTS},0)((S^1)^{\cts}))= \varpi_n(\ugr X^{\rho,\mal}_{\MTS},0)((S^1)^{\cts}).
$$
Now, write $\ugr X^{\rho,\mal}_{\MTS} \simeq \bar{W}_R(R \ltimes \exp(N^s\g))= \bar{W}N^s\g$ under the equivalences of Theorem \ref{bigequiv}, for $\g \in s\hat{\cN}(R)$. By \cite[Lemma \ref{htpy-wworks}]{htpy}, the left-hand side becomes $\pi_n(\bar{W} (R \ltimes \exp(\g))((S^1)^{\cts}))$, which is just $\pi_{n-1}((R \ltimes \exp(\g))((S^1)^{\cts}))$, giving $(R \ltimes \exp(\pi_0\g))((S^1)^{\cts})$ for $n =1$, and $(\pi_{n-1}\g)((S^1)^{\cts})$ for $n \ge 2$.  Meanwhile, the right-hand side is $(R \ltimes \exp(\H_0N\g))((S^1)^{\cts})$ for $n=1$, and $(\H_{n-1}N\g)((S^1)^{\cts})$ for $n \ge 2$. Thus the required isomorphism follows from the Dold--Kan correspondence. 
\end{proof}

 Hence   (for  $\rho \co \varpi_1(X,x)^{\red}_{\R}\to R$ any quotient  to which the $(S^1)^{\delta}$-action       descends), we have:

\begin{corollary} \label{kmhspinanal}
If the group $\varpi_n(X,x)^{\rho,\mal}$  is finite-dimensional and spanned by the image of  $\pi_n(X,x)$, then the former carries a natural $\cS$-split mixed Hodge structure, which extends the mixed twistor structure of Corollary \ref{kmtspin}. This is functorial in $(X,x)$ and compatible with the action of $\varpi_1$ on $\varpi_n$ and the Whitehead bracket.
\end{corollary}
\begin{proof}
The splittings of Theorem \ref{pinsplit} and Proposition  \ref{redenrich} combine with Corollary \ref{kmtspinen} to show that the map
$$
\pi_n(X,x) \by S^1\to \varpi_n(\ugr X^{\rho,\mal}_{\MTS},0)
$$
is continuous. 
Since the splitting also gives an isomorphism $\varpi_n(\ugr X^{\rho,\mal}_{\MTS},0) \cong \varpi_n(X,x)^{\rho,\mal}$, we deduce that $\pi_n(X,x)$ spans $\varpi_n(\ugr X^{\rho,\mal}_{\MTS},0)$, so the $S^1$ action on $\varpi_n(\ugr X^{\rho,\mal}_{\MTS},0)$ is continuous. 

Since any finite-dimensional continuous 
$S^1$-action is algebraic, this gives us an algebraic $S^1$-action on $\varpi_n(\ugr X^{\rho,\mal}_{\MTS},0)$. Retracing our steps through the splitting isomorphisms, this  implies that the  $S^1$-action on $\varpi_n(X,x)^{\rho, \mal}_{\MTS}$ is algebraic.
As in  Lemma \ref{tfilenrich}, this gives an algebraic $\bG_m \by S$-action on $\row_1^*\varpi_n(X^{\rho, \mal}_{\MTS})$, so we have a  mixed Hodge structure. That this is  $\cS$-split follows from Proposition \ref{redenrich}, since the $\cS$-splitting of the MTS in Corollary \ref{kmtspin} is $S^1$-equivariant.
\end{proof}

 \begin{remark}
   Observe that if $\pi_1(X,x)$ is algebraically good with respect to $R$ and the homotopy groups $\pi_n(X,x)$ have finite rank for all $n\ge 2$, with the local system $\pi_n(X,-)\ten_{\Z} \R$ an extension of $R$-representations, then   Theorem \ref{classicalpimal} implies that $\varpi_n(X^{\rho, \mal},x)\cong \pi_n(X,x)\ten \R$, ensuring that  the hypotheses of  Corollary \ref{kmhspinanal} are satisfied.
   \end{remark}

\section{Variations of mixed Hodge and mixed twistor structures}\label{vmhssn}

Fix a connected compact K\"ahler manifold $X$. 

\subsection{Representations in MHS/MTS}

\begin{definition}\label{sAtildedef}
Define the sheaf $\tilde{\sA}^{\bt}(X)$ of CDGAs  on $X \by C$ by
$$
\tilde{\sA}^{\bt}= (\sA^* \ten_{\R} O(C), ud + v\dc),
$$
for co-ordinates $u,v$ as in Remark \ref{Ccoords}. We denote the differential by $\tilde{d}:=ud + v\dc$.
Note that 
$
\Gamma(X, \tilde{\sA}^{\bt})=\tilde{A}^{\bt}(X),
$
as given in Definition \ref{Atildedef}.
\end{definition}

\begin{definition}
 Define   a real $\C^{\infty}$ family  of mixed Hodge (resp. mixed twistor) structures $\sE$ on $X$ to be of a finite locally free $S$-equivariant (resp. $\bG_m$-equivariant)  $j^{-1}\tilde{\sA^0}_X$-sheaf  on $X \by C^*$ equipped with a finite increasing filtration $W_i\sE$ by locally free $S$-equivariant (resp. $\bG_m$-equivariant) sub-bundles such that for all $x \in X$, the pullback of $\sE$ to $x$ corresponds under Proposition \ref{flatmhs} to a mixed Hodge structure (resp. corresponds under Corollary \ref{flatmts} to a mixed twistor structure).
\end{definition}

\begin{lemma}\label{vmhslemma}
A (real) variation of mixed Hodge structures (in the sense of \cite{VMHS1}) on $X$ is equivalent to a real $\C^{\infty}$ family  of mixed twistor structures $\sE$ on $X$, equipped with a flat $S$-equivariant $\tilde{d}$-connection
$$
\tilde{D}:\sE \to \sE\ten_{j^{-1}\tilde{\sA^0_X}}j^{-1}\tilde{\sA^1_X},
$$
compatible with the filtration $W$.
\end{lemma}
\begin{proof}
Given a real VMHS $\vv$, we obtain a $\C^{\infty}$ family   $\sE:=\xi(\vv\ten \sA^0, \bF)$ of mixed Hodge structures (in the notation of Corollary \ref{flathfil}), and the connection $D:\vv\ten \sA^0 \to \vv\ten \sA^1$ gives $\tilde{D}=\xi(D, \bF) $.  Writing $\sA^n_{\Cx}:= \sA^n\ten \Cx$, $S$-equivariance of $\tilde{D}$ is equivalent to   the condition
$$
D:F^p(\vv\ten \sA^0_{\Cx}) \to (F^p(\vv\ten \sA^0_{\Cx}) \ten_{\sA^0_{\Cx}}\sA^{01}) \oplus  (F^{p-1}(\vv\ten \sA^0_{\Cx}) \ten_{\sA^0_{\Cx}}\sA^{10}),
$$ corresponding to  a  Hodge filtration on $\vv\ten \O_X$, with $D:F^p(\vv\ten \O_X) \to F^{p-1}(\vv\ten \O_X)\ten_{\O_X}\Omega_X^1$. 
\end{proof}

\begin{definition}
Adapting \cite[\S 1]{MTS} from complex to real structures, we define  a (real) variation of mixed twistor structures (or VMTS) on $X$ to consist of a real $\C^{\infty}$ family  of mixed twistor structures $\sE$ on $X$, equipped with a flat $\bG_m$-equivariant $\tilde{d}$-connection
$$
D:\sE \to \sE\ten_{j^{-1}\tilde{\sA^0}}j^{-1}\tilde{\sA^1},
$$
compatible with the filtration $W$.
\end{definition}

\begin{definition}
Given an ind-MHS (resp. ind-MTS) structure on a Hopf algebra $O(\Pi)$,
define an MHS (resp. MTS) representation of $G$ to consist of a MHS (resp. MTS) $V$, together with a morphism 
$$
V \to V\ten O(\Pi)
$$
of  ind-MHS (resp. ind-MTS),  codistributive with respect to the Hopf algebra comultiplication.
\end{definition}


\begin{theorem}\label{kvmhsrep}\label{kvmtsrep}
For $\varrho: \pi_1(X,x) \to {}^{\VHS}\!\varpi_1(X,x)$ (resp. $\varrho: \pi_1(X,x) \to \varpi_1(X,x)^{\red}$),
the category of MHS (resp. MTS)  representations of $\varpi_1(X^{\varrho, \mal},x)$  is equivalent to the category of real variations of mixed Hodge structure (resp. variations of mixed twistor structure) on $X$. 
Under this equivalence, the forgetful functor to real MHS (resp. MTS) sends a real VMHS (resp. VMTS) $\vv$ to $\vv_x$. 

For $R$ any $S^1$-equivariant quotient of ${}^{\VHS}\!\varpi_1(X,x)$ (resp. any quotient of  $\varpi_1(X,x)^{\red}$), MHS (resp. MTS) representations of $\varpi_1(X,x)^{R, \mal}$ correspond to real VMHS (resp. VMTS) $\vv$ whose underlying local systems are extensions of $R$-representations.
\end{theorem}
\begin{proof}
 We will prove this for VMHS. The proof for VMTS is almost identical, replacing $S$ with $\bG_m$, and Proposition \ref{flatmhs} with Proposition \ref{flatmts}.

Given an MHS representation $\psi:\varpi_1(X,x)^{R, \mal} \to \GL(V) $ for an MHS $V$, observe that because $\varpi_1$ is of non-positive weights, $\psi$ maps to 
\[
 W_0\GL(V):= \Spec O(\GL(V))/W_{-1} = (W_0\End(V))\cap \GL(V),
\]
 the group of automorphisms of $V$ preserving the weight filtration. Next, note that 
\[
 \gr^W_0\GL(V):= \Spec \gr^W_0O(\GL(V))
\]
is the group $\prod_i \GL(\gr^W_iV)$ of graded automorphisms of $\gr^WV$, so
\[
 \gr^W_0\GL(V) =W_0\GL(V)/ (I + W_{-1}\End(V)).
\]

Since the tangent space of $\Ru\varpi_1$ is of strictly negative weights, this means that $\psi( \Ru \varpi_1(X,x)^{R, \mal}) \subset I + W_{-1}\End(V) $, so $\psi$ induces an $S$-equivariant morphism $R \to \gr^W_0\GL(V)$ on the reductive quotients. Since $R$ is a quotient of  ${}^{\VHS}\!\varpi_1(X,x)$, this corresponds by Proposition \ref{vhsequiv} to a VHS $\gr^W\vv$ on $X$ with $(\gr^W\vv)_x= \gr^WV$.

We then set 
\[
 G:= R\by_{\gr^W_0\GL(V) }W_0\GL(V), \quad U:= I + W_{-1}\End(V),
\]
so we have $\psi\co \varpi_1(X^{\rho, \mal},x) \to G$ a homomorphism over $R$, with $G \to R$ a surjection with pro-unipotent kernel $U$. Note that 
\[
 \gr^WG= R\by_{\gr^W_0\GL(V) }W_0\GL(\gr^WV)= R \ltimes (I + W_{-1}\End(\gr^WV)).
\]

This gives rise to  $\bG_m \by S$-equivariant  affine group objects $U_{\MHS} \lhd G_{\MHS}$ over $\bA^1 \by C^*$, given by 
$$
G_{\MHS} = \oSpec \xi(O(G), \MHS), \quad U_{\MHS} = \oSpec \xi(O(U), \MHS),
$$
and we then have a morphism $G_{\MHS} \to \bA^1 \by R \by C^*$ with kernel $U_{\MHS}$. We construct $\gr^WG_{\MHS}, \gr^WU_{\MHS}$ similarly, with $\gr^WG_{\MHS} =  (\bA^1 \by R \by C^*)\ltimes \gr^WU_{\MHS} $.

The $\SL_2$-splitting of $(X^{\rho, \mal}_{\MHS},x)$ from Theorem \ref{mhsmal} then guarantees that $\row_1^*O(X^{\rho, \mal}_{\MHS},x)$ satisfies the conditions of Proposition \ref{repbarY} on  $\bA^1 \by \SL_2$, and hence on $\bA^1 \by Y$ for any $S$-equivariant affine scheme $Y$ over $C^*$. Thus a $\bG_m \by S$-equivariant morphism
\[
 \varpi_1(X,x)_{\MHS}^{\rho, \mal}|_{\bA^1 \by Y} \to \gr^WG_{\MHS}|_{\bA^1 \by Y}
\]
is given by an element $(\omega, \alpha)$ of
\[
\mc( \Gamma(Y, W_0\tilde{A}^{\bt}(X, \fu_{\psi}))^S)\by^{I+\Gamma(Y, W_0\tilde{A}^0(X, \fu_{\psi}))^S , x^*}(I +   W_0(O(Y) \ten \fu)^S),
\]
where $\fu_{\psi}$ is the local system $W_{-1}\bEnd(\gr^W\vv)$ and $\fu= W_{-1}\End(\gr^WV)$.

Since the weight filtration on $\tilde{A}$ is defined by d\'ecalage and $\fu$ is of strictly negative weights, this simplifies to
\[
 \mc( \Gamma(Y, \tilde{A}^{\bt}(X,  \fu_{\psi}))^S)\by^{I +\Gamma(Y, \tilde{A}^0(X,  \fu_{\psi}))^S ), x^*}(I +   (O(Y) \ten \fu)^S).
\]

If we write $(\sE_0, \tilde{D}_0)$ for the $S$-equivariant $\C^{\infty}$-family of mixed twistor structures with flat $\tilde{d}$-connection coming from $\gr^W\vv$  as in Lemma \ref{vmhslemma}, then $\tilde{D}_0 +\omega$ is another such connection on $\sE_0|_Y$. Now, Proposition \ref{cSubiq} gives an isomorphism $\beta \co V\ten \cS \cong (\gr^WV)\ten \cS$ of quasi-MHS, so by Lemmas \ref{sluniv} and \ref{slhodge}, on any such $Y$ we have a non-canonical isomorphism
\[
  G_{\MHS}|_{\bA^1 \by Y} \cong \gr^WG_{\MHS}|_{\bA^1 \by Y}.
\]

Putting these together,  a $\bG_m \by S$-equivariant morphism $\varpi_1(X^{\rho, \mal},x)_{\MHS}|_{\bA^1 \by Y} \to  G_{\MHS}|_{\bA^1 \by Y}$ is the same as a pair 
\[
 (\sE|_Y, \tilde{D}|_Y):= (\sE_0, \tilde{D}_0+\omega)
\]
as above
with $\gr^W(\sE|_Y, \tilde{D}|_Y)=(\sE_0, \tilde{D}_0)|_Y$, together with a $W$-filtered isomorphism $\beta^{-1}\circ \alpha \co (\sE|_Y, \tilde{D}|_Y)_x \cong  \xi(V, \bF)$. 

This description does not depend on the choice $\beta$ of splitting, so is functorial in $Y$. Gluing these for all $Y$ shows that a $\bG_m \by S$-equivariant morphism $\varpi_1(X^{\rho, \mal},x)_{\MHS}\to  G_{\MHS}$
 is the same as a pair $(\sE, \tD)$ as in Lemma \ref{vmhslemma}, equipped with isomorphisms
 \[
 \gr^W(\sE, \tilde{D})\cong (\sE_0, \tilde{D}_0), \quad (\sE, \tilde{D})_x \cong \xi(V, \bF).
\]
By Lemma \ref{vmhslemma}, this gives a VMHS $\vv$ with $\gr^W\vv=\gr^W\vv$ and $\vv_x=V$. 

That every VMHS arises in this way follows from the observation that there exist $\C^{\infty}$-isomorphisms $\sA^0_X(\vv) \cong \sA^0(\gr^W\vv)$ for all VMHS $\vv$, because all extensions of $\C^{\infty}$ vector bundles are trivial.
\end{proof}

For $R$ as in Theorem \ref{kvmhsrep}, we now have the following.
\begin{corollary}\label{kvmhspin}\label{kvmtspin}
There is a canonical commutative algebra $\bO(\varpi_1X^{R, \mal})$ in ind-VMHS (resp. ind-VMTS) on $X\by X$, with $\bO(\varpi_1X^{R, \mal})_{x,x}= O(\varpi_1(X^{R, \mal}, x))$. This has a  comultiplication 
$$
\pr_{13}^{-1}\bO(\varpi_1X^{R, \mal})\to \pr_{12}^{-1}\bO(\varpi_1X^{R, \mal})\ten \pr_{23}^{-1}\bO(\varpi_1X^{R, \mal})
$$  on 
$X \by X \by X$, 
a co-identity $\Delta^{-1}\bO(\varpi_1X^{R, \mal})\to \R$ on $X$ (where $\Delta(x) =(x,x)$) and a co-inverse $\tau^{-1}\bO(\varpi_1X^{R, \mal}) \to  \bO(\varpi_1X^{R, \mal})$ (where $\tau(x,y)= (y,x)$), all of which are morphisms of algebras in ind-VMHS (resp. ind-VMTS).

There are canonical ind-VMHS (resp. ind-VMTS) $\Pi^n(X^{R, \mal})$ on $X$ for all $n \ge 2$, with $\Pi^n(X^{R, \mal})_x= \varpi_n(X^{R, \mal}, x)^{\vee}$.
\end{corollary}
\begin{proof}
The left and right actions of   $\varpi_1(X^{R, \mal}, x)$ on itself make $O(\varpi_1(X^{R, \mal}, x))$ into an ind-MHS (resp. ind-MTS) representation of  $\varpi_1(X^{R, \mal}, x)^2$, so it corresponds under Theorem \ref{kvmhsrep} to an ind-VMHS (resp. ind-VMTS) $\bO(\varpi_1X^{R, \mal}) $ with the required properties. Theorem \ref{kmhspin} makes  $\varpi_n(X^{R, \mal}, x)^{\vee}$ into an ind-MHS-representation of $\varpi_1(X^{R, \mal}, x)$, giving $\Pi^n(X^{R, \mal})$.
\end{proof}

Note that for any VMHS (resp. VMTS) $\vv$, this means that we have a canonical morphism $\pr_2^{-1}\vv \to \pr_1^{-1}\vv \ten \bO(\varpi_1X^{\varrho, \mal})$ of ind-VMHS (resp. ind-VMTS) on $X \by X$, for $\varrho$ as in Theorem \ref{kvmhsrep}.

\begin{remark}\label{kvmhsbasepts}
Using Remarks \ref{baseptslike} and \ref{kmhspinbasepts}, we can adapt Theorem \ref{kvmhsrep} to any MHS/MTS representation $V$ of the groupoid $\varpi_1(X; T)^{R, \mal}$ with several basepoints (where we define representations by requiring that each map  $V(x) \to O(\varpi_1(X; x,y)^{R, \mal})\ten V(y)$ be a morphism of ind-MHS/MTS). This gives a VMHS/VMTS $\vv$, with canonical isomorphisms $\vv_x \cong V(x)$ of MHS/MTS for all $x \in T$. 

Corollary \ref{kvmhspin} then adapts to multiple basepoints, since there is a natural representation of $\varpi_1(X; T)^{R, \mal}\by \varpi_1(X; T)^{R, \mal}$ given by $(x, y) \mapsto O(\varpi_1(X; x,y)^{R, \mal})$. This gives a canonical Hopf algebra $\bO(\varpi_1X^{R, \mal})$ in ind-VMHS/VMTS on $X\by X$, with $\bO(\varpi_1X^{R, \mal})_{x,y}= O(\varpi_1(X^{R, \mal}; x,y)$ for all $x, y \in T$. Since this construction is functorial for sets of basepoints, we deduce that this is the VMHS/VMTS $\bO(\varpi_1X^{R, \mal})$ of Corollary \ref{kvmhspin} (which is therefore independent of the basepoint $x$). This generalises \cite[Corollary 13.11]{hainrelative} (which takes $R= \GL(\vv_x)$ for a VHS $\vv$).

Likewise, the representation $x \mapsto  \varpi_n(X^{R, \mal}, x)^{\vee}$ of $\varpi_1(X; T)^{R, \mal}$ gives an ind-VMHS/VMTS $\Pi^n(X^{R, \mal})$ (independent of $x$) on $X$ with $\Pi^n(X^{R, \mal})_x= \varpi_n(X^{R, \mal}, x)^{\vee}$, for all $x \in X$.
\end{remark}

\begin{remark}\label{arapuraconj2}
\cite{arapurapi1} introduces a quotient $\varpi_1(X,x)^{\alg}_k \to \pi_1(X,x)_k^{\mathrm{hodge}}$ over any field $k \subset \R$, characterised by the  the property that representations of $\pi_1(X,x)_k^{\mathrm{hodge}}$ correspond to local systems underlying $k$-VMHS on $X$.

Over any field $k \subset \R$, there is a pro-algebraic group $\mathrm{MT}_k$ over $k$, whose representations correspond to mixed Hodge structures over $k$. If $\rho: \varpi_1(X,x)^{\alg}_k\to{}^{\VHS}\!\varpi_1(X,x)_k$ is the largest quotient of the $k$-pro-algebraic completion with the property that the surjection $ \varpi_1(X,x)^{\alg}_{\R} \to \varpi_1(X,x)^{\alg}_k\ten_k\R$ factors through ${}^{\VHS}\!\varpi_1(X,x)_{\R}$, then Theorem \ref{kmhspin} and Remark \ref{arapuraconj} give an algebraic action of $\mathrm{MT}_k$ on  $\varpi_1(X,x)^{\rho,\mal}_k$, with representations of $\varpi_1(X,x)^{\rho,\mal}_k \rtimes \mathrm{MT}_k$ being representations of $\varpi_1(X,x)^{\rho,\mal}_k$ in $k$-MHS. Theorem \ref{kvmhsrep} implies that these are precisely $k$-VMHS on $X$, so  \cite[Lemma 2.8]{arapurapi1} implies that $\varpi_1(X,x)^{\rho,\mal}_k = \pi_1(X,x)_k^{\mathrm{hodge}}$.  

For any quotient $R$ of ${}^{\VHS}\!\varpi_1(X,x)_k \to R$ (in particular if $R$ is the image of the monodromy representation of a $k$-VHS), Theorem \ref{kvmhsrep} then implies that $\varpi_1(X,x)^{R,\mal}_k$ is a quotient of  $\pi_1(X,x)_k^{\mathrm{hodge}}$,  proving the second part of  \cite[Conjecture 5.5]{arapurapi1}.

Note that this also implies that if $\vv$ is a local system  on $X$ whose semisimplification $\vv^{ss}$ underlies a VHS, then $\vv$ underlies a VMHS (which need not be compatible with the VHS on $\vv^{ss}$).
\end{remark}

\begin{example}
One application of the ind-VMHS on $\bO(\varpi_1X^{\rho, \mal})$ from Corollary \ref{kvmhspin} is to look at deformations of the representation associated to a VHS $\vv$. Explicitly, $\vv$ gives representations $\rho_x:   {}^{\VHS}\!\varpi_1(X,x) \to \GL(\vv_x)$ for all $x \in X$, and for any Artinian local commutative $\R$-algebra $A$ with residue field $\R$, we consider the formal scheme $F_{\rho_x}$ given by
$$
F_{\rho_x}(A) = \Hom(\pi_1(X,x), \GL(\vv_x \ten A))\by_{\Hom(\pi_1(X,x), \GL(\vv_x ))}\{\rho_x\}.
$$

Now, $\GL(\vv_x \ten A)= \GL(\vv_x) \ltimes \exp(\gl(\vv_x) \ten \m(A))$, where $\m(A)$ is the maximal ideal of $A$. If $R(x)$ is the image of $\rho_x$, and $\rho_x':{}^{\VHS}\!\varpi_1(X,x) \to R(x)$ is the induced morphism,  then 
$$
F_{\rho_x}(A)= \Hom( \varpi_1(X,x)^{\rho'_x, \mal}, R(x) \ltimes  \exp(\gl(\vv_x) \ten \m(A))_{\rho_x}.
$$
Thus $F_{\rho_x}$ is a formal  subscheme contained in the germ at $0$ of $O(\varpi_1(X,x)^{\rho'_x, \mal})\ten \gl(\vv_x)$, defined by the conditions 
$$
f( a\cdot b)= f(a)\star (\ad_{\rho'_x(a)}( f(b)))
$$ 
for $a,b \in \varpi_1(X,x)^{\rho'_x, \mal} $, where $\star$ is the Campbell--Baker--Hausdorff  product $a\star b= \log(\exp(a)\cdot \exp(b))$.

Those same conditions define a family $\bF(\rho)$ on $X$ of formal subschemes contained in $ (\Delta^{-1}\bO(\varpi_1X^{\rho', \mal}))\ten \gl(\vv)$, with $\bF(\rho)_x= F_{\rho_x}$. If $\bF = \Spf \sB$, the VMHS on $\bO(\varpi_1X^{\rho', \mal})$ and $\vv$ then give $\sB$ the natural structure of a (pro-Artinian algebra in) pro-VMHS. This generalises \cite{carlosrepvhs} to real representations, and also adapts easily to $S$-equivariant representations in more general groups than $\GL_n$. Likewise, if we took $\vv$ to be any variation of twistor structures, the same argument would make $\sB$ a pro-VMTS.
\end{example}

\subsection{Enriching VMTS}

Say we have some quotient $R$  of $\varpi_1(X,x)^{\red}$  to which the action of the discrete group $(S^1)^{\delta}$ descends, but does not necessarily act algebraically. Corollary \ref{kmtspin} puts an ind-MTS  on the Hopf algebra $O(\varpi_1(X,x)^{R, \mal})$, and Proposition \ref{redenrich} puts a $(S^1)^{\delta}$ action on $\xi(O(\varpi_1(X,x)^{R, \mal}), \MTS)$, satisfying the conditions of Lemma \ref{tfilenrich}.

Now take an MHS $V$, and assume that we have an MTS  representation $\varpi_1(X,x)^{R, \mal} \to \GL(V)$, with the additional property that  the corresponding morphism 
$$
\xi(V, \MHS) \to \xi(V, \MHS) \ten  \xi(O(\varpi_1(X,x)^{R, \mal}), \MTS)
$$
of ind-MTS is equivariant for the $(S^1)^{\delta}$-action.

Now, $\gr^W_nV$ is an MTS representation of $\gr^W_0\varpi_1(X,x)^{R, \mal}= R$, giving a $(S^1)^{\delta}$-equivariant map
$$
\gr^W_nV \to \gr^W_nV \ten O(R).
$$

If $\vv$ is the local system associated to $V$, then
this is equivalent to giving a compatible system of isomorphisms $\gr^W_n\vv \cong t\circledast \gr^W_n\vv$ for $t \in S^1$. Therefore Proposition \ref{vhsequiv} implies that $\gr^WV$ is a representation of ${}^{\VHS}\!\varpi_1(X,x)$. Letting $R'$ be the largest common quotient of $R$ and ${}^{\VHS}\!\varpi_1(X,x)$, this means that $\gr^WV$ is an $R'$-representation, so $V$ is a representation of $\varpi_1(X,x)^{R', \mal}$. 

Then we have a $S^1$-equivariant morphism
  $$
\xi(V, \MHS) \to \xi(V, \MHS) \ten  \xi(O(\varpi_1(X,x)^{R', \mal}), \MHS)
$$
of ind-MTS, noting that $S^1$ now acts algebraically on both sides (using Corollary \ref{mhspin}), so Lemma \ref{tfilenrich} implies that this is a morphism of ind-MHS, and therefore that $V$ is an MHS representation of $\varpi_1(X,x)^{R', \mal}$. Theorem \ref{kvmhsrep} then implies that this amounts to $\vv$ being a VMHS on $X$.

Combining this argument  with Corollary \ref{kmtspinen} immediately gives:

\begin{proposition} \label{kvmhspinen}
Under the conditions of Corollary \ref{kmhspinanal}, the local system associated to the $\pi_1(X,x)$-representation $\varpi_n(X,x)^{R,\mal}$ naturally underlies a VMHS, which is independent of the basepoint $x$.
\end{proposition}

\subsection{Absolute Hodge and twistor homotopy types}\label{hodgeHTsn}

Given a group $G$ acting on a topological space $Y$, with homotopy quotient $Z$, we have a homotopy fibration sequence
\[
 Y \xra{\pi} Z \to BG,
\]
so can apply Theorem \ref{fibrations} to describe the relative Malcev homotopy types of $Z$ in terms of those of $Y$. A very similar phenomenon arises in \cite{weiln} when comparing arithmetic and geometric \'etale homotopy types, with $G$ being the Galois group, giving fibration sequences
\[
 (Z\ten_F \bar{F})_{\et}^{K, \mal} \to Z_{\et}^{R, \mal} \to (\Spec F)_{\et}^{T, \mal}= (B\Gal(F))^{T,\mal}
\]
for a scheme $Z$ over a field $F$.

\subsubsection{Definitions}

\begin{definition}\label{hodgeHT}
Given an algebraic mixed Hodge structure $Y^{R, \mal}_{\MHS}$ on a relative Malcev homotopy type $Y^{R, \mal} $, define the absolute Hodge homotopy type  $Y^{R, \mal}_{\cH}$ by
\[
 O(Y^{R, \mal}_{\cH}):=\oR \Gamma(\bA^1 \by C^*, O(Y^{R, \mal}_{\MHS}))^{\bG_m},
\]
where the $\bG_m$-action is given by the morphism 
\begin{eqnarray*}
 \bG_m &\xra{\eta}& \bG_m \by S\\
a &\mapsto&  (a, a^{-1}).
\end{eqnarray*}
Since $(\bG_m \by S)/\eta(\bG_m)= S$, 
the $\bG_m \by S$-action on $O(Y^{R, \mal}_{\MHS})$ then induces  an $S$-action on $O(Y^{R, \mal}_{\cH})$, making it an $R \rtimes S$-equivariant CDGA. 

For an explicit model of $O(Y^{R, \mal}_{\cH})$, we will always take 
\[
 \Gamma(\bA^1 \by C^*, O(Y^{R, \mal}_{\MHS})\ten_{\O_{C^*}}\oR O(C^*))^{\bG_m}.
\]
\end{definition}

\begin{definition}
 Given an algebraic mixed twistor structure $Y^{R, \mal}_{\MTS}$ on a relative Malcev homotopy type $Y^{R, \mal} $, define the absolute twistor homotopy type  $Y^{R, \mal}_{\cT}$ by
\[
 O(Y^{R, \mal}_{\cT}):=\oR \Gamma(\bA^1 \by C^*, O(Y^{R, \mal}_{\MTS}))^{\bG_m},
\]
for the $\bG_m$-action $\eta \co \bG_m \to \bG_m \by \bG_m$. For an explicit model of  $O(Y^{R, \mal}_{\cT})$, we will always take the $R \by \bG_m$-equivariant CDGA 
\[
 \Gamma(\bA^1 \by C^*, O(Y^{R, \mal}_{\MTS})\ten_{\O_{C^*}}\oR O(C^*))^{\bG_m}.
\]
\end{definition}

Note that for a homotopy type with a MHS, the    absolute twistor homotopy type associated to the underlying  MTS is just given by 
the absolute Hodge homotopy type, equipped only with the action of $\bG_m \subset S$.

\begin{definition}\label{cWdef}
Given a $\bG_m$-representation $V$, there is a grading on $V$ which we denote by
\[
 V= \bigoplus_{n \in \Z} \cW_nV,
\]
with $a \in \bG_m$ acting as $a^n$ on $\cW_nV$.

Given an $S$-representation $V$, the inclusion $\bG_m \into S$ (given by $v=0$ in the co-ordinates of Remark \ref{Ccoords}) then gives a grading $\cW_nV$ on $V$.
  Equivalently, $ \cW_n(V\ten \Cx)$ is the sum of elements of type $(p,q)$ for $p+q=n$.
 \end{definition}

When $O(Y^{R, \mal}_{\MHS})$ is given by a Rees construction $\xi(O(Y^{R, \mal}_{\bF}), W)$, we can simplify the expression for the Hodge homotopy type to
\[
 O(Y^{R, \mal}_{\cH})= \bigoplus_m \cW_m(  \Gamma(C^*, W_mO(Y^{R, \mal}_{\bF})\ten_{\O_{C^*}}\oR O(C^*))),
\]
and the analogous statement is true for absolute twistor homotopy types.

At a point, this gives
\[
 O(\{*\}^{1, \mal}_{\cH})= \cW_{\ge 0}\oR O(C^*).
\]
The $S$-action on the identity matrix $I$ gives a morphism $\jmath \co S \to \SL_2$, and hence $\jmath^* \co\oR O(C^*) \to O(S)$. Since this is $S$-equivariant, we have
\[
 (\cW_{\ge 0}\oR O(C^*)\xra{\jmath^*} O(S)) \in DG\Alg(S)_{0*}.
\]

Any morphism $y^* \co  O(Y^{R, \mal}_{\MHS} \to O(R)\ten  O(\{*\}^{1, \mal}_{\MHS})$ of algebraic mixed Hodge structures then gives rise to 
\[
  (O(Y^{R, \mal}_{\cH}) \xra{\jmath^*y^*  } O(R \rtimes S)) \in DG\Alg(R \rtimes S)_*.
\]
Similarly, a morphism $y^* \co  O(Y^{R, \mal}_{\MTS} \to O(R)\ten  O(\{*\}^{1, \mal}_{\MTS})$ of algebraic mixed twistor structures gives rise to 
\[
 (O(Y^{R, \mal}_{\cT}) \xra{\jmath^*y^*  } O(R \by \bG_m)) \in DG\Alg(R \by \bG_m)_*.
\]

\begin{definition}\label{GHdef}\label{GTdef}
 Define 
\begin{eqnarray*}
 G(Y,y)_{\cH}^{R, \mal}&:=& \bar{G}_{R \rtimes S}(\jmath^*y^*\co  O(Y^{R, \mal}_{\cH}) \xra{\jmath^*y^*  } O(R \rtimes S))\\
G(Y,y)_{\cT}^{R, \mal}&:=& \bar{G}_{R \by \bG_m}(\jmath^*y^*\co  O(Y^{R, \mal}_{\cT}) \xra{\jmath^*y^*  } O(R \by \bG_m)),
\end{eqnarray*}
  noting these lie in $dg\cE(R \rtimes S)$ and $dg\cE(R \by \bG_m)$, respectively.

We then define the absolute Hodge and absolute twistor homotopy groups by 
\begin{eqnarray*}
 \varpi_n(Y,y){R, \mal}_{\cH}&:=& \pi_{n-1}(\Spec D)G(Y,y)_{\cH}^{R, \mal}\\ 
\varpi_n(Y,y){R, \mal}_{\cT}&:=& \pi_{n-1}(\Spec D)G(Y,y)_{\cT}^{R, \mal}.
\end{eqnarray*}
\end{definition}
In Proposition \ref{hodgefibration}, we will see that the absolute Hodge and absolute twistor homotopy groups admit a particularly simple description.

\subsubsection{MHS, MTS and fibrations}

If we write $\MHS$ and $\MTS$ for the categories of real mixed Hodge and mixed twistor structures (Definitions \ref{quasiMHS} and \ref{abmtsdef}), then these are neutral Tannakian categories over real vector spaces, with the fibre functor sending each MHS $(V,W,F)$ to the underlying real vector space $V$, and each MTS $\sE$ to $\sE_1$.

We then write $\Pi(\MHS)$ and $\Pi(\MTS)$ for the Tannakian duals (in the sense of \cite{tannaka}) of these categories. These are real affine group schemes, with the property that a $\Pi(\MHS)$-action (resp. $\Pi(\MTS)$-action) on a real vector space $V$ is equivalent to a MHS (resp. MTS) on $V$.

\begin{proposition}\label{SHSequiva}
There is a natural isomorphism
\[ 
 G(\{*\})_{\cH}^{1, \mal}=\bar{G}_{S}(\jmath^*\co \cW_{\ge 0}\oR O(C^*) \to O(S))\cong \Pi(\MHS).
\]
of  pro-algebraic   groups.
\end{proposition}
\begin{proof}
 This will follow by combining Theorem \ref{SHSequiv} with Lemma \ref{reducedequivlemma}.
\end{proof}

\begin{proposition}\label{STSequiva}
There is a natural isomorphism
\[ 
G(\{*\})_{\cT}^{1, \mal}=\bar{G}_{\bG_m}(\jmath^*\co\cW_{\ge 0}\oR O(C^*)\to O(\bG_m))\cong \Pi(\MTS)
\]
of pro-algebraic   groups.
\end{proposition}
\begin{proof}
 This will follow by combining Theorem \ref{STSequiv} with Lemma \ref{reducedequivlemma}.
\end{proof}

\begin{proposition}\label{hodgefibration}
If the weights on $\H^*O(\ugr Y^{R,\mal}_{\MHS})$ (resp. $\H^*O(\ugr Y^{R,\mal}_{\MTS})$) are bounded below, then
the grouplike homotopy type $G(Y,y)^{R, \mal}$ is the homotopy fibre of 
\[
 G(Y,y)^{R, \mal}_{\cH} \to \Pi(\MHS),
 \text{ resp. } G(Y,y)^{R, \mal}_{\cT} \to \Pi(\MTS).
\]

Moreover, the basepoint $y$ gives a section of these maps, so $G(Y,y)^{R, \mal}_{\cH}  $  (resp. $G(Y,y)^{R, \mal}_{\cT}  $)  is the semidirect product of $\Pi(\MHS) $ (resp. $\Pi(\MTS) $) and a simplicial algebraic group quasi-isomorphic to $ G(Y,y)^{R, \mal}$. In particular, the homotopy groups of Definition \ref{GHdef} are given by
\[
 \varpi_n(Y,y)^{R, \mal}_{\cH} \cong \varpi_n(Y,y)^{R, \mal}, \quad \varpi_n(Y,y)^{R, \mal}_{\cT} \cong \varpi_n(Y,y)^{R, \mal} 
\]
for $n \ge 2$, and
\[
  \varpi_1(Y,y)^{R, \mal}_{\cH} \cong \varpi_1(Y,y)^{R, \mal} \rtimes \Pi(\MHS), \quad \varpi_1(Y,y)^{R, \mal}_{\cT} \cong \varpi_1(Y,y)^{R, \mal} \rtimes \Pi(\MTS).
\]
\end{proposition}
\begin{proof}
We will prove this for MHS --- the proof for MTS is almost identical, replacing $S$ with  $\bG_m$ and $\cH$ with $\cT$.

If we write $\cF$ for the homotopy fibre of $G(Y,y)^{R, \mal}_{\cH} \to \Pi(\MHS) $, then because $\bar{W}$ from Definition \ref{barwg} is a right adjoint, the functor $\bar{W}_R$ of Corollary \ref{WRcor} satisfies 
\[
 \bar{W}_R(N^s\cF)\simeq \bar{W}_{R \rtimes S}(N^sG(Y,y)^{R, \mal}_{\cH})\by^h_{\bar{W}_{S}\Pi(\MHS))} \Spec \R.
\]
 The equivalences of Theorem \ref{bigequiv} allow us to rewrite this as
\[
 O( \bar{W}_R(N^s\cF))\simeq O(Y^{R, \mal}_{\cH})\ten^{\oL}_{ O(\{*\}^{R, \mal}_{\cH}) }\R.
\]

It therefore suffices to show that the natural map
\[
\psi \co O(Y^{R, \mal}_{\cH})\ten^{\oL}_{ O(\{*\}^{R, \mal}_{\cH}) }\R \to O(Y^{R, \mal})
\]
is  quasi-isomorphism, since the equivalences of Theorem \ref{bigequiv} then give
\[
 N^s\cF \simeq \bar{G}_{ R \rtimes S}( \jmath^*y^*\co O(Y^{R, \mal}_{\cH}) \to O(R\rtimes S)),
\]
as required.

We may replace $O( Y^{R,\mal}_{\MHS})$ with a quasi-isomorphic flat complex over $\bA^1$, in which case $O(Y^{R, \mal}_{\MHS})$ is given by a Rees construction $\xi(O(Y^{R, \mal}_{\bF}), W)$, by Lemma \ref{flatfiltrn}. Then we are looking at the derived tensor product
\[
( \bigoplus_m \cW_m(  \Gamma(C^*, W_mO(Y^{R, \mal}_{\bF})\ten_{\O_{C^*}}\oR O(C^*)))\ten^{\oL}_{ \cW_{\ge 0}\oR O(C^*)}\R.
\]
The weight filtration $W$ on $O(Y^{R, \mal}_{\bF}) $ induces a filtration $W$ on $ O(Y^{R, \mal}_{\cH})$, given by
\[
W_i O(Y^{R, \mal}_{\cH}):=  \bigoplus_m \cW_m\Gamma(C^*, W_mW_iO(Y^{R, \mal}_{\bF})\ten_{\O_{C^*}}\oR O(C^*)).
\]

The graded pieces are then
\[
 \gr^W_iO(Y^{R, \mal}_{\cH})= \bigoplus_{m} \cW_m\Gamma(C^*, \gr^W_iO(Y^{R, \mal}_{\bF})\ten_{\O_{C^*}}\oR O(C^*)),
\]
which by the opposedness isomorphism is quasi-isomorphic to 
\[
 \bigoplus_{m} \cW_m( \cW_i(W_mO(\ugr Y^{R, \mal}_{\MHS}))\ten_{\R}\oR O(C^*))= \cW_i(O(\ugr Y^{R, \mal}_{\MHS}))\ten_{\oR} \bigoplus_{m\ge i} \cW_{m-i}\oR O(C^*)).
\]
We then have
\[
 W_i O(Y^{R, \mal}_{\cH})\ten^{\oL}_{ \cW_{\ge 0}\oR O(C^*)}\R \simeq \cW_i(O(\ugr Y^{R, \mal}_{\MHS})),
\]
so $\psi$ induces quasi-isomorphisms on $\gr^W$. Our hypotheses ensure that $W_i$ is acyclic for $i \ll 0$, so  $\psi$ is a quasi-isomorphism, as required.

Alternatively, note that we could prove this proposition by adapting Theorem \ref{fibrations}, using the 
absolute Hodge spectral sequence
\[
 \Ext^i_{\MHS}(\R, \H^j(Y, \vv)) \abuts \H^{i+j}(O(Y^{R, \mal}_{\cH}))\ten^{R \rtimes S}\vv_y. 
\]
instead of the Leray--Serre spectral sequence.
\end{proof}

\subsubsection{Absolute Hodge and absolute twistor cohomology}

For our pointed connected compact K\"ahler manifold $X$, we now give alternative descriptions of the  absolute Hodge and twistor homotopy types.

\begin{definition}
Given a semisimple local system $\vv$ on $X$ equipped with a grading $V= \bigoplus_m \cW_m\vv$, 
we define the weight filtration $W$ on 
\[
 \tilde{A}^{\bt}(X, \vv)= (A^*(X,\vv) \ten_{\R} O(C), uD + vD^c)
\]
by d\'ecalage. Explicitly,
\[
 W_n \tilde{A}^{\bt}(X, \vv) = \bigoplus_m  \tau^{\le n-m}\tilde{A}^{\bt}(X, \cW_m\vv).
\]
 Note that this agrees with the definition in the special case of Theorem \ref{mhsmal} because $O(\Bu_{\rho})$ has weight $0$.
 \end{definition}

\begin{definition}
Given a   semisimple local system $\vv$ on $X$,  equipped with a grading, we set
 \[
\tilde{A}_{\cT}^{\bt}(X, \vv):= ( (W_0 \tilde{A}^{\bt}(X, \vv))\ten_{O(C)}\oR O(C^*))^{\bG_m},
\]
for $\bG_m$ acting via $\eta \co \bG_m \to \bG_m \by \bG_m$.

Given a VHS $\vv$ on $X$, we set
\[
\tilde{A}_{\cH}^{\bt}(X, \vv):= ( (W_0 \tilde{A}^{\bt}(X, \vv))\ten_{O(C)}\oR O(C^*))^{S}.
\]
\end{definition}

We now let  $O(\Bu_{\rho} \by \bG_m)$ be the graded semisimple local system $O(\Bu_{\rho}) \ten O(\bG_m) $, with grading given by the $\bG_m$-action.
If $R$ is moreover  an $S^1$-equivariant quotient $R$ of ${}^{\VHS}\!\varpi_1(X,x)$, we set  $O(\Bu_{\rho} \rtimes S)=  O(\Bu_{\rho}) \ten O(S)$, with the right $R$-action defined as a semi-direct tensor product, as in Definition \ref{semitendef}.

\begin{lemma}\label{hodgeHTlemma}
For any quotient  $\rho \co \varpi_1(X,x)^{\red}\onto R$,  the absolute twistor homotopy type of $X$  relative to $R$ is given by
 \[
O(X^{R, \mal}_{\cT})\simeq  \tilde{A}_{\cT}^{\bt}(X, O(\Bu_{\rho} \by \bG_m)).
\]

If $R$ is moreover  an $S^1$-equivariant quotient $R$ of ${}^{\VHS}\!\varpi_1(X,x)$, then the absolute 
 Hodge homotopy type of $X$ relative to $R$ is given by
  \[
 O(X^{R, \mal}_{\cH})\simeq  \tilde{A}_{\cH}^{\bt}(X, O(\Bu_{\rho} \rtimes S)).
\]
\end{lemma}
\begin{proof}
We first prove this in the twistor case.
As observed above, we have
\begin{eqnarray*}
O(X^{R, \mal}_{\cT})&=& \bigoplus_m \cW_m(  \Gamma(C^*, W_mO(X^{R, \mal}_{\bT})\ten_{\O_{C^*}}\oR O(C^*))) \\    
&=&  \bigoplus_m \cW_m (W_m \tilde{A}^{\bt}(X, O(\Bu_{\rho}))\ten_{O(C)}\oR O(C^*)).
\end{eqnarray*}

Meanwhile, note  that 
\[
W_0 \tilde{A}^{\bt}(X, O(\Bu_{\rho} \by \bG_m))= \bigoplus_m  \tau^{\le m}\tilde{A}^{\bt}(X, O(\Bu_{\rho}))\ten \cW_{-m}O(\bG_m),
\]
and that for any $\bG_m$-representation $V$ we have $\cW_{-m}O(\bG_m)\ten^{\bG_m}V \cong\cW_mV$, the weight $m$ subspace for the $\bG_m$-action, so
\[
\tilde{A}_{\cT}^{\bt}(X, O(\Bu_{\rho} \by \bG_m))= \bigoplus_m\cW_m(W_m \tilde{A}^{\bt}(X, O(\Bu_{\rho}))\ten_{O(C)}\oR O(C^*) ),
\]
as required.

The proof in the Hodge case is almost identical, replacing $\bG_m$ with $S$ as required, and noting that  $\cW_{-m}O(S)\ten^{S}V \cong\cW_mS$ for all $S$-representations $V$.
\end{proof}

For an explicit description, note that 
\[
 \tilde{A}_{\cT}^n(X, O(\Bu_{\rho} \by \bG_m)) \subset \bigoplus_{m\ge n} \cW_m(A^n(X,  O(\Bu_{\rho}))\ten \oR O(C^*))= \bigoplus_{i\ge 0} A^n(X,  O(\Bu_{\rho}))\ten \cW_i\oR O(C^*),
\]
because $A^n$ is of weight $n$ with respect to the $\bG_m$-action. 

Moreover, observe that $\tD \co A^n \ten \cW_i\oR O(C^*) \to A^{n+1} \ten  \cW_{i-1}\oR O(C^*)$, so 
\[
 \tilde{A}_{\cT}^n(X, O(\Bu_{\rho} \by \bG_m)) \subset \bigoplus_{i\ge 0} A^n(X,  O(\Bu_{\rho})\ten \cW_i\oR O(C^*)
\]
consists of $(a_0, a_1, \ldots)$ with $\tD a_0 =0$. The same description also holds for $ \tilde{A}_{\cH}^n(X, O(\Bu_{\rho} \by S))$. 

\begin{lemma}\label{AHcSlemma}
 There is a canonical isomorphism
\[
  \tilde{A}_{\cH}^{\bt}(X, O(\Bu_{\rho} \rtimes S)) \cong \gamma^0((W_0 A^{\bt}(X, O(\Bu_{\rho} \rtimes S)))\ten \Omega^{\bt}(\cS/\R)),
\]
for the filtration $\gamma^pV= V \cap (F^pV_{\Cx})$.
\end{lemma}
\begin{proof}
This is an immediate consequence of Lemma \ref{slhodge}, because $\oR O(C^*)$ is the flat algebraic Hodge filtration corresponding to $\Omega^{\bt}(\cS/\R)$.
\end{proof}

\begin{lemma}\label{vhshodgecoho}
 For any VHS $\vv$ on $X$, we have a quasi-isomorphism of cochain complexes 
\[
  \tilde{A}_{\cH}^{\bt}(X, \vv) \simeq \oR \Gamma_{\cH}(X, \vv)
\]
to the absolute Hodge cohomology of $X$ with coefficients in $\vv$, as in \cite{beilinson}. 
\end{lemma}
\begin{proof}
First observe that $\tilde{A}_{\cH}^{\bt}(X, \vv) = \tilde{A}_{\cH}^{\bt}(X, O(\Bu_{\rho} \rtimes S))\ten^{R \rtimes S}V$, where $V$ is the $R$-representation associated  to $\vv$, for any suitable quotient $R$ of ${}^{\VHS}\!\varpi_1(X,x)$. Then
we have 
\begin{eqnarray*}
 \tilde{A}_{\cH}^{\bt}(X, \vv) &\simeq& \oR \Gamma(\bA^1 \by C^*, O(X^{R, \mal}_{\MHS}) )^{\bG_m}\ten^{R \rtimes S}V\\
&\simeq& \oR \Gamma(\bA^1 \by C^*, O(X^{R, \mal}_{\MHS}\ten V))^{\bG_m \by R \rtimes S}\\ 
&\simeq& \oR \Gamma([\bA^1/\bG_m] \by [C^*/S],  A^{\bt}(X,\vv)),
\end{eqnarray*}
which is just $\oR \Gamma_{\cH}(X, \vv) $, as required. 
\end{proof}

\subsubsection{VMHS and VMTS}

\begin{proposition}\label{pi1Hprop}
 A representation of $ \varpi_1(X,x)^{R, \mal}_{\cH}$  (resp. $\varpi_1(X,x)^{R, \mal}_{\cT}$) on a vector space $V$ is the same as a VMHS (resp. VMTS) $\vv$ with $\vv_x=V$ whose underlying local system is an extension of $R$-representations.
\end{proposition}
\begin{proof}
Proposition \ref{hodgefibration} gives 
\begin{eqnarray*}
  \varpi_1(X,x)^{R, \mal}_{\cH}&=&\varpi_1(X,x)^{R, \mal} \rtimes \Pi(\MHS)\\
 \varpi_1(X,x)^{R, \mal}_{\cT}&=&\varpi_1(X,x)^{R, \mal} \rtimes \Pi(\MTS),
\end{eqnarray*}
 so a representation is the same as a representation of $\varpi_1(X,x)^{R, \mal}$ in MHS (resp. MTS), which is a VMHS (resp. VMTS) by Theorem \ref{kvmhsrep}. 
\end{proof}

We can use this to study deformations of a VHS, generalising Proposition \ref{cSubiq} from MHS to VMHS:

\begin{proposition}
Given a VHS (resp. VTS) $\vv$ on $X$, the category of VMHS (resp. VMTS) $\vv'$ on $X$ with fixed isomorphism $\gr^W\vv'\cong \vv$ is given by 
\[
 \Del(\tilde{A}_{\cH}^{\bt}(X,W_{-1}\bEnd\vv) ),\text{ resp. }\Del(\tilde{A}_{\cT}^{\bt}(X,W_{-1}\bEnd\vv) )
\]
 for the Deligne groupoid $\Del$ of Definition \ref{defDel}.
\end{proposition}
\begin{proof}
We address the  VMHS case -- the proof of the VMTS case is entirely similar.  If $R$ is the Zariski closure of the representation $\rho \co \pi_1(X,x) \to \GL(\vv_x)$ associated to $\vv$, then 
it follows from Proposition \ref{pi1Hprop} that the category of VMHS $\vv'$ we require is equivalent to the groupoid of homomorphisms
\[
 \varpi_1(X,x)^{R, \mal}_{\cH} \to (R \rtimes S) \ltimes(1+ W_{-1}\End(\vv_x))
\]
over $R\rtimes S$, with morphisms given by the conjugation action of $1+ W_{-1}\End(\vv_x)$.

From the definition of $\bar{G}$, it follows that this is the same as the groupoid
\[
 \Del( \tilde{A}_{\cH}^{\bt}(X, O(\Bu_{\rho} \rtimes S))\ten^{R \rtimes S}W_{-1}\End(\vv_x)),
\]
 and then we just observe that we automatically have an isomorphism
\[
 \tilde{A}_{\cH}^{\bt}(X, O(\Bu_{\rho} \rtimes S))\ten^{R \rtimes S}W_{-1}\End(\vv_x) \cong \tilde{A}_{\cH}^{\bt}(X,W_{-1}\bEnd\vv)
\]
of DG Lie algebras.
\end{proof}

\begin{remark}\label{cSubiq2}
Explicitly, an object of $\Del(\tilde{A}_{\cT}^{\bt}(X,W_{-1}\bEnd\vv) )$ is an element
\[
 \phi \in \tilde{A}_{\cT}^1(X,W_{-1}\bEnd\vv) 
\]
with $[d, \phi] + \phi \wedge \phi =0$. 

The decomposition $\oR O(C^*)= O(\SL_2) \oplus O(\SL_2)(-1)[-1]\eps$ as a graded algebra gives us a decomposition
 $\phi =\omega+\beta\eps$. The conditions on $\phi$ then become
\[
 [\tD, \omega] +\omega \wedge \omega =0, \quad  [N, \omega] +[ \tD, \beta] +  [\omega,\beta]=0.
\]

Writing $\tD':= \tD + \omega$ and $N':= N+\beta$, these conditions reduce to 
\[
 (\tD')^2=0, \quad [\tD', N']=0.
\]

Thus we have a flat $\bG_m$-equivariant $\tilde{d}$-connection 
\[
\tD'\co  \sA_X^0(\vv)\ten O(\SL_2) \to \sA_X^1(\vv)\ten O(\SL_2)
\]
such that $\gr^W\tD'= \tD$,
together with a compatible  $N$-derivation 
\[
 N'\co  \sA_X^0(\vv)\ten O(\SL_2)\to   \sA_X^0(\vv)\ten O(\SL_2)(-1)
\]
with $\gr^WN'=N$.

Thus $N'$ is surjective (since $N$ is so), and $\tD'$ gives a flat $\bG_m$-equivariant  $\tilde{d}$-connection 
\[
\tD'\co  (\ker N') \to \sA_X^1\ten_{\sA^0_X}(\ker N'),
\]
which are precisely the data required for a VMTS $\vv'$.

If we replace $\cT$ with $\cH$ and $\bG_m$ with $S$, the same construction  gives a VMHS $\vv'$ by Lemma \ref{vmhslemma}.
\end{remark}

\section{Monodromy at the Archimedean place}\label{archmon}

Remark \ref{gadata}  shows that the mixed Hodge (resp. mixed twistor)  structure on $G(X,x_0)^{R, \mal}$ can be recovered from a nilpotent monodromy operator $\beta:O(R \ltimes \exp(\g)) \to O(R \ltimes \exp(\g))\ten \cS(-1)$, where   
$\g= \bar{G}(\H^*(X,O(\Bu_{\rho})) )$. In this section, we show how to calculate  the monodromy operator in terms of standard operations on the de Rham complex. 


\begin{definition}
If there is an algebraic action of $S^1$ on the reductive pro-algebraic group $R$, set  $S':=S$. Otherwise, set  $S':=\bG_m$. These two cases will correspond to mixed Hodge and mixed twistor structures, respectively.  
\end{definition}

We now show how to recover $\beta$ explicitly from the formality quasi-isomorphism of Theorem \ref{mhsmal}. By Corollary \ref{keysplit2}, $\beta$ can be regarded as an element of
$$
W_{-1}\EExt^0_{\H^*(X,O(\Bu_{\rho}) ) }(\bL_{\H^*(X,O(\Bu_{\rho}))}^{\bt},   (\H^*(X,O(\Bu_{\rho})) \to O(R)  ) \ten O(\SL_2)(-1))^{R \rtimes S'}.
$$

\begin{definition}
Recall that we set $\tD= uD+vD^c$, and define $\tDc:= xD+yD^c$, for co-ordinates $\left(\begin{smallmatrix} u & v  \\ x & y \end{smallmatrix} \right) $ on $\SL_2$. Note that $\tDc$ is of type $(0,0)$ with respect to the $S$-action, while $\tDc$ is of type $(1,1)$.
\end{definition}

As in the proof of Theorem \ref{mhsmal}, 
there are
$R \ltimes S'$-equivariant quasi-isomorphisms
$$
\H^*(X,O(\Bu_{\rho}))\ten O(\SL_2) \xla{p} Z_{\tDc} \xra{i} \row_1^*\tilde{A}^{\bt}(X,O(\Bu_{\rho}))
$$
of CDGAs,
where $\z_{\tDc}:=\ker(\tDc)\cap \row_1^*\tilde{A}^{\bt} $ (so has differential $\tD$). These are moreover compatible with the augmentation maps  to $O(\Bu_{\rho})_{x_0}\ten O(\SL_2)= O(R)\ten O(\SL_2)$.

\begin{definition}
For simplicity of exposition, we write
\begin{eqnarray*}
 \cH^*&:=& \H^*(X,O(\Bu_{\rho}))\\
\cZ^{\bt}&:=& \ker(\tDc)\cap \row_1^*\tilde{A}^{\bt}(X,O(\Bu_{\rho}))\\
\cA^{\bt}&:=& \row_1^*\tilde{A}^{\bt}(X,O(\Bu_{\rho})),
\end{eqnarray*}so  the quasi-isomorphisms become
$$
\cH^*\ten O(\SL_2) \xla{p}\cZ^{\bt} \xra{i} \cA^{\bt}.
$$
We also set $\cO:= O(R)$,  $\uline{\cH}^*:= \cH^*\ten O(\SL_2)$ and $\uline{\cO}:= \cO \ten O(\SL_2)$.
\end{definition}

This gives the following $R \ltimes S'$-equivariant quasi-isomorphisms of $\Hom$-complexes:
$$
\xymatrix{
  \oR\cHom_{\cA}( \bL_{\cA/O(C)}, \cA(-1)\xra{{x_0}^*} \uline{\cO}(-1) ) \ar[r]^{i^*}             &\oR\cHom_{\cZ}( \bL_{\cZ/O(C)}, \cA(-1)\xra{{x_0}^*} \uline{\cO}(-1))\\
  \oR\cHom_{\cZ}( \bL_{\cZ/O(C)}, \cZ(-1)\to \uline{\cO}(-1))\ar[ru]^{i_*} \ar[r]^{p_*}  & \oR\cHom_{\cZ}( \bL_{\cZ/O(C)}, \uline{\cH}(-1)\to \uline{\cO}(-1))\\
\oR\cHom_{\uline{\cH}}( \bL_{\uline{\cH}/O(C)}, \uline{\cH}(-1)\to \uline{\cO}(-1))\ar[ru]^{p^*}.
}
$$
(Note that, since $\cH^0=\R$ and $\cZ^0=O(\SL_2)$, in both cases the augmentation maps to  $\cO$ are independent of the basepoint ${x_0}$.) 

The derivation $N:O(\SL_2) \to O(\SL_2)(-1)$  has kernel $O(C)$, so yields an $O(C)$-derivation $\cA \to \cA(-1)$, and hence an  element 
$$
(N,0) \in \Hom_{\cA,R \ltimes S'}( \bL_{\cA/O(C)}, \cA(-1)\xra{{x_0}^*} \cO(-1) )^0 \quad\text{ with } \quad d(N,0)= (0,  {x_0}^*\circ N ).
$$
Moreover, the map $O(\SL_2) \to \cA$ induces a map
\[
 \oR\cHom_{\cA,R \ltimes S'}( \bL_{\cA}, \cA(-1)\xra{{x_0}^*} \cO(-1) )\to \oR\cHom_{\cA,R \ltimes S'}(\cA\ten \Omega(\SL_2/C) , \cA(-1)\xra{{x_0}^*} \cO(-1) ),
\]
under which the image of $(N,0)$ is the isomorphism  $ \cA\ten \Omega(\SL_2/C) \cong  \cA(-1)$.

The chain of quasi-isomorphisms then yields a homotopy-equivalent element 
\[
 f \in \Hom_{\uline{\cH}}( \bL_{\uline{\cH}/O(C)}, \uline{\cH}(-1)\to \uline{\cO}(-1))^0,
\]
whose  restriction to  $\cH\ten\Omega(\SL_2/C)$  is  the identification $N \co \cH\ten\Omega(\SL_2/C)\cong \cH(-1)\ten O(\SL_2)$, with $df= (0,  {x_0}^*\circ N )$.

Since  
$$
\bL_{(\cH\ten O(\SL_2))/O(C)} \cong (\bL_{\cH/\R}\ten O(\SL_2)) \oplus (\cH\ten\Omega(\SL_2/C)),
$$
we can rewrite this as $f= \beta +N$, for 
$$
\beta\in \Hom_{\cH,R \ltimes S'}( \bL_{\cH}, \cH(-1)\ten O(\SL_2)\to \cO(-1))^0. 
$$
Note that $N \circ {x_0}^*=0$ on $\cH \subset \cH\ten O(\SL_2)$, so $d\beta = 0$.

\subsection{Reformulation via $E_{\infty}$ derivations}

\begin{definition}
Given a  commutative DG algebra $B$ without unit, define $E(B)$ to be the real  graded Lie coalgebra $\CoLie(B[1])$ freely cogenerated by $B[1]$. Explicitly,
$
\CoLie(V)= \bigoplus_{n \ge 1} \CoLie^n(V),
$
where $\CoLie^n(V)$ is the quotient of $V^{\ten n}$ by the elements 
$$
\mathrm{sh}_{pq}(v_1 \ten \ldots v_n ):= \sum_{\sigma \in \mathrm{Sh}(p,q)} \pm v_{\sigma(1)} \ten \ldots \ten  v_{\sigma(n)},  
$$
for $p,q >0$ with $p+q=n$. Here, $\mathrm{Sh}(p,q)$ is the set of $(p,q)$ shuffle permutations, and $\pm$ is the Koszul sign.

$E(B)$ is equipped with a differential $d_{E(B)}$ defined on cogenerators $B[1]$ by
$$
(q_B+ d_B): ({\bigwedge}^2(B[1]) \oplus B[1])[-1] \to B[1],
$$
where $q_B: \Symm^2 B \to B$ is the product on $B$.  Since $d_{E(B)}^2=0$, this turns $E(B)$ into a differential graded Lie coalgebra.
\end{definition}

Freely cogenerated differential graded Lie coalgebras are known as strong homotopy commutative algebras (SHCAs). A choice of cogenerators $V$ for an SHCA $E$ is then known as an $E_{\infty}$ or $C_{\infty}$ algebra. For more details, and analogies with $L_{\infty}$ algebras associated to DGLAs, see \cite{Kon}. 
Note that when $B$ is concentrated in strictly positive degrees, $E(B)$ is dual to the dg Lie algebra $G(B \oplus \R)$ of Definition \ref{barwg}. 

\begin{definition}
The functor $E$ has a left adjoint $O(\bar{W}_+)$, given by 
$O(\bar{W}_+(C)):= \bigoplus_{n >0}\Symm^n(C[-1])$, with differential as in Definition \ref{barwg}. In particular, if $C=\g^{\vee}$, for $\g \in dg\hat{\cN}$, then $\R \oplus O(\bar{W}_+ C)= O(\bar{W}\g)$, for the functor $\bar{W}$ of Definition \ref{barwg}.

For any dg Lie coalgebra $C$, we therefore define $O(\bar{W}C)$ to be the unital dg algebra $\R \oplus O(\bar{W}_+C)$. 
\end{definition}

Now, the crucial property of this construction is that $O(\bar{W}_+E(B))$ is a cofibrant replacement for $B$ in the category of non-unital dg algebras (as follows for instance from the proof of \cite[Theorem \ref{ddt1-mcequiv}]{ddt1}, interchanging the r\^oles of Lie and commutative algebras). Therefore for any CDGA $B$ over $A$, $O(\bar{W}E(B))\ten_{O(\bar{W}E(A))}A$ is a cofibrant replacement for $B$ over $A$, so 
$$
\bL_{B/A} \simeq \ker(\Omega( O(\bar{W}E(B)))\ten_{O(\bar{W}E(B))}B \to \Omega(O(\bar{W}E(A)))\ten_{O(\bar{W}E(A))}B  ).
$$

Thus
$$
\oR\cHom_{\cZ}( \bL_{\cZ}, B)\simeq \cDer(O(\bar{W}E(\cZ))\ten_{O(\bar{W}E(\R))}\R , B),
$$
the complex of derivations over $\R$. This in turn is isomorphic to  the complex
$$
\cDer_{E(\R)}(E(\cZ), E(B) )
$$
of dg Lie coalgebra derivations. The remainder of this section is devoted to constructing explicit homotopy inverses for the equivalences above,   thereby deriving the element 
$$
\beta=(\alpha, \gamma) \in \cDer(E(\cH), E(\cH)\ten O(\SL_2))^0 \by \cDer(E(\cH), E(\R)\ten O(\SL_2))^{-1}
$$
required by Remark \ref{gadata}, noting that the second term can be rewritten to give $\gamma \in G(\cH)^0$. 

\subsection{K\"ahler identities}\label{kahlerids}

By \cite[\S 1]{Simpson}, we have first-order K\"ahler identities
$$
D^*= -[\L, D^c], \quad (D^c)^*=[\L, D]
$$
(noting that our operator $D^c$ differs from Simpson's by a factor of $-i$), with Laplacian
$$
\Delta= [D,D^*] = [D^c, (D^c)^*]=-D\L D^c +D^c\L D +DD^c\L+ \L DD^c  .
$$

Since $uy-vx=1$, we also have
$$
\Delta= -\tD\L \tDc+ \tDc\L \tD+  \tD\tDc \L +\L\tD\tDc.
$$

\begin{definition}
Define a semilinear involution $*$ on $O(\SL_2)\ten \Cx$ by $u^*=y, v^*=-x$. This corresponds to the map $A \mapsto (A^{\dagger})^{-1}$ on $\SL_2(\Cx)$. The corresponding involution on $S$ is given by $\lambda^*= \bar{\lambda}^{-1}$, for $\lambda \in S(\R) \cong \Cx^{\by}$. 
\end{definition}

The calculations above combine to give:
\begin{lemma}
$$
\tD^*= -[\L,\tDc] \quad \tDc^*:= [\L, \tD].
$$
\end{lemma}

Note that this implies that $\tD\tDc^* +\tDc^*\tD=0$. Also note that  Green's operator $G$ commutes with $\tD$ and $\tDc$ as well as with $\L$, and hence with $\tD^*$ and $\tDc^*$.

The working above yields the following.
\begin{lemma}\label{laplacian}
$$
\Delta= [\tD, \tD^*]= [\tDc, \tDc^*]=-\tD\L \tDc +\tDc\L \tD +\tD\tDc\L+ \L \tD\tDc  .
$$
\end{lemma}

\subsection{ Monodromy calculation }\label{calcn}

Given any operation $f$ on $\cA$ or $\cZ$, we will simply denote the associated dg Lie coalgebra derivation on $E(\cA)$ or $E(\cZ)$ by $f$, so $d_{E(\cA)}= \tD +q= d_{E(\cZ)}$.

Note that the complex $\cDer(C,C)$ of  coderivations of a dg Lie coalgebra $C$ has the natural structure of a DGLA, with bracket $[f,g] = f\circ g - (-1)^{\deg f\deg g}g\circ f$. When $C=E(B)$, this DGLA is moreover pro-nilpotent, since $E(B)= \bigoplus_{n \ge 1} \CoLie^n (B[1])$, so
$$
\cDer(E(B),E(B)) = \Lim_n \cDer(\bigoplus_{1 \le m \le n} \CoLie^m (B[1]), \bigoplus_{1 \le m \le n} \CoLie^m (B[1])).
$$

Since $[\tDc, \tD^*]=0$ and 
$$
\id = \pr_{\cH} +G\Delta= \pr_{\cH} +G(\tDc \tDc^*+ \tDc^*\tDc),
$$ 
it follows that $\im(\tDc^*)$ is a subcomplex of $\cZ= \ker(\tDc)$, and $\cZ= \cH\ten O(\SL_2) \oplus \im(\tDc^*)$.

\begin{definition}
Decompose $\im(\tDc^*)$ as $B\oplus C$, where $B= \ker(\tD) \cap \im(\tDc^*)$, and $C$ is its orthogonal complement. Since $i:\cZ \to \cA $ is a quasi-isomorphism, $\tD: C \to B$ is an isomorphism, and we may  define $h_i: \cA \to \cA[-1]$ by $h_i(z+b+c)= \tD^{-1}b \in C$, for $z \in \cZ, b \in B, c \in C$. Thus $h_i^2=0$, and $\id= \pr_{\cZ} + \tD h_i +h_i\tD$. Explicitly,
$$
h_i:= G\tD^* \circ (1-\pr_{\cZ})= G^2 \tD^*\tDc^*\tDc=  G^2 \tDc^*\tDc\tD^*,
$$ where $G$ is  Green's operator and $\pr_{\cZ}$ is orthogonal projection onto $\cZ$. Since $\tD\tDc= DD^c$, we can also rewrite this as $G^2 D^*D^{c*}\tDc $.
\end{definition}

\begin{lemma}\label{Li}
Given a derivation $f \in \cDer(E(\cZ),E(\cA))^0$ with $[q,f] +[\tD,f]=0$, let
$$
\gamma_i(f):= \sum_{n \ge 0} (-1)^{n+1}   h_i  \circ [q,h_i]^n\circ  (f+ h_i \circ [q,f]).
$$
Then $\gamma_i(f) \in \cDer(E(\cZ),E(\cA))^{-1}$, and 
$$
f + [d_E, \gamma_i(f)] = \pr_{\cZ} \circ (\sum_{n \ge 0} (-1)^n    \circ [q,h_i]^n\circ  (f+ h_i \circ [q,f])),
$$
so lies in $\cDer(E(\cZ),E(\cZ))^0$.
\end{lemma}
\begin{proof}
First, observe that $h_i$ is $0$ on $\cZ$, so $g\circ h_i =0$ for all $ g \in \cDer(E(\cZ),E(\cA))$, and therefore $h_i\circ g= [h_i,g]$ is a derivation. If we write $\ad_q(g)= [q,g]$, then $\ad_q(h_i \circ e)= [q,h_i]\circ e$, for any $e \in \cDer(E(\cZ),E(\cA))^0$  with $[q,e]=0$. Then 
$$
\ad_q(h_i\circ ad_q(h_i \circ e))= [q,h_i \circ [q,h_i]\circ e] = [q,h_i] \circ [q,h_i] \circ e + h_i\circ  \half[[q,q], h_i] \circ e, 
$$
which is just $[q,h_i]^2\circ e$, since $q^2=0$ (which amounts to saying that the multiplication on $\cA$ is associative), so $\ad_q^2=0$. 

Now, 
$$
\ad_q(h_i \circ f)= [q, h_i \circ f]= [q,h_i] \circ f - h_i \circ [q,f], 
$$
and this lies in $\ker(\ad_q)$.
Proceeding inductively, we get
$$
\gamma_i(f):= \sum_{n \ge 0} (-1)^{n+1}   (\ad_{h_i}\ad_q)^n \ad_{h_i} f, 
$$
which is clearly a derivation. Note that the  sum is locally finite because the $n$th term maps $\CoLie^{m}(\cZ)$ to $\CoLie^{m-n}(\cA)$.

Now, let $y:= \sum_{n \ge 0} (-1)^n   (\ad_q \ad_{h_i})^n f$, so $\gamma_i(f)= -[h_i,y ]= -h_i \circ y$. Set $f':=f + [d_E, \gamma_i(f)]$; we wish to show that $[\tD, h_i \circ f'] + h_i \circ [\tD, f']=0$. Note that $f+ [q,\gamma_i(f)]=y$,
so
$$
f'= f- [q, h_i \circ y] -[\tD, h_i \circ y] = y - [\tD, h_i] \circ y - h_i \circ [\tD,y].
$$
Since $\pr_{\cZ}= (\id - [\tD, h_i])$, it only remains to show that $h_i \circ [\tD,y]=0$, or equivalently that $h_i \circ [\tD,f']=0$.

Now,
$
0 = [d_E,f']= [\tD,f'] +[q, f'].
$
Since $[q, \cZ] \subset \cZ$, this means that
$$
h_i \circ [\tD,f'] = -h_i \circ [q, h_i \circ [\tD,f']].
$$
Since $h_i \circ \ad_q$ maps  $\CoLie^n(\cA[1])$ to $\CoLie^{n-1}(\cA[1])$, this means that $h_i \circ [\tD,f']=0$, since $h_i \circ [\tD,f']= (-h_i \circ \ad_q)^n\circ (h_i \circ [\tD,f'])$ for all $n$, and this is $0$ on $\CoLie^n(\cA[1])$. 
\end{proof}

\begin{definition}
On the complex $\cZ$, define $h_p:= G\tD^*$, noting that this is also isomorphic to $G\tDc\L$ here.
\end{definition}

\begin{lemma}\label{Lp}
Given a derivation $f \in \cDer(E(\cZ),E(\cH))^0$ with $[q,f] +[\tD,f]=0$, let
$$
\gamma_p(f):= \sum_{n \ge 0} (-1)^{n+1}     (f+  [q,f]\circ h_p )\circ [q,h_p]^n\circ  h_p . 
$$
Then $\gamma_p(f) \in \cDer(E(\cZ),E(\cH))^{-1}$, and 
$$
f + [d_E, \gamma_p(f)] =   (\sum_{n \ge 0} (-1)^n     (f+  [q,f]\circ h_p )\circ [q,h_p]^n)\circ \pr_{\uline{\cH}},
$$
where $\pr_{\uline{\cH}}$ is orthogonal projection onto harmonic forms. Thus
$f+ [d_E, \gamma_p(f)]$ lies in $\cDer(E(\uline{\cH}),E(\uline{\cH}))^0$.
\end{lemma}
\begin{proof}
The proof of Lemma \ref{Li} carries over, since the section of $p:\cZ \to \uline{\cH}$ given by harmonic forms corresponds to a decomposition $\cZ = \uline{\cH} \oplus \im(\tDc)$. Then  $h_p$ makes $p$ into a deformation retract, as $[h_p, \tD]= \pr_{\uline{\cH}}$ on $\cZ$.
\end{proof}

\begin{theorem}\label{archmonthm}
For $\g= G(\H^*(X, O(\Bu_{\rho})))$, the monodromy operator 
$$
\beta: O(R \rtimes \exp(\g))\to O(R \rtimes \exp(\g))\ten O(\SL_2)(-1)
$$
at infinity, corresponding to the MHS (or MTS) on the homotopy type $(X,{x_0})^{\rho, \mal}$   
is given $\beta =\alpha + \ad_{\gamma_{x_0}}$, where
$
\alpha: \g^{\vee} \to \g^{\vee} \ten O(\SL_2)(-1)
$
is 
\begin{eqnarray*}
\alpha
&=& \sum_{b>0  a\ge 0} (-1)^{a+b+1}\pr_{\uline{\cH}} \circ ([q,G^2 D^*D^{c*}] \circ \tDc )^b \circ (\tD \circ [q, G\L]) \circ  (\tDc \circ [q,G\L])^a\circ s \\
&&+ \sum_{b>0,a  > 0} (-1)^{a+b}[q,\pr_{\uline{\cH}}] \circ ([q,G^2 D^*D^{c*}] \circ \tDc )^b\circ (\tD \circ G\L)  \circ (\tDc \circ [q,G\L])^a\circ  s,
\end{eqnarray*}
for $s:\cH \to \cA$ the inclusion of harmonic forms.
Meanwhile, $\gamma_{x_0} \in \g\hat{\ten} O(\SL_2)(-1)$ is  
\begin{eqnarray*}
\gamma_{x_0} 
&=& \sum_{a \ge 0,b \ge 0} (-1)^{a+b}    
x_0^* \circ h_i\circ ([q,G^2 D^*D^{c*}] \circ \tDc )^b\circ  (\tD \circ [q, G\L])\circ  (\tDc \circ [q,G\L])^{a} \circ s\\
&&+ \sum_{a>0,b > 0} (-1)^{a+b}     x_0^*\circ ([q,G^2 D^*D^{c*}] \circ \tDc )^{b}\circ  (\tD \circ G\L) \circ (\tDc \circ [q,G\L])^a \circ s.
\end{eqnarray*}
\end{theorem}
\begin{proof}
The derivation $N:\cA \to \cA(-1)$  yields a coderivation $N \in \cDer(E(\cZ),E(\cA))^0$ with $[q,f] =[\tD,f]=0$.
Lemma \ref{Li} then gives $\gamma_i(N)\in \cDer(E(\cZ),E(\cA))^{-1}$ with $N+ [d_E,\gamma_i(N)] \in  \cDer(E(\cZ),E(\cZ))^0$. Therefore, in the cone complex $ \cDer(E(\cZ),E(\cA)) \xra{x^*} \cDer(E(\cZ),E(\uline{cO}))$, the derivation $N$ is homotopic to
$$
(N+ [d_E,\gamma_i(N)], \gamma_i(N)_{x_0}) \in \cDer(E(\cZ),E(\cZ))^0\oplus  \cDer(E(\cZ),E(\uline{\cO}))^{-1}.
$$
Explicitly,
\begin{eqnarray*}
\gamma_i(N)&=& \sum_{n \ge 0} (-1)^{n+1}   h_i  \circ [q,h_i]^n\circ  N\\
N + [d_E, \gamma_i(N)] &=& \pr_{\cZ} \circ (\sum_{n \ge 0} (-1)^n    \circ [q,h_i]^n\circ N.
\end{eqnarray*}

Setting $f:= N + [d_E, \gamma_i(N)]$, we next apply Lemma \ref{Lp} to the pair $(p\circ f, \gamma_i(N)_{x_0})$. If $s: \uline{\cH} \to \cZ$ denotes the inclusion of harmonic forms, we obtain 
\begin{eqnarray*}
\alpha\circ s &=& p \circ f+ [d_E, \gamma_p(p \circ f)],\\ 
\gamma_{x_0} \circ s &=& \gamma_i(N)_{x_0} +\gamma_p(p \circ f)_{x_0}  +[d_E, \gamma_p(\gamma_i(N)_{x_0}) + \gamma_p(p \circ f)_{x_0})].
\end{eqnarray*}

 Now, 
\begin{eqnarray*}
\alpha &=& \sum_{m \ge 0} (-1)^{m}     (p \circ f+  [q,p \circ f]\circ h_p )\circ [q,h_p]^m \circ s\\
&=& \sum_{m \ge 0, n \ge 0} (-1)^{m+n}\pr_{\uline{\cH}} \circ [q,h_i]^n\circ  N\circ [q,h_p]^m\circ s \\
&&+ \sum_{m \ge 0, n \ge 0} (-1)^{m+n}[q,\pr_{\uline{\cH}}] \circ [q,h_i]^n\circ  N\circ h_p \circ [q,h_p]^m\circ  s
\end{eqnarray*}
since $p \circ \pr_{\cZ}= \pr_{\uline{\cH}}$, $[q,N]=0$ and $\ad_q^2=0$.

Now, $N \circ g = [N,g] + g \circ N$, but $N$ is $0$ on $\cH \subset \uline{\cH}$, while $[N,s]$=0 (since $s$ is  $\SL_2$-linear). Since $h_i=G^2 D^*D^{c*}\tDc$, $h_p= \tDc G\L$ and $[q, \tDc]=0$, we get $[q, h_i] = [q,G^2 D^*D^{c*}] \circ \tDc$ and $[q,h_p] = \tDc \circ [q, G\L]$.  In particular, this implies that $[q, h_i] \circ \tDc=0$ and that $[N, [q, h_p]]= \tD \circ [q, G\L]$, since $[N,G\L]=[N,q]=0$. 

Thus
\begin{eqnarray*}
\alpha
&=& \sum_{n , a,c \ge 0} (-1)^{n+a+c+1}\pr_{\uline{\cH}} \circ [q,h_i]^n\circ  [q,h_p]^c \circ (\tD \circ [q, G\L]) \circ  [q,h_p]^a\circ s \\
&&+ \sum_{m,n  \ge 0} (-1)^{m+n}[q,\pr_{\uline{\cH}}] \circ [q,h_i]^n\circ (\tD \circ G\L)  \circ [q,h_p]^m\circ  s\\
&&+ \sum_{n,a,c  \ge 0} (-1)^{m+n}[q,\pr_{\uline{\cH}}] \circ [q,h_i]^n\circ \tDc \circ G\L  \circ [q,h_p]^c \circ (\tD \circ [q, G\L])   \circ [q,h_p]^a \circ    s.
\end{eqnarray*}
When $n=0$, all terms are $0$, since $\pr_{\uline{\cH}}\circ \tD = \pr_{\uline{\cH}}\circ \tDc=0$, and $[q,h_p]= \tDc \circ [q, G\L]$. For $n\ne 0$, the first sum is $0$ whenever $c \ne 0$, and the final sum is always $0$ (since $[q,h_i]\circ \tDc =0$). If $m=0$, the second sum is also $0$, as $\tD \circ G\L$ equals $G \tDc^*$ on $\ker (\tD)$, so is $0$ on $\cH$.  Therefore (writing $b=n$), we get
\begin{eqnarray*}
\alpha
&=& \sum_{b>0,  a\ge 0} (-1)^{a+b+1}\pr_{\uline{\cH}} \circ [q,h_i]^b \circ (\tD \circ [q, G\L]) \circ  [q,h_p]^a\circ s \\
&&+ \sum_{b>0,a  > 0} (-1)^{a+b}[q,\pr_{\uline{\cH}}] \circ [q,h_i]^b\circ (\tD \circ G\L)  \circ [q,h_p]^a\circ  s,
\end{eqnarray*}
and substituting for $[q,h_i] $ and $[q,h_p]$ gives the required expression.

Next, we look at $\gamma_{x_0}$. First, note that $\L|_{\cZ^1}=0$, so $h_p|_{\cZ^1}=0$, and therefore $h_p$ (and hence $\gamma_p(p \circ f)$) restricted to $\CoLie^n(\cZ^1)$ is $0$, so $x_0^*\circ \gamma_p(p \circ f)=0$. Thus
\begin{eqnarray*}
\gamma_{x_0} \circ s &=& \gamma_i(N)_{x_0} +  [d_E, \gamma_p(\gamma_i(N)_{x_0})]\\
&=& \sum_{m \ge 0} (-1)^{m}     ( \gamma_i(N)_{x_0} +  [q,\gamma_i(N)_{x_0}]\circ h_p )\circ [q,h_p]^m \circ s\\
&=&\sum_{m,n \ge 0} (-1)^{m+n+1}    x_0^*\circ h_i  \circ [q,h_i]^n\circ  N\circ [q,h_p]^m \circ s\\
&&+ \sum_{m,n \ge 0} (-1)^{m+n+1}    [q, x_0^*\circ h_i  \circ [q,h_i]^n\circ  N] \circ  h_p\circ [q,h_p]^m \circ s.
\end{eqnarray*}

On restricting to $\cH\subset \uline{\cH}$, we may replace $N\circ g$ with $[N,g]$ (using the same reasoning as for $\alpha$). Now, $[q,h_i]^{n+1}\circ h_p=0$, and on expanding out $\tDc\circ [N, [q,h_p]^m]$, all terms but one vanish, giving
\begin{eqnarray*}
\gamma_{x_0} \circ s 
&=& \sum_{m>0,n \ge 0} (-1)^{m+n+1}    x_0^*\circ h_i  \circ [q,h_i]^n\circ  (\tD \circ [q, G\L])\circ  [q,h_p]^{m-1} \circ s\\
&&+ \sum_{m>0,n \ge 0} (-1)^{m+n+1}   x_0^*\circ  [q,h_i]^{n+1}\circ  (\tD \circ G\L) \circ [q,h_p]^m \circ s,
\end{eqnarray*}
which expands out to give the required expression.
\end{proof}

\begin{remark}\label{gadata2}
This implies that  the MHS $O(\varpi_1(X,{x_0})^{\rho, \mal})$ is just the kernel of 
$$
\beta\ten \id  + \ad_{\gamma_{x_0}}\ten \id + \id\ten N: O(R \rtimes \exp(\H_0\g))\ten \cS \to  O(R \rtimes \exp(\H_0\g))\ten \cS(-1),
$$
where $\beta, \gamma_{x_0}$ here denote the restrictions of $\beta, \gamma_{x_0}$ in Theorem \ref{archmonthm} to $\Spec \cS = \left(\begin{smallmatrix} 1 &  0 \\ \bA^1 & 1\end{smallmatrix}\right) \subset \SL_2$.

Likewise, $(\varpi_n(X,{x_0})^{\rho, \mal})^{\vee}$ is the kernel of
$$
\beta\ten \id  + \ad_{\gamma_{x_0}}\ten \id + \id\ten N: (\H_{n-1}\g)^{\vee}\ten \cS\to (\H_{n-1}\g)^{\vee}\ten \cS(-1)
$$
\end{remark}

\begin{examples}\label{archmonegs}
Since $q$ maps $\CoLie^n(\cH)$ to $\CoLie^{n-1}(\cH)$, we need only look at the truncations of the sums in Theorem \ref{archmonthm} to calculate the MHS or MTS on $G(X,{x_0})^{R, \mal}/[\Ru G(X,{x_0})^{R, \mal}]_m$, where $[K]_1=K$ and $[K]_{m+1}= [K,[K]_m]$. 

\begin{enumerate}
\item
Since all terms involve $q$, this means that $G(X,{x_0})^{R, \mal}/[\Ru G(X,{x_0})^{R, \mal}]_2 \simeq R \ltimes \H^{>0}(X, O(\Bu_{\rho}))^{\vee}[1]$, the equivalence respecting the MHS (or MTS). This just corresponds to the quasi-isomorphism $s:\H^*(X, O(\Bu_{\rho})) \to A^{\bt}(X, O(\Bu_{\rho}))$ of cochain complexes, since the ring structure on $A^{\bt}( (X, O(\Bu_{\rho}))$ is not needed to recover $G(X,{x_0})^{R, \mal}/[\Ru G(X,{x_0})^{R, \mal}]_2$. 

\item The first non-trivial case is $G(X,{x_0})^{R, \mal}/[\Ru G(X,{x_0})^{R, \mal}]_3$. The only contributions to $\beta$ here come from terms of degree $1$ in $q$. Thus $\alpha$ vanishes on this quotient, which means that the obstruction to splitting the  MHS is a unipotent inner automorphism.

The element $\gamma_{x_0}$ becomes
\begin{eqnarray*}
x_0^*\circ h_i \circ \tD \circ [q, G\L]\circ s &=& x_0^*\circ G^2D^*D^{c*}\tDc\tD\circ [q, G\L]\circ s\\
&=& x_0^*\circ G^2D^*D^{c*}D^cD\circ [q, G\L]\circ s,
\end{eqnarray*}
which we can rewrite as $x_0^*\circ\pr_{\im (D^*D^{c*})} \circ [q, G\L]\circ s$, where $\pr_{\im (D^*D^{c*})} $ is orthogonal projection onto $\im(D^*D^{c*})$. Explicitly, $\gamma_{x_0} \in ([\g]_2/ [\g]_3)\hat{\ten} O(\SL_2)$  corresponds to the morphism $\bigwedge^2\cH^1 \to O(R)\ten O(\SL_2)$ given by
$$
v\ten w \mapsto (\pr_{\im (D^*D^{c*})} G\L (s(v)\wedge s(w)))_{x_0},
$$ 
since $\L|_{\cH^1}=0$.

Since $[\g]_2/ [\g]_3$ lies in the centre of $\g/[\g]_3$, this means that $\ad_{\gamma_{x_0}}$ acts trivially on $\Ru(G(X,{x_0})^{R, \mal})/[\Ru G(X,{x_0})^{R, \mal}]_3 $, so $G(X,{x_0})^{R, \mal}/[\Ru G(X,{x_0})^{R, \mal}]_3$ is an extension
$$
1 \to \Ru(G(X,{x_0})^{R, \mal})/[\Ru G(X,{x_0})^{R, \mal}]_3\to G(X,{x_0})^{R, \mal}/[\Ru G(X,{x_0})^{R, \mal}]_3 \to R \to 1
$$
of split MHS. Thus $\gamma_{x_0}$ is the obstruction to any Levi decomposition respecting the MHS, and allowing ${x_0}$ to vary gives us the associated VMHS on $X$.

In particular, taking $R=1$,  the MHS on $G(X,{x_0})^{1, \mal}/[G(X,{x_0})^{1, \mal}]_3$ is split, and specialising further to the case when  $X$ is simply connected,
$$
(\pi_3(X,{x_0})\ten \R)^{\vee} \cong \H^3(X, \R) \oplus \ker (\Symm^2\H^2(X, \R) \xra{\cup} \H^4(X, \R))
$$
is an isomorphism of real MHS. This means that the non-split phenomena for integral MHS in \cite{mhsnonsplit} do not apply to real MHS, and  are describing the lattice $\pi_3(X,{x_0})$ in $\pi_3(X,{x_0})\ten \R$.

\item The first  case in which $\alpha$ is non-trivial is $\Ru(G(X,{x_0})^{R, \mal})/[\Ru G(X,{x_0})^{R, \mal}]_4 $. We then have
\begin{eqnarray*}
\alpha &=& \pr_{\cH} \circ [q, G^2D^*D^{c*}] \circ  D^cD \circ  [q,G\L] \circ s\\
& =& \pr_{\cH} \circ q \circ \pr_{\im (D^*D^{c*})}\circ  [q,G\L] \circ s, 
\end{eqnarray*}
and this determines the MHS on $G(X,{x_0})^{R, \mal} $ up to pro-unipotent inner automorphism. In particular, if $X$ is simply connected, this determines the MHS on $\pi_4(X,{x_0})\ten \R$ as follows. 

Let $V:= \CoLie^3(\H^2(X,\R)[1])[-2]$, i.e. the quotient of $\H^2(X,\R)^{\ten 3}$ by the subspace generated by $a\ten b \ten c -a \ten c\ten b+c\ten a \ten b$ and $a\ten b\ten c - b\ten a \ten c+b\ten c\ten a$, then set $K$ to be the kernel of the map $q:V \to \H^4(X, \R) \ten \H^2(X, \R)$ given by $q(a\ten b\ten c)=(a\cup b)\ten c - (b\cup c) \ten a$.
If we let
$C:= \coker (\Symm^2\H^2(X, \R) \xra{\cup} \H^4(X, \R)) $ and $L:=\ker (\H^2(X, \R)\ten\H^3(X, \R)  \xra{\cup} \H^5(X, \R)) $, then
$$
\gr^W (\pi_4(X)\ten \R)^{\vee} \cong C\oplus L \oplus K.
$$ 

The MHS is then determined by $\alpha: K \to C(-1)$, corresponding to the restriction to $K$ of the map $\alpha': V \to C(-1)$ given by setting  $\alpha'(a\ten b\ten c) $ to be
\begin{eqnarray*}
&&\pr_{\cH}(\pr_I((G\L \tilde{a}) \wedge \tilde{b})\wedge \tilde{c}) 
- \pr_{\cH}((\pr_IG\L \tilde{a}) \wedge (\tilde{b} \wedge \tilde{c})) \\ 
&&-\pr_{\cH}(\pr_I(\tilde{a} \wedge (G\L \tilde{b}))\wedge \tilde{c})
- \pr_{\cH}(\tilde{a} \wedge \pr_I((G\L \tilde{b})\wedge \tilde{c}))\\
&&- \pr_{\cH}( \tilde{a} \wedge \tilde{b}\wedge(\pr_IG\L \tilde{c}) )  
+ \pr_{\cH}( \tilde{a} \wedge \pr_I(\tilde{b}\wedge (G\L \tilde{c} )))
\end{eqnarray*}
where $\tilde{a}:=sa$, for $s$ the identification of cohomology with harmonic forms, while  $\pr_{I}$ and $\pr_{\cH}$ are  orthogonal projection onto $\im(d^*d^{c*})$ and harmonic forms, respectively.

Explicitly, the MHS $(\pi_4(X)\ten \R)^{\vee}$ is then given by the subspace
$$
(c -x \alpha(k) ,l,k) \subset ( C\oplus L \oplus K)\ten \cS,
$$
for $c \in C$, $l \in L$ and $k \in K$, with $\cS$ the quasi-MHS of Lemma \ref{slhodge}. 
\end{enumerate}
\end{examples}

\section{Simplicial and singular varieties}\label{singsn}

In this section, we will show how the techniques of cohomological descent allow us to extend real mixed  Hodge and twistor structures to all proper complex varieties. By \cite[Remark 4.1.10]{SD}, the method of \cite[\S 9]{effective} shows that a surjective proper morphism of topological spaces is universally of effective cohomological descent. 

\begin{lemma}\label{weakcoholemma}
If $f:X \to Y$ is a map of compactly generated Hausdorff topological spaces inducing an equivalence on fundamental groupoids, such that $\RR^if_*\vv=0$ for all local systems $\vv$ on $X$ and all $i>0$, then $f$ is a weak equivalence. 
\end{lemma}
\begin{proof}
Without loss of generality, we may assume that $X$ and $Y$ are path-connected. If $\tilde{X}\xra{\pi} X,\tilde{Y}\xra{\pi'} Y$ are the universal covering spaces of $X,Y$, then it will suffice to show that $\tilde{f}:\tilde{X}\to\tilde{Y}$ is a weak equivalence, since the fundamental groups are isomorphic.

As $\tilde{X},\tilde{Y}$ are simply connected, it suffices to show that $\RR^i\tilde{f}_*\Z=0$ for all $i>0$. By the Leray-Serre spectral sequence, $\RR^i\pi_*\Z=0$ for all $i>0$, and similarly for $Y$. The result now follows from the observation that $\pi_*\Z$ is a local system on $X$.
\end{proof}

\begin{proposition}\label{effectivewks}
If $a:X_{\bt} \to X$ is a morphism (of simplicial topological spaces) of effective cohomological descent, then $|a|:|X_{\bt}| \to X$ is a weak equivalence, where $|X_{\bt}|$ is the geometric realisation of $X_{\bt}$.
\end{proposition}
\begin{proof}
We begin by showing that the fundamental groupoids are equivalent. Since $\H^0(|X_{\bt}|,\Z) \cong \H^0(X,\Z)$, we know that $\pi_0|X_{\bt}|\cong \pi_0X$, so we may assume that $|X_{\bt}|$ and $X$ are both connected. 

Now the fundamental groupoid of $|X_{\bt}|$ is isomorphic to the fundamental groupoid of the simplicial set $\diag\Sing(X_{\bt})$ (the diagonal of the bisimplicial set given by the singular simplicial sets of the  $X_n$). For any group $G$,  the groupoid of $G$-torsors on $|X_{\bt}|$ is thus equivalent to the groupoid of pairs $(T,\omega)$, where $T$ is a $G$-torsor on $X_0$, and the descent datum $\omega:\pd_0^{-1}T \to \pd_1^{-1}T$ is a morphism of $G$-torsors satisfying 
$$
\pd_2^{-1}\omega \circ \pd_0^{-1}\omega=\pd_1^{-1}\omega, \quad \sigma_0^{-1}\omega = 1.
$$
Since $a$ is effective, this groupoid is equivalent to the groupoid of $G$-torsors on $X$, so the fundamental groups are isomorphic.

Given a local system $\vv$ on $|X_{\bt}|$, there is a corresponding $\GL(V)$-torsor $T$, which therefore descends to $X$. Since $\vv=T\by^{\GL(V)}V$ and $T=|a|^{-1}|a|_*T$, we can deduce that $\vv=|a|^{-1}|a|_*\vv$, so $\oR^i|a|_*\vv=0$ for all $i>0$, as $a$ is  of effective cohomological descent. Thus $|a|$ satisfies the conditions of Lemma \ref{weakcoholemma}, so is a weak equivalence.
\end{proof}

\begin{corollary}\label{hodge3stuff}
Given a proper complex variety $X$, there exists a smooth proper simplicial variety $X_{\bt}$, unique up to homotopy, and a map $a:X_{\bt} \to X$, such that $|X_{\bt}| \to X$ is a weak equivalence. 

%
In fact, we may take each $X_n$ to be projective, and these resolutions are unique up to simplicial homotopy.
\end{corollary}
\begin{proof}
Apply \cite[6.2.8, 6.4.4 and \S 8.2]{Hodge3}.
\end{proof}

\subsection{Semisimple local systems}

From now on, $X_{\bt}$ will be a fixed simplicial proper complex variety (\emph{a fortiori}, this allows us to consider any proper complex variety).

In this section, we will define the real holomorphic $S^1$-action on a suitable quotient of the real reductive pro-algebraic fundamental group $\varpi_1(|X_{\bt}|,x)^{\red}$. 

Recall that a local system on a simplicial diagram $X_{\bt}$ of topological spaces  is  equivalent to the category of pairs $(\vv, \alpha)$, where $\vv$ is a local system on $X_0$, and $\alpha:\pd_0^{-1}\vv \to \pd_1^{-1}\vv$ is an isomorphism of local systems  satisfying 
$$
\pd_2^{-1}\alpha \circ \pd_0^{-1}\alpha=\pd_1^{-1}\alpha, \quad \sigma_0^{-1}\alpha = 1.
$$

\begin{definition}\label{pinormdef}
Given a  simplicial diagram $X_{\bt}$ of smooth proper varieties and a point $x \in X_0$, define the fundamental group $\varpi_1(|X_{\bt}|,x)^{\norm}$ to be the quotient of $\varpi_1(|X_{\bt}|,x)$ by the normal subgroup generated  by the image of $\Ru\varpi_1(X_0,x)$. We call its  representations  normally semisimple local systems on $|X_{\bt}|$ --- these correspond to  local systems $\ww$  (on the connected component of $|X|$ containing $x$) for which $a_0^{-1}\ww$ is semisimple, for $a_0:X_0 \to |X_{\bt}|$. 

Then define $\varpi_1(|X_{\bt}|,x)^{\norm,\red}$ to be the reductive quotient of $\varpi_1(|X_{\bt}|,x)^{\norm}$. Its representations are semisimple and normally semisimple local systems on  the connected component of $|X|$ containing $x$.
\end{definition}

\begin{lemma}\label{normwell}
If $f\co X_{\bt} \to Y_{\bt}$ is a homotopy equivalence of  simplicial smooth proper varieties, then   $\varpi_1(|X_{\bt}|,x)^{\norm}\simeq \varpi_1(|Y_{\bt}|,fx)^{\norm}$. 
\end{lemma}
\begin{proof}
Without loss of generality, we may assume that the matching maps
$$
X_n \to Y_n \by_{\Hom_{\bS}(\pd \Delta^n, Y)}\Hom_{\bS}(\pd \Delta^n, X)
$$
  of $f$ are faithfully flat and proper for all $n \ge 0$ (since morphisms of this form generate all homotopy equivalences), and that $|X|$ is connected. Here, $\bS$ is the category of simplicial sets and $\pd\Delta^n$ is the boundary of $\Delta^n$, with the convention that $\pd\Delta^0=\emptyset$.

Topological and algebraic effective descent then imply that $f^{-1}$ induces an equivalence on the categories of local systems, and that $f^*$ induces an equivalence on the categories of quasi-coherent sheaves, and hence on the categories of Higgs bundles. Since   representations of $\varpi_1(|X_{\bt}|,x)^{\norm}$ correspond to  objects in the category of Higgs bundles on $X_{\bt}$, this completes the proof.
\end{proof}

\begin{definition}
If $X_{\bt} \to X$ is any resolution as in Corollary \ref{hodge3stuff}, with $x_0 \in X_0$ mapping to $x \in X$, 
we  denote the corresponding  pro-algebraic group by $ \varpi_1(X,x)^{\norm}:=\varpi_1(|X_{\bt}|,x_0)^{\norm}$, noting that this is independent of the choice of $x_0$, since $|X_{\bt}| \to X$ is a weak equivalence.
\end{definition}

\begin{proposition}\label{norm}
If $X$ is a proper complex variety with a smooth proper resolution $a:X_{\bt} \to X$, then normally semisimple local systems on $X_{\bt}$ correspond to  local systems on $X$ which become semisimple on pulling back to the normalisation $\pi: X^{\norm}\to X$ of $X$. 
\end{proposition}
\begin{proof}
First observe that  $\varpi_1(|X_{\bt}|,x_0)^{\norm}= \varpi_1(X,x_0)/\langle a_0\Ru(\varpi_1(X_0,x_0)) \rangle$. Lemma \ref{normwell} ensures that  $\varpi_1(|X_{\bt}|,x_0)^{\norm}$ is independent of the choice of resolution $X_{\bt}$ of $X$, so can be defined as $\varpi_1(X,x_0)/\langle f\Ru(\varpi_1(Y,y)) \rangle$ for any smooth projective variety $Y$ and proper faithfully flat $f:Y \to X$, with $fy=x$.

Now, since $X^{\norm}$ is normal,  we may make use of an observation on  pp. 9--10 of \cite{Am} (due to M. Ramachandran). $X^{\norm}$ admits a proper faithfully flat morphism $g$ from a smooth variety $Y$ with connected fibres over $X^{\norm}$. If $\tilde{x} \in X^{\norm}$ is a point above $x \in X$, and $y \in Y$ is a point above $\tilde{x}$, then this implies the morphism  $\pi_1g\co \pi_1(Y,y) \to \pi_1(X^{\norm},\tilde{x})$ is surjective (from the long exact sequence of homotopy), and therefore $g(\Ru  \varpi_1(Y,y)) = \Ru \varpi_1(X^{\norm},\tilde{x})$. 

Taking $f:Y \to X$ to be the composition $Y \xra{g} X^{\norm} \xra{\pi} X$, we see that  $f\Ru\varpi_1(Y,y)= \pi(\Ru\varpi_1(X^{\norm},\tilde{x}))$.
This shows that $\varpi_1(X,x)^{\norm}= \varpi_1(X,x)/\langle \pi(\Ru\varpi_1(X^{\norm}, \tilde{x})) \rangle$, as required.
\end{proof}

\begin{proposition}\label{properred}
If $X_{\bt}$ is a simplicial diagram of compact K\"ahler manifolds, then there is a discrete action of the  circle group $S^1$ on $\varpi_1(|X_{\bt}|,x)^{\norm}$, such that the composition $S^1 \by \pi_1(X_{\bt},x) \to \varpi_1(|X_{\bt}|,x)^{\norm}$ is continuous. We denote this last map by $\sqrt h:\pi_1(|X_{\bt}|,x) \to \varpi_1(|X_{\bt}|,x)^{\norm}((S^1)^{\cts})$.

This also holds if we replace $X_{\bt}$ with any proper complex variety $X$.
\end{proposition}
\begin{proof}
The key observation is that the $S^1$-action defined in \cite{Simpson} is functorial in $X$, and that semisimplicity is preserved by pullbacks between compact K\"ahler manifolds (since Higgs bundles pull back to Higgs bundles), 
so there is a canonical isomorphism $t(\pd_i^{-1}\vv)\cong \pd_i^{-1}(t\vv)$ for $t \in S^1$; thus it makes sense for us to define
$$
t(\vv,\alpha):= (t\vv, t(\alpha)),
$$
whenever $\vv$ is semisimple on $X_0$.

If $\C$ is the category of finite-dimensional real local systems on $X_{\bt}$,  
this defines an $S^1$-action on the full subcategory $\C' \subset \C$ consisting of those local systems $\vv$ on $X_{\bt}$ whose restrictions to $X_0$ (or equivalently to all $X_n$)  are semisimple.
Now, the category of $\varpi_1(|X_{\bt}|,x)^{\norm}$-representations  is equivalent to  $\C'$ (assuming, without loss of generality, that $|X_{\bt}|$ is connected). By Tannakian duality, this defines an $S^1$-action on $\varpi_1(|X_{\bt}|,x)^{\norm}$. 

Since $X_0,X_1$ are smooth and proper, the actions of $S^1$ on their reductive pro-algebraic fundamental groupoids are continuous by Lemma \ref{analyticpi}, corresponding to maps 
$$
\pi_1(X_i; T) \to \varpi_1(X_i;T)^{\red}((S^1)^{\cts}).
$$
The morphisms $\varpi_1(X_i; a_i^{-1}\{x\}) \to \varpi_1(|X_{\bt}|,x)$ (coming from $a_i: X_i \to |X_{\bt}|$) then give us maps
$$
\pi_1(X_i;a_i^{-1}\{x\} ) \to \varpi_1(|X_{\bt}|,x)^{\norm, \red}((S^1)^{\cts}),
$$
compatible with the simplicial operations on $X_{\bt}$. Since 
$$
\xymatrix@1{\pi_1(X_1; a_1^{-1}\{x\}) \ar@<.5ex>[r]^-{\pd_0} \ar@<-.5ex>[r]_-{\pd_1}& \pi_1(X_0;a_0^{-1}\{x\}) \to \pi_1(|X_{\bt}|,x)}
$$
is a coequaliser diagram in the category of groupoids, this gives us a map
$$
\pi_1(|X_{\bt}|,x) \to \varpi_1(|X_{\bt}|,x)^{\norm, \red}((S^1)^{\cts}).
$$

For the final part, replace a  proper complex variety with a simplicial smooth proper resolution, as in Corollary \ref{hodge3stuff}.
\end{proof}

\subsection{The Malcev homotopy type}

Now fix a simplicial diagram $X_{\bt}$ of compact K\"ahler manifolds, and take a full and essentially surjective representation $\rho: \varpi_1(|X_{\bt}|,x)^{\norm, \red}\to R$. As in Definition \ref{Budef}, this gives rise to an $R$-torsor $\Bu_{\rho}$ on $X$. 

\begin{definition}
Define the cosimplicial CDGAs 
$$
A^{\bt}(X_{\bt},  O(\Bu_{\rho})), \,\H^*(X_{\bt},  O(\Bu_{\rho}))  \in cDG\Alg(R)
$$
by $n \mapsto A^{\bt}(X_n, O(\Bu_{\rho}))$ and $n \mapsto\H^*(X_n,  O(\Bu_{\rho}))$.
\end{definition}

\begin{definition}
Given a point $x_0 \in X_0$, define $x_0^*: A^{\bt}(X_{\bt},  O(\Bu_{\rho})) \to O(R)$ to be given in cosimplicial degree $n$ by $((\sigma_0)^nx_0)^*: A^{\bt}(X_{n},  O(\Bu_{\rho})) \to  O(\Bu_{\rho})_{ (\sigma_0)^nx_0} \cong O(R)$.
\end{definition}

\begin{lemma}\label{singmalcev}
There is a canonical chain of quasi-isomorphisms between  
the relative Malcev homotopy type $G(|X_{\bt}|,x_0)^{\rho, \mal}$ and
$$
(\Spec D)\bar{G}_R(\Th(A^{\bt}(X_{\bt}, O(\Bu_{\rho}))) \xra{x_0^*} O(R))     
$$
where $\Th: cDG\Alg(R) \to DG\Alg(R)$ is the Thom-Sullivan functor (Definition \ref{Th}) mapping cosimplicial DG algebras to DG algebras.
\end{lemma}
\begin{proof}
This is true for any simplicial diagram of manifolds, and follows by combining Propositions \ref{bigequiv} and \ref{propforms}.
\end{proof}

\subsection{Mixed Hodge structures}

Retaining the hypothesis that $X_{\bt}$ is a simplicial proper complex variety, observe that a representation of $\varpi_1(|X_{\bt}|,x)^{\norm, \red}$ corresponds to a semisimple representation of $X_{\bt}$ whose pullbacks to each $X_n$ are all semisimple. This follows because the morphisms $X_n \to X_0$ of compact K\"ahler manifolds all preserve semisimplicity under pullback, as observed in Proposition \ref{properred}.

\begin{theorem}\label{singmhsmal}\label{singmtsmal}
If  $R$ is any quotient of $\varpi_1(|X_{\bt}|,x)^{\norm, \red}$ (resp. any quotient to which the $(S^1)^{\delta}$-action of Proposition \ref{properred} descends and acts algebraically), then there is an algebraic mixed twistor  structure (resp. mixed Hodge structure) $(|X_{\bt}|,x)^{\rho,\mal}_{\MTS}$ (resp. $(|X_{\bt}|,x)^{\rho,\mal}_{\MHS}$) on the relative Malcev homotopy type $(|X_{\bt}|,x)^{\rho, \mal}$, where $\rho$ denotes the quotient map.

There is also a $\bG_m$-equivariant (resp. $S$-equivariant)  splitting
$$
\bA^1 \by (\ugr (|X_{\bt}|^{\rho,\mal},0)_{\MTS}) \by \SL_2 \simeq (|X_{\bt}|,x)^{\rho, \mal}_{\MTS}\by_{C^*, \row_1}^{\oR}\SL_2
$$
(resp.
$$
\bA^1 \by (\ugr (|X_{\bt}|^{\rho,\mal},0)_{\MHS}) \by \SL_2 \simeq (|X_{\bt}|,x)^{\rho, \mal}_{\MHS}\by_{C^*, \row_1}^{\oR}\SL_2)
$$
on pulling back along $\row_1:\SL_2 \to C^*$, whose pullback over $0 \in \bA^1$ is given by the opposedness isomorphism. 
\end{theorem}
\begin{proof}
We define the cosimplicial CDGA $\tilde{A}(X_{\bt}, O(\Bu_{\rho}))$ on $C$ by  $n \mapsto \tilde{A}^{\bt}(X_n, O(\Bu_{\rho}))$, observing that functoriality (similarly to Proposition \ref{kmhsfun}) ensures that the simplicial and CDGA structures are compatible. This has an augmentation $x^*: \tilde{A}(X_{\bt}, O(\Bu_{\rho}))\to O(R)\ten O(C)$ given in level $n$ by $((\sigma_0)^nx)^*$. 

We then define the  mixed twistor structure by 
$$
|X_{\bt}|^{\rho, \mal}_{\MHS}:= (\Spec \Th \xi(\tilde{A}(X_{\bt}, O(\Bu_{\rho})), \tau_{\tilde{A}})  )  \by_CC^* \in dg_{\Z}\Aff_{\bA^1\by C^*}(\bG_m \by R\rtimes S),
$$
with 
$$
\ugr |X_{\bt}|^{\rho,\mal}_{\MHS}= \Spec(\Th \H^*(X_{\bt}, O(\Bu_{\rho})))\in dg_{Z}\Aff(R\rtimes S);
$$
the definitions of $|X_{\bt}|^{\rho, \mal}_{\MTS}$ are similar, replacing $S$ with $\bG_m$.

For any CDGA $B$, we may regard $B$ as a cosimplicial CDGA (with constant cosimplicial structure), and then $\Th(B)=B$. In particular, $\Th(O(R))= O(R)$, so we have a basepoint $\Spec \Th(x^*): \bA^1 \by R \by C^* \to |X_{\bt}|^{\rho, \mal}_{\MHS} $, giving  
$$
(|X_{\bt}|,x)^{\rho, \mal}_{\MHS}\in  dg_{\Z}\Aff_{\bA^1\by C^*}(R)_*(\bG_m \by S),
$$
and similarly for $|X_{\bt}|^{\rho, \mal}_{\MTS}$.

The proof of Theorem \ref{mhsmal} now carries over. For a singular variety $X$, apply Proposition \ref{effectivewks} to substitute a simplicial smooth proper variety $X_{\bt}$.
\end{proof}

\begin{corollary}\label{singmhspin}\label{singmtspin}
In the scenario of Theorem \ref{singmhsmal}, the
 homotopy groups $\varpi_n(|X_{\bt}|^{\rho, \mal},x)$ for $n\ge 2$, and the Hopf algebra $O(\varpi_1(|X_{\bt}|^{\rho, \mal},x))$ carry natural ind-MTS (resp. ind-MHS), functorial in $(X_{\bt},x)$, and compatible with the action of $\varpi_1$ on $\varpi_n$, the Whitehead bracket  and the Hurewicz maps $\varpi_n(|X_{\bt}|^{\rho, \mal},x) \to \H^n(|X_{\bt}|,O(\Bu_{\rho}))^{\vee}$. 

Moreover, there are $\cS$-linear isomorphisms
\begin{eqnarray*}
\varpi_n(|X_{\bt}|^{\rho, \mal},x)^{\vee}\ten\cS &\cong& \pi_n(\Th \H^*(X_{\bt}, O(\Bu_{\rho})))^{\vee}\ten\cS\\
O(\varpi_1(|X_{\bt}|^{\rho, \mal},x)) \ten\cS &\cong& O(R \ltimes\pi_1(\Th \H^*(X_{\bt}, O(\Bu_{\rho})))\ten\cS
\end{eqnarray*}
of quasi-MTS (resp. quasi-MHS). The associated graded map from the weight filtration is just the pullback of the standard isomorphism $\gr_W\varpi_*(|X_{\bt}|^{\rho, \mal})\cong\pi_*(\Th \H^*(X_{\bt}, O(\Bu_{\rho})))$.

Here, $\pi_*(B)$ are the  homotopy groups $\H_{*-1}\bar{G}(B)$  associated to the $R \rtimes S$-equivariant  CDGA $\H^*(X,O(\Bu_{\rho}))$ (as constructed in Definition \ref{barwg}), with the induced real twistor (resp. Hodge) structure.

Furthermore, $W_0O(\varpi_1(|X_{\bt}|^{\rho, \mal},x))= O(\varpi_1(|X_{\bt}|^{\rho, \mal},x)^{\norm})$.
\end{corollary}
\begin{proof}
This is essentially the same as Corollary \ref{kmhspin}. Note that we may simplify the calculation of $\pi_*( \Th \H^*(X_{\bt}, O(\Bu_{\rho})))$ by observing that $\pi_*(C^{\bt})= \pi_*\Spec (D C^{\bt})$, where $D$ denotes cosimplicial denormalisation, so $\pi_*( \Th \H^*(X_{\bt}, O(\Bu_{\rho})))= \pi_*\Spec (\diag D\H^*(X_{\bt}, O(\Bu_{\rho})))$. 

For the final statement, note that representations of $\gr^W_0 \varpi_1(|X_{\bt}|^{\rho, \mal},x):= \Spec W_0O(\varpi_1(|X_{\bt}|^{\rho, \mal},x))$ correspond to representations of $\varpi_1(|X_{\bt}|^{\rho, \mal},x)$ which annihilate the image of $W_{-1}\varpi_1(X_n^{\rho, \mal},x)$ for all $n$. Since $X_n$ is smooth and projective, we just have $W_{-1}\varpi_1(X_n^{\rho, \mal},x)= \Ru\varpi_1(X_n^{\rho, \mal},x)$, so these are precisely the normally semisimple representations.
\end{proof}

\begin{corollary}
If $\pi_1(|X_{\bt}|,x)$ is algebraically good with respect to $R$ and the homotopy groups $\pi_n(|X_{\bt}|,x)$ have finite rank for all $n\ge 2$, with $\pi_n(|X_{\bt}|,x)\ten_{\Z} \R$ an extension of $R$-representations, then  Corollary \ref{singmhspin} gives mixed twistor (resp. mixed Hodge) structures on $\pi_n(|X_{\bt}|,x)\ten\R$ for all $n \ge 2$, by Theorem \ref{classicalpimal}.  
\end{corollary}

\begin{proposition}\label{Hainhodge}
When $R=1$, the mixed Hodge structures of Corollary \ref{singmhspin} agree with those defined in \cite[Theorem 6.3.1]{Hainhodge}.
\end{proposition}
\begin{proof}
The algebraic mixed Hodge structure of Theorem \ref{singmhsmal} is given by applying the Rees functor to the real mixed Hodge complex $(\Th A^{\bt}(X_{\bt}, \R),W,F)$. Up to a shift in the weight filtration caused by d\'ecalage, this is naturally quasi-isomorphic to the real analogue of the  mixed Hodge complex  of \cite[Theorem 5.6.4]{Hainhodge}.

Since $\bar{G}_1(A)= \Spec \bar{B}(A)$, it follows from Theorem \ref{bigequiv} that our characterisation of homotopy groups (Definition \ref{varpidef}) is the same as that given in \cite{Hainhodge}, so our construction of mixed Hodge structures on homotopy groups agrees with \cite[Theorem 4.2.1]{Hainhodge} (for details on the shift in the weight filtration, see Remark \ref{trueweightrmk}).
\end{proof}

\subsection{Enriching twistor structures}
For the remainder of this section, assume that  $R$ is any quotient of $(\pi_1(|X_{\bt}|,x))^{\red, \norm}_{\R}$  to which the $(S^1)^{\delta}$-action descends, but does not necessarily act algebraically.

\begin{proposition}\label{singmtsmalenrich}
There is a natural  $(S^1)^{\delta}$-action on $\ugr |X_{\bt}|^{\rho,\mal}_{\MTS}$, giving   a $(S^1)^{\delta}$-invariant  map
$$
h \in\Hom_{\Ho(\bS \da BR((S^1)^{\cts}))}(\Sing(|X_{\bt}|,x), \holim_{\substack{\lra \\ R((S^1)^{\cts})}}  (|X_{\bt}|,x)^{\rho,\mal}_{\bT}((S^1)^{\cts})_{C^*}),
$$
where $(|X_{\bt}|,x)_{\bT}^{\rho,\mal}((S^1)^{\cts})_{C^*}:= \Hom_{C^*}((S^1)^{\cts}, (|X_{\bt}|^{\rho,\mal},x)_{\bT})$.

Moreover, for  $1: \Spec \R \to (S^1)^{\cts}$, the map
$$
1^*h:\Sing(|X_{\bt}|,x)\to \holim_{\substack{\lra \\ R(\R)}} (|X_{\bt}|^{\rho,\mal},x)_{\bT}((S^1)^{\cts})_{C^*}\by_{BR((S^1)^{\cts})}BR(\R)
$$
in $\Ho(\bS \da BR(\R))$ is just the canonical map
$$
\Sing(|X_{\bt}|,x)\to \holim_{\substack{\lra \\ R(\R)}} (|X_{\bt}|^{\rho, \mal},x)(\R).
$$
\end{proposition}
\begin{proof}
We first note that Proposition \ref{redenrich} adapts by functoriality to give a $(S^1)^{\delta}$-action on the  mixed twistor structure $|X_{\bt}|^{\rho, \mal}_{\MTS}$ of Theorem \ref{singmtsmal}. It also gives a $(S^1)^{\delta}$-action on $\ugr |X_{\bt}|^{\rho, \mal}_{\MTS}$, for which the $\bG_m\by \bG_m$-equivariant  splitting
$$
\bA^1 \by \ugr |X_{\bt}|^{\rho, \mal}_{\MTS}\by \SL_2 \cong (|X_{\bt}|^{\rho, \mal}_{\MTS},x )\by_{C^*, \row_1}^{\oR}\SL_2
$$
of Theorem \ref{singmtsmal} is also $(S^1)^{\delta}$-equivariant.

The proof of Proposition \ref{mtsmalenrich} also adapts  by functoriality, with $h$ above 
extending the map  $ h:(|X_{\bt}|,x) \to BR((S^1)^{\cts})$ corresponding to the group homomorphism $h: \pi_1(|X_{\bt}|,x) \to R((S^1)^{\cts})$ given by $h(t)= \sqrt h(t^2)$, for $\sqrt h$ as in  Proposition \ref{properred}.
\end{proof}

Thus   (for  $R$  any quotient of $\varpi_1(|X_{\bt}|,x)^{\red, \norm}$  to which the $(S^1)^{\delta}$-action       descends), we have:

\begin{corollary} \label{singmhspinanal}
If the group $\varpi_n(|X_{\bt}|^{\rho,\mal},x)$  is finite-dimensional and spanned by the image of  $\pi_n(|X_{\bt}|,x)$, then the former carries a natural mixed Hodge structure, which splits on tensoring with $\cS$ and extends the mixed twistor structure of Corollary \ref{singmtspin}. This is functorial in $X_{\bt}$ and compatible with the action of $\varpi_1$ on $\varpi_n$, the Whitehead bracket, and the Hurewicz maps $\varpi_n(|X_{\bt}|^{\rho, \mal},x) \to \H^n(|X_{\bt}|,O(\Bu_{\rho}))^{\vee}$. 
\end{corollary}
\begin{proof}
 We first note that Corollary \ref{kmtspinen} adapts to show that 
for all $n$, the homotopy class of maps
$
\pi_n(|X_{\bt}|,x) \by (S^1)^{\delta}\to \varpi_n(|X_{\bt}|^{\rho,\mal},x)_{\bT},
$
given by composing the maps $\pi_n (|X_{\bt}|,x) \to \varpi_n(|X_{\bt}|^{\rho,\mal},x)$  with the $(S^1)^{\delta}$-action on $(|X_{\bt}|^{\rho,\mal},x)_{\bT}$, is analytic. The proof of Corollary \ref{kmhspinanal} then carries over to this context.
\end{proof}

 \begin{remark}
   Observe that if $\pi_1(|X_{\bt}|,x)$ is algebraically good with respect to $R$ and the homotopy groups $\pi_n(|X_{\bt}|,x)$ have finite rank for all $n\ge 2$,  with $\pi_n(|X_{\bt}|,x)\ten_{\Z} \R$ an extension of $R$-representations, then  Theorem \ref{classicalpimal} implies that $\varpi_n(|X_{\bt}|^{\rho, \mal},x)\cong \pi_n(|X_{\bt}|,x)\ten \R$, ensuring that  the hypotheses of  Corollary \ref{kmhspinanal} are satisfied.
   \end{remark}


\section{Algebraic MHS/MTS for quasi-projective varieties I}\label{quprojsn}

Fix a smooth compact K\"ahler manifold $X$, a divisor $D$ locally of normal crossings, and set $Y:= X-D$. Let $j: Y \to X$ be the inclusion morphism.

\begin{definition}\label{logdef}
Denote the sheaf  of real $\C^{\infty}$ $n$-forms on $X$ by $\sA^n_X$, and let  $\sA_X^{\bt}$ be the resulting complex (the real sheaf de Rham complex on $X$).

Let $ \sA_X^{\bt} \llbracket D \rrbracket\subset j_*\sA_Y^{\bt}$ be the sheaf of dg $\sA_X^{\bt}$-subalgebras locally generated by      $\{\log r_i, d\log r_i, \dc \log r_i\}_{1\le i\le m}$,  where $D$ is given in local co-ordinates  by $D= \bigcup_{i=1}^m\{z_i=0\}$, and $r_i= |z_i|$.

Let  $\sA^{\bt}_X\langle  D \rangle\subset j_*\sA_Y^{\bt}\ten \Cx $  be the sheaf of dg $\sA_X^{\bt}\ten \Cx$-subalgebras locally generated by  
$\{ d \log z_i\}_{1\le i\le m} $. 
\end{definition}

Note that $\dc\log r_i= d\arg z_i$.

\begin{definition}
 Construct  increasing  filtrations on $\sA^{\bt}_X\langle  D \rangle$ and $\sA^{\bt}_X \llbracket D \rrbracket$ by setting
\begin{eqnarray*}
 J_0\sA_X^{\bt} \llbracket D \rrbracket&=& \sA^{\bt}_X,\\
 J_0\sA^{\bt}_X\langle  D \rangle&=&\sA^{\bt}_X\ten \Cx,
\end{eqnarray*}
 then forming $J_r\sA_X^{\bt} \langle  D \rangle\subset \sA_X^{\bt}\langle  D \rangle $ and $J_r\sA_X^{\bt} \llbracket D \rrbracket\subset \sA_X^{\bt} \llbracket D \rrbracket$ inductively by the local expressions
\begin{align*}
J_r\sA_X^{\bt} \langle  D \rangle&=  \sum_i  J_{r-1}\sA_X^{\bt}\langle  D \rangle  d\log z_i, \\
J_r\sA_X ^{\bt}\llbracket D \rrbracket&= \sum_i  J_{r-1}\sA_X^{\bt} \llbracket D \rrbracket\log r_i+ \sum_i  J_{r-1}\sA_X^{\bt} \llbracket D \rrbracket d\log r_i   + \sum_i  J_{r-1}\sA_X^{\bt} \llbracket D \rrbracket \dc\log r_i,
\end{align*}
for local co-ordinates as above.
\end{definition}

Given any cochain complex $V$, we denote the good truncation filtration by $\tau_nV:= \tau^{\le n}V$.  

\begin{lemma}\label{Jworks}
 The maps
\begin{eqnarray*}
 (\sA_X^{\bt}\langle  D \rangle , J) \la &(\sA_X^{\bt} \langle D \rangle, \tau) &\to (j_*\sA_Y^{\bt}\ten \Cx, \tau)\\
(\sA_X^{\bt} \llbracket D \rrbracket, J) \la &(\sA_X^{\bt}\llbracket  D \rrbracket, \tau) &\to (j_*\sA_Y^{\bt}, \tau)
\end{eqnarray*}
are filtered quasi-isomorphisms of complexes of sheaves on $X$.
\end{lemma}
\begin{proof}
This is essentially the same as \cite{Hodge2} Prop 3.1.8, noting that the inclusion $ \sA_X^{\bt}\langle  D \rangle \into \sA_X^{\bt}\llbracket D \rrbracket\ten \Cx $ is  a filtered quasi-isomorphism, because  $ \sA_X^{\bt}\llbracket D \rrbracket\ten \Cx $ is locally freely generated over $ \sA_X^{\bt}\langle  D \rangle$ by the elements $\log r_i$ and $d\log r_i$.
\end{proof}

An immediate consequence of this lemma is that for all $m \ge 0$, the flabby complex $\gr^J_m\sA_X \llbracket D \rrbracket$ is quasi-isomorphic to   $\oR^mj_*\R$.

\begin{definition}
For any real local system $\vv$ on $X$, define 
\begin{align*}
 \sA^{\bt}_X(\vv)&:= \sA^{\bt}_X\ten_{\R} \vv, & \sA^{\bt}_X(\vv) \langle  D \rangle&:= \sA^{\bt}_X\langle  D\rangle\ten_{\R} \vv,  & \sA^{\bt}_X(\vv) \llbracket D \rrbracket&:= \sA^{\bt}_X\llbracket D \rrbracket\ten_{\R} \vv. 
\end{align*}
\begin{align*}
 A^{\bt}(X,\vv)&:=\Gamma(X, \sA^{\bt}_X(\vv )),& A^{\bt}(X,\vv)\langle  D \rangle&:= \Gamma(X, \sA^{\bt}_X(\vv)\langle  D \rangle ), \\
&  & A^{\bt}(X,\vv) \llbracket D \rrbracket&:= \Gamma(X, \sA^{\bt}_X(\vv) \llbracket D \rrbracket).
\end{align*}
These inherit  filtrations, given by
\begin{align*}
J_rA^{\bt}(X,\vv) \langle  D \rangle&:= \Gamma(X, J_r\sA^{\bt}_X \langle  D \rangle\ten \vv ),\\
 J_rA^{\bt}(X,\vv)\llbracket D \rrbracket&:= \Gamma(X, J_r\sA^{\bt}_X \llbracket D \rrbracket\ten \vv ).
\end{align*}
\end{definition}

Note that Lemma \ref{Jworks} implies that for all $m \ge 0$, the flabby complex $\gr^J_m\sA_X(\vv) \llbracket D \rrbracket$ (resp.  $\gr^J_m\sA^{\bt}_X(\vv) \langle  D \rangle$) is quasi-isomorphic to   $\oR^mj_*(j^{-1}\vv)\cong \vv\ten  \oR^mj_*\R$ (resp. $\oR^mj_*(j^{-1}\vv)\ten \Cx$).

\begin{remark}\label{trueweightrmk}
 The filtration $J$ essentially corresponds to the weight filtration $W$ of \cite[3.1.5]{Hodge2}. However, the true weight filtration on cohomology, and hence on homotopy types, is given by the d\'ecalage $\Dec J$ (as in \cite[Theorem 3.2.5]{Hodge2} or  \cite{Morgan}). As explained for instance in \cite{poids}, $\Dec J$ gives the correct weights on $\oR\Gamma(Y, \Q)$, not only for mixed  Hodge structures but also for Frobenius eigenvalues on  $\ell$-adic completions, making it the natural filtration to use on $A^{\bt}(X,\vv) \langle  D \rangle $. It also tallies with Frobenius eigenvalues on homotopy types, as in \cite{weiln}.

Since $\Dec J$ gives the correct notion of weights, we reserve the terminology ``weight filtration'' for $W:= \Dec J$ on $A^{\bt}(X,\vv) \langle  D \rangle$. This considerably simplifies several constructions; for instance, this weight filtration passes directly to a weight filtration on the bar construction, without having to take convolution with the bar filtration as in \cite{Hainhodge}.
\end{remark}

\subsubsection{Decreasing Hodge and twistor filtrations}\label{nonabfil2}

We now introduce refinements of the constructions from \S \ref{nonabfil} in order to deal with the non-abelian analogue of decreasing filtrations $F^0 \supset F^1 \supset \ldots$.

\begin{definition}
 $\Mat_n$ is the algebraic monoid of $n \by n$-matrices. Thus $\Mat_1\cong \bA^1$, so acts on $\bA^1$ by multiplication.
Note that the inclusion $\bG_m \into \Mat_1$ identifies $\Mat_1$-representations with non-negatively weighted $\bG_m$-representations. 

Let $\bar{S}:= (\Mat_1\by S^1)/(-1,-1)$, giving a real  algebraic monoid whose subgroup of units is $S$, via the isomorphism $ S \cong (\bG_m\by S^1)/(-1,-1)$. There is thus a morphism $\bar{S} \to S^1$ given by $(m,u) \mapsto u^2$, extending the isomorphism $S/ \bG_m \cong S^1$.
\end{definition}

Note that $\bar{S}$-representations correspond via the morphism $S \to \bar{S}$ to real Hodge structures of non-negative weights. In the co-ordinates of Remark \ref{Ccoords}, 
\[
 \bar{S}= \Spec\R[u,v,\frac{u^2-v^2}{u^2+v^2}, \frac{2uv}{u^2+v^2}].
\]

The following adapts Definition \ref{algmhsdef} to non-positive weights, replacing $\bG_m$ and $S$ with $\Mat_1$ and $\bar{S}$ respectively.
\begin{definition}\label{algmhsdef2}
A non-positively weighted algebraic  mixed Hodge structure $(X,x)_{\MHS}^{R, \mal}$ on a pointed Malcev homotopy type $(X,x)^{R, \mal}$ consists of the following data:
\begin{enumerate}
\item an algebraic action of $S^1$ on $R$,
\item an object 
$$
(X,x)_{\MHS}^{R, \mal} \in \Ho(dg_{\Z}\Aff_{\bA^1 \by C^*}(R)_*(\Mat_1 \by S)),
$$
where $S$ acts on $R$ via the $S^1$-action, using the canonical isomorphism $S^1 \cong S/\bG_m$, 
\item an object 
$$
\ugr (X,x)_{\MHS}^{R, \mal} \in \Ho(dg_{\Z}\Aff(R)_*(\bar{S})),
$$
\item an isomorphism $(X,x)^{R, \mal} \cong  (X,x)_{\MHS}^{R, \mal} \by_{(\bA^1\by C^*), (1,1)}^{\oR}\Spec \R \in \Ho(dg_{\Z}\Aff(R)_*)$,
\item an isomorphism (called the opposedness isomorphism)
$$
\theta^{\sharp}(\ugr (X,x)_{\MHS}^{R, \mal}) \by C^* \cong  (X,x)_{\MHS}^{R, \mal} \by_{\bA^1, 0}^{\oR}\Spec \R \in \Ho( dg_{\Z}\Aff_{C^*}( R)_*(\Mat_1 \by S)),
$$
 for the canonical map $\theta: \Mat_1 \by S \to \bar{S}$ given by combining the inclusion $\Mat_1 \into \bar{S}$ with the inclusion $S \into \bar{S}$. 
\end{enumerate}
\end{definition}

\begin{definition}\label{algmtsdef}
An non-positively weighted algebraic  mixed twistor structure $(X,x)_{\MTS}^{R, \mal}$ on a pointed Malcev homotopy type $(X,x)^{R, \mal}$ consists of the following data:
\begin{enumerate}
\item an object 
$$
(X,x)_{\MTS}^{R, \mal} \in \Ho(  dg_{\Z}\Aff_{\bA^1 \by C^*}(R)_*( \Mat_1 \by \bG_m)),
$$
\item an object $\ugr (X,x)_{\MTS}^{R, \mal} \in \Ho( dg_{\Z}\Aff(R)_*(\Mat_1 ))$,
\item an isomorphism $(X,x)^{R, \mal} \cong  (X,x)_{\MTS}^{R, \mal} \by_{(\bA^1\by C^*), (1,1)}^{\oR}\Spec \R \in \Ho( dg_{\Z}\Aff(R)_*)$,
\item an isomorphism (called the opposedness isomorphism)
$$
\theta^{\sharp}(\ugr (X,x)_{\MTS}^{R, \mal}) \by C^* \cong  (X,x)_{\MTS}^{R, \mal} \by_{\bA^1, 0}^{\oR}\Spec \R \in \Ho( dg_{\Z}\Aff_{C^*}( R)_*(\Mat_1 \by\bG_m)),
$$
 for the canonical  map $\theta: \Mat_1 \by \bG_m \to \Mat_1$ given by combining the identity on  $\Mat_1$ with the inclusion $\bG_m \into \Mat_1$. 
\end{enumerate}
\end{definition}

\subsection{The Hodge and twistor filtrations}

We begin by generalising some constructions from \S \ref{vmhssn}.

\begin{definition}\label{sAtildedef2}
Given a semisimple real local system $\vv$ on $X$, define the sheaf $\tilde{\sA}_{X}^{\bt}(\vv) \llbracket D \rrbracket $ of cochain complexes on $X_{\an} \by C_{\Zar}$ by
\[
  \tilde{\sA}^{\bt}_{X}(\vv)\llbracket D \rrbracket= (\sA^*_X(\vv) \llbracket D \rrbracket\ten_{\R} \O_C, u D + v D^c),
\]
for co-ordinates $u,v$ as in Remark \ref{Ccoords}. We denote the differential by $\tilde{D}:=uD + vD^c$.

Define the quasi-coherent sheaf $\tilde{A}^{\bt}(X,\vv) \llbracket D \rrbracket$ of cochain complexes on $C$ by
$
\tilde{A}^{\bt}(X,\vv) \llbracket D \rrbracket:=\Gamma (X_{\an}, \tilde{\sA}^{\bt}_X(\vv)\llbracket D \rrbracket).
$
\end{definition}

\begin{definition}\label{Saction}
 Note that the $\dmd$ action on $\sA$ from Definition \ref{dmd}  gives an action of $\bG_m \subset S$ on $ \tilde{\sA}^{\bt}_{X}(\vv)\llbracket D \rrbracket$ over $C$. If we have  a  semisimple local system  $\vv$, equipped 
 with  a discrete (resp. algebraic) action of $S^1$ on $\sA^0_X(\vv)$, recall that the proof of Proposition \ref{redenrich} (resp. Theorem \ref{mhsmal}) gives a discrete  $S(\R)=\Cx^{\by}$-action (resp. an  algebraic $S$-action)  $\boxast$  on  $\tilde{\sA}^{\bt}_{X}(\vv)$, and note that this extends naturally to   $ \tilde{\sA}^{\bt}_{X}(\vv)\llbracket D \rrbracket$.
\end{definition}

\begin{definition}\label{Ytfil}
Given a Zariski-dense representation $\rho\co \pi_1(X,jy) \to R(\R)$, for $R$ a pro-reductive pro-algebraic group, define an algebraic twistor filtration on the relative Malcev homotopy type $(Y,y)^{R, \mal}$ by
\[
 (Y,y)^{R, \mal}_{\bT}:=  (R  \by C^* \xra{\oSpec(jy)^*} \oSpec_{C^*}   \tilde{A}^{\bt}(X, O(\Bu_{\rho}))\llbracket D \rrbracket|_{C^*}),      
\]
in $\Ho( dg_{\Z}\Aff_{C^*}(R)_*(\bG_m))$,
where $O(\Bu_{\rho})$ is the local system of Proposition \ref{propforms}, which is necessarily a sum of finite-dimensional semisimple local systems, and $\bG_m \subset S$ acts via the $\boxast$ action above. 
 \end{definition}

A Zariski-dense representation $\rho\co \pi_1(X,jy) \to R(\R)$ is equivalent to a morphism $\varpi_1(X,jy)^{\red} \onto R$ of pro-algebraic groups, where $\varpi_1(X,jy)^{\red}$ is the reductive quotient of the real pro-algebraic fundamental group $\varpi_1(X,jy)$. \cite{Simpson} effectively gives a discrete $S^1$-action on $\varpi_1(X,jy)^{\red}$, corresponding (as in Lemma \ref{discreteact}) to the $\circledast$ action on semisimple local systems from Lemma \ref{discreteact}. This $S^1$-action thus descends to $R$ if and only if $O(\Bu_{\rho})$ satisfies the conditions of Definition \ref{Saction}. Moreover, the $S^1$-action is algebraic on $R$ if and only if $O(\Bu_{\rho})$ becomes a weight $0$ variation of Hodge structures under the $\circledast$ action, by Proposition \ref{vhsequiv}. 

\begin{definition}\label{Yhfil}
Take a Zariski-dense representation $\rho\co \pi_1(X,jy) \to R(\R)$, for $R$ a pro-reductive pro-algebraic group to which the $S^1$-action on $\varpi_1(X,jy)^{\red}$ descends and acts algebraically. Then define 
        an algebraic Hodge filtration on the relative Malcev homotopy type $(Y,y)^{R, \mal}$ by
\[
 (Y,y)^{R, \mal}_{\bF}:=  (R  \by C^* \xra{\oSpec(jy)^*} \oSpec_{C^*}   \tilde{A}^{\bt}(X, O(\Bu_{\rho}))\llbracket D \rrbracket|_{C^*}),      
\]
in $\Ho( dg_{\Z}\Aff_{C^*}(R)_*(S))$, where the $S$-action is given by the $\boxast$ action of Definition \ref{Saction}.
\end{definition}

If the $S^1$ action descends to $R$ but is not algebraic, we still have the following:
\begin{proposition}\label{redenrich2}
The  algebraic twistor filtration $(Y,y)^{R, \mal}_{\bT}$ of Definition \ref{Ytfil} is equipped with an $(S^1)^{\delta}$-action (i.e. a discrete $S^1$-action) with the properties that
\begin{enumerate}
 \item the $S^1$-action and $\bG_m$-actions commute,
\item the projection $(Y,y)^{R, \mal}_{\bT} \to C^*$ is $S^1$-equivariant, and
\item  $-1 \in S^1$ acts as $-1 \in \bG_m$.
\end{enumerate}
\end{proposition}
\begin{proof}
 This is the same as the proof of Proposition \ref{redenrich}. The action comes from Definition \ref{Saction}, with $t\in (S^1)^{\delta}$ acting on 
$\sA^*_{X}(O(\Bu_{\rho}))\llbracket D \rrbracket$
by  $t \boxast (a\ten v) = (t \dmd a) \ten (t^2 \circledast v)$.
\end{proof}

\subsection{Higher direct images and residues}

\begin{definition}\label{Dmdef}
Let $D^m \subset X$  denote the union of all $m$-fold intersections of local components of the divisor $D \subset X$, and set $D^{(m)}$ to be its normalisation. Write $\nu_m\co D^{(m)}\to X $ for the composition of the normalisation map with  the embedding of $D^m$, and set $C^{(m)}:= \nu_m^{-1}D^{m+1}$.
       \end{definition}

As in \cite[1.2]{Timm}, observe that $D^m- D^{m+1} $ is  a smooth quasi-projective variety, isomorphic to $D^{(m)}- C^{(m)}$. Moreover, $D^{(m)} $ is a smooth projective variety, with $C^{(m)}$ a normal crossings divisor.

\begin{definition}
 Recall from \cite{Hodge2} Definition 2.1.13 that for $n \in \Z$, $\Z(n)$ is the lattice $(2\pi i)^n\Z $, equipped with the pure Hodge structure of type $(-n,-n)$. Given an abelian group $A$, write $A(n):= A\ten_{\Z}\Z(n)$.
\end{definition}

\begin{definition}
 On $D^{(m)}$, define $\vareps^m$ by the property that $\vareps^m(m) $  is the integral local system  of orientations of $D^m$ in $X$. Thus $\vareps^n$ is  the local system $\vareps^n_{\Z}$ defined in \cite[3.1.4]{Hodge2}.      
\end{definition}

\begin{lemma}\label{higherdirectlemma}
 $\oR^m j_*\Z \cong \nu_{m*} \vareps^m$.      
\end{lemma}
\begin{proof}
 This is      \cite[Proposition 3.1.9]{Hodge2}.  
\end{proof}

\begin{lemma}\label{residuedef}
For any local system $\vv$ on $X$, there is  a canonical quasi-isomorphism
\[
 \Res_m \co \gr^J_m \sA^{\bt}_X(\vv)\langle  D \rangle \to \nu_{m*} \sA^{\bt}_{ D^{(m)}}( \vv\ten_{\R}\vareps^m_{\Cx})[-m]      
\]
of cochain complexes on $X$.
\end{lemma}
\begin{proof}
We follow the construction of \cite[3.1.5.1]{Hodge2}. In a neighbourhood where $D$ is given locally by $\bigcup_i \{z_i=0\}$, with $\omega \in  \sA^{\bt}_X(\vv)$, we  set
\[
 \Res_m(\omega \wedge d\log z_1 \wedge \ldots \wedge d \log z_m) := \omega|_{D^{(m)}} \ten  \eps(z_1, \ldots, z_m),    
\]
 where $\eps(z_1, \ldots, z_m) $ denotes the orientation of the components $\{z_1=0\}, \ldots, \{z_m=0\}$.
       
That $\Res_m$  is a quasi-isomorphism follows immediately from Lemmas \ref{Jworks} and \ref{higherdirectlemma}.
\end{proof}
%

\subsection{Opposedness}

Fix a Zariski-dense representation $\rho\co \pi_1(X,jy) \to R(\R)$, for $R$ a pro-reductive pro-algebraic group.

\begin{proposition}\label{grjpuremhs}
If the $S^1$-action  on $\varpi_1(X,jy)^{\red}$ descends to an algebraic action on $R$, then
 for the algebraic Hodge filtration $(Y,y)^{R, \mal}_{\bF}$ of Definition \ref{Yhfil}, the $R\rtimes S$-equivariant  cohomology sheaf
\[
 \sH^a(\gr^J_b\O(Y,y)^{R, \mal}_{\bF} )       
\]
on $C^*$ defines a pure ind-Hodge  structure of weight $a+b$, corresponding to the $\boxast$ $S$-action on
\[
 \H^{a-b} (D^{(b)}, O(\Bu_{\rho})\ten_{\Z}\vareps^b).
\]
\end{proposition}
\begin{proof}
 We need to show that $ \H^a(\gr^J_b\tilde{A}^{\bt}(X, O(\Bu_{\rho}))\llbracket D \rrbracket)|_{C^*}$ corresponds to a   pure ind-Hodge structure of weight $a+b$, or equivalently to a sum of vector bundles of slope $a+b$. We are therefore led to study the  complex
$\gr^J_b\tilde{\sA}^{\bt}_X (O(\Bu_{\rho}))\llbracket D \rrbracket)|_{C^*} $ on $X \by C^*$, since
\[
  \H^a(\gr^J_b\tilde{A}^{\bt}(X, O(\Bu_{\rho}))\llbracket D \rrbracket)|_{C^*} =   \H^a(X,\gr^J_b\tilde{\sA}^{\bt}_X (O(\Bu_{\rho}))\llbracket D \rrbracket )|_{C^*}.   
\]

In a neighbourhood where $D$ is given locally by $\bigcup_i \{z_i=0\}$, $\gr^J\tilde{\sA}^{\bt}_X\llbracket D \rrbracket $ is the $\tilde{\sA}^{\bt}_X $-algebra generated by the classes $[\log |z_i|]$ $[d\log |z_i|]$ and $[\dc\log |z_i|]$ in $\gr^J_1$.
Let $\widetilde{C^*}\to C^*$ be the \'etale covering of Definition \ref{tildeC}. Now, $\tilde{d}= ud+ v\dc= (u+iv)\pd + (u-iv)\bar{\pd}$, so  $\gr^J\tilde{\sA}^{\bt}_X\llbracket D \rrbracket|_{\widetilde{C^*} } $  is the  $\gr^J\tilde{\sA}^{\bt}_X|_{\widetilde{C^*}}$-algebra generated by  $[\log |z_i|], \tilde{d}[\log |z_i|], [d\log z_i]$.

Since $\tilde{\sA}^{\bt}_X(O(\Bu_{\rho}))\llbracket D \rrbracket= \tilde{\sA}^{\bt}_X(O(\Bu_{\rho}))\ten_{\sA^{\bt}_X }\sA^{\bt}_X\llbracket D \rrbracket$, we have  an $S$-equivariant  quasi-isomorphism
\[
 \tilde{\sA}^{\bt}_X(O(\Bu_{\rho}))\ten_{\sA^{\bt}_X }\gr^J\sA^{\bt}_X \langle D \rangle|_{\widetilde{C^*}} \into   \gr^J\tilde{\sA}^{\bt}_X(O(\Bu_{\rho}))\llbracket D \rrbracket|_{\widetilde{C^*} },     
\]
as the right-hand side is generated over the left by $ [\log |z_i|], \tilde{d}[\log |z_i|]$. 

Now, Lemma \ref{residuedef} gives a quasi-isomorphism
\[
\Res_b \co \tilde{\sA}^{\bt}_X(O(\Bu_{\rho}))\ten_{\sA^{\bt}_X }\gr^J_b \sA^{\bt}_X(\vv)\langle  D \rangle \to \nu_{b*} \tilde{\sA}^{\bt}_X\ten_{\sA^{\bt}_X }\sA^{\bt}_{ D^{(b)}}((O(\Bu_{\rho})\ten_{\R}\vareps^b_{\Cx})[-b],  
\]
and the right-hand side is just
\[
 \nu_{b*}\tilde{\sA}^{\bt}_{ D^{(b)}}( O(\Bu_{\rho})\ten_{\R}\vareps^b_{\Cx})[-b].
\]
Therefore
\[
 \gr^J_b\tilde{\sA}^{\bt}_X(O(\Bu_{\rho}))\llbracket D \rrbracket|_{\widetilde{C^*} } \simeq \nu_{b*}\tilde{\sA}^{\bt}_{ D^{(b)}}( O(\Bu_{\rho})\ten_{\R}\vareps^b_{\Cx})[-b]|_{\widetilde{C^*} },
\]
and in particular $\Res_b$ defines an isomorphism
\[
 \H^a (\gr^J_b\tilde{A}^{\bt}(X,O(\Bu_{\rho}))\llbracket D \rrbracket|_{\widetilde{C^*} }) \cong \H^{a-b}( \tilde{A}^{\bt}(D^{(b)}, O(\Bu_{\rho})\ten_{\R}\vareps^b_{\Cx})|_{\widetilde{C^*}}).
\]

As in \S \ref{Bei1}, we have an  \'etale pushout $C^*= \widetilde{C^*}\cup_{S_{\Cx}}S$ of affine schemes, so to give an isomorphism $\sF \to \sG$ of quasi-coherent sheaves on $C^*$ is the same as giving an isomorphism $f: \sF|_{\widetilde{C^*}} \to \sG|_{\widetilde{C^*}}$ such that $f|_{S_{Cx}}$ is real, in the sense that $f= \bar{f}$ on $S_{\Cx}$. Since $\tilde{d}\log |z_i|=  (u+iv)d \log z_i + (u-iv)d\log \bar{z}_i$ is a boundary, we deduce that $[i(u-iv)^{-1}d \log z_i] \sim [-i(u+iv)^{-1}d \log \bar{z}_i ]  $, so 
\[
 \overline{(u-iv)^b\Res_b}= (u-iv)^b\Res_b,
\]
making use of the fact that $\vareps^b$ already contains a factor of $i^b$ (coming from $\Z(-b)$). 

Therefore $(u-iv)^b\Res_b$ gives an isomorphism
\[
 \H^a (\gr^J_b\tilde{A}^{\bt}(X,O(\Bu_{\rho}))\llbracket D \rrbracket)|_{C^*} \cong \H^{a-b}( \tilde{A}^{\bt}(D^{(b)}, O(\Bu_{\rho})\ten_{\Z}\vareps^b))|_{C^*}.
\]
Now, $d\log z_i$ is of type $(1,0)$, while $\vareps^b$ is of type $(b,b)$ and $(u-iv)$ is of type $(0,-1)$, so it follows that $(u-iv)^b\Res_b$ is of type $(0,0)$, i.e. $S$-equivariant.

As in Theorem \ref{mhsmal}, inclusion of harmonic forms gives an $S$-equivariant isomorphism
\[                                                                                                                                    
 \H^{a-b}( \tilde{A}^{\bt}(D^{(b)}, O(\Bu_{\rho})\ten_{\Z}\vareps^b))|_{C^*}    \cong                \H^{a-b}(D^{(b)}, O(\Bu_{\rho})\ten_{\Z}\vareps^b)\ten\O_{C^*},                                                                                                          \]
which is a 
 pure twistor structure of weight $(a-b)+2b= a+b$. 
Therefore
\[
\sH^a(\gr^J_b\tilde{\A}^{\bt}_X(O(\Bu_{\rho}))\llbracket D \rrbracket|_{C^*}) \cong \H^{a-b}(D^{(b)}, O(\Bu_{\rho})\ten_{\Z}\vareps^b)\ten\O_{C^*}
\]
is pure of weight $a+b$, as required.
\end{proof}
 
\begin{proposition}\label{grjpuremts}
  For the algebraic twistor filtration $(Y,y)^{R, \mal}_{\bT}$ of Definition \ref{Ytfil}, the $R\by \bG_m$-equivariant  cohomology sheaf
\[
 \sH^a(\gr^J_b\O(Y,y)^{R, \mal}_{\bT} )       
\]
on $C^*$ defines a pure ind-twistor structure of weight $a+b$, corresponding to the canonical $\bG_m$-action  on
\[
 \H^{a-b} (D^{(b)}, O(\Bu_{\rho})\ten_{\Z}\vareps^b).
\]
\end{proposition}
\begin{proof}
 The proof of Proposition \ref{grjpuremhs} carries over, replacing $S$-equivariance with  $\bG_m$-equivariance, and Theorem \ref{mhsmal} with  Theorem \ref{mtsmal}.
 \end{proof}

\begin{proposition}\label{grjpuremtsen}
If  the $S^1$-action on $\varpi_1(X,jy)^{\red}$ descends to $R$, then the associated discrete $S^1$-action of Proposition \ref{redenrich2} on
$
 \sH^a(\gr^J_b\O(Y,y)^{R, \mal}_{\bT} )       
$ corresponds to the $\boxast$ action of  $S^1\subset S$ (see Definition \ref{Saction}) on
\[
 \H^{a-b} (D^{(b)}, O(\Bu_{\rho})\ten_{\Z}\vareps^b).
\]
\end{proposition}
\begin{proof}
 The proof of Proposition \ref{grjpuremhs} carries over, replacing $S$-equivariance with discrete $S$-equivariance.         
\end{proof}

\begin{theorem}\label{qmts}
There is a canonical non-positively weighted mixed twistor structure $(Y,y)^{R, \mal}_{\MTS}$ on $(Y,y)^{R, \mal}$, in the sense of Definition \ref{algmtsdef}.
\end{theorem}
\begin{proof}
On $ \O(Y,y)^{R, \mal}_{\bT} =\tilde{A}^{\bt}(X, O(\Bu_{\rho}))\llbracket D \rrbracket|_{C^*}$, we define the filtration $\Dec J$ by
\[
(\Dec J)_r     (\O(Y,y)^{R, \mal}_{\bT})^n= \{ a \in J_{r-n}(\O(Y,y)^{R, \mal}_{\bT})^n\,:\, \tilde{D}a \in  J_{r-n-1}(\O(Y,y)^{R, \mal}_{\bT})^{n+1}\}.   
\]

For the Rees algebra construction $\xi$ of Lemma \ref{flatfiltrn}, we then set $\O(Y,y)^{R, \mal}_{\MTS}\in DG_{\Z}\Alg_{\bA^1 \by C^*}(R)_*( \Mat_1 \by \bG_m)$ to be 
\[
\O(Y,y)^{R, \mal}_{\MTS}:= \xi(  \O(Y,y)^{R, \mal}_{\bT}, \Dec J),      
\]
noting that this is flat and that $(Y,y)_{\MTS}^{R, \mal} \by_{\bA^1, 1}\Spec \R = (Y,y)^{R, \mal}_{\bT}$, so
\[
  (Y,y)_{\MTS}^{R, \mal} \by_{(\bA^1\by C^*), (1,1)}^{\oR}\Spec \R\simeq      (Y,y)^{R, \mal}. 
\]

We define $\ugr (Y,y)_{\MTS}^{R, \mal} \in dg_{\Z}\Aff(R)_*(\Mat_1 )$ by
\[
  \ugr (Y,y)_{\MTS}^{R, \mal}= \Spec (\bigoplus_{a,b}\H^{a-b} (D^{(b)}, O(\Bu_{\rho})\ten_{\Z}\vareps^b)[-a], d_1),   
\]
where $d_1\co \H^{a-b} (D^{(b)}, O(\Bu_{\rho})\ten_{\Z}\vareps^b) \to \H^{a-b+2} (D^{(b-1)}, O(\Bu_{\rho})\ten_{\Z}\vareps^{b-1})$ is the differential in the $E_1$ sheet of the spectral sequence associated to the filtration $J$. Combining Lemmas \ref{higherdirectlemma} and \ref{residuedef}, it follows that this is the same as the differential $ \H^{a-b} (X, \oR^b j_*j^{-1}O(\Bu_{\rho})) \to \H^{a-b+2}( X, \oR^{b-1} j_*j^{-1}O(\Bu_{\rho}))$ in the $E_2$ sheet of the Leray spectral sequence for $j: Y \to X$. The augmentation $ \bigoplus_{a,b}\H^{a-b} (D^{(b)}, O(\Bu_{\rho})\ten_{\Z}\vareps^b) \to O(R)$ is just defined to be the unique ring homomorphism  $\H^0(X, O(\Bu_{\rho}))= \R \to O(R)$.

In order to show that this defines a mixed twistor structure, it only remains to establish opposedness. Since $(Y,y)_{\MTS}^{R, \mal}$ is flat, 
\[
  (Y,y)_{\MTS}^{R, \mal} \by_{\bA^1, 0}^{\oR}\Spec \R \simeq    (Y,y)_{\MTS}^{R, \mal} \by_{\bA^1, 0}\Spec \R,   
\]
 and properties of Rees modules mean that this is just given by
\[
 \oSpec_{C^*} (\gr^{\Dec J}  \O(Y,y)^{R, \mal}_{\bT}) \in  dg_{\Z}\Aff_{C^*}( R)_*(\Mat_1 \by\bG_m),    
\]
where the $\Mat_1$-action  assigns $ \gr^{\Dec J}_n$ the weight $n$.

By \cite[Proposition 1.3.4]{Hodge2}, d\'ecalage has the formal property that the canonical map
\[
\gr^{\Dec J}_n  \O(Y,y)^{R, \mal}_{\bT}) \to (\bigoplus_a \sH^a(\gr^J_{n-a} \O(Y,y)^{R, \mal}_{\bT})[-a], d_1)    
\]
is a quasi-isomorphism. Since the right-hand side is just 
\[
(\bigoplus_{a}\H^{2a-n} (D^{(n-a)}, O(\Bu_{\rho})\ten_{\Z}\vareps^{n-a})[-a], d_1)\ten \O_{C^*}     
\]
by Proposition \ref{grjpuremts}, we have a quasi-isomorphism
\[
 (\ugr (Y,y)_{\MTS}^{R, \mal}) \by C^* \cong  (Y,y)_{\MTS}^{R, \mal} \by_{\bA^1, 0}^{\oR}\Spec \R.       
\]
That this is $(\Mat_1 \by\bG_m)$-equivariant follows because   $\H^{2a-n} (D^{(n-a)}, O(\Bu_{\rho})\ten_{\Z}\vareps^{n-a})$ is of weight $2a-n +2(n-a)= n$ for the $\bG_m$-action, and of weight $n$ for the $\Mat_1$-action, being $ \gr^{\Dec J}_n$.
\end{proof}

\begin{theorem}\label{qmhs}
If the local system on $X$ associated to any $R$-representation underlies a polarisable variation of Hodge structure, then  there is a canonical non-positively weighted mixed Hodge structure $(Y,y)^{R, \mal}_{\MHS}$ on $(Y,y)^{R, \mal}$, in the sense of Definition \ref{algmhsdef2}.
\end{theorem}
\begin{proof}
 We adapt the proof of Theorem \ref{qmts}, replacing Proposition \ref{grjpuremts} with Proposition \ref{grjpuremhs}. The first condition is equivalent to saying that the $S^1$-action descends to  $R$ and is algebraic, by Proposition \ref{vhsequiv}. We therefore set
\[
 \O(Y,y)^{R, \mal}_{\MHS}:= \xi(  \O(Y,y)^{R, \mal}_{\bF}, \Dec J),       
\]
for $ (Y,y)^{R, \mal}_{\bF}$ as in Definition \ref{Yhfil}, and let 
\[
  \ugr (Y,y)_{\MTS}^{R, \mal}= \Spec (\bigoplus_{a,b}\H^{a-b} (D^{(b)}, O(\Bu_{\rho})\ten_{\Z}\vareps^b)[-a], d_1),   
\]
  which is now in    $dg_{\Z}\Aff(R)_*(\bar{S} )$, since $O(\Bu_{\rho})$ is a sum of weight $0$ VHS, making $ \H^{a-b} (D^{(b)}, O(\Bu_{\rho})\ten_{\Z}\vareps^b)$
a weight $a-b+2b=a+b$ Hodge structure, and hence an $\bar{S}$-representation. 
\end{proof}

\begin{proposition}\label{qmtsenrich}
If the discrete $S^1$-action on $\varpi_1(X,jy)^{\red}$ descends to $R$, then there are natural $(S^1)^{\delta}$-actions on  $(Y,y)^{R, \mal}_{\MTS}$ and $\ugr (Y,y)^{R, \mal}_{\MTS}$, compatible with the opposedness isomorphism, and with $-1 \in S^1$ acting as $-1 \in \bG_m$.
\end{proposition}
\begin{proof}
 This is  a direct consequence of Proposition \ref{redenrich2} and Proposition \ref{grjpuremtsen}, since the Rees module construction transfers the discrete $S^1$-action.       
\end{proof}

\subsection{Singular and simplicial varieties}\label{singsn2} 

\begin{proposition}\label{hodge3stuff2}
If $Y$ is any separated  complex scheme of finite type, there exists  a simplicial smooth proper   complex variety $X_{\bt}$, a simplicial divisor $D_{\bt} \subset X_{\bt}$ with normal crossings, and a map $(X_{\bt} - D_{\bt}) \to Y$ such that $|X_{\bt} - D_{\bt}| \to Y$ is a weak equivalence, where $|Z_{\bt}|$ is the geometric realisation of the simplicial space $Z_{\bt}(\Cx)$.
\end{proposition}
\begin{proof}
 The results in \cite[\S 8.2]{Hodge3} and \cite[Propositions 5.1.7 and 5.3.4]{SD}, adapted as in  Corollary \ref{hodge3stuff}, give the equivalence required.      
\end{proof}

Now, let $X_{\bt}$ be a simplicial smooth proper complex variety, and $D_{\bt} \subset X_{\bt}$ a simplicial divisor  with normal crossings. Set $Y_{\bt}=X_{\bt} - D_{\bt}$,  assume that $|Y_{\bt}|$ is connected, and pick a point $y \in |Y_{\bt}|$. Let $j:|Y_{\bt}| \to |X_{\bt}|$  be the natural inclusion map. We will look at representations of the fundamental group  $\varpi_1(|X_{\bt}|,jy)^{\norm, \red}$ from Definition \ref{pinormdef}.

Using Proposition \ref{hodge3stuff2}, the following gives mixed twistor or mixed Hodge structures on relative Malcev homotopy types of arbitrary complex varieties. 
\begin{theorem}\label{singmhsmal2}\label{singmtsmal2}
If  $R$ is any quotient of $\varpi_1(|X_{\bt}|,jy)^{\norm, \red}$ (resp. any quotient to which the $(S^1)^{\delta}$-action of Proposition \ref{properred} descends and acts algebraically), then there is an algebraic mixed twistor  structure (resp. mixed Hodge structure) $(|Y_{\bt}|,y)^{R,\mal}_{\MTS}$ (resp. $(|Y_{\bt}|,y)^{R,\mal}_{\MHS}$) on the relative Malcev homotopy type $(|Y_{\bt}|,y)^{R, \mal}$.

There is also a canonical $\bG_m$-equivariant (resp. $S$-equivariant)  splitting
$$
\bA^1 \by (\ugr (|Y_{\bt}|^{R,\mal},0)_{\MTS}) \by \SL_2 \simeq (|Y_{\bt}|,y)^{R, \mal}_{\MTS}\by_{C^*, \row_1}^{\oR}\SL_2
$$
(resp.
$$
\bA^1 \by (\ugr (|Y_{\bt}|^{R,\mal},0)_{\MHS}) \by \SL_2 \simeq (|Y_{\bt}|,y)^{R, \mal}_{\MHS}\by_{C^*, \row_1}^{\oR}\SL_2)
$$
on pulling back along $\row_1:\SL_2 \to C^*$, whose pullback over $0 \in \bA^1$ is given by the opposedness isomorphism. 
\end{theorem}
\begin{proof}
We adapt the proof of Theorem \ref{singmtsmal}.
Define the cosimplicial CDGA $\tilde{A}(X_{\bt}, O(\Bu_{\rho})) \llbracket D_{\bt} \rrbracket $ on $C$ by  $n \mapsto \tilde{A}^{\bt}(X_n, O(\Bu_{\rho})) \llbracket D_n \rrbracket $, observing that functoriality  ensures that the cosimplicial and CDGA structures are compatible. This has an augmentation $(jy)^*: \tilde{A}(X_{\bt}, O(\Bu_{\rho})) \llbracket D_{\bt} \rrbracket  \to O(R)\ten O(C)$ given in level $n$ by $((\sigma_0)^nx)^*$, and inherits a filtration $J$ from the CDGAs $\tilde{A}^{\bt}(X_n, O(\Bu_{\rho})) \llbracket D_n \rrbracket $.

We then define the  mixed Hodge structure to be the object of $dg_{\Z}\Aff_{\bA^1\by C^*}(\Mat_1 \by R\rtimes S)$ given  by 
$$
|Y_{\bt}|^{R, \mal}_{\MHS}:= (\Spec \Th \xi(\tilde{A}(X_{\bt}, O(\Bu_{\rho})) \llbracket D_{\bt} \rrbracket  , \Dec\Th(J)  ))  \by_CC^*, 
$$
for $\Th$ the Thom--Sullivan construction of Definition \ref{Th}. $|Y_{\bt}|^{R, \mal}_{\MTS}$ is defined similarly, replacing $S$ with $\bG_m$. The graded object is given by
$$
\ugr |Y_{\bt}|^{R,\mal}_{\MHS}=  \Spec \Th(\bigoplus_{a,b}\H^{a-b} (D^{(b)}_{\bt}, O(\Bu_{\rho})\ten_{\Z}\vareps^b)[-a], d_1)
$$
in $dg_{Z}\Aff(R\rtimes \bar{S}) $,
with $\ugr |Y_{\bt}|^{R, \mal}_{\MTS}$ given by replacing $S$ with $\bG_m$.

For any CDGA $B$, we may regard $B$ as a cosimplicial CDGA (with constant cosimplicial structure), and then $\Th(B)=B$. In particular, $\Th(O(R))= O(R)$, so we have a basepoint $\Spec \Th((jy)^*): \bA^1 \by R \by C^* \to |Y_{\bt}|^{R, \mal}_{\MHS} $, giving  
$$
(|Y_{\bt}|,y)^{R, \mal}_{\MHS}\in  dg_{\Z}\Aff_{\bA^1\by C^*}(R)_*(\Mat_1 \by S),
$$
and similarly for $|Y_{\bt}|^{R, \mal}_{\MTS} $.

The proofs of Theorems \ref{qmhs} and \ref{qmts} now carry over for the remaining statements.
\end{proof}

\section{Algebraic MHS/MTS for quasi-projective varieties II --- non-trivial monodromy}\label{nontrivsn}

In this section, we assume that $X$ is a smooth projective complex variety, with $Y=X-D$ (for $D$ still a divisor locally of normal crossings).
The hypothesis in Theorems \ref{qmts} and \ref{qmhs} that $R$ be a quotient of $\varpi_1(X,jy)$ is  unnecessarily strong, and corresponds to allowing only those semisimple local systems on $Y$ with trivial monodromy around the divisor. By \cite{mochi}, every semisimple local system on $Y$ carries an essentially unique tame imaginary pluriharmonic metric, so it is conceivable that Theorem \ref{qmts}
 could hold for any reductive quotient $R$ of $\varpi_1(Y,y)$. 

However, Simpson's discrete $S^1$-action on $\varpi_1(X,jy)^{\red}$ does not extend to the whole of $ \varpi_1(Y,y)^{\red}$, but only to a quotient $ {}^{\nu}\!\varpi_1(Y,y)^{\red}$. This is because given a tame pure imaginary Higgs form $\theta$ and $\lambda \in S^1$, the Higgs form $\lambda \theta$ is only pure imaginary if either $\lambda=\pm 1$ or $\theta$ is nilpotent.
The group $ {}^{\nu}\!\varpi_1(Y,y)^{\red}$ is characterised by the property that its representations are   semisimple local systems whose associated Higgs form has nilpotent residues. This is equivalent to saying that $ {}^{\nu}\!\varpi_1(Y,y)^{\red}$-representations are semisimple local systems on $Y$ for which the monodromy around any component of $D$ has unitary eigenvalues. Thus the greatest generality in which Proposition \ref{qmtsenrich} could possibly hold is for any $S^1$-equivariant quotient $R$ of ${}^{\nu}\!\varpi_1(Y,y)^{\red}$.

Denote the maximal quotient of ${}^{\nu}\!\varpi_1(Y,y)^{\red}$ on which the $S^1$-action is algebraic by ${}^{\VHS}\!\varpi_1(Y,y)$. Arguing as in Proposition \ref{vhsequiv}, representations of ${}^{\VHS}\!\varpi_1(Y,y)$ correspond to real local systems underlying variations of Hodge structure on $Y$, and representations of ${}^{\VHS}\!\varpi_1(Y,y) \rtimes S^1$ correspond to weight $0$ real VHS. The greatest generality in which  Theorem \ref{qmhs} could hold is for any $S^1$-equivariant quotient $R$ of ${}^{\VHS}\!\varpi_1(Y,y)^{\red}$.

\begin{definition}
Given a semisimple real local system $\vv$ on $Y$, use Mochizuki's tame imaginary pluriharmonic metric to
decompose the associated connection $D:\sA^0_Y(\vv) \to \sA^1_Y(\vv)$ 
as $D= d^++\vartheta$ into  antisymmetric and symmetric  parts, and let $D^c:= i\dmd d^+ -i\dmd \vartheta$. Also write $D'= \pd +\bar{\theta}$ and $D''= \bar{\pd} +\theta$.
Note that these definitions are independent of the choice of pluriharmonic metric, since the metric is unique up to global automorphisms $\Gamma(X, \Aut(\vv))$.
\end{definition}

\subsection{Constructing mixed Hodge structures}
We now  outline a strategy for adapting  Theorem \ref{qmhs}   to more general $R$.

\begin{proposition}\label{mochilemmamhs}
Let $R$ be a quotient of  ${}^{\VHS}\!\varpi_1(Y,y)$ to which the  $S^1$-action descends, and assume we have the following data.
\begin{itemize}
 \item For each  weight $0$ real VHS $\vv$ on $Y$ corresponding to an $R \rtimes S^1$-representation,  an $S$-equivariant $\R$-linear graded subsheaf 
\[
 \sT^*(\vv) \subset j_*\sA^*_Y(\vv)\ten \Cx, 
\]
  on $X$, closed under the operations $D$ and $D^c$. This must be   functorial in $\vv$,  with 
\begin{itemize}
  \item
  $\sT^*(\vv\oplus \vv')=\sT^{*}(\vv)\oplus \sT^*(\vv')$,
  \item
  the image of 
  $\sT^*(\vv) \ten \sT^*(\vv')  \xra{\wedge} j_*\sA^*_Y(\vv\ten \vv')\ten \Cx$ contained in $\sT^*(\vv\ten \vv')$, and
  \item $1 \in \sT^*(\R)$.
\end{itemize}

\item  
An 
increasing non-negative $S$-equivariant filtration $J$ of $\sT^*(\vv) $ with $J_r\sT^n(\vv)=\sT^n(\vv)$ for all $n \le r$, compatible with the tensor structures, and closed under the operations $D$ and $D^c$.
\end{itemize}

Set $F^p\sT^{\bt}(\vv) := \sT^{\bt}(\vv)\cap F^p\sA^{\bt}(Y, \vv)_{\Cx}$, where the Hodge filtration $F$ is defined in the usual way in terms of the $S$-action,  and assume that 
\begin{enumerate}
\item\label{matchcdn0} The map $\sT^{\bt}(\vv) \to j_*\sA^{\bt}_Y(\vv)_{\Cx}$ is a quasi-isomorphism of sheaves  on $X$ for all $\vv$.

\item\label{taucdn0}  For all $i \ne r$, the sheaf $\sH^i(\gr^J_r\sT^{\bt}(\vv)) $  on $X$ is $0$.

\item\label{mhscdn} For all $a,b$ and $p$, the map
\[
  \bH^{a+b}(X,  F^p\gr^J_b \sT^{\bt}(\vv))\to   \H^a(X, \oR^b j_*\vv)_{\Cx}
\]
is injective, giving a  Hodge filtration $F^p \H^a(X, \oR^b j_*\vv)_{\Cx}$ which defines a pure Hodge structure of  weight $a+2b$ on $\H^a(X, \oR^b j_*\vv)$.
\end{enumerate}

Then there is a non-negatively weighted mixed  Hodge structure $(Y,y)^{R, \mal}_{\MHS}$, with 
\[
 \ugr (Y,y)^{R, \mal}_{\MHS} \simeq \Spec (\bigoplus_{a,b} \H^a(X, \oR^b j_* O(\Bu_{\rho}))[-a-b], d_2),       
\]
where $\H^a(X, \oR^b j_* O(\Bu_{\rho})) $ naturally becomes a pure Hodge structure of weight $a+2b$, and $d_2\co \H^a(X, \oR^b j_* O(\Bu_{\rho})) \to \H^{a+2}(X, \oR^{b-1} j_* O(\Bu_{\rho})) $ is the differential from the $E_2$ sheet of the Leray spectral sequence for $j$.
\end{proposition}
\begin{proof}
We proceed along similar lines to \cite{Morgan}. To construct the Hodge filtration, we first  define $\tilde{\sT}^{\bt}(\vv) \subset j_*\tilde{\sA}^{\bt}_Y(\vv)_{\Cx}$ to be given by the differential $\tilde{D}$  on the graded sheaf $\sT^*(\vv)\ten O(C)$, then  let $\sE_{\bF}(O(\Bu_{\rho}))$ be the homotopy fibre product
\[
(\tilde{\sT}^{\bt}(O(\Bu_{\rho}))\ten_{O(C)\ten \Cx}O(\widetilde{C^*}))\by^h_{(j_*\tilde{\sA}^{\bt}_Y(O(\Bu_{\rho}))\ten_{O(C)}O(S)\ten \Cx)}(j_*\tilde{\sA}^{\bt}_Y(O(\Bu_{\rho}))\ten_{O(C)}O(S))
\]
in the category of $R\rtimes S$-equivariant CDGAs on $X \by C^*_{\Zar}$, quasi-coherent over $C^*$. Here, we are extending $\sT^{\bt}$ to ind-VHS by setting $\sT^{\bt}(  \LLim_{\alpha}\vv_{\alpha}):= \LLim_{\alpha} \sT^{\bt}( \vv_{\alpha})$, and similarly for $\tilde{\sT}^{\bt}$.

Explicitly, a homotopy fibre product $C\by^h_DF$  is defined by replacing  $C \to D$ with a quasi-isomorphic surjection $C' \onto D$, then setting $C\by^h_DF:= C'\by_DF$. Equivalently, we could replace $F \to D$ with a surjection. That such surjections exist and give well-defined homotopy fibre products up to quasi-isomorphism follows from the observation in Proposition \ref{quaffworks1} that the homotopy category of quasi-coherent CDGAs on a quasi-affine scheme can be realised as the homotopy category of a right proper model category. 

Observe that for co-ordinates $u,v$ on $C$ as in Remark \ref{Ccoords},
\[
 \tilde{\sT}^{\bt}(O(\Bu_{\rho}))\ten_{O(C)\ten \Cx}O(\widetilde{C^*})\cong (\bigoplus_{p \in \Z}F^p{\sT}^{\bt}(O(\Bu_{\rho}))(u+iv)^{-p})[(u-iv), (u-iv)^{-1}],
\]
while $(j_*\tilde{\sA}^{\bt}_Y(O(\Bu_{\rho}))\ten_{O(C)}O(S))\cong  j_*{\sA}^{\bt}_Y(O(\Bu_{\rho}))\ten O(S)$ (with the same reasoning as Lemma \ref{hodgenice}).

Note that $\widetilde{C^*}\by_C\widetilde{C^*} \cong \widetilde{C^*} \sqcup S_{\Cx}$, so $\sE_{\bF}(O(\Bu_{\rho}))|_{\widetilde{C^*}}$ is
\begin{align*}
  &[\tilde{\sT}^{\bt}(O(\Bu_{\rho}))|_{\widetilde{C^*}} \oplus \tilde{\sT}^{\bt}(O(\Bu_{\rho}))|_{S_{\Cx}}]\by^h_{[j_*\tilde{\sA}^{\bt}_Y(O(\Bu_{\rho}))|_{S_{\Cx}} \oplus j_*\tilde{\sA}^{\bt}_Y(O(\Bu_{\rho}))|_{S_{\Cx}}] }[j_*\tilde{\sA}^{\bt}_Y(O(\Bu_{\rho}))|_{S_{\Cx}}\\
\simeq &[\tilde{\sT}^{\bt}(O(\Bu_{\rho}))|_{\widetilde{C^*}} \oplus  j_*\tilde{\sA}^{\bt}_Y(O(\Bu_{\rho}))|_{S_{\Cx}}]\by^h_{[j_*\tilde{\sA}^{\bt}_Y(O(\Bu_{\rho}))|_{S_{\Cx}} \oplus j_*\tilde{\sA}^{\bt}_Y(O(\Bu_{\rho}))|_{S_{\Cx}}] }[j_*\tilde{\sA}^{\bt}_Y(O(\Bu_{\rho}))|_{S_{\Cx}}\\
\simeq &\tilde{\sT}^{\bt}(O(\Bu_{\rho}))|_{\widetilde{C^*}}.
\end{align*}
Similarly, $\sE_{\bF}(O(\Bu_{\rho}))|_S \simeq  j_*\tilde{\sA}^{\bt}_Y(O(\Bu_{\rho}))\ten_{O(C)}O(S)$.
 
If we let $\CC^{\bt}(X, -)$ denote either the cosimplicial \v Cech or Godement resolution on $X$, then the Thom--Sullivan  functor $\Th$ of Definition \ref{Th} gives us a functor $\Th \circ \CC^{\bt}(X, -)$ from sheaves of DG algebras on $X$ to DG algebras. We denote this by $\oR \Gamma(X, -)$, since it gives a canonical choice for derived global sections. We then
define the Hodge filtration by  
\[
 \O(Y,y)_{\bF}^{R, \mal}:= \oR \Gamma(X,\sE_{\bF}(O(\Bu_{\rho})))
\]
as an object of
$\Ho( DG_{\Z}\Alg_{C^*}(R)_*(S))$. Note that condition (\ref{matchcdn0}) above ensures that the  pullback of $(Y,y)_{\bF}^{R, \mal} $ over $1 \in C^*$ is quasi-isomorphic to 
$
\Spec  \oR \Gamma(X, j_*{\sA}^{\bt}_Y(O(\Bu_{\rho})).
$
Since the map
\[
A^{\bt}(Y, O(\Bu_{\rho})) \to  \oR \Gamma(X, j_*{\sA}^{\bt}_Y(O(\Bu_{\rho}))
\]
is a quasi-isomorphism, this means that  $(Y,y)_{\bF}^{R, \mal} $ indeed defines an algebraic Hodge filtration on $ (Y,y)^{R, \mal}$.

To define the mixed Hodge structure, we first note that  condition (\ref{taucdn0}) above implies that 
\[
 (\tilde{\sT}^{\bt}(O(\Bu_{\rho}))\ten_{O(C)}O(S), \tau) \to  (\tilde{\sT}^{\bt}(O(\Bu_{\rho}))\ten_{O(C)}O(S), J)
\]
is a filtered quasi-isomorphism of complexes, where $\tau$ denotes the good truncation filtration. We then define
$\O(Y,y)_{\MHS}^{R, \mal}$ to be the homotopy limit of the diagram
\[
 \xymatrix@R=0ex{\xi(\oR \Gamma(X, \tilde{\sT}^{\bt}(O(\Bu_{\rho}))|_{\widetilde{C^*}}), \Dec \oR \Gamma(J)) \ar[r] &
\xi(\oR \Gamma(X, \tilde{\sT}^{\bt}(O(\Bu_{\rho}))|_{S_{\Cx}}), \Dec \oR \Gamma(J)) \\
  \xi(\oR \Gamma(X, \tilde{\sT}^{\bt}(O(\Bu_{\rho}))|_{S_{\Cx} }), \Dec \oR \Gamma(\tau)) \ar[ur] \ar[r] &
\xi(\oR \Gamma(X, j_*\tilde{\sA}_Y^{\bt}(O(\Bu_{\rho}))|_{S_{\Cx}}), \Dec \oR \Gamma(\tau)) \\
\xi(\oR \Gamma(X, j_*\tilde{\sA}_Y^{\bt}(O(\Bu_{\rho}))|_{S}), \Dec \oR \Gamma(\tau)) \ar[ur], 
}
\]
which can be expressed as an iterated homotopy fibre product of the form $E_1\by^h_{E_2}E_3\by^h_{E_4}E_5$.  Here, $\xi$ denotes the Rees algebra construction as in Lemma \ref{flatfiltrn}. The basepoint $jy \in X$ gives an augmentation of this DG algebra, so we have defined an object of $\Ho(DG_{\Z}\Alg_{\bA^1 \by C^*}(R)_*(\Mat_1 \by S))$.

Conditions (\ref{taucdn0}) and (\ref{matchcdn0}) above ensure that the second and third maps in the diagram above are both quasi-isomorphisms, with the second map  becoming an isomorphism on pulling back along $1 \in \bA^1$ (corresponding to forgetting the filtrations). The latter observation means that  we do indeed have
\[
 (Y,y)_{\MHS}^{R, \mal}\by^{\oR}_{\bA^1,1 }\Spec \R \simeq  (Y,y)_{\bF}^{R, \mal}.
\]
Setting $\ugr (Y,y)^{R, \mal}_{\MHS}$ as in the statement above, it only remains to establish opposedness. 

Now, the pullback  of $\xi(M,W)$ along $0 \in \bA^1$ is just $\gr^WM$. Moreover, \cite[Proposition 1.3.4]{Hodge2} shows that for any filtered complex $(M, J)$, the map
\[
\gr^{\Dec J}M \to (\bigoplus_{a,b}\H^a(\gr^J_bM)[-a], d_1^J)
\]
is a quasi-isomorphism, where  $d_1^J$ is the differential in the $E_1$ sheet of the spectral sequence associated to $J$. Thus the structure sheaf $\sG $ of $(Y,y)_{\MHS}^{R, \mal}\by^{\oR}_{\bA^1,0 }\Spec \R$ is the homotopy limit of the diagram
\[
 \xymatrix@R=0ex{
(\bigoplus_{a,b} \bH^a(X, \gr^J_b\tilde{\sT}^{\bt}(O(\Bu_{\rho}))|_{\widetilde{C^*}})[-a], d_1^J ) \ar[r] &
 (\bigoplus_{a,b} \bH^a(X, \gr^J_b\tilde{\sT}^{\bt}(O(\Bu_{\rho}))|_{S_{\Cx}})[-a], d_1^J )\\
  (\bigoplus_{a,b} \H^a(X, \sH^b\tilde{\sT}^{\bt}(O(\Bu_{\rho}))|_{S_{\Cx}})[-a], d_2)  \ar[ur] \ar[r] &
(\bigoplus_{a,b} \H^a(X, \oR^bj_*(O(\Bu_{\rho}))|_{S_{\Cx}})[-a], d_2) \\
(\bigoplus_{a,b} \H^a(X, \oR^bj_*(O(\Bu_{\rho}))|_{S})[-a], d_2) \ar[ur], 
}
\]
where $d_2$ denotes the differential on the $E_2$ sheet of the spectral sequence associated to a bigraded complex.

The second and third maps in the diagram above are isomorphisms, so we can write  $\sG$ as the homotopy fibre product of 
\[
\xymatrix@R=0ex{ (\bigoplus_{a,b} \bH^{a+b}(X, \gr^J_b\tilde{\sT}^{\bt}(O(\Bu_{\rho}))|_{\widetilde{C^*}})[-a-b], d_1^J )\ar[r] & {(\bigoplus_{a,b} \H^{a}(X, \oR^bj_*(O(\Bu_{\rho}))|_{S_{\Cx}})[-a-b], d_2) }\\
(\bigoplus_{a,b} \H^{a}(X, \oR^bj_*(O(\Bu_{\rho}))|_{S})[-a-b], d_2) \ar[ur].
}
\]

By condition (\ref{mhscdn}) above, $ \bH^{a}(X, \oR^bj_*(O(\Bu_{\rho}))$ has the structure of an $S$-representation of weight $a+2b$ --- denote this by $E^{ab}$, and set $E:=(\bigoplus_{a,b}E^{ab}, d_2)$. Then we can apply Lemma \ref{flatmhs} to rewrite $\sG$ as 
\[
 (\bigoplus_{p \in \Z}F^p(E\ten \Cx)(u+iv)^{-p})[(u-iv), (u-iv)^{-1}]\by^h_{E\ten O(S_{\Cx})}E\ten O(S).
\]
Since $(\bigoplus_{p \in \Z}F^p(E\ten \Cx)(u+iv)^{-p})[(u-iv), (u-iv)^{-1}]\cong E\ten O(\widetilde{C^*})$, this is just
\[
 E\ten (O(\widetilde{C^*})\by^h_{O(S_{\Cx}) }O(S)) \simeq E\ten \O(C^*),
\]
as required.
\end{proof}

\subsection{Constructing mixed twistor structures}

Proposition \ref{mochilemmamhs} does not easily adapt to mixed twistor structures,  
 since an $S$-equivariant morphism $M \to N$ of quasi-coherent sheaves on $S$  is an isomorphism if and only if the fibres $M_1 \to N_1$ are isomorphisms of vector spaces, but the same is not true of a $\bG_m$-equivariant morphism of quasi-coherent sheaves on $S$. Our solution is to introduce holomorphic properties, the key idea being that for $t$ the co-ordinate on $S^1$, the connection $t\circledast D\co \sA^0_Y(\vv)\ten O(S^1) \to \sA^1_Y(\vv)\ten O(S^1)$ does not define a local system of $O(S^1)$-modules, essentially because iterated integration takes us outside $O(S^1)$. However, as observed in \cite[end of \S 3]{MTS}, $t\circledast D$ defines a holomorphic family  of local systems on $X$, parametrised by $S^1(\Cx)= \Cx^{\by}$.

\begin{definition}
Given a smooth complex affine variety $Z$, define $O(Z)^{\hol}$ to be the ring of  holomorphic functions $f\co Z(\Cx) \to \Cx$. Given a smooth real affine variety $Z$, define $O(Z)^{\hol}$ to be the ring of $\Gal(\Cx/\R)$-equivariant holomorphic functions $f\co Z(\Cx) \to \Cx$.
\end{definition}

In particular, $O(S^1)^{\hol}$ is the ring of functions $f: \Cx^{\by} \to \Cx$ for which $\overline{f(z)}= f(\bar{z}^{-1})$, or equivalently convergent Laurent series $\sum_{n \in \Z} a_n t^n$ for which $\bar{a}_n= a_{-n}$.

\begin{definition}
 Given a smooth complex  variety $Z$, define $\sA_Y^0\O_Z^{\hol}$ to be the sheaf on $Y \by Z(\Cx)$ consisting of smooth complex functions which are holomorphic along $Z$. Write $\sA_Y^{\bt}\O_Z^{\hol}:= \sA_Y^{\bt}\ten_{\sA_Y^0} \sA_Y^0\O_Z^{\hol}$, and, given a local system $\vv$ on $Y$, set $\sA_Y^{\bt}\O_Z^{\hol}(\vv):= \sA_Y^{\bt}(\vv)\ten_{\sA_Y^0} \sA_Y^0\O_Z^{\hol}$.

Given a smooth real  variety $Z$, define $\sA_Y^0\O_Z^{\hol}$ to be the $\Gal(\Cx/\R)$-equivariant  sheaf $\sA_Y^0\O_{Z_{\Cx}}^{\hol}$
on $Y \by Z(\Cx)$, where the the non-trivial element $\sigma \in \Gal(\Cx/\R)$ acts by $\sigma(f)(y,z)= \overline{f(y, \sigma z)}$.
\end{definition}

\begin{definition}\label{Pcoords}
 Define $P:= C^*/\bG_m$ and $\tilde{P}:= \widetilde{C^*}/\bG_m$, for $C^*$ from Definition \ref{CCdef} and $\widetilde{C^*}$ from Definition \ref{tildeC}. As in Definition \ref{U1def}, we have $S^1= S/\bG_m$, and hence a canonical inclusion $S^1 \into P$ (given by cutting out the divisor $\{(u:v)\,:\, u^2+v^2=0\}$). For co-ordinates $u,v$ on $C$ as in Remark \ref{Ccoords}, fix co-ordinates $t= \frac{u+iv}{u-iv}$ on $\tilde{P}$, and $a= \frac{u^2-v^2}{u^2+v^2}$, $b=\frac{2uv}{u^2+v^2} $ on $S^1$ (so $a^2+b^2=1$).
\end{definition}

Thus $P\cong \bP^1_{\R}$ and $\tilde{P}\cong \bA^1_{\Cx}$, the latter isomorphism using   the co-ordinate $t$.   The canonical map $\tilde{P} \to P$ is  given by $t \mapsto (1+t: i-it)$, and the map $S^1_{\Cx} \to \tilde{P}$ by $(a,b) \mapsto a+ib$.

Also note that the \'etale pushout $C^*= \widetilde{C^*}\cup_{S_{\Cx}}S$ from \S \ref{Bei1} corresponds to an \'etale pushout
\[
 P= \tilde{P}\cup_{S^1_{\Cx}}S^1,       
\]
where $S^1_{\Cx}\cong \bG_{m, \Cx}$ is given by the subscheme $t\ne 0$ in $\bA^1_{\Cx}$. Note that the  $\Gal(\Cx/\R)$-action  on $O(S^1_{\Cx})=\Cx[t,t^{-1}]$ given by the real form $S^1$  is determined by the condition that the non-trivial element $\sigma \in \Gal(\Cx/\R)$ maps $t$ to $t^{-1}$.

\begin{definition}\label{sAbrevedef}
 Define $\breve{\sA}_Y^{\bt}(\vv)$ to be the sheaf $\bigoplus_{n\ge 0}\sA_Y^n(\vv)\O_P^{\hol}(n)$ of graded algebras on $Y \by P(\Cx)$, equipped with the differential $uD+vD^c$, where $u,v \in \Gamma(P, \O_P(1))$ correspond to the weight $1$ generators $u,v \in O(C)$.
 \end{definition}

\begin{definition}
 Given a polarised scheme $(Z, \O_Z(1))$ (where $Z$ need not be projective), and a sheaf $\sF$  of $\O_Z$-modules, define $\uline{\Gamma}(Z,\sF):=\bigoplus_{n \in \Z}\Gamma(Z, \sF(n))$. This is regarded as a $\bG_m$-representation, with $\Gamma(Z, \sF(n))$ of weight $n$.      
\end{definition}

\begin{lemma}\label{brevetilde}
The $\bG_m$-equivariant sheaf $\tilde{\sA}_Y^{\bt}(\vv)$ of $O(C)$-complexes on $Y$ (from Definition \ref{sAtildedef}) is given by
\[
 \tilde{\sA}_Y^{\bt}(\vv)\cong  \uline{\Gamma}( P(\Cx), \breve{\sA}_Y^{\bt}(\vv))^{\Gal(\Cx/\R)}.      
\]
\end{lemma}
\begin{proof}
We first consider $\Gamma( P(\Cx), \breve{\sA}_Y^0(\vv))$. This is the sheaf on $Y$ which sends any open subset $U \subset Y$ to the ring of consisting of those smooth functions $f\co U \by \bP^1(\Cx) \to \Cx$ which are holomorphic along $\bP^1(\Cx)$. Thus for any $y \in U$, $f(y,-)$ is a global holomorphic function on $\bP^1(\Cx)$, so is constant. Therefore $ \Gamma( P(\Cx), \breve{\sA}_Y^0(\vv))= \sA^0_Y\ten \Cx$.

For general $n$, a similar argument using finite-dimensionality of $\Gamma(\bP^1(\Cx), \O(n)^{\hol})$ shows that 
\[
        \Gamma( P(\Cx),\breve{\sA}_Y^0(\vv)(n))\cong \sA^0_Y\ten \Gamma(P(\Cx), \O_P(n)^{\hol}).
\]
Now by construction of $P$, we have $\uline{\Gamma}(P(\Cx), \O_P^{\hol}) \cong O(C)\ten \Cx$ with the grading corresponding the the $\bG_m$-action. Thus
\[
 \uline{\Gamma}( P(\Cx), \breve{\sA}_Y^*(\vv)(n))^{\Gal(\Cx/\R)} \cong    \tilde{\sA}_Y^*(\vv).    
\]
Since the differential in both cases is given by $uD+vD^c$, this establishes the isomorphism of complexes.
\end{proof}

\begin{definition}
 On the schemes $S^1$ and $\tilde{P}$, define the sheaf $\O(1)$ by pulling back   $\O_P(1)$ from $P$. Explicitly, on any affine scheme $\Spec A$ over $P$,  the corresponding module $A(1) := \Gamma(\Spec A, \O(1))$  is given by   
\[
 A(1)= A(u,v)/(t(u-iv)-(u+iv)).
\]
 Hence $\O_{\tilde{P}}(1)= \O_{\tilde{P}}(u-iv)$ and $ \O_{S^1}(1)\ten \Cx = \O_{S^1}\ten \Cx (u-iv)$ are trivial line bundles, but $ \O_{S^1}(1)= \O_{S^1}(u,v)/(au+bv-u,bu-av-v)$.    
\end{definition}

\begin{proposition}\label{mochilemmamts}
 Let $R$ be a quotient of  $\varpi_1(Y,y)^{\red}$, and assume that we have the following data.
\begin{itemize}
 \item For each finite rank local real system $\vv$ on $Y$ corresponding to an $R$-representation,  a   
flat graded $(\sA^0_X\ten \Cx)$-submodule  
\[
 \sT^*(\vv) \subset j_*\sA_Y^*(\vv)\ten \Cx, 
\]
closed under the operations $D$ and $D^c$. This must be   functorial in $\vv$,  with 
\begin{itemize}
  \item
  $\sT^*(\vv\oplus \vv')=\sT^{*}(\vv)\oplus \sT^*(\vv')$,
  \item
  the image of 
  $\sT^*(\vv) \ten \sT^*(\vv')  \xra{\wedge} j_*\sA^*_Y(\vv\ten \vv')\ten \Cx$ contained in $\sT^*(\vv\ten \vv')$, and
  \item $1 \in \sT^*(\R)$.
\end{itemize}

\item  
An 
increasing non-negative filtration $J$ of $\sT^*(\vv) $ with $J_r\sT^n(\vv)=\sT^n(\vv)$ for all $n \le r$, compatible with the tensor  
structure, and closed under the operations $D$ and $D^c$.
\end{itemize}

Set $\breve{\sT}^{\bt}(\vv)\subset j_*\breve{\sA}^{\bt}_Y(\vv) $ to be the complex on $X \by P(\Cx)$ whose underlying sheaf is $\bigoplus_{n \ge 0} \sT^n(\vv)\ten_{\sA^0_X\ten \Cx} \sA^0_X\O_P^{\hol}(n)$,  and assume that 
\begin{enumerate}
\item\label{matchcdn} For $S^1(\Cx) \subset P(\Cx)$, the map $\breve{\sT}^{\bt}(\vv)|_{S^1(\Cx)} \to j_*\breve{\sA}_Y^{\bt}(\vv)|_{S^1(\Cx)} $ is a quasi-isomorphism of sheaves of $\O_{S^1}^{\hol}$-modules on $X\by S^1(\Cx)$ for all $\vv$.

\item\label{taucdn}  For all $i \ne r$, the sheaf $\sH^i(\gr^J_r\breve{\sT}^{\bt}(\vv)|_S) $ of  $\O_{S^1}^{\hol}$-modules on $X\by S^1(\Cx)$ is $0$.

\item\label{mtscdn} For all $a,b\ge 0$, the $\Gal(\Cx/\R)$-equivariant sheaf
\[
\ker( \bH^a(X, \gr^J_b\breve{\sT}^{\bt}(\vv))|_{\tilde{P}(\Cx) } \oplus \sigma^*\overline{\bH^a(X, \gr^J_b\breve{\sT}^{\bt}(\vv))|_{\tilde{P}(\Cx)}} \to \H^a(X, \sH^b( j_*\breve{\sA}_Y^{\bt}(\vv)))|_{S^1(\Cx)})
\]
 is a finite locally free $\O_{P}^{\hol}$-module of slope  $a+2b$.
\end{enumerate}

Then there is a non-negatively weighted mixed twistor structure $(Y,y)^{R, \mal}_{\MTS}$, with 
\[
 \ugr (Y,y)^{R, \mal}_{\MTS} \simeq \Spec (\bigoplus_{a,b} \H^a(X, \oR^b j_* O(\Bu_{\rho}))[-a-b], d_2),       
\]
where $\H^a(X, \oR^b j_* O(\Bu_{\rho})) $ is assigned the weight $a+2b$, and $d_2\co \H^a(X, \oR^b j_* O(\Bu_{\rho})) \to \H^{a+2}(X, \oR^{b-1} j_* O(\Bu_{\rho})) $ is the differential from the $E_2$ sheet of the Leray spectral sequence for $j$.
 \end{proposition}
\begin{proof}
Define $\O(Y,y)_{\bT}^{R, \mal}$ to be the homotopy fibre product 
\[
\oR \Gamma(X, \uline{\Gamma}(\tilde{P}(\Cx), \breve{\sT}^{\bt}(O(\Bu_{\rho}))))\by^h_{\oR \Gamma(X, j_*\uline{\Gamma}(S^1(\Cx),\breve{\sA}_Y^{\bt}(O(\Bu_{\rho}))))} \oR \Gamma(X, j_*  \uline{\Gamma}(S^1(\Cx),\breve{\sA}_Y^{\bt}(O(\Bu_{\rho}))))^{\Gal(\Cx/\R)}
\]
as an object of $\Ho(DG_{\Z}\Alg_{C^*}(R)_*(\bG_m))$, and let
$\O(Y,y)_{\MTS}^{R, \mal}$ be  the homotopy limit of the diagram
\[
 \xymatrix@=2ex
{\xi(\oR \Gamma(X, \uline{\Gamma}(\tilde{P}(\Cx), \breve{\sT}^{\bt}(O(\Bu_{\rho})))), \Dec \oR \Gamma(J)) \ar[d]\\  
\xi(\oR \Gamma(X,     \uline{\Gamma}(S^1(\Cx), \breve{\sT}^{\bt}(O(\Bu_{\rho})))       ), \Dec \oR \Gamma(J)) \\
  \xi(\oR \Gamma(X,   \uline{\Gamma}(S^1(\Cx), \breve{\sT}^{\bt}(O(\Bu_{\rho})))      ), \Dec \oR \Gamma(\tau)) \ar[u] \ar[d] \\ 
\xi(\oR \Gamma(X, j_*\uline{\Gamma}(S^1(\Cx),\breve{\sA}_Y^{\bt}(O(\Bu_{\rho})))), \Dec \oR \Gamma(\tau)) \\
\xi(\oR \Gamma(X, j_*  \uline{\Gamma}(S^1(\Cx),\breve{\sA}_Y^{\bt}(O(\Bu_{\rho}))))^{\Gal(\Cx/\R)}  ), \Dec \oR \Gamma(\tau))   \ar[u], 
}
\]
as an object of $\Ho(DG_{\Z}\Alg_{\bA^1 \by C^*}(R)_*(\Mat_1 \by \bG_m))$.
Here, we are extending $\sT^{\bt}$ to ind-local systems by setting $\sT^{\bt}(  \LLim_{\alpha}\vv_{\alpha}):= \LLim_{\alpha} \sT^{\bt}( \vv_{\alpha})$, and similarly for $\breve{\sT}^{\bt}$.

Given a $\Gal(\Cx/\R)$-equivariant sheaf $\sF$ of $ \O_{P}^{\hol}$-modules on  $X \by P(\Cx)$, the group cohomology complex gives a $\Gal(\Cx/\R)$-equivariant cosimplicial sheaf $\CC^{\bt}(\Gal(\Cx/\R),\sF  )$ on $X \by P(\Cx) $ --- this is a resolution of $\sF$, with $\CC^0(\Gal(\Cx/\R),\sF  ) = \sF\oplus \sigma^*\bar{\sF}$.
Applying the Thom--Whitney functor $\Th$, this means that
\[
 \Th\CC^{\bt}(\Gal(\Cx/\R),j_*\breve{\sA}_Y^{\bt}(\vv))
\]
is a  $\Gal(\Cx/\R)$-equivariant $\O_{P}^{\hol}$-CDGA on $X \by P(\Cx) $, equipped with a surjection to $j_*\breve{\sA}_Y^{\bt}(\vv) \oplus \sigma^*\overline{j_*\breve{\sA}_Y^{\bt}(\vv)} $.

This allows us to consider the $\Gal(\Cx/\R)$-equivariant  sheaf $\sB^{\bt}_{\bT}$ of $\O_{P}^{\hol}$-CDGAs on $P(\Cx)$ given by the fibre product of
\[
\xymatrix@R=0ex{
 (\breve{\sT}^{\bt}(O(\Bu_{\rho})))|_{\tilde{P}(\Cx)}  \oplus \sigma^*\overline{\breve{\sT}^{\bt}(O(\Bu_{\rho}))}|_{\tilde{P}(\Cx)})\ar[r] &
{ (j_*\breve{\sA}_Y^{\bt}(O(\Bu_{\rho}))\oplus \sigma^*\overline{j_*\breve{\sA}_Y^{\bt}(O(\Bu_{\rho}))})|_{S^1(\Cx)} }\\  
\Th\CC^{\bt}(\Gal(\Cx/\R),j_*\breve{\sA}_Y^{\bt}(O(\Bu_{\rho})))|_{S^1(\Cx)}. \ar[ur]}
\]
Note that since the second map is surjective, this fibre product is in fact a homotopy fibre product.
In particular,
\[
\O(Y,y)_{\bT}^{R, \mal} \simeq  \oR \Gamma(X,  \uline{\Gamma}(P(\Cx),\sB^{\bt}_{\bT})^{\Gal(\Cx/\R)})|_{C^*}.
\]

Now, $\uline{\Gamma}(P(\Cx), -)$ gives a functor from  Zariski sheaves to $\O_{P}^{\hol}$-modules to $O(C)$-modules, and  we consider the functor $\uline{\Gamma}(P(\Cx), -)|_{C^*}$  to quasi-coherent sheaves on $C^*$. 
There is a right derived functor $\oR\uline{\Gamma}(P(\Cx), -)$; by \cite{GAGA}, the map
\[
 \uline{\Gamma}(P(\Cx), \sF)|_{C^*}\to  \oR\uline{\Gamma}(P(\Cx), \sF)|_{C^*}      
\]
is a quasi-isomorphism for all coherent $\O_{P}^{\hol}$-modules $\sF$. Given a morphism $f: Z \to P_{\Cx}$ of polarised varieties, with $Z$ affine, and a quasi-coherent Zariski sheaf $\sF$ of  $\O_Z^{\hol}$-modules on $Z$, note that
\[
\oR\uline{\Gamma}(P(\Cx), f_*\sF)\simeq  \oR\uline{\Gamma}(P(\Cx), \oR f_*\sF)\simeq  \oR\uline{\Gamma}(Z(\Cx), \sF)\simeq  \uline{\Gamma}(Z(\Cx), \sF).   
\]

There are convergent spectral sequences
\[
 \H^a(P(\Cx),  \sH^b(\sB^{\bt}_{\bT})(n)) \abuts  \bH^{a+b}(P(\Cx), \sB^{\bt}_{\bT}(n))     
\]
for all $n$, and Condition (\ref{mtscdn}) above ensures that $\sH^b(\sB^{\bt}_{\bT}) $ is a direct sum of coherent sheaves. 
Since $\H^i\oR\uline{\Gamma}(P(\Cx), \sB^{\bt}_{\bT}) = \bigoplus_{n \in \Z} \bH^{i}(P(\Cx), \sB^{\bt}_{\bT}(n))$, this means that the map
\[
 \uline{\Gamma}(P(\Cx), \sB^{\bt}_{\bT})|_{C^*} \to   \oR\uline{\Gamma}(P(\Cx), \sB^{\bt}_{\bT})|_{C^*}    
\]
 is a quasi-isomorphism. 
Combining these observations shows that
\[
 \O(Y,y)_{\bT}^{R, \mal} \simeq  \oR \Gamma(X,\oR\uline{\Gamma}(P(\Cx), \sB^{\bt}_{\bT}))^{\Gal(\Cx/\R)} |_{C^*}.     
\]

In particular,
\[
\O(Y,y)_{\bT}^{R, \mal}\ten^{\oL}_{\O_{C^*}}O(\bG_m) \to  \oR \Gamma(X,\uline{\Gamma}(\Spec \Cx,\sB^{\bt}_{\bT}\ten_{\O_{P}^{\hol}, (1:0)}\Cx )^{\Gal(\Cx/\R)}))
\]
is a quasi-isomorphism,  and note that right-hand side is just
\[
\oR \Gamma(X,(\sB^{\bt}_{\bT}\ten_{\O_{P}^{\hol}, (1:0)}\Cx )^{\Gal(\Cx/\R)}) \ten O(\bG_m),
\]
which is the homotopy fibre
\[
 \oR \Gamma(X, [{\sT}^{\bt}(O(\Bu_{\rho})) \by^h_{ j_*{\sA}_Y^{\bt}(O(\Bu_{\rho}))\ten \Cx }  j_*\breve{\sA}_Y^{\bt}(O(\Bu_{\rho}))] ),
\]
and hence quasi-isomorphic to $\oR \Gamma(X,j_*{\sA}_Y^{\bt}(O(\Bu_{\rho})))$ by condition (\ref{matchcdn}) above. This proves that
\[
(Y,y)_{\bT}^{R, \mal}\by^{\oR}_{C^*,1}\Spec \R \simeq  (Y,y)^{R, \mal},
\]
so $(Y,y)_{\bT}^{R, \mal}$ is indeed  a twistor filtration on $(Y,y)^{R, \mal}$. 

The proof that $\O(Y,y)_{\bT}^{R, \mal}\simeq \O(Y,y)_{\MTS}^{R, \mal}\ten^{\oL}_{\O_{\bA^1}, 1}\Spec \R $ follows along exactly the same lines as  in  Proposition \ref{mochilemmamhs}, so it only remains to establish opposedness.

Arguing as in the proof of Proposition \ref{mochilemmamhs}, we see  that the structure sheaf
$\sG $ of $\ugr (Y,y)_{\MTS}^{R, \mal}\by^{\oR}_{\bA^1,0 }\Spec \R$ is 
 the homotopy fibre product of the diagram
\[
\xymatrix@=2ex{ (\bigoplus_{a,b} \uline{\Gamma}(\tilde{P}(\Cx), \bH^{a+b}(X, \gr^J_b\breve{\sT}^{\bt}(O(\Bu_{\rho}))))[-a-b], d_1^J )\ar[d] \\
 {(\bigoplus_{a,b} \uline{\Gamma}(S^1(\Cx),\H^{a}(X, \sH^b(j_*\breve{\sA}_Y(O(\Bu_{\rho})))))[-a-b], d_2) }\\
(\bigoplus_{a,b} \uline{\Gamma}(S^1(\Cx),\H^{a}(X, \sH^b(j_*\breve{\sA}_Y(O(\Bu_{\rho})))))^{\Gal(\Cx/\R)}[-a-b], d_2), \ar[u]
}
\]
as a $(\Mat_1 \by R\by\bG_m )$-equivariant sheaf of  CDGAs over $C^*$.

Set $\ugr \sB_{\MHS}^{a,b}$ to be the sheaf on $P(\Cx)$ given by the fibre product of the diagram
\[
\xymatrix@=2ex{ 
\bH^{a+b}(X, \gr^J_b\breve{\sT}^{\bt}(O(\Bu_{\rho})))|_{\tilde{P}(\Cx)}  \oplus \sigma^*\overline{\bH^{a+b}(X, \gr^J_b\breve{\sT}^{\bt}(O(\Bu_{\rho}))|_{\tilde{P}(\Cx)})}  \ar[d]\\
{ \H^{a}(X, \sH^b(j_*\breve{\sA}_Y^{\bt}(O(\Bu_{\rho}))))\oplus \sigma^*\overline{\H^{a}(X, \sH^b(j_*\breve{\sA}_Y^{\bt}(O(\Bu_{\rho}))))}|_{S^1(\Cx)} } \\
 \Th\CC^{\bt}(\Gal(\Cx/\R),\H^{a}(X, \sH^b(j_*\breve{\sA}_Y^{\bt}(O(\Bu_{\rho}))))|_{S^1(\Cx)},  \ar[u]
}
\]
and observe that
\[
 \sG \simeq (\bigoplus_{a,b} \uline{\Gamma}(P(\Cx),\ugr \sB_{\MHS}^{a,b})^{\Gal(\Cx/\R)}|_{C^*}, d_1^J). 
\]

Now, $\ugr \sB_{\MHS}^{a,b} $ is just the homotopy fibre product of 
\[
\xymatrix@=2ex{
 \bH^{a+b}(X, \gr^J_b\breve{\sT}^{\bt}(O(\Bu_{\rho})))|_{\tilde{P}(\Cx)}  \oplus \sigma^*\overline{\bH^{a+b}(X, \gr^J_b\breve{\sT}^{\bt}(O(\Bu_{\rho}))|_{\tilde{P}(\Cx)})} \ar[d] \\
  \H^{a}(X, \sH^b(j_*\breve{\sA}_Y^{\bt}(O(\Bu_{\rho}))))\oplus \sigma^*\overline{\H^{a}(X, \sH^b(j_*\breve{\sA}_Y^{\bt}(O(\Bu_{\rho}))))}|_{S^1(\Cx)}\\ \ar[u]  \H^{a}(X, \sH^b(j_*\breve{\sA}_Y^{\bt}(O(\Bu_{\rho}))))|_{S^1(\Cx)};    
}
\]
condition (\ref{matchcdn}) ensures that the first map is injective, so $\ugr \sB_{\MHS}^{a,b}  $ is quasi-isomorphic to
the kernel of 
\[
 \bH^a(X, \gr^J_b\breve{\sT}^{\bt}(O(\Bu_{\rho})))|_{\tilde{P}(\Cx) } \oplus \sigma^*\overline{\bH^a(X, \gr^J_b\breve{\sT}^{\bt}(O(\Bu_{\rho})))|_{\tilde{P}(\Cx)}} \to \H^a(X, \sH^b( j_*\breve{\sA}_Y^{\bt}(O(\Bu_{\rho}))))|_{S^1(\Cx)}.        
\]
By condition (\ref{mtscdn}), this is a holomorphic vector bundle on $P(\Cx)$ of slope $a+2b$. 

Now, we just observe that for any holomorphic vector bundle $\sF$ of slope $m$, the map $\Gamma(P(\Cx), \F(-m)) \to 1^*\sF$, given by taking the fibre at $1 \in P(\R)$, is an isomorphism of complex vector spaces, and that the maps
\[
 \Gamma(P(\Cx), \F(-m))\ten   {\Gamma}(P(\Cx),\O(n))    \to {\Gamma}(P(\Cx),\sF(n-m)) 
\]
are isomorphisms  for $n \ge 0$. This gives an  isomorphism
\[
 \uline{\Gamma}(P(\Cx),  \sF)|_{C^*} \cong  (1^*\sF) \ten \O_{C^*}   
\]
over $C^*$, which becomes $\bG_m$-equivariant if we set $1^*\sF$ to have weight $m$.

Therefore
\[
        \uline{\Gamma}(P(\Cx),\ugr \sB_{\MHS}^{a,b})|_{C^*}^{\Gal(\Cx/\R)}\cong \H^a(X, \oR^b j_* O(\Bu_{\rho}))\ten \O_{C^*},
\]
making use of condition (\ref{matchcdn}) to show that $\H^a(X, \oR^b j_* O(\Bu_{\rho}))\ten \Cx$ is the fibre of  $\ugr \sB_{\MHS}^{a,b}$ at $1 \in P(\R)$. This completes the proof of opposedness.
\end{proof}

\begin{proposition}\label{mochilemmaenrich}
Let $R$ be a quotient of  ${}^{\nu}\!\varpi_1(Y,y)^{\red}$ to which the discrete $S^1$-action descends, assume that the conditions of Proposition \ref{mochilemmamts} hold, and assume in addition that for all $\lambda \in \Cx^{\by}$, the map $\lambda\dmd\co j_*\tilde{\sA}^{\bt}_Y(\vv) \to j_*\tilde{\sA}^{\bt}_Y( \lambda\bar{\lambda}^{-1}\circledast \vv) $ maps   $\sT(\vv)$ isomorphically to $\sT(\lambda\bar{\lambda}^{-1}\circledast \vv) $.  Then there are natural $(S^1)^{\delta}$-actions on  $(Y,y)^{R, \mal}_{\MTS}$ and $\ugr (Y,y)^{R, \mal}_{\MTS}$, compatible with the opposedness isomorphism, and with the action of $-1 \in S^1$ coinciding with that of  $-1 \in \bG_m$.
\end{proposition}
\begin{proof}
 The proof of Proposition \ref{qmtsenrich} carries over, substituting Proposition \ref{mochilemmamts} for Theorem \ref{qmts}.        
\end{proof}

\subsection{Unitary monodromy}

In this section, we will consider only semisimple local systems $\vv$ on $Y$ with unitary monodromy around the local components of $D$ (i.e. semisimple monodromy with unitary eigenvalues),

\begin{definition}
For $\vv$ as above, let $\sM(\vv) \subset j_*\sA^0(\vv)\ten \Cx$ consist of locally $L_2$-integrable functions for the Poincar\'e metric, holomorphic in the sense that they lie in $\ker \bar{\pd}$, where $D= \pd+\bar{\pd} + \theta + \bar{\theta}$. 

Then set 
\[
 \sA^*_X(\vv)\langle  D \rangle:= \sM(\vv)\ten_{\O_X}\sA^*_X(\R)\langle D \rangle \subset j_*\sA_Y(\vv)\ten \Cx,
\]
where $\O_X$ denotes the sheaf of holomorphic functions on $X$.
\end{definition}
The crucial observation which we now make is that $\sA^*_X(\vv)\langle  D \rangle $ is closed under the operations $D$ and $D^c$. Closure under $\bar{\pd}$ is automatic, and closure under $\pd$ follows because Mochizuki's metric is tame, so $\pd\co \sM(\vv)\to \sM(\vv)\ten_{\O_X}\Omega^1_X\langle D \rangle$. Since $\vv$ has unitary monodromy around the local components of $D$, the Higgs form $\theta$ is holomorphic, which ensures that $\sA^*_X(\vv)\langle  D \rangle$ is closed under both
  $\theta$ and $\bar{\theta}$. We can thus write $\sA^{\bt}_X(\vv)\langle D \rangle $ for the complex given by  $\sA^*_X(\vv)\langle  D \rangle $ with differential $D$.

\begin{lemma}
 For all $m \ge 0$, there is a morphism 
\[
\Res_m\co \sA^{\bt}_X(\vv)\langle  D \rangle \to \nu_{m*} \sA^{\bt}_{D^{(m)}}(\nu_m^{-1}j_*\vv\ten \vareps^m)\langle C^{(m)} \rangle[-m],                      
\]
 compatible with both $D$ and $D^c$, for $D^{(m)},C^{(m)}$ as in Definition \ref{Dmdef} . 
\end{lemma}
\begin{proof}
As in \cite[1.4]{Timm}, $\Res_m$ is given in level $q$ by the composition
\begin{align*}
 \sA^q_X(\vv)\langle  D \rangle &= \sM(\vv)\ten_{\O_X}\sA^q_X\langle  D \rangle\\
&\xra{\id \ten \Res_m}  \sM(\vv)\ten_{\O_X}\nu_{m*}\sA^{q-m}_{D^{(m)}}(\vareps^m_{\R})\langle  C^{(m)}\rangle\\
&= \nu_{m*}[\vareps^m\ten_{\Z}\nu_m^*\sM(\vv)\ten_{\O_{D^{(m)}}}\sA^{q-m}_{D^{(m)}}(\vareps^m)\langle  C^{(m)}\rangle\\
&\to \nu_{m*}[\vareps^m\ten_{\Z}\nu_m^*\sM(\vv)\ten_{\O_{D^{(m)}}}\sA^{q-m}_{D^{(m)}}(\vareps^m)\langle  C^{(m)}\rangle,
\end{align*}
where the final map is given by orthogonal projection.
 The proof of \cite[Lemma 1.5]{Timm} then adapts to show that $\Res_m $is compatible with both $D$ and $D^c$.
\end{proof}

Note that $(j_*\vv\ten \vareps^m)|_{D^m-D^{m+1}}$ inherits a pluriharmonic metric from $\vv$, so is necessarily a semisimple local system on the quasi-projective variety $D^m-D^{m+1}=D^{(m)}-C^{(m)} $.

\begin{definition}
 Define a filtration on $ \sA^{\bt}_X(\vv)\langle  D \rangle $ by 
\[
 J_r\sA^{\bt}_X(\vv)\langle  D \rangle := \ker(\Res_{r+1}),
\]
for $r \ge 0$. This generalises \cite[Definition 1.6]{Timm}.
\end{definition}

\begin{definition}\label{L2def}
Define the graded sheaf $\sL_{(2)}^*(\vv)$ on $X$ to consist of  $j_*\vv$-valued $L^2$ distributional forms $a$ for which $\pd a$ and $\bar{\pd}a$ are also $L^2$. Write $L_{(2)}^*(X,\vv):= \Gamma(X, \sL_{(2)}^*(\vv))$. 
\end{definition}

Since $\theta$ is holomorphic, note that the operators $\theta $ and $\bar{\theta}$ are bounded, so also act on $\sL_{(2)}^*(\vv)\ten \Cx$.

\subsubsection{Mixed Hodge structures}

\begin{theorem}\label{unitmhs}
If $R$ is a quotient of ${}^{\VHS}\!\varpi_1(Y,y)$ for which the representation $\pi_1(Y,y)\to R(\R)$ has unitary monodromy around the local components of $D$, then
there is a canonical non-positively weighted mixed Hodge structure $(Y,y)^{R, \mal}_{\MHS}$ on $(Y,y)^{R, \mal}$, in the sense of Definition \ref{algmhsdef2}. The associated split MHS is given by
\[
 \ugr (Y,y)^{R, \mal}_{\MHS} \simeq \Spec (\bigoplus_{a,b} \H^a(X, \oR^b j_* O(\Bu_{\rho}))[-a-b], d_2),       
\]
with $\H^a(X, \oR^b j_* O(\Bu_{\rho})) $ a pure ind-Hodge structure of weight $a+2b$.
\end{theorem}
\begin{proof}
We apply Proposition \ref{mochilemmamhs}, taking $\sT^*(\vv):= \sA^*_X(\vv)\langle  D \rangle$, equipped with its filtration $J$. The first condition to check is compatibility with tensor operations. This follows because, although a product of arbitrary $L^2$ functions is not $L^2$, a product of meromorphic $L^2$ functions is so.

Next, we check that $\sA^{\bt}_X(\vv)\langle  D \rangle \to j_*\sA^{\bt}_Y(\vv)_{\Cx}$ is a quasi-isomorphism, with $\gr^J_m\sA^{\bt}_X(\vv)\langle  D \rangle \simeq \oR^mj_*\vv[-m]$.
 \cite[Proposition 1.7]{Timm} (which deals with unitary local systems), adapts to show that $\Res_m$ gives a quasi-isomorphism
\[
 \gr^J_m\sA^{\bt}_X(\vv)\langle  D \rangle \to J_0\nu_{m*}\sA^{\bt}_{D^{(m)}}(\nu_m^{-1}j_*\vv\ten \vareps^m)\langle C^{(m)} \rangle[-m].
\]
Since $\oR^mj_*\vv \cong \nu_{m*}(\nu_m^{-1}j_*\vv\ten \vareps^m)$, this means that it suffices to establish the quasi-isomorphism for $m=0$ (replacing $X$ with $D^{(m)} $ for the higher cases). The proof of \cite[Theorem D.2(a)]{Timm2} adapts to this generality, showing that $j_*\vv \to J_0\sA^{\bt}_X(\vv)\langle  D \rangle$ is a quasi-isomorphism.

It only remains to show that for all $a,b$, the groups $\bH^{a+b}(X,  F^p\gr^J_b\sA^{\bt}_X(\vv)\langle  D \rangle)$ define a Hodge filtration on $\H^a(X, \oR^b j_*\vv)_{\Cx}$, giving a pure Hodge structure of  weight $a+2b$. This is essentially \cite[Proposition 6.4]{Timm}: the quasi-isomorphism induced above by $\Res_m$ is in fact a filtered quasi-isomorphism, provided we set $\vareps^m$ to be of type $(m,m)$. 
By applying a twist, we can therefore reduce to the case $b=0$ (replacing $X$ with $D^{(b)} $ for the higher cases),  so we wish to show that the groups $\bH^{a}(X,  F^pJ_0\sA^{\bt}_X(\vv)\langle  D \rangle)$  define a Hodge filtration on $\H^a(X,  j_*\vv)$ of weight $a$. 

The proof of \cite[Proposition D.4]{Timm2} adapts to give this result, by identifying  $\H^*(X,  j_*\vv)$ with $L^2$ cohomology, which in turn is identified with the space of harmonic forms.  We have a bicomplex $ (\Gamma(X,\sL_{(2)}^{*}(\vv)\ten \Cx), D', D'')$  satisfying the principle of two types, with $F^pJ_0\sA^{\bt}_X(\vv)\langle  D \rangle \to F^p\sL_{(2)}^{\bt}(\vv)\ten \Cx $  and $j_*\vv \to \sL_{(2)}^{\bt}(\vv) $ both being quasi-isomorphisms. 
\end{proof}

\subsubsection{Mixed twistor structures}

\begin{definition}

Given a smooth complex variety $Z$, let  $\sL^*_{(2)}(\vv)\O_{Z}^{\hol} $ be the sheaf on $X \by Z(\Cx)$ consisting of holomorphic families  of $L^2$ distributions on $X$, parametrised by $Z(\Cx)$.
Explicitly, given   a local co-ordinate $z$ on $Z(\Cx)$, the space $\Gamma(U \by \{|z|<R\}, \sL^n_{(2)}(\vv)\O_P^{\hol}) $ consists of power series
\[
 \sum_{m \ge 0} a_mz^n       
\]
with $a_m \in \Gamma(U,\sL^*_{(2)}(\vv))\ten \Cx$, such that for all $K \subset U$ compact and all $r<R$, the sum
\[
  \sum_{m \ge 0} {\|a_m\|}_{2,K} r^m      
\]
converges, where ${\|-\|}_{2,K} $ denotes the $L^2$ norm on $K$.
\end{definition}

\begin{definition}
Set $\breve{L}^n_{(2)}(X,\vv) $ to be the complex of $\O_P^{\hol}$-modules  on $P(\Cx)$ given by
\[
\breve{L}^n_{(2)}(X,\vv):=\Gamma(X, \sL^n_{(2)}(\vv)\O_P^{\hol}(n) ), 
\]
with differential $uD+vD^c$. Note that locally on $P(\Cx)$, elements of $\breve{L}^n_{(2)}(X,\vv) $ can be characterised as convergent power series with coefficients in ${L}^n_{(2)}(X,\vv)\ten \Cx $.       
\end{definition}

\begin{theorem}\label{unitmts}
If  $\pi_1(Y,y)\to R(\R)$ is Zariski-dense, with  unitary monodromy around the local components of $D$, then
there is a canonical non-positively weighted mixed twistor structure $(Y,y)^{R, \mal}_{\MTS}$ on $(Y,y)^{R, \mal}$, in the sense of Definition \ref{algmtsdef}. The associated split MTS is given by
\[
 \ugr (Y,y)^{R, \mal}_{\MTS} \simeq \Spec (\bigoplus_{a,b} \H^a(X, \oR^b j_* O(\Bu_{\rho}))[-a-b], d_2),       
\]
with $\H^a(X, \oR^b j_* O(\Bu_{\rho})) $ of weight $a+2b$.
\end{theorem}
\begin{proof}
We verify the conditions of Proposition \ref{mochilemmamts}, setting 
\[
\sT^*(\vv) \subset j_*\sA_Y(\vv)\ten \Cx
\]
to be $\sT^*(\vv)=: {\sA}^*_X(\vv)\langle  D \rangle$, with its filtration $J$ defined above.
This gives the complex $\breve{\sT}^{\bt}(\vv)\subset j_*\breve{\sA}^{\bt}_Y(\vv) $  on $X \by P(\Cx)$ whose underlying sheaf is $\bigoplus_{n \ge 0} \sT^n(\vv)\ten_{\sA^0_X} \sA^0_X\O_P^{\hol}(n)$, with differential $uD+vD^c$. 

This leads us to study the restriction to
 $S^1(\Cx) \subset P(\Cx)$, where we can divide  $\sT^{pq}(\vv)$ by $(u+iv)^p(u-iv)^q$,
giving 
\[
j_*\breve{\sA}^{\bt}_Y(\vv)|_{S^1(\Cx)} \cong (j_*{\sA}^*_Y(\vv)\O_{S^1}^{\hol}, t^{-1}\circledast D),
\]
where (adapting Lemma \ref{discreteact}),
\[
 t^{-1}\circledast D:=d^++ t^{-1}\dmd \vartheta=  \pd +\bar{\pd}+t^{-1}\theta +t \bar{\theta},
\]
for $t \in \Cx^{\by} \cong S^1(\Cx)$. There is a similar expression for $\breve{\sT}^{\bt}(\vv)|_{S^1(\Cx)}$.

 Now, as observed in \cite[end of \S 3]{MTS}, $t^{-1}\circledast D$ defines a holomorphic family $\sK(\vv)$ of local systems on $Y$, parametrised by $S^1(\Cx)= \Cx^{\by}$. Beware that for non-unitary points $\lambda \in \Cx^{\by}$, the canonical metric is not pluriharmonic on the  fibre $\sK(\vv)_{\lambda}$, since $\lambda^{-1}\theta +\lambda \bar{\theta} $ is not Hermitian. The proof of Theorem \ref{unitmhs}  (essentially \cite[Proposition 1.7]{Timm} and  \cite[Theorem D.2(a)]{Timm2}) still adapts to verify conditions (\ref{matchcdn}) and (\ref{taucdn}) from Proposition \ref{mochilemmamts}, replacing $\vv$ with $\sK(\vv)$, so that for instance
\[
j_*\sK(\vv) \to J_0 \breve{\sT}^{\bt}(\vv)|_{S^1(\Cx)}
\]
 is a quasi-isomorphism.

It remains to verify condition (\ref{mtscdn}) from  Proposition \ref{mochilemmamts}: we need to show that 
for all $a,b\ge 0$, the $\Gal(\Cx/\R)$-equivariant sheaf
\[
\ker( \bH^a(X, \gr^J_b\breve{\sT}^{\bt}(\vv))|_{\tilde{P}(\Cx) } \oplus \sigma^*\overline{\bH^a(X, \gr^J_b\breve{\sT}^{\bt}(\vv))|_{\tilde{P}(\Cx)}} \to \H^a(X, \sH^b( j_*\breve{\sA}_Y^{\bt}(\vv)))|_{S^1(\Cx)})
\]
 is a finite locally free $\O_{P}^{\hol}$-module of slope  $a+2b$.

Arguing as in the proof of Theorem \ref{unitmhs}, we may apply a twist to reduce to the case $b=0$ (replacing $X$ with $D^{(b)} $ for the higher cases),  so we wish to show that 
\[
\sE^a:=\ker( \bH^a(X, J_0\breve{\sT}^{\bt}(\vv))|_{\tilde{P}(\Cx) } \oplus \sigma^*\overline{\bH^a(X, J_0\breve{\sT}^{\bt}(\vv))|_{\tilde{P}(\Cx)}} \to \H^a(X,  j_*\sK(\vv))|_{S^1(\Cx)})
\]
is a holomorphic vector bundle on $P(\Cx)$ of slope  $a$.

We do this by considering the graded sheaf $\sL^*_{(2)}(\vv)$ of $L^2$-integrable distributions from Definition \ref{L2def}, and observe that  \cite[Proposition D.4]{Timm2} adapts to show that
\[
 j_*\sK(\vv) \to (\sL^*_{(2)}(\vv)\O_{S^1}^{\hol}, t^{-1}\circledast D)
\]
is a quasi-isomorphism on $X \by S^1(\Cx)$. 

On restricting to $\tilde{P}(\Cx) \subset P(\Cx)$, Definition \ref{Pcoords} gives the co-ordinate $t$ on $\tilde{P}(\Cx)$ as 
$t= \frac{u+iv}{u-iv}$, and dividing $\sT^{n}(\vv)$ by $(u-iv)^n$ gives an isomorphism
\[
 \breve{\sT}^{\bt}(\vv)|_{\tilde{P}(\Cx)} \cong ({\sA}^*_X(\vv)\langle  D \rangle\O_{\tilde{P}}^{\hol}, tD' +D''),
\]
and similarly for  $j_*\breve{\sA}^{\bt}_Y(\vv)|_{\tilde{P}(\Cx)}$

Thus we also wish to show that
\[
 J_0\breve{\sT}^{\bt}(\vv))|_{\tilde{P}(\Cx) } \to (\sL^*_{(2)}(\vv)\O_{\tilde{P}}^{\hol}, t D' + D'')
\]
is a quasi-isomorphism. Condition (\ref{matchcdn}) from  Proposition \ref{mochilemmamts} combines with the quasi-isomorphism above to show that we have a quasi-isomorphism on $S^1(\Cx) \subset \tilde{P}(\Cx)$, so cohomology of the quotient is supported on $0 \in \tilde{P}(\Cx)$. Studying the fibre over this point, it thus suffices to show that
\[
 (J_0\sT(\vv)), D'') \to (\sL^*_{(2)}(\vv)\ten \Cx,  D'')
\]
is a quasi-isomorphism, which also follows by adapting  \cite[Proposition D.4]{Timm2}. 

 Combining the
 quasi-isomorphisms above gives an isomorphism
\[
 \sE^a \cong \sH^a(\breve{L}^{\bt}_{(2)}(X,\vv) ),
\]
and inclusion of harmonic forms $\cH^a(X,\vv) \into {L}^a_{(2)}(X,\vv)$ gives a map
\[
 \cH^a(X,\vv)\ten_{\R}\O_P^{\hol}(a) \to   \sH^a(\breve{L}^{\bt}_{(2)}(X,\vv) ).     
\]
The Green's operator $G$ behaves well in holomorphic families, so gives a decomposition
\[
 \breve{L}^{a}_{(2)}(X,\vv) =   (\cH^a(X,\vv)\ten_{\R}\O_P^{\hol}(a)) \oplus \Delta \breve{L}^{a}_{(2)}(X,\vv),     
\]
making use of finite-dimensionality of $\cH^a(X,\vv)$ to give the isomorphism $ \cH^a(X,\vv)\ten_{\R}\O_P^{\hol}(a) \cong \ker \Delta \cap \breve{L}^{a}_{(2)}(X,\vv)$.

Since these expressions are $\Gal(\Cx/\R)$-equivariant,  it suffices to work on $\tilde{P}(\Cx)$. Dividing $\sT^{n}(\vv)$ by $(u-iv)^n$ gives
\[
 \breve{L}^{\bt}_{(2)}(X,\vv)|_{\tilde{P}(\Cx)} \cong  ({L}^{*}_{(2)}(X,\vv)\O_{\tilde{P}}^{\hol}, t D' + D''). 
\]
Now, since $D'(D'')^* + (D'')^*D'=0$, we can write 
\[
 \half \Delta= (t D' + D'') (D'')^*+ (D'')^* (t D' + D''),      
\]
giving us a direct sum  decomposition
\[
 \breve{L}^{a}_{(2)}(X,\vv)|_{\tilde{P}(\Cx)} =   (\cH^a(X,\vv)\ten_{\R}\O_{\tilde{P}}^{\hol}) \oplus  (t D' + D'') \breve{L}^{a}_{(2)}(X,\vv)|_{\tilde{P}(\Cx)} \oplus   (D'')^* \breve{L}^{n}_{(2)}(X,\vv)|_{\tilde{P}(\Cx)},  
\]
with the principle of two types (as in \cite{Simpson} Lemmas 2.1 and 2.2) showing that  $(t D' + D'')\co \im((D'')^*) \to \im(t D' + D'')  $ is an isomorphism.

We have therefore shown that $\sE^a \cong \cH^a(X,\vv)\ten_{\R}\O_P^{\hol}(a)$, which is indeed of slope $a$.
\end{proof}

\begin{proposition}\label{unitmtsenrich}
Assume that a Zariski-dense representation $\pi_1(Y,y)\to R(\R)$ has unitary monodromy around the local components of $D$, and that 
the discrete $S^1$-action on ${}^{\nu}\!\varpi_1(Y,y)^{\red}$ descends to $R$. Then there are natural $(S^1)^{\delta}$-actions on  $(Y,y)^{R, \mal}_{\MTS}$ and $\ugr (Y,y)^{R, \mal}_{\MTS}$, compatible with the opposedness isomorphism, and with the action of $-1 \in S^1$ coinciding with that of $-1 \in \bG_m$.
\end{proposition}
\begin{proof}
We just observe that the construction $\sT^*(\vv)= {\sA}^*_X(\vv)\langle  D \rangle$ of Theorem \ref{unitmts} satisfies the conditions of Proposition \ref{mochilemmaenrich}, 
being closed under the $\boxast$-action of $\Cx^{\by}$.
    \end{proof}

\subsection{Singular and simplicial varieties}\label{unitsingsn}

Fix 
a smooth proper simplicial complex variety $X_{\bt}$, and a simplicial divisor $D_{\bt} \subset X_{\bt}$ with normal crossings. Set $Y_{\bt}:=X_{\bt} - D_{\bt}$, with a point $y \in Y_0$, and write $j:Y_{\bt} \to X_{\bt}$ for the embedding. Note that Proposition \ref{hodge3stuff2} shows that for any separated complex scheme $Y$ of finite type, there exists such a simplicial variety $Y_{\bt}$ with 
an augmentation $ a\co Y_{\bt}\to Y$ for which $|Y_{\bt}| \to Y$ is a weak equivalence.

\begin{theorem}\label{singunitmts}
Take $\rho\co \pi_1(|Y_{\bt}|,y)\to R(\R)$  Zariski-dense with $R$ pro-reductive,  and assume that for every local system $\vv$ on $|Y_{\bt}|$ corresponding to an $R$-representation, the local system $a_0^{-1}\vv$ on $Y_0$ is semisimple, with unitary monodromy around the local components of $D_0$. Then
there is a canonical non-positively weighted mixed twistor structure $(|Y_{\bt}| ,y)^{R, \mal}_{\MTS}$ on $(|Y_{\bt}|,y)^{R, \mal}$, in the sense of Definition \ref{algmtsdef}. 

The associated split MTS is given by
\[
 \ugr (|Y_{\bt}|,y)^{R, \mal}_{\MTS} \simeq \Spec \Th (\bigoplus_{p,q} \H^p(X_{\bt}, a^{-1}\oR^q j_* O(\Bu_{\rho}))[-p-q], d_2),       
\]
with $\H^p(X_n, \oR^q j_* a_n^{-1}O(\Bu_{\rho})) $ of weight $p+2q$. Here, $\H^p(X_{\bt},a^{-1} \vv)$ denotes the cosimplicial vector space $n\mapsto \H^p(X_n, a_n^{-1}\vv)$, and $\Th$ is the Thom-Whitney functor of Definition \ref{Th}.
\end{theorem}
\begin{proof}
Our first observation is that the pullback of a holomorphic pluriharmonic metric is holomorphic, so for any local system $\vv$ corresponding to an $R$-representation, the local system $a_n^{-1}\vv$ on $Y_n$ is semisimple for all $n$, with unitary monodromy around the local components of $D_n$. We may therefore form objects 
\[
(Y_n,(\sigma_0)^ny )_{\MTS}^{R, \mal} \in   dg_{\Z}\Aff_{\bA^1 \by C^*}(R)_*( \Mat_1 \by \bG_m),
\] 
and $\ugr (Y_n,(\sigma_0)^ny)_{\MTS}^{R, \mal} \in  dg_{\Z}\Aff(R)_*(\Mat_1 )$  as in the proof of  Theorem \ref{unitmts}, together with opposedness quasi-isomorphisms.

These constructions are functorial, giving cosimplicial CDGAs 
\[
 \O(Y_{\bt},y )_{\MTS}^{R, \mal} \in   cDG_{\Z}\Alg_{\bA^1 \by C^*}(R)_*( \Mat_1 \by \bG_m),
\]
and $\O(\ugr (Y_{\bt},y)_{\MTS}^{R, \mal}) \in  cDG_{\Z}\Aff(R)_*(\Mat_1 )$. We now apply the Thom-Whitney functor, giving an algebraic MTS with  $\ugr (|Y_{\bt}|,y)_{\MTS}^{R, \mal}$ as above, and
\[
 \O(|Y_{\bt}|,y )_{\MTS}^{R, \mal}:= \Th(\O(Y_{\bt},y )_{\MTS}^{R, \mal} ).
\]
Taking the fibre over $(1,1) \in \bA^1 \by C^*$ gives $ \Th(\O(Y_{\bt},y )^{R, \mal} )$, which is quasi-isomorphic to $\O(|Y_{\bt}|,y )^{R, \mal}$, by Lemma \ref{singmalcev}. 
\end{proof}

\begin{theorem}\label{singunitmhs}
Take $\rho \co \pi_1(|Y_{\bt}|,y)\to R(\R)$  Zariski-dense with $R$ pro-reductive, and assume that for every local system $\vv$ on $|Y_{\bt}|$ corresponding to an $R$-representation, the local system $a_0^{-1}\vv$  underlies a variation of Hodge structure with unitary monodromy around the local components of $D_0$. Then
there is a canonical non-positively weighted mixed Hodge structure $(Y,y)^{R, \mal}_{\MHS}$ on $(Y,y)^{R, \mal}$, in the sense of Definition \ref{algmhsdef2}. The associated split MTS is given by
\[
 \ugr (Y,y)^{R, \mal}_{\MHS} \simeq \Spec \Th(\bigoplus_{p,q} \H^p(X_{\bt}, \oR^q j_* a^{-1}O(\Bu_{\rho}))[-p-q], d_2),       
\]
with $\H^p(X_n, \oR^q j_* a_n^{-1}O(\Bu_{\rho})) $ a pure ind-Hodge structure of weight $p+2q$.
\end{theorem}
\begin{proof}
The proof of Theorem \ref{singunitmts} carries over, replacing Theorem \ref{unitmts} with Theorem \ref{unitmhs}, and observing that variations of Hodge structure are preserved by pullback.
\end{proof}

\begin{definition}
 Define  ${}^{\nu}\!\varpi_1(|Y_{\bt}|,y)^{\norm}$ to be the quotient of $\varpi_1(|Y_{\bt}|,y)^{\norm}$ characterised as follows. 
Representations of ${}^{\nu}\!\varpi_1(|Y_{\bt}|,y)^{\norm}$  correspond to local systems $\vv$ on $|Y_{\bt}|$ for which $a_0^{-1}\vv$ is a semisimple local system on $Y_0$ whose monodromy around local components of $D_0$ has unitary eigenvalues.
\end{definition}

\begin{proposition}\label{singproperred}
There is a discrete action of the  circle group $S^1$ on ${}^{\nu}\!\varpi_1(|Y_{\bt}|,y)^{\norm}$, such that the composition $S^1 \by \pi_1(|Y_{\bt}|,y) \to {}^{\nu}\!\varpi_1(|Y_{\bt}|,y)^{\norm}$ is  continuous. We denote this last map by $\sqrt h:\pi_1(|Y_{\bt}|,y) \to {}^{\nu}\!\varpi_1(|Y_{\bt}|,y)^{\norm}((S^1)^{\cts})$.
\end{proposition}
\begin{proof}
 The proof of Proposition \ref{properred} carries over to the  quasi-projective case.
\end{proof}

\begin{proposition}\label{singunitmtsenrich}
Take a pro-reductive $S^1$-equivariant quotient $R$  of ${}^{\nu}\!\varpi_1(|Y_{\bt}|,x)^{\norm} $, and assume  that for every local system $\vv$ on $|Y_{\bt}|$ corresponding to an $R$-representation, the local system $a_0^{-1}\vv$  has unitary monodromy around the local components of $D_0$. Then there are natural $(S^1)^{\delta}$-actions on  $(|Y_{\bt}|,y)^{R, \mal}_{\MTS}$ and $\ugr (|Y_{\bt}|,y)^{R, \mal}_{\MTS}$, compatible with the opposedness isomorphism, and with the action of $-1 \in S^1$ coinciding with that of $-1 \in \bG_m$.
\end{proposition}
\begin{proof}
This just follows from the observation that the $S^1$-action of Proposition \ref{unitmtsenrich} is functorial, hence compatible with the construction of Theorem \ref{singunitmts}.
    \end{proof}

\subsection{More general monodromy}\label{generalmonodromy}

It is natural to ask whether the hypotheses of Theorems \ref{unitmhs} and \ref{unitmts} are optimal, or whether algebraic mixed Hodge and mixed twistor structures can be defined more widely. The analogous results to Theorem \ref{unitmhs} for $\ell$-adic pro-algebraic homotopy types in \cite{weiln} hold in full generality (i.e. for any Galois-equivariant quotient $R$ of $\varpi_1(Y,y)^{\red}$). However the proofs of Theorems \ref{unitmhs} and \ref{unitmts} clearly do not extend to non-unitary monodromy, since if $\theta$ is not holomorphic, then $\bar{\theta}$ does not act on ${\sA}^*_X(\vv)\langle  D \rangle$. Thus any proof adapting those theorems 
would  have to take some modification of ${\sA}^*_X(\vv)\langle  D \rangle$ closed under the operator $\bar{\theta}$.

A serious obstruction to considering non-semisimple monodromy around the divisor  is that the principle of two types plays a crucial r\^ole in the proofs of Theorems \ref{unitmhs} and \ref{unitmts}, and for quasi-projective varieties this is only proved for $L^2$ cohomology. The map $ \H^*(X, j_*\vv)\to \H^*_{(2)}(X,\vv)$ is only an isomorphism either for $X$ a curve or for semisimple monodromy, so $\sL^{\bt}_{(2)}(\vv)$ will no longer have the properties we require. There is not even any prospect of modifying  the filtrations in Propositions \ref{mochilemmamhs} or \ref{mochilemmamts} so that $J_0 \H^*(Y, \vv):= \H^*_{(2)}(X,\vv)$, because $L^2$ cohomology does not carry a cup product \emph{a priori} (and nor does intersection cohomology). This means that there is little prospect of applying the  decomposition theorems of \cite{sabbah} and \cite{mochiasymp}, except possibly in the case of curves.

If the groups $ \H^n(X, j_*\vv) $ all carry natural MTS or MHS, then the other terms in the Leray spectral sequence should   inherit MHS or MTS via the isomorphisms   
\[
 \H^n(X,\oR^m j_*\vv) \cong \H^n(X,\oR^m j_*\R \ten  (j_*\vv^{\vee})^{\vee}) \cong \H^n(D^{(m)}, j_{m*}j_m^{-1}\nu_m^{-1} (j_*\vv^{\vee})^{\vee}\ten \vareps^m),
\]
for $j_m\co (D^m-D^{m+1}) \to D^{(m)}$ the canonical open immersion. Note that $j_m^{-1}\nu_m^{-1} (j_*\vv^{\vee})^{\vee} $ is a local system on $D^m-D^{m+1}$ --- this will hopefully inherit  a tame pluriharmonic metric from $\vv$ by taking residues. 

It is worth noting that even for  non-semisimple monodromy, the weight filtration on homotopy types should just be the one associated  to the Leray spectral sequence. Although the monodromy filtration is often involved in such weight calculations, \cite{Weil2} shows that for $\vv$ pure of weight $0$ on $Y$, we still expect $j_*\vv$ to be pure of weight $0$ on $X$. It is only at generic (not closed) points of $X$ that the monodromy filtration affects purity.
 
Adapting $L^2$ techniques to the case of  non-semisimple monodromy around the divisor would have to involve some complex of Fr\'echet spaces to replace $L^{\bt}_{(2)}(X,\vv)$, with the properties that it calculates $ \H^*(X, j_*\vv)$ and is still amenable to 
Hodge theory. When monodromy around $D$ is trivial, a suitable complex is $A^{\bt}(X,j_*\vv)$, since $j_*\vv$ is a local system. In general, one possibility is a modification of Foth's complex $\sB^{\bt}(\vv)$ from \cite{Foth}, based on bounded forms. Another possibility might be the complex  given by $\bigcap_{p \in (0, \infty)}L_{(p)}^{\bt}(X,\vv)$, i.e. the complex consisting of distributions which are $L^p$ for all $p< \infty$.
Beware that these are not the same as bounded forms --- $p$-norms are all defined, but the limit $\lim_{p \to \infty} \|f\|_p$ might be infinite (as happens for  $\log|\log |z||$). 

Rather than using  Fr\'echet space techniques directly, another  approach to  defining the MHS or MTS we need  (including for $\vv$ with  non-semisimple monodromy) might be  via Saito's mixed Hodge modules  or Sabbah's mixed twistor modules. Since  $\H^n(X, j_*\vv)\cong \mathrm{IH}^n(X,\vv)$ for curves $X$, fibring by families of curves then  opens the possibility of putting MHS or MTS on $\H^n(X, j_*\vv)$ for general $X$. Again, the main difficulty would lie in defining the cup products needed to construct CDGAs.

\section{Canonical splittings}\label{splitsn}

\subsection{Splittings of mixed Hodge structures}

\begin{definition}
 Define $\MHS$ to be the category of finite-dimensional mixed Hodge structures. 
\end{definition}

Write $\row_2: \SL_2 \to \bA^2$ for projection onto the second row, so $\row_2^{\sharp}O(\bA^2)$ is a subring of $O(\SL_2)$. This subring is equivariant for the  $S$-action on $\SL_2$ from Definition \ref{rowdef}. 

\begin{definition}\label{SHSdef}
  Define   $\SHS$ (resp. $\ind(\SHS)$)  to be the category of  pairs $(V,\beta)$, where $V$ is  a finite-dimensional $S$-representation (resp. an $S$-representation) in real vector spaces and $\beta: V \to V\ten \row_2^{\sharp}O(\bA^2)(-1)$ is $S$-equivariant. A morphism $(V,\beta) \to (V', \beta')$ is an $S$-equivariant map $f:V\to V'$ with $\beta'\circ f= (f\ten \id) \circ \beta$.
\end{definition}

\begin{definition}\label{SHSdual}
 Given $(V,\beta) \in \SHS$, observe that taking duals gives rise to a map $\beta^{\vee}: V^{\vee} \to V^{\vee}\ten \row_2^{\sharp}O(\bA^2)(-1) $.
Then define the dual in $\SHS$ by  $(V,\beta)^{\vee}:= ( V^{\vee},\beta^{\vee})$. 

Likewise, we define the tensor product $(U, \alpha)\ten (V, \beta):= (U\ten V, \alpha \ten \id + \id\ten \beta)  $.
\end{definition}

Observe that for $(V, \beta), (V', \beta') \in \SHS$,
\[
 \Hom_{\SHS}((V, \beta), (V', \beta'))\cong \Hom_{\SHS}( (\R,0), (V, \beta)^{\vee}\ten (V', \beta')).
\]

\begin{lemma}\label{SHSalg}
A (commutative)  algebra $(A, \delta)$ in  $\ind(\SHS)$  consists of an $S$-equivariant (commutative) algebra $A$, together with an $S$-equivariant derivation $\delta: A \to A \ten \row_2^{\sharp}O(\bA^2)(-1)$.
 \end{lemma}
\begin{proof}
We need to endow  $(A, \delta) \in \SHS$ with a unit $(\R, 0) \to (A, \delta)$, which is the same as a unit $1 \in A$, and  with a (commutative) associative multiplication
\[
 \mu\co (A, \delta) \ten  (A, \delta) \to    (A, \delta).   
\]
Substituting for $\ten$, this becomes $\mu \co (A\ten A, \delta \ten \id + \id \ten\delta) \to (A, \delta)$, so $\mu$ is a (commutative) associative multiplication on $A$, and for $a,b \in A$, we must have $\delta(ab)= a\delta(b) + b\delta(a)$. 
 \end{proof}

\begin{theorem}\label{SHSequiv}
 The categories $\MHS$ and $\SHS$ are equivalent. This equivalence is additive, and compatible with tensor products and duals.
\end{theorem}
\begin{proof}
 Given  $(V,\beta) \in \SHS$ as above, define a weight filtration on $V$  by $W_rV= \bigoplus_{i \le r} \cW_iV$, for the decomposition $\cW$ of Definition \ref{cWdef} given by the action of $\bG_m \subset S$. Since $\beta$ is $S$-equivariant and $ \row_2^{\sharp}O(\bA^2)(-1) $ is of strictly positive weights, we have
\[
 \beta: W_rV \to (W_{r-1}V)\ten \row_2^{\sharp}O(\bA^2)(-1).
\]
Thus $\beta$ gives rise  to an $S$-equivariant map $V \to V \ten O(\SL_2)(-1)$ for which $\beta(W_rV) \subset (W_{r-1}V)\ten O(\SL_2)(-1)$ for all $r$. In particular, $(W_rV, \beta|_{W_rV})\in \SHS$ for all $r$. 

We now form $V\ten O(\SL_2)$, then look at the $S$-equivariant derivation $N_{\beta}: V\ten O(\SL_2) \to V\ten O(\SL_2)(-1)$ given by $N_{\beta}= \id \ten N + \beta\ten \id$. Since $\ker N= O(C)$, this map is $O(C)$-linear; by Lemma \ref{slhodge}, it corresponds under Lemma \ref{flathfil} to a real derivation
\[
 N_{\beta}: V\ten \cS \to V(-1)\ten \cS
\]
 such that $N_{\beta}\ten_{\R}\Cx$ preserves Hodge filtrations $F$. The previous paragraph shows that    $N_{\beta}( (W_rV)\ten \cS) \subset  (W_rV)(-1)\ten \cS$, with 
\[
 \gr^W N_{\beta}= (\id\ten N) : (\gr^WV)\ten \cS \to (\gr^WV)(-1)\ten \cS.
\]

Therefore $ M(V, \beta):=\ker(N_{\beta})\subset V\ten \cS$ is a real vector space, equipped with an increasing filtration $W$, and a decreasing filtration $F$ on $M(V, \beta)\ten \Cx$. We need to show that $ M(V, \beta)$ is a mixed Hodge structure.

Since $N:\cS \to \cS(-1)$ is surjective,   the observation above that $\gr^W N_{\beta}= (\id\ten N) $ implies that $N_{\beta}$ must also be surjective (as the filtration $W$ is bounded), so
\[
 0 \to M(V, \beta) \to V\ten \cS \xra{N_{\beta} } V(-1)\ten \cS\to 0
\]
is a exact sequence; this implies that the functor $M$ is exact.

Since $\gr^W_r(V, \beta)= (\cW_rV,0)$, we get that $M(\gr^W_r(V, \beta))= \cW_rV$. As $M$ is exact, $ \gr^W_rM(V, \beta)=M(\gr^W_r(V, \beta))$, so we have shown that $\gr^W_rM((V, \beta))$ is a pure weight $r$ Hodge structure, and hence that $M(V, \beta) \in \MHS$. Thus we have an exact functor
\[
 M: \SHS \to \MHS;
\]
 it is straightforward to check that this is compatible with tensor products and duals.

We need to check that $M$ is an equivalence of categories. First, observe that for any $S$-representation $V$, we have an object $(V,0) \in \SHS$ with $M(V)=V$.

Write 
\[
 \Ext^1_{\SHS}((U, \alpha),(V, \beta)):= \coker(  \Hom_{S} (U,V) \xra{\beta_*-\alpha^*} \Hom_S(U, V \ten O(C)) ).
\]
This gives a  an exact sequence
\begin{align*}
  0 \to &\Hom_{\SHS}( (U, \alpha), (V, \beta))  \to \Hom_{S} (U,V) \xra{\beta_*-\alpha^*} \Hom_S(U, V \ten O(C))  \\  \to&\Ext^1_{\SHS}((U, \alpha),(V, \beta)) \to 0. 
\end{align*}

Note that $\Ext^1_{\SHS}((U, \alpha),(V, \beta))$ does indeed parametrise extensions of  $ (U, \alpha)$ by $(V, \beta) $: given an exact sequence
\[
 0 \to  (V, \beta) \to (W, \gamma) \to   (U, \alpha) \to 0,    
\]
we may choose an $S$-equivariant  section $s$ of $W \onto U$, so $W \cong U \oplus V$. The obstruction to this being a morphism in $\SHS$ is  $o(s):= s^*\gamma -\alpha  \in \Hom_S(U, V \ten O(C))$, and another choice of section differs from $s$ by some  $f \in \Hom_{S} (U,V)$, with $o(s+f)= o(s) +  \beta_*f-\alpha^*f$.

Write $\oR^i \Gamma_{\SHS}(V, \beta):= \Ext^i((\R,0), (V, \beta))$ for $i=0,1$, noting that 
\[
 \Ext^i_{\SHS}( (U, \alpha),(V, \beta)))= \oR^i \Gamma_{\SHS}((V, \beta)\ten (U, \alpha)^{\vee}).
\]
We thus have morphisms
\[
\xymatrix@=2ex{0 \ar[r] &\Gamma_{\SHS}(V, \beta) \ar[r]\ar[d] &  V^S\ar[d] \ar[rr]^-{\beta} &&(V \ten \row_2^{\sharp}O(\bA^2)(-1))^S \ar[r] \ar[d]& \oR^1 \Gamma_{\SHS}(V, \beta)\ar[r]\ar[d] & 0\\
0 \ar[r] &\Gamma_{\cH}M(V, \beta) \ar[r] &  (V\ten O(\SL_2))^S \ar[rr]^-{\beta+N} &&(V \ten \row_2^{\sharp}O(\bA^2)(-1))^S \ar[r] & \oR^1 \Gamma_{\cH}M(V, \beta) \ar[r] & 0
} 
\]
of  exact sequences, making use of the calculations of \S \ref{Bei2}. 
For any short exact sequence
in $\SHS$, the morphisms $\rho^i\co\oR^i \Gamma_{\SHS}(V, \beta)\to \oR^i \Gamma_{\cH}M(V, \beta) $ are thus compatible with the long exact sequences of cohomology. 

The crucial observation on which the construction hinges is   that the map $\row_2^{\sharp}O(\bA^2)(-1) \to \coker(N: O(\SL_2)\to O(\SL_2)(-1))$ is an isomorphism, making $\row_2^{\sharp}O(\bA^2)(-1) $ a section for $O(\SL_2)(-1)\onto \H^1(C^*, \O_{C^*})$. This implies that
when $\beta=0$, the maps $\rho^i$ are isomorphisms. Since each  object $(V, \beta) \in \SHS$ is an Artinian extension of $S$-representations,  we deduce that the maps $\rho^i$ must be isomorphisms for all such objects.

Taking $i=1$ gives that  $\Ext^1_{\SHS}((U, \alpha),(V, \beta)) \to \Ext^1_{\cH}(M(U, \alpha),M(V, \beta))$ is an isomorphism; we deduce that every extension in $\MHS$ lifts uniquely to an extension in $\SHS$, so $M: \SHS \to \MHS$ is essentially surjective. Taking $i=0$ shows that $M$ is full and faithful. 
\end{proof}

\begin{remark}
 Note that  the Tannakian fundamental group (in the sense of \cite{tannaka}) of the category $\SHS$ is
\[
 \Pi(\SHS)= S \ltimes \Fr( \row_2^{\sharp}O(\bA^2)(-1)^{\vee}),
\]
 where $\Fr(V)$ denotes the free pro-unipotent group generated by the pro-finite-dimensional vector space $V$. In other words, $\SHS$ is canonically equivalent to the category of finite-dimensional $\Pi(\SHS)$-representations. Likewise, $\ind(\SHS)$ is equivalent to the category of all $\Pi(\SHS)$-representations.

The categories $\SHS$ and $\MHS$ both have  vector space-valued forgetful functors. Tannakian formalism shows that the functor $\SHS \to \MHS$, together 
with a choice of natural isomorphism between the respective forgetful functors, gives a morphism $\Pi(\MHS) \to \Pi(\SHS)$.
The choice of natural isomorphism  amounts  to choosing a Levi decomposition for $\Pi(\MHS)$, or   equivalently a functorial isomorphism $V \cong \gr^WV$ of vector spaces for $V \in \MHS$. 

A canonical choice $b_0$ of such an isomorphism is given by composing the embedding $b\co M(V, \beta) \into V\ten \cS$ with the map $p_0\co \cS \to \R$ given by $x \mapsto 0$. This allows us to put a new MHS on $V$, with Hodge filtration $b_0(F)$ and the same weight filtration as $V$, so $b_0\co M(V, \beta) \to (V, W, b_0(F))$ is an isomorphism of MHS. To describe this new MHS, first observe that  $\cS(-1)\cong \Omega(\cS/\R)= \cS dx$, and that  for $\beta: V \to V \ten\Omega(\cS/\R)$, we get an isomorphism  $\exp( -\int_0^x \beta)\co V \to  M(V, \beta)$, which is  precisely $b_0^{-1}$. 

Since the map $p_i\co \cS \to \Cx$ given by $x \mapsto i$ preserves $F$, it follows that the map 
\[
 p_i\circ b_0^{-1}= \exp( -\int_0^i \beta)\co V \to  V\ten \Cx
\]
  satisfies $\exp( -\int_0^i \beta)(b_0(F))=F$, so the new MHS  is 
\[
 (V, W, b_0(F))= (V,W, \exp( \int_0^i \beta)(F)).
\]
\end{remark}

\begin{remark}
In Proposition \ref{cSubiq}, it was shown that every mixed Hodge structure $M$ admits a non-unique  splitting $M\ten \cS \cong (\gr^WM) \ten \cS$, compatible with the filtrations. Theorem \ref{SHSequiv} is a refinement of that result, showing that such a splitting can be chosen canonically, by requiring that the image of $\gr^WM$ under the 
derivation $(\id_M \ten N\co M\ten O(\SL_2) \to M\ten O(\SL_2)(-1)$ lies in $\row_2^{\sharp}O(\bA^2)(-1) $. This is because $\beta$ is just the restriction of $ \id_M \ten N$ to $V:=\gr^WM$.

This raises the question of which $F$-preserving maps $\beta: V \to V \ten\Omega(\cS/\R)$ correspond to maps $V \to V \ten \row_2^{\sharp}O(\bA^2)(-1)$ (rather than just $ V \to V \ten O(\SL_2)(-1)$. Using the explicit description from the proof of Lemma \ref{slhodge}, we see that this amounts to the restriction that
\[
 \beta(V_{\Cx}^{p,q}) \subset \sum_{a\ge 0, b \ge 0} V_{\Cx}^{p-a-1,q-b-1} (x-i)^a(x+i)^bdx.
\]
\end{remark}

\begin{remark}\label{cfRMHS}
In \cite{RMHS}, Deligne established a characterisation of real MHS in terms of $S$-representations equipped with additional structure.

For any $\lambda \in \Cx$, we have a map $p_{\lambda}\co \cS \to \Cx$ given by $x \mapsto \lambda$, and $b_{\lambda}^{-1}:= (p_{\lambda} \circ b)^{-1}= \exp( -\int_{\lambda}^x \beta)\co V \to  M(V, \beta)$. 
Comparing the filtrations $b_0(F)$ and $b_0(\bar{F})$ on $V$, we are led  to consider
\[
d:=b_{-i}\circ b_i^{-1}=  \exp( \int_{-i}^{i} \beta). 
\]
This maps $V$ to $V$, and has the properties that $\bar{d}=d^{-1}$ and
\[
 (d-\id)(V_{\Cx}^{pq}) \subset \bigoplus_{r<p, s<q} V_{\Cx}^{rs}.
\]
 
This is precisely the data of an $\fM$-representation in the sense of \cite[Proposition 2.1]{RMHS}, so corresponds to a MHS. Explicitly, we first find the unique operator $d^{1/2}$ with $d:= (d^{1/2})^2$ satisfying the  properties  above, then define the mixed Hodge structure $M(V,d)$ to have underlying vector space $V$, with the same weight filtration, and with $F^pM(V,d) :=d^{1/2}(F^pV)$.

For our choice of $d$ as above, we then have an isomorphism 
\[
 a:= d^{1/2}\circ b_i =d^{-1/2}\circ b_{-i} \co M(V, \beta) \to V
\]
 of vector spaces. Since $b_i(F^pM(V, \beta))= F^pV$, this means that $a(F^pM(V, \beta))= F^pM(V,d)$, so $a$ is an isomorphism of MHS. 

We have therefore shown directly  how our category $\SHS$ is equivalent to Deligne's category of $\fM$-representations by sending the pair $(V, \beta)$ to $(V, \exp( \int_{-i}^{i} \beta))$. This also gives a canonical isomorphism $\fM \cong \Pi(\SHS)$, once we specify the associated isomorphism $a\circ b_0^{-1}\co V \to V$ on fibre functors.
The Archimedean monodromy operator $\beta$ thus provides a more canonical generator for the Lie algebra of $\Ru\fM$
than is given by the operator $d$ of \cite{RMHS}. Providing such a generator was also the goal of  the Hodge correlator $\bG$ over a point in \cite[\S 4.2(v)]{GoncharovCorrelators1} --- see Remark \ref{cfgonchrk} for a fuller comparison. 

For an explicit quasi-inverse functor from $\fM$-representations to $\SHS$, take a pair $(V,d)$. Since $d$ is unipotent, $\delta:=\log d \co V_{\Cx} \to V_{\Cx}$ is well-defined, and decomposes into types as $\delta= \sum_{p,q<0} \delta^{pq}$. We now just set 
\begin{align*}
  \beta &:= \sum_{a\ge 0,b\ge 0}\frac{\delta^{-a-1,-b-1}(x-i)^a(x+i)^bdx}{\int_{-i}^{i} (x-i)^a(x+i)^bdx},\\
&= \sum_{a\ge 0,b\ge 0}\frac{(-1)^a(a+b+1)!\delta^{-a-1,-b-1}(x-i)^a(x+i)^bdx}{(2i)^{a+b+1}a!b!}.
\end{align*}

This equivalence  $\fM\simeq \SHS$ can be understood in terms of  identifying the generating elements of \cite[Construction 1.6]{RMHS} with explicit elements of $ (\row_2^{\sharp}(O(\bA^2)(-1)))^{\vee}\ten \Cx$. 
Explicitly, Deligne's generating set $\{\delta^{-a-1,-b-1}\}$ is the dual basis to 
\[
 \{ \frac{(-1)^a(a+b+1)!(x-iy)^a(x+iy)^b}{(2i)^{a+b+1}a!b!}\}\subset (\row_2^{\sharp}(O(\bA^2)(-1)))\ten \Cx.
\]
 
\end{remark}

\subsection{Splittings of mixed twistor structures}

The following lemma ensures that a mixed twistor structure can be regarded as an Artinian extension of $\bG_m$-representations.
\begin{lemma}\label{MTShom}
 If $\sE$ and $\sF$  are pure twistor structures of weights $m$ and $n$ respectively, then
\[
 \Hom_{\MTS}(\sE, \sF) \cong \left\{\begin{matrix} \Hom_{\R}(\sE_1, \sF_1) & m=n \\ 0 & m \ne n.
                                     
                                    \end{matrix}\right.
\]
\end{lemma}
\begin{proof}
By hypothesis, $\sE= \gr^W_m\sE$ and $\sF= \gr^W_n\sF$. Thus we may assume that $\sE = \O(m)$ and $\sF= \O(n)$. Since homomorphisms must respect the weight filtration, we have
\[
  \Hom_{\MTS}(\O(m), \O(n))=  \Hom_{\bP^1}(\O(m), W_m\O(n)),
\]
which is $0$ unless $m \ge n$. When $m \ge n$, we have $ W_m\O(n)= \O(n)$, so
\[
 \Hom_{\MTS}(\O(m), \O(n))= \Gamma(\bP^1, \O(n-m)),
\]
which is $0$ for $m>n$ and $\R$ for $n=m$, as required.
\end{proof}

\begin{definition}\label{STSdef}
  Define $\STS$ to be the category of  pairs $(V,\beta)$, where $V$ is an $\bG_m$-representation in real vector spaces  and $\beta: V \to V\ten \row_2^{\sharp}O(\bA^2)(-1)$ is $\bG_m$-equivariant. A morphism $(V,\beta) \to (V', \beta')$ is a $\bG_m$-equivariant map $f:V\to V'$ with $\beta'\circ f= (f\ten \id) \circ \beta$.
\end{definition}

Note that the only difference between Definitions \ref{SHSdef} and \ref{STSdef} is that the latter replaces $S$ with $\bG_m$ throughout.

\begin{definition}\label{STSdual}
 Given $(V,\beta) \in \STS$, observe that taking duals gives rise to a map $\beta^{\vee}: V^{\vee} \to V^{\vee}\ten \row_2^{\sharp}O(\bA^2)(-1)$.
Then define the dual in $\STS$ by  $(V,\beta)^{\vee}:= ( V^{\vee},\beta^{\vee})$. 

Likewise, we define the tensor product by $(U, \alpha)\ten (V, \beta):= (U\ten V, \alpha \ten \id + \id \ten \beta)  $.
\end{definition}

Observe that for $(V, \beta), (V', \beta') \in \STS$,
\[
 \Hom_{\STS}((V, \beta), (V', \beta'))\cong \Hom_{\STS}( (\R,0), (V, \beta)^{\vee}\ten (V', \beta')).
\]

\begin{theorem}\label{STSequiv}
 The categories $\MTS$ and $\STS$ are equivalent. This equivalence is additive, and compatible with tensor products and duals.
\end{theorem}
\begin{proof}
As in the proof of Theorem \ref{SHSequiv}, every object $(V, \beta)\in \STS$ inherits a weight filtration $W$ from $V$, and $\beta$ gives rise to a $\bG_m$-equivariant map
\[
 N_{\beta}: V\ten O(\SL_2) \to V\ten O(\SL_2)(-1)
\]
 respecting the weight filtration on $V$, with $\gr^WN_{\beta}= (\id \ten N)$. 

For the projection $\row_1: \SL_2 \to C^*$ of Definition \ref{rowdef}, we then get a $\bG_m$-equivariant map
\[
 \row_{1*}N_{\beta}: \row_{1*}(V\ten \O_{\SL_2})  \to \row_{1*}(V\ten \O_{\SL_2}(-1));
\]
Then $\ker(\row_{1*}N_{\beta} ) $ is a $\bG_m$-equivariant vector bundle on $C^*$. Using the isomorphism $C\cong \bA^2$ of Remark \ref{Ccoords} and the projection $\pi:(\bA^2-\{0\}) \to \bP^1$ , this corresponds to a vector bundle $M(V, \beta):= (\pi_*\ker(\row_{1*}N_{\beta} ))^{\bG_m} $  on $\bP^1$.

Now, $M(V, \beta) $ inherits a weight filtration $W$ from $V$, and surjectivity of $N_{\beta}$ implies that 
\[
 0 \to \ker(\row_{1*}N_{\beta} )\to \row_{1*}(V\ten \O_{\SL_2})  \to \row_{1*}(V\ten \O_{\SL_2}(-1))\to 0
\]
is an exact sequence, so $M$ is an exact functor. In particular, this gives $\gr^W_nM(V, \beta)= M(\cW_nV,0)$, which is just the vector bundle on $\bP^1$ corresponding to the $\bG_m$-equivariant vector bundle $(\cW_nV)\ten \O_{C^*}$ on $C^*$. Since $\cW_nV$ has weight $n$ for the $\bG_m$-action, this means that $\gr^W_nM(V, \beta)$ has slope $n$, so we have defined an exact functor
\[
 M: \STS \to \MTS,
\]
which is clearly compatible with tensor products and duals.

If we define $\Gamma_{\STS}(V, \beta):= \ker(\beta\co V \to V\ten \row_2^{\sharp}O(\bA^2)(-1))^{\bG_m}$ and $\oR^1 \Gamma_{\STS}(V, \beta):= (\coker \beta)^{\bG_m}$, then the proof of Theorem \ref{SHSequiv} gives us morphisms
\[
 \rho^i\co \oR^i \Gamma_{\STS}(V, \beta)\to W_0 \H^i(\bP^1, M (V, \beta))
\]
for $i=0,1$. These are automatically isomorphisms when $\beta=0$, and the long exact sequences of cohomology then give that $\rho^i$ is an isomorphism for all $(V, \beta)$. We therefore have isomorphisms
\[
 \Ext^i_{\STS}((U, \alpha), (V, \beta))\to W_0 \Ext^i_{\bP^1}( M(U, \alpha), M (V, \beta)),
\]
and
arguing as in Theorem \ref{SHSequiv}, this shows that $M$ is an equivalence of categories, using Lemma \ref{MTShom} in the pure case.
\end{proof}

\begin{remark}
 Note that  the Tannakian fundamental group (in the sense of \cite{tannaka}) of the category $\STS$ is
\[
 \Pi(\STS)= \bG_m \ltimes \Fr( \row_2^{\sharp}O(\bA^2) (-1))^{\vee},
\]
 where $\Fr(V)$ denotes the free pro-unipotent group generated by the pro-finite-dimensional vector space $V$. 
   
The functor $\STS \to \MTS$ then gives a morphism $\Pi(\MTS) \to \Pi(\STS)$,  but this is not unique, since it depends on a choice of natural isomorphism between the  fibre functors (at $1 \in C^*$)  on  $\MTS$ and on $\STS$.
 This amounts  to choosing a Levi decomposition for $\Pi(\MTS)$, or   equivalently a functorial isomorphism $\sE_1 \cong \gr^W\sE_1$ of vector spaces for $\sE \in \MHS$. A canonical choice of such an isomorphism is to take the fibre at $I \in \SL_2$.

We can think of Theorem \ref{STSequiv} as an analogue of \cite{RMHS} for real mixed twistor structures, in that for any MTS $\sE$, it gives a canonical splitting of the weight filtration on $\sE_1$, together with unique additional data required to recover $\sE$.
\end{remark}

\subsection{Reduced forms for Hodge and twistor homotopy types}\label{hodgeHTsn2}

In \S \ref{hodgeHTsn}, we looked at absolute Hodge and twistor homotopy types. Over a point, the absolute Hodge homotopy type consists of the non-negatively weighted part $\cW_{\ge 0}\oR O(C^*)$ of the resolution $\oR O(C^*)$ of the structure sheaf $\O_{C^*}$, equipped with an $S$-action and an augmentation $\jmath^*\co \cW_{\ge 0}\oR O(C^*) \to O(S)$. The absolute twistor homotopy type is the same CDGA, but only equipped with the underlying $\bG_m$-action  and augmentation to $O(\bG_m)$.

\subsubsection{Over a point}

\begin{lemma}\label{reducedequivlemma}
There are natural isomorphisms
\begin{eqnarray*}
 \bar{G}_{S}(\jmath^*\co \cW_{\ge 0}\oR O(C^*) \to O(S))&\cong& \Pi(\SHS)\\
\bar{G}_{\bG_m}(\jmath^*\co \cW_{\ge 0}\oR O(C^*) \to O(\bG_m))&\cong& \Pi(\STS)
\end{eqnarray*}
of  pro-algebraic   groups.
\end{lemma}
\begin{proof}
First, observe that the bar construction gives pro-algebraic   groups (rather than just  dg pro-algebraic   groups) in these cases because the CDGA is concentrated in degrees $[0,1]$.

It follows immediately from the definition of the bar construction that $\SHS$ is equivalent to the category of representations of
$
 \bar{G}_{S}( \R \oplus \row_2^{\sharp}O(\bA^2)(-1)[-1]),
$
which is thus isomorphic to $\Pi(\SHS)$. Now, the embedding 
\[
 \R \oplus \row_2^{\sharp}O(\bA^2)(-1)[-1] \to \cW_{\ge 0}\oR O(C^*)
\]
is a quasi-isomorphism, so gives an isomorphism on applying $\bar{G}_S$, completing the proof of the first statement. 

For the second statement, repeat the argument replacing $S$ with $\bG_m$.
\end{proof}

\begin{remark}
Lemma \ref{reducedequivlemma} means that we can interpret the canonical splittings of Theorems \ref{SHSequiv} and \ref{STSequiv} as consequences of the $S$-equivariant  quasi-isomorphism 
\[
 \R \oplus \row_2^{\sharp}O(\bA^2)(-1)[-1] \to \cW_{\ge 0}\oR O(C^*)
\]
of CDGAs.
 It is worth noting that the non-unique splittings of Proposition \ref{cSubiq} can similarly be interpreted in terms of the quasi-isomorphism
\[
 \R \oplus  \cW_{\ge 1}\oR O(C^*)\to \cW_{\ge 0}\oR O(C^*).
\]
\end{remark}

\subsubsection{Compact K\"ahler manifolds}

We now consider generalisations of this phenomenon to absolute Hodge and twistor homotopy types of compact K\"ahler manifolds $X$.  In \cite[\S 3.1]{GoncharovCorrelators1}, Goncharov constructs a Hodge complex $\C_{\cH}\vv$ for any VHS $\vv$, which plays the same r\^ole as our $\tilde{A}^{\bt}_{\cH}(\vv)$ but is much smaller.  
  Taking coefficients $ \C_{\cH}O(\Bu_{\rho} \rtimes S)$ then gives a  CDGA in VHS with the same key properties as our absolute homotopy type $O(X)_{\cH}^{\rho,\mal}$. When $\rho$ is the canonical map $\pi_1(X,x) \to {}^{\VHS}\!\varpi_1(X,x)$, this recovers the CDGA $\cD_X$ of  \cite[\S 1.8]{GoncharovCorrelators1} (at least after correcting Goncharov's definition 
along the lines of Remark \ref{irreprmk} to compensate for the failure of $\R$ to be algebraically closed). 

A key feature of Goncharov's Hodge complex is an injective morphism (called the twistor transformation in \cite[Definition 3.4]{GoncharovCorrelators1})  from $\C_{\cH}\vv $ to $A^{\bt}(X \by \R, \pr_1^{-1} \vv)$. In Lemma \ref{AHcSlemma}, we identified $\tilde{A}^{\bt}_{\cH}(\vv) $ with a subcomplex of $A^{\bt}(X, \vv)\ten \Omega^{\bt}(\cS/\R)$, which itself embeds into $A^{\bt}(X \by \R, \pr_1^{-1} \vv)$ by identifying $\Spec \cS$ with $\bA^1_{\R}$.  Inspection shows that the twistor transform also maps to $ A^{\bt}(X, \vv)\ten \Omega^{\bt}(\cS/\R)$,  and we now show how  Goncharov's Hodge complex arises naturally as a quasi-isomorphic subcomplex of $ \tilde{A}^{\bt}_{\cH}(\vv)$.

\begin{definition}\label{gonchdef}
Given a VHS $\vv$ on $X$,  define $\grave{A}_{\cH}^{\bt}(X,\vv) \subset \tilde{A}^{\bt}_{\cH}(X,\vv)$ to be spanned by:
\begin{enumerate}
 \item $(\ker D)\cap (\ker D^c)\subset  \cW_0A^*(X, \vv) \ten^S \R$, for $\R \subset O(\SL_2) \subset \oR O(C^*)$.

\item elements of the form
\[
 (\frac{(-1)^{n-1}\tDc a}{n-1}, a\eps)\in  [ \cW_{1-n}A^*(X,\vv)\ten^S \cW_{n-1}\oR O(C^*)^0] \oplus [\cW_{-n}A^*(X,\vv)\ten^S \cW_{n}\oR O(C^*)^1],
\]
 for $a \in \cW_{-n}A^*(X,\vv)\ten^S \cW_{n-2}(\row_2^{\sharp} O(\bA^2))$
and $n \ge 2$. This  uses the description $\oR O(C^*)^0= O(\SL_2)$, $\oR O(C^*)^1= O(\SL_2)(-1)\eps$.
\end{enumerate}

Given a semisimple local system $\vv$ on $X$, define $\grave{A}_{\cT}^{\bt}(X,\vv) \subset \tilde{A}^{\bt}_{\cT}(X,\vv)$ to be spanned by:
\begin{enumerate}
 \item $(\ker D)\cap (\ker D^c)\subset  \cW_0A^*(X, \vv) \ten \R$, for $\R \subset O(\SL_2) \subset \oR O(C^*)$.

\item elements of the form
\[
 (\frac{(-1)^{n-1}\tDc a}{n-1}, a\eps)\in  [ \cW_{1-n}A^*(X,\vv)\ten \cW_{n-1}\oR O(C^*)^0] \oplus [\cW_{-n}A^*(X,\vv)\ten \cW_{n}\oR O(C^*)^1],
\]
for $a \in \cW_{-n}A^*(X,\vv)\ten \cW_{n-2}(\row_2^{\sharp} O(\bA^2))$
and $n \ge 2$.
\end{enumerate}
\end{definition}
Note that $\cW_n(\row_2^{\sharp} O(\bA^2)(-1)) =0$ for $n<2$, and that $U\ten^{\bG_m}\cW_nV= (\cW_{-n}U) \ten \cW_n V$. Also observe that
\[
 \grave{A}_{\cH}^0(X,\vv)= \H^0(X, \vv)^S, \quad \grave{A}_{\cT}^0(X,\vv)= \H^0(X, \vv)^{\bG_m}=\H^0(X, \cW_0\vv).
\]

\begin{lemma}\label{OCmodOA2}
 The $O(C)$-module structure on $O(\bA^2)$ induced by the isomorphism $\row_2^{\sharp}\co O(\bA^2) \to \coker N$ is given by 
\[
 u[x^my^n]= \frac{n}{m+n+1}[x^my^{n-1}], \quad v[x^my^n]= \frac{-m}{m+n+1}[x^{m-1}y^n].
\]
 \end{lemma}
\begin{proof}
Since $vx=uy-1$, we have 
\[
  N(x^{m+1}y^n)= (m+1)ux^my^n +nx^my^{n-1}(uy-1)= (m+n+1)u x^my^n - nx^my^{n-1}.
 \]
Since $uy=vx+1$, we have
\[
 N(x^my^{n+1})= (n+1)vx^my^n +mx^{m-1}y^n(vx+1)= (m+n+1)v x^my^n +mx^{m-1}y^n.
 \]
\end{proof}

\begin{proposition}\label{gonchprop}
For any VHS $\vv$ on $X$, the subspace
\[
 \grave{A}_{\cH}^{\bt}(X,\vv) \subset \tilde{A}^{\bt}_{\cH}(X,\vv)
\]
is a quasi-isomorphic subcomplex. It is closed under multiplication in the sense that
\[
 \grave{A}_{\cH}^{\bt}(X,\vv) \cdot \grave{A}_{\cH}^{\bt}(X,\ww) \subset \grave{A}_{\cH}^{\bt}(X,\vv\ten \ww).
\]

For any semisimple local system $\vv$ on $X$, the same is true of the subspace  $\grave{A}_{\cT}^{\bt}(X,\vv) \subset \tilde{A}^{\bt}_{\cT}(X,\vv) $.
\end{proposition}
\begin{proof}
We prove this for Hodge complexes, the proof for twistor complexes being almost identical.
We first have to check that  $\grave{A}_{\cH}^{\bt}(X,\vv)$ is closed under the differential $\tD$. Observe that on $\tilde{A}^{\bt}(X,\vv))\ten_{O(C)}\oR O(C^*) $, we have
\[
 \tD \co A^n(X, \vv) \ten \cW_n\oR O(C^*) \to A^{n+1}(X, \vv) \ten \cW_{n-1}\oR O(C^*).
\]
 
For any $\xi \in (\ker D)\cap (\ker D^c)$, we clearly have $\tD \xi=0$.

Other elements are of the form
\[
 (\frac{\pm 1}{m+n+1}(  x^{m+1}y^nDa + x^my^{n+1}D^c a), x^my^na\eps) 
\]
for $a \in \cW_{-2-m-n}A^*(X,\vv)$. Then for $m+n>0$, we have
\begin{eqnarray*}
&& d_{\grave{A}}(\frac{\pm 1}{m+n+1}(  x^{m+1}y^nDa + x^my^{n+1}D^c a, x^my^na\eps))\\
 &=& (\frac{\pm x^my^n DD^ca}{m+n+1} , \frac{-x^my^n((m+1) uDa+ nvxy^{-1}Da +mux^{-1}yD^ca + (n+1)vD^ca)\eps}{m+n+1})\\
&& + (0,\tD x^my^na\eps).
\end{eqnarray*}

Since $uy \equiv vx \mod \row_2^{\sharp}O(\bA^2)$, the co-ordinate $\tD x^my^na$ lies in
\begin{align*}
& \frac{-x^my^n}{m+n+1}((m+n+1) uDa + (m+n+1)vD^ca) +\tD x^my^na + A^*(X,\vv)\ten\row_2^{\sharp}O(\bA^2)\\
&=A^*(X,\vv)\ten\row_2^{\sharp}O(\bA^2). 
\end{align*}
Applying Lemma \ref{OCmodOA2}, we see that this co-ordinate is 
\[
\eta:= (u[x^my^nDa]+v[x^my^nD^ca])= \frac{1}{m+n+1}(n x^my^{n-1}Da -mx^{m-1}y^nD^ca).  
\]
Applying $\mp \tDc$ to this gives
\[
 \frac{\mp 1}{m+n+1}( n x^my^nD^cDa -m x^my^n DD^ca)= \frac{\pm (m+n)}{m+n+1}( x^my^n DD^c a),
\]
so 
\[
 d_{\grave{A}}(\frac{\pm 1}{m+n+1}\tDc x^my^na, x^my^na\eps) = ( \frac{\mp 1}{m+n}\tDc \eta, \eta\eps),
\]
which is of the required form.

When $m=n=0$, the same calculations give 
\[
 d_{\grave{A}}(-\tDc a, a\eps)=  (- DD^c a,0),
\]
so $\grave{A}_{\cH}^{\bt}(X,\vv)$ is indeed a subcomplex. 

To see that this subcomplex is quasi-isomorphic, we filter by weights for the $S$-action on $\oR O(C^*)$, giving a convergent spectral sequence, and we may compare the $E_0$ terms.  For strictly positive weights, we get the quasi-isomorphism
\[
 \cW_{-n}A^*(X,\vv)\ten \cW_n(\row_2^{\sharp} O(\bA^2)(-1)[-1]) \to \cW_{-n}A^*(X,\vv)\ten \cW_n\oR O(C^*).
\]

For weight $0$, because the weight filtration on $\tilde{A}$ is defined by good truncation we have the map
\[
 \psi\co (\ker D)\cap (\ker D^c)\cap\cW_0A^*(X, \vv)\to  (\ker \tD) \cap (\cW_0A^*(X, \vv) \ten  \cW_0\oR O(C^*)). 
\]
Now look at the differential $N\co \cW_0O(\SL_2) \to \cW_0O(\SL_2)(-1)$, which is surjective with kernel $\R$. For $a \in A^*(X, \vv)$, $\ker \tD=  (\ker D)\cap (\ker D^c)$, so the map $\psi$ above gives an isomorphism between the left-hand side and $\ker N$. 

To establish quasi-isomorphism, it then suffices to show that 
\[
 N\co (\ker \tD) \cap (\cW_0A^*(X, \vv) \ten  \cW_0O(\SL_2)) \to (\ker \tD) \cap (\cW_0A^*(X, \vv) \ten \cW_0O(\SL_2)(-1))
\]
is surjective. If we replace $\ker \tD$ with $\im \tD$, surjectivity follows from surjectivity of 
$N\co \cW_1O(\SL_2) \to \cW_1O(\SL_2)(-1)$. Since opposedness gives $\H^*j^*\tilde{A}\cong \O_{C^*}\ten\H^*A$, the map on the quotient $(\ker \tD)/(\im \tD)$ is 
\[
 N\co (\cW_0\H^*(X, \vv) \ten  \cW_0O(\SL_2)) \to  (\cW_0\H^*(X, \vv) \ten \cW_0O(\SL_2)(-1)),
\]
 which is surjective, giving the required surjectivity of $N$ on $\ker \tD$.


Finally, we have to show that $\grave{A}_{\cH}^{\bt}(X,-)$ is closed under external multiplication. For elements in $\ker D \cap \ker D^c$ this is automatic, and otherwise we have
\[
 (\frac{\pm\tDc a }{m-1} , a\eps)(\frac{\pm\tDc b}{n-1}, b\eps)= (\frac{\pm 1}{ (m-1)(n-1)} (\tDc a) \wedge (\tDc b),  \frac{\pm \eps}{ (m-1)} (\tDc a)\wedge b \pm \frac{\pm}{n-1} a\eps\wedge (\tDc b)),
\]
and setting $\zeta:= \frac{\pm 1}{ (m-1)} (\tDc a)\wedge b \pm \frac{\pm 1}{n-1} a\wedge (\tDc b)$ gives
\[
 \pm \tDc \zeta = \frac{\pm(m+n-2)}{(m-1)(n-1) }(\tDc a) \wedge (\tDc b),
\]
which has the required form because $\zeta$ is of weight $m+n-1$.
\end{proof}

\begin{remark}\label{cfgonchrk}
Goncharov's Hodge complex $\C^{\bt}_{\cH}\vv$ has the same dimension as the complex $\grave{A}_{\cH}^{\bt}(X,\vv)$ for each type and degree, and it is not too hard to construct an explicit isomorphism. The isomorphism is given by the identity on forms of type $0$, and then by renormalising other forms so that the projection of $d_{\grave{A}}$ to $\row_2^{\sharp} O(\bA^2)(-1)$ agrees with 
 \cite[Equation 81]{GoncharovCorrelators1}. For $\phi$ of type $(s,t)$, the image is a scalar multiple of 
\[
 (\frac{(-1)^{s+t-1}\tDc}{s+t-1}, \eps)(x+iy)^{s-1}(x-iy)^{t-1} \phi.  
\]
This isomorphism is multiplicative for the product of  \cite[Definition 3.6]{GoncharovCorrelators1}. 
\end{remark}

\begin{corollary}\label{gonchcor}
For any quotient  $\rho \co \varpi_1(X,x)^{\red}\onto R$,  the absolute twistor homotopy type of $X$  relative to $R$ is given by the $R \by \bG_m$-equivariant CDGA
 \[
 \grave{A}_{\cT}^{\bt}(X, O(\Bu_{\rho} \by \bG_m)).
\]

If $R$ is moreover  an $S^1$-equivariant quotient $R$ of ${}^{\VHS}\!\varpi_1(X,x)$, then the absolute 
 Hodge homotopy type of $X$ relative to $R$ is given by the $R \rtimes S$-equivariant CDGA
  \[
 O(X^{R, \mal}_{\cH})\simeq  \grave{A}_{\cH}^{\bt}(X, O(\Bu_{\rho} \rtimes S)).
\]
\end{corollary}
\begin{proof}
This just combines Lemma \ref{hodgeHTlemma} with the monoidal quasi-isomorphism of Proposition \ref{gonchprop}. 
\end{proof}

\begin{remark}\label{cSubiq3}
Because $\grave{A}_{\cH}^0(X,O(\Bu_{\rho} \rtimes S))= \R$, the bar construction and reduced bar construction agree for $\grave{A}_{\cH}^{\bt}(X, O(\Bu_{\rho} \rtimes S))$, and similarly for $ \grave{A}_{\cT}^{\bt}(X, O(\Bu_{\rho} \by \bG_m))$.

We can thus strengthen  Remark \ref{cSubiq2} to say that every VMHS (resp. VMTS) can be represented \emph{uniquely} by a suitable triple $(\vv, \omega, \beta)$ for a VHS (resp. semisimple local system) $\vv$, and $\omega, \beta$ as in Remark \ref{cSubiq2}. 
 Again, objects are given by forms $\phi =\omega+\beta\eps$ with 
\[
 (\tD+ \omega)^2=0,\quad  [\tD +\omega, N+\beta]=0, 
\]
 but we now have the restrictions
\[
 \beta \in \bigoplus_{n \ge 2}A^0(X, \cW_{-n}\bEnd\vv)\ten \cW_n(\row_2^{\sharp} O(\bA^2)(-1)),
\]
\[
 \omega - \sum_n \frac{(-1)^{n-1}\tDc}{n-1}\beta_n \in A^1(X, \cW_0 \bEnd\vv)\cap (\ker D)\cap (\ker D^c).
\]
These data correspond to the Green data of \cite[\S 4.1]{GoncharovCorrelators1}.

Because our CDGAs are reduced, we can now describe morphisms easily as well: if $M$ is the functor from our triples to  VMHS/VMTS, then a morphism $M(\vv, \omega, \beta)\to M(\vv', \omega', \beta')$ is just given by a morphism $\vv \to \vv'$ intertwining $\beta, \beta'$ and $\omega, \omega'$.

It is also worth noting that we get the same objects regardless of whether we use the DGLA $\grave{A}_{\cH}^0(X,\g)$ for $\g=\bEnd\vv$ or for  the coefficients $\g=W_{-1}\bEnd(\vv)$ used in Remark \ref{cSubiq2}. That is because both DGLAs are nilpotent, and $\grave{A}_{\cH}^{\bt}(X,\g)= \H^0(X,\g)^S \oplus \grave{A}_{\cH}^{\bt}(X,W_{-1}\g) $ for any $\g$. With hindsight, the coefficients $W_{-1}\bEnd(\vv)$ from Remark \ref{cSubiq2} can be understood as a manifestation of the quasi-isomorphism
\[
 H^0(X,\g)^S \oplus \tilde{A}_{\cH}^{\bt}(X,W_{-1}\g)\to \tilde{A}_{\cH}^{\bt}(X,\g).
\]
\end{remark}


\section{$\SL_2$ splittings of non-abelian MTS/MHS and strictification}\label{nonabsplitsn}

We now return to abstract algebraic MHS and MTS as constructed for quasi-projective varieties with non-trivial monodromy in \S \ref{nontrivsn}, and show that canonical $\SL_2$-splittings for these algebraic structures exist in great generality. 

\subsection{Simplicial structures}

\begin{definition}
 Let $s\Cat$ be the category of simplicially enriched  categories, which we will refer to as simplicial categories. Explicitly, an object $\C \in s\Cat$ consists of a class $\Ob \C$ of objects, together with simplicial sets $\HHom_{\C}(x,y) $ for all $x,y \in \Ob \C$, equipped with an associative composition law and identities. 
\end{definition}

\begin{lemma}\label{algmodel2}
For a reductive pro-algebraic monoid $M$ and an $M$-representation $A$ in DG algebras, there is a cofibrantly generated model structure on $DG_{\Z}\Alg_A(M)$, in which fibrations are surjections, and weak equivalences are quasi-isomorphisms.
\end{lemma}
\begin{proof}
When $M$ is a group, this is Lemma \ref{algmodel}, but the same proof carries over to the monoid case.
\end{proof}

\begin{definition}
 Given   $B \in DG_{\Z}\Alg_A(M)$ define $B^{\Delta^n}:= B \ten_{\Q}\Omega(|\Delta^n|)$, for $\Omega(|\Delta^n|)$ as in Definition \ref{Th}.  Make $DG_{\Z}\Alg_A(M)$  into a simplicial category by setting $\HHom(B,B')$ to be the simplicial set
\[
\HHom_{ DG_{\Z}\Alg_A(M)}(B,C)_n:= \Hom_{ DG_{\Z}\Alg_A(M)}(B,  C^{\Delta^n}).     
\]
\end{definition}

Beware that $DG_{\Z}\Alg_A(M)$ does not then satisfy the axioms of a simplicial model category from \cite[Ch. II]{sht}, because
$\HHom(-, B):DG_{\Z}\Alg_A(M)^{\op} \to \bS  $ does not have a left adjoint. However,  $DG_{\Z}\Alg_A(M)$ is a simplicial model category in the weaker sense of \cite{QHA}.

Now, as in \cite[\S 5]{Hovey}, for any pair $X,Y$ of objects in a model category $\C$, there is a derived function complex $\oR\Map_{\C}(X, Y) \in \bS$, defined up to weak equivalence. 
One construction is to take a cofibrant replacement $\tilde{X}$ for $X$
 and a fibrant resolution $\hat{Y}_{\bt}$ for $Y$ in the Reedy category of simplicial diagrams in $\C$, then to set
 $$
 \oR\Map_{\C}(X,Y)_n:= \Hom_{\C}(\tilde{X}, \hat{Y}_n).
 $$ 

In fact, Dwyer and Kan showed in \cite{simploc2} that $\oR\Map_{\C}$ is completely determined by the weak equivalences in $\C$. In particular, $\pi_0\oR\Map_{\C}(X,Y)= \Hom_{\Ho(\C)}(X,Y)$, where $\Ho(\C)$ is the homotopy category of $\C$, given by formally inverting weak equivalences.

To see that $C^{\Delta^{\bt}}$ is a Reedy fibrant simplicial resolution of $C$ in $DG_{\Z}\Alg_A(M)$, note that the matching object $M_n C^{\Delta^{\bt}}$ is given by
\[
  C\ten M_n \Omega(|\Delta^{\bt}|)=  C\ten \Omega(|\Delta^n|) /( t_0\cdots t_n, \sum_i t_0\cdots t_{i-1}  (dt_i) t_{i+1} \cdots t_n), 
\]
so the matching map $C^{\Delta^n} \to M_n C^{\Delta^{\bt}}$ is a fibration (i.e. surjective).

Therefore for $\tilde{B}\to B$ a cofibrant replacement, 
\[
 \oR\Map_{ DG_{\Z}\Alg_A(M)}(B,C) \simeq \HHom_{ DG_{\Z}\Alg_A(M)}(\tilde{B},C).        
\]

\begin{definition}
 Given  an object $D \in DG_{\Z}\Alg_A(M)$,  make the comma category $DG_{\Z}\Alg_A(M) \da D$  into a simplicial category by setting
\[
\HHom_{ DG_{\Z}\Alg_A(M)\da D}(B,C)_n:= \Hom_{ DG_{\Z}\Alg_A(M)}(B,  C^{\Delta^n}\by_{D^{\Delta^n}}D).     
\]
\end{definition}

Now, $ C \to C^{\Delta^{\bt}}\by_{D^{\Delta^{\bt}}}D$ is a Reedy fibrant resolution of $C$ in $DG_{\Z}\Alg_A(M) \da D $ for every fibration $C \to D$. Thus
for $\tilde{B}\to B$ a cofibrant replacement and $C\to \hat{C}$ a fibrant replacement, 
\[
 \oR\Map_{ DG_{\Z}\Alg_A(M)\da D}(B,C) \simeq \HHom_{ DG_{\Z}\Alg_A(M)\da A}(\tilde{B},\hat{C}).        
\]
 
\begin{definition}\label{pi0C}
Given a simplicial category $\C$, recall from \cite{bergner} that the category $\pi_0\C$ is defined to have the same objects as $\C$, with morphisms 
$$
\Hom_{\pi_0\C}(x,y)=\pi_0\HHom_{\C}(x,y). 
$$
A morphism in $\HHom_{\C}(x,y)_0$ is said to be a homotopy equivalence if its image in $\pi_0\C$ is an isomorphism.
\end{definition}

If the objects of a simplicial category $\C$ are the fibrant cofibrant objects of a model category $\cM$, with $\HHom_{\C}= \oR \Map_{\cM}$, then observe that homotopy equivalences in $\C$ are precisely weak equivalences in $\cM$.

\subsection{Functors parametrising Hodge and twistor structures}

Recall from Definition \ref{rjc} that we write $\oR O(C^*)$ for the 
DG algebra $O(\SL_2) \xra{N\eps} O(\SL_2)(-1)\eps$, with $\eps$ of degree $1$. By Proposition \ref{othermodelc*}, this induces an equivalence 
$$
\Ho(DG_{\Z}\Alg_{A\ten\oR O(C^*) }(R')\da B\ten \oR O(C^*)) \to \Ho(DG_{\Z}\Alg_{\Spec A\by C^* }(R')\da B\ten \O_{C^*})
$$
for any $R'$-representation $B$ in $A$-algebras.

\begin{definition}
For  $A \in \Alg(\Mat_1)$, define $\cP\cT(A)_*$ (resp. $\cP\cH(A)_*$) to be the  full simplicial subcategory of the category 
\begin{align*}
  &DG_{\Z}\Alg_{ A \ten \oR O(C^*)}( \Mat_1 \by  R\by\bG_m) \da A\ten O(R)\ten \oR O(C^*)\\
 (\text{resp. }
 &DG_{\Z}\Alg_{ A \ten \oR O(C^*)}( \Mat_1 \by  R\rtimes S)\da A\ten O(R)\ten \oR O(C^*))
\end{align*}
 on  fibrant cofibrant objects. These define functors
\[
 \cP\cT_*, \, \cP\cH_*\co  DG_{\Z}\Alg(\Mat_1) \to s\Cat.
\]
\end{definition}

 \begin{remark}
Since  
$\cP\cT(A)_* $ and $\cP\cH(A)_*$ are defined in terms of derived function complexes, it follows that a morphism in any of these categories is a homotopy equivalence (in the sense of Definition \ref{pi0C}) if and only if it is weak equivalence in the associated model category, i.e. a quasi-isomorphism. 
\end{remark}

\begin{remark}\label{othermodelcrk}
Let $\R[t] \in \Alg(\Mat_1)$ be given by setting $t$ to be of weight $1$.
After applying Proposition \ref{othermodelc*} and  taking fibrant cofibrant replacements, observe that a pointed  algebraic  non-abelian mixed twistor structure consists of 
\[
 O(\ugr X_{\MTS}) \in   DG_{\Z}\Alg(R \by \Mat_1)\da  O(R),
\]
 together with  an object $O(X_{\MTS}) \in \cP\cT_*(\R[t])$ and a weak equivalence 
\[
 O(X_{\MTS})\ten_{\R[t]}\R \to O(\ugr X_{\MTS})
\]
  in $\cP\cT_*(\R)$. 

Likewise, a pointed  algebraic  non-abelian
  mixed Hodge structure consists of 
\[
 O(\ugr X_{\MHS})\in   DG_{\Z}\Alg(R \rtimes \bar{S})\da  O(R),
\]
together with  an object  $O(X_{\MHS}) \in \cP\cH_*(\R[t])$, and a weak equivalence 
\[
 O(X_{\MHS})\ten_{\R[t]}\R \to O(\ugr X_{\MHS})
\]
 in $\cP\cH_*(\R)$.
\end{remark}

\subsection{Deformations}

\subsubsection{Quasi-presmoothness}

The following is \cite[Definition \ref{dmsch-2fibrn}]{dmsch}:
\begin{definition}\label{2fibrn}
Say that a morphism $F:\cA \to \cB$ in $s\Cat$ is  is a $2$-fibration if
\begin{enumerate} 
\item[(F1)] for any objects $a_1$ and $a_2$ in $\cA$, the map
$\HHom_{\cA}(a_1, a_2)\to \HHom_{\cB}(Fa_1, Fa_2)$ 
is a fibration of simplicial sets;
\item[(F2)] for any objects $a_1 \in\cA$, $b \in \cB$, and any homotopy equivalence $e :
Fa_1 \to b$ in $\cB$, there is an object $a_2 \in \C$,  a homotopy equivalence
$d : a_1 \to a_2$ in $\C$ and an isomorphism $\theta: Fa_2 \to b$   such that $\theta\circ Fd = e$.
\end{enumerate}
\end{definition}

The following are adapted from \cite{dmsch}:

\begin{definition}
Say that a  functor $\cD:  \Alg(\Mat_1) \to s\Cat$ is formally 2-quasi-presmooth  if for all square-zero extensions $A \to B$, the map
\[
\cD(A) \to \cD(B)
\]
is a 2-fibration. 

Say that $\cD$ is formally 2-quasi-presmooth 
if $\cD \to \bt$ is so.
\end{definition}

\begin{proposition}\label{relalghgs}
The functors
 $
\cP\cT_*,\cP\cH_*: \Alg(\Mat_1)  \to s\Cat$  are  formally 2-quasi-presmooth.
\end{proposition}
\begin{proof}
Apart from the augmentation maps, 
this is essentially the same as  \cite[Proposition \ref{dmsch-relalghgs}]{dmsch}, which proves the corresponding statements for the functor on algebras given by sending $A$ to the simplicial category of cofibrant DG $(T\ten A)$-algebras, for $T$ cofibrant. The same proof carries over, the only change being to take $\Mat_1 \by  R\by\bG_m $-invariants (resp. $\Mat_1 \by  R\rtimes S$-invariants) of the Andr\'e-Quillen cohomology groups. We now sketch the argument.

Let $\cP$ be  $\cP\cT_*$ (resp. $\cP\cH_*$), and write $S'$ for  $\bG_m$ (resp. $S$). 
Fix a square-zero extension $A \to B$ in $\Alg(\Mat_1)$.
Thus an object $P \in \cP(B) $ is a $\Mat_1 \by  R\rtimes S'$-equivariant diagram
$B \ten \oR O(C^*) \to P \to  B\ten O(R)\ten \oR O(C^*)  $, with the first map a cofibration and the second a fibration. Since $P$ is cofibrant, the underlying graded algebra is smooth over $B \ten \oR O(C^*) $, so lifts essentially uniquely to give a smooth morphism $A^* \ten \oR O(C^*)^* \to \tilde{P}^* $ of graded algebras, with $\tilde{P}^*\ten_AB\cong P^*$. As $A\ten O(R)\ten \oR O(C^*) \to B\ten O(R)\ten \oR O(C^*)$ is square-zero, smoothness of $\tilde{P}^*$ gives us a lift   $\tilde{p}:\tilde{P}^* \to  A^*\ten O(R)\ten \oR O(C^*)^*$. Since $\Mat_1 \by  R\rtimes S'$ is reductive, these maps can all be chosen equivariantly.

Now, choose some equivariant $A$-linear derivation $\delta$ on $\tilde{P}$ lifting $d_P$. The obstruction to lifting $P \in \cP(B)$ to $\cP(A)$ up to isomorphism is then the class
\begin{align*}
 [(\delta^2, p\circ \delta - d\circ p) ] \in &\H^2\HOM_P( \Omega(P/ (B \ten \oR O(C^*))), I\ten_BP \xra{p} I \ten O(R)\ten \oR O(C^*))\\
=&\Ext^2_P(\bL^{P/(B \ten \oR O(C^*))}_{\bt}, I\ten_BP \xra{p} I \ten O(R)\ten \oR O(C^*)).
\end{align*}
This is because any other choice of $(\delta, \tilde{p})$ amounts to adding the boundary of an element in $\HOM_P^1( \Omega(P/ (B \ten \oR O(C^*))), I\ten_BP \xra{p} I \ten O(R)\ten \oR O(C^*)) $.

The key observation now is that the cotangent complex is an invariant of the quasi-isomorphism class, so $P$ lifts to $\cP(A)$ up to isomorphism if and only if all quasi-isomorphic objects also lift. The treatment of morphisms is similar.  Although augmentations are not addressed in \cite[Proposition \ref{dmsch-relalghgs}]{dmsch}, the same proof adapts. It is important to note  that the Andr\'e--Quillen characterisation of obstructions to lifting morphisms does not require the target to be cofibrant.
\end{proof}

\subsubsection{Strictification}

\begin{proposition}\label{strictlift}
Let $\cP:  \Alg(\Mat_1) \to s\Cat$ be  one of the functors  
$\cP\cT_*$ or $\cP\cH_*$. Given an object $E$ in $\cP(\R)$,   an
object $P$ in  $\cP(\R[t])$, and a quasi-isomorphism 
\[
 f: P/tP \to E       
\]
in $\cP(\R)$, there is an object $M\in  \cP(\R[t])$, a quasi-isomorphism $g: P \to M$, and  an isomorphism $\theta: M/tM \to E$   such that $\theta\circ \bar{g} = f$.
\end{proposition}
\begin{proof}
 If we replace $\R[t]$ with $\R[t]/t^r$, then the statement holds immediately from Proposition \ref{relalghgs} and the definition of formal 2-quasi-presmoothness, since  the extension $\R[t]/t^r \to \R$ is nilpotent. Proceeding inductively, we get a system of objects $M_r \in \cP(\R[t]/t^r)$,   quasi-isomorphisms $g_r: P/t^rP \to M_r$ and isomorphisms $\phi_r: M_r/t^{r-1}M_r \to M_{r-1}$ with $M_0=E$, $g_0=f$ and  $\phi_r\circ \bar{g_r} = g_{r-1}$.  

We may therefore set $M$ to be the inverse limit of the system
\[
  \ldots \xra{\phi_{r+1}} M_r \xra{\phi_r} M_{r-1} \xra{\phi_{r-1}} \ldots \xra{\phi_1} M_0=E       
\]
in the category of $\Mat_1$-representations. Explicitly, this says that the maps
\[
 \cW_nM \to \Lim_r \cW_nM/(t^{r}\cW_{n-r}M)       
\]
are isomorphisms for all $n$. In particular, beware that the forgetful functor from $\Mat_1$-representations to vector spaces does not preserve inverse limits.

Let $\cM(A)$ be one of the model categories 
\begin{align*}
 &DG_{\Z}\Alg_{ A \ten \oR O(C^*)}( \Mat_1 \by  R\by\bG_m) \da A\ten O(R)\ten \oR O(C^*)\\
  \text{or }\quad
  &DG_{\Z}\Alg_{ A \ten \oR O(C^*)}( \Mat_1 \by  R\rtimes S)\da A\ten O(R)\ten \oR O(C^*),
\end{align*}
so $\cP(A)$ is the full simplicial subcategory on fibrant cofibrant objects.  The
  maps $g_r$ give a morphism $g: P \to M$ in $\cM(\R[t] )$ and the maps $\phi_r$ give an isomorphism $\theta:M/tM \to E$ in $\cP(\R)$. We  need to show that $M$ is fibrant and  cofibrant (so $M \in \cP(\R[t])$) and that $g$ is  a quasi-isomorphism. Fibrancy is immediate, since the deformation of a surjection is a surjection.

Given an object $A \in \cM(\R[t])$, the $\Mat_1$-action gives a weight decomposition $A = \bigoplus_{n \ge 0} \cW_nA$, and $$
A= {\Lim_n}^{\cM(\R[t])} A/\cW_{\ge n}A.
$$ 
Moreover, if $A \to B$ is a quasi-isomorphism, then so is $A/\cW_{\ge n}A \to B/\cW_{\ge n}B$ for all $n$. In order to show that $M$ is cofibrant, take a trivial fibration $A \to B$ in $\cM(\R[t])$ (i.e. a surjective quasi-isomorphism) and a map $M \to B$. Then $A/\cW_{\ge n}A \to B/\cW_{\ge n}B$ is a trivial fibration in $\cM(\R[t])$, and in fact in $\cM(\R[t]/t^n)$. Since $M_n\cong M/t^nM$ is cofibrant in $\cM(\R[t]/t^n)$, the map $M \to B$ lifts to a map 
$M \to (A/\cW_{\ge n}A)\by_{ B/\cW_{\ge n}B} B$. We now proceed inductively, noting that 
$$
(A/\cW_{\ge n+1}A)\by_{ (B/\cW_{\ge n+1}B)} B\to (A/\cW_{\ge n}A)\by_{ (B/\cW_{\ge n}B)} B
$$ 
is a trivial fibration in $\cM(\R[t]/t^{n+1})$. This gives us a compatible system of lifts $M \to (A/\cW_{\ge n}A)\by_{ (B/\cW_{\ge n}B)} B$, and hence 
$$
M \to \Lim_n [(A/\cW_{\ge n}A)\by_{ (B/\cW_{\ge n}B)} B]=A.
$$ 
Therefore $M$ is cofibrant.

To show that $g$ is a quasi-isomorphism, observe that for $A \in \cM(\R[t])$, the map $\cW_nA \to \cW_n(A/t^rA)$ is an isomorphism for $n<r$. Since $g_r$ is a quasi-isomorphism for all $r$, this means that $g$ induces quasi-isomorphisms $\cW_nP \to \cW_nM$ for all $n$, so $g$ is a quasi-isomorphism.        
\end{proof}

\begin{definition}
 Given an    $R$-equivariant $O(R)$-augmented CDGA $\sM$  in the category of ind-MTS (resp. ind-MHS) of non-negative weights, define the associated  non-positively weighted  algebraic    mixed twistor (resp. mixed Hodge) structure $\oSpec \xi(\sM)$    as follows.  Under Lemma \ref{flatmts} (resp. Lemma \ref{flatmhs}), the Rees module construction gives a flat $\Mat_1 \by R \by \bG_m$-equivariant  (resp. $\Mat_1 \by R\rtimes S$-equivariant) quasi-coherent $\O_{\bA^1}\ten O(R)\ten \O_{C^*}$-augmented  algebra $\xi(\sM):=\xi(\sM, \MTS)$ (resp. $\xi(\sM):=\xi(\sM, \MHS)$)   on $\bA^1 \by C^*$ associated to $\sM$. We therefore define $\oSpec \xi(\sM):= \oSpec_{\bA^1 \by C^*} \xi(\sM)$.

Now, 
$\gr^W\sM$ is an  $O(R)$-augmented CDGA in the category of $\Mat_1$-representations (resp. $\bar{S}$-representations), so we may set $\ugr  \oSpec \xi(\sM):= \Spec \gr^W\sM$. Since $\xi(\sM)$ is flat, 
\[
    (\oSpec \xi(\sM)) \by_{\bA^1, 0}^{\oR}\Spec \R \simeq  (\oSpec \xi(\sM)) \by_{\bA^1, 0}\Spec \R,    
\]
so
Lemma \ref{flatmts} (resp. Lemma \ref{flatmhs}) gives the required opposedness isomorphism.
\end{definition}

\begin{theorem}\label{strictmodel}
For every non-positively weighted  algebraic    mixed twistor (resp. mixed Hodge) structure $(X,x)_{\MTS}^{R, \mal}$ (resp. $(X,x)_{\MHS}^{R, \mal}$) on a pointed Malcev homotopy type $(X,x)^{R, \mal}$, there exists an $R$-equivariant $O(R)$-augmented CDGA $\sM$  in the category of ind-MTS (resp. ind-MHS) with  $(X,x)_{\MTS}^{R, \mal}$ (resp. $(X,x)_{\MHS}^{R, \mal}$)  quasi-isomorphic in the category of  algebraic    mixed twistor (resp. mixed Hodge) structures to $\oSpec \xi(\sM)$, for $\xi$ as above.
\end{theorem}
\begin{proof}
 Making use of Remark \ref{othermodelcrk}, choose a fibrant cofibrant replacement $E$ for $O(\ugr (X,x)_{\MTS}^{R, \mal} ) $ (resp. $O(\ugr (X,x)_{\MHS}^{R, \mal} ) $) in the category $DG_{\Z}\Alg(R)_* (\Mat_1)$ (resp. $DG_{\Z}\Alg(R)_* (\bar{S})$), and a fibrant   cofibrant replacement $P$ for 
\begin{align*}
 &\Gamma(C^*,\O( (X,x)_{\MTS}^{R, \mal}) \ten_{\O_{C^*}} \oR \O_{C^*}) \\
( \text{resp. }&\Gamma(C^*, \O( (X,x)_{\MHS}^{R, \mal} )\ten_{\O_{C^*}} \oR \O_{C^*} ) )
\end{align*}
 in the category 
\[
 DG_{\Z}\Alg_{\R[t] \ten \oR O(C^*)}(R)_* (\Mat_1 \by S')),
\]
where $S'=\bG_m$ (resp. $S'=S$).
Since $P$ is cofibrant, it is flat, so the data of an algebraic    mixed twistor (resp. mixed Hodge) structure give a quasi-isomorphism
\[
 f:P/tP \to E\ten \oR O(C^*)    
\]
\[
\text{in }\quad DG_{\Z}\Alg_{ \oR O(C^*)}(R)_* (\Mat_1 \by S')),
\]
 so we may apply Proposition \ref{strictlift} to obtain a fibrant cofibrant object
\[
 M \in DG_{\Z}\Alg_{\R[t] \ten \oR O(C^*)}(R)_* (\Mat_1 \by S')) 
\]
with an isomorphism $M/tM \cong E\ten \oR O(C^*)$, and a  quasi-isomorphism $g: P \to M$ lifting $f$.

Since $M$ is cofibrant, it is flat as an $\oR O(C^*)$-module. For the canonical map $\row_1^*\co \oR O(C^*) \to O(\SL_2)$, this implies that we have a short exact sequence
\[
 0 \to \row_1^*M(-1)\eps \to M \to \row_1^*M \to 0,
\]
and the section $O(\SL_2) \to \oR O(C^*)$ of graded rings (not respecting differentials) gives a canonical splitting of the short exact sequence for the underlying graded objects. Thus we may write $M^* = \row_1^*M \oplus \row_1^*M(-1)\eps$, and decompose the differential $d_M$ as $d_M:= \delta_M +N_M\eps$, where $\delta_M= \row_1^*d_M$.

Now, since $M/tM=E\ten \oR O(C^*)$,  we know that 
$$
N_M\co \row_{1*}\row_1^*(M/tM) \to \row_{1*}\row_1^*(M/tM)(-1)
$$ 
is a surjection of sheaves on $C^*$. Since $M= \Lim_r M/t^rM$ in the $\Mat_1$-equivariant category and $M$ is flat, this means that $N_M$ is also surjective. We therefore set
\[
 K:= \ker(N_M\co \row_{1*}\row_1^*M\to \row_{1*}\row_1^*M(-1));
\]
as $\ker(N: \row_{1*}O(\SL_2) \to \row_{1*}O(\SL_2)(-1))=\O_{C^*}$, we have 
\[
 K \in DG_{\Z}\Alg_{ \bA^1 \by C^* }( \Mat_1 \by  R\rtimes S')\da O(\bA^1\by R) \ten \O_{C^*} ),
\]
  with 
\[
 M= \Gamma(C^*,K\ten_{\O_{C^*}}\oR \O_{C^*}),
\]
for $\oR \O_{C^*} $ as in Definition \ref{rjc}.

Since $M$ is flat over $\oR O(C^*)\ten O(\bA^1)$, it follows that $K$ is flat over $C^* \by \bA^1$.  Moreover, for $0 \in \bA^1$, we have $0^*K= K/tK$, so
\begin{eqnarray*}
 0^*K&=& \ker(N_M\co \row_{1*}\row_1^*(M/tM)\to \row_{1*}\row_1^*(M/TM)(-1))\\
&=& E\ten\ker(N: \row_{1*}O(\SL_2) \to \row_{1*}O(\SL_2)(-1))\\
&=&E\ten \O_{C^*}.
\end{eqnarray*}
Thus $K$ satisfies the opposedness condition, so by  Lemma \ref{flatmts} (resp. Lemma \ref{flatmhs}) it corresponds to an ind-MTS (resp. ind-MHS)  on the $R$-equivariant $O(R)$-augmented CDGA  $(1,1)^*K$ given by pulling back along $(1,1)\co \Spec \R \to \bA^1 \by C$. Letting this ind-MTS (resp. ind-MHS) be $\sM$ completes the proof.
 \end{proof}

\subsubsection{Homotopy fibres}

In Proposition \ref{strictlift}, it is natural to ask how unique the model $M$ is. We cannot expect it to be unique up to isomorphism, but only up to quasi-isomorphism. As we will see in Corollary \ref{strictlift2}, that quasi-isomorphism is unique up to homotopy, which in turn is unique up to $2$-homotopy, and so on. 

\begin{definition}\label{scatfibdef}
Recall from \cite{bergner} Theorem 1.1 that  a morphism  $F:\C \to \D$ in $s\Cat$ is said to be  a  weak equivalence (a.k.a. an $\infty$-equivalence) whenever
\begin{enumerate} 
\item[(W1)] for any objects $a_1$ and $a_2$ in $\C$, the map
$\HHom_{\C}(a_1, a_2)\to \HHom_{\cD}(Fa_1, Fa_2)$ 
is a weak equivalence of simplicial sets;
\item[(W2)] the induced functor $\pi_0F : \pi_0\C \to \pi_0\cD$ is an equivalence of
categories.
\end{enumerate}

A morphism $F:\C \to \cD$ in $s\Cat$ is said to be a fibration whenever
\begin{enumerate} 
\item[(F1)] for any objects $a_1$ and $a_2$ in $\C$, the map
$\HHom_{\C}(a_1, a_2)\to \HHom_{\cD}(Fa_1, Fa_2)$ 
is a fibration of simplicial sets;
\item[(F2)] for any objects $a_1 \in\C$, $b \in \cD$, and homotopy equivalence $e :
Fa_1 \to b$ in $\cD$, there is an object $a_2 \in \C$ and a homotopy equivalence
$d : a_1 \to a_2$ in $\C$ such that $Fd = e$.
\end{enumerate}
\end{definition}

\begin{definition}\label{2fibre}
Given functors $\cA \xra{F} \cB \xla{G} \C$ between categories, define the 2-fibre product $\cA\by^{(2)}_{\cB}\C$ as follows. Objects of $\cA\by^{(2)}_{\cB}\C  $ are triples $(a,\theta, c)$, for $a \in \cA, c \in \C$ and $\theta: Fa \to Gc$ an isomorphism in $\cB$. A morphism in   $\cA\by^{(2)}_{\cB}\C$ from $(a,\theta, c)$ to $(a',\theta', c')$ is a pair $(f,g)$, where $f:a \to a'$ is a morphism in $\cA$ and $g: c\to c'$ a morphism in $\C$, satisfying the condition that
\[
Gg \circ \theta = \theta' \circ Ff.
\] 
\end{definition}

\begin{remark}
This definition has the property  that $\cA\by^{(2)}_{\cB}\C$ is a model for the 2-fibre product in the 2-category of categories. However, we will always use the notation $\cA\by^{(2)}_{\cB}\C$ to mean the specific model of Definition \ref{2fibre}, and not merely any equivalent category. 

Also note that
\[
\cA\by^{(2)}_{\cB}\C= (\cA\by^{(2)}_{\cB}\cB)\by_{\cB}\C, 
\]
and  that a morphism $F:\cA \to \cB$ in $s\Cat$ is a  2-fibration in the sense of Definition \ref{2fibrn} if and only if  $\cA\by_{\cB}^{(2)}\cB \to \cB$ is a  fibration in the sense of Definition \ref{scatfibdef}.
\end{remark}

\begin{corollary}\label{strictlift2}
Let $\cP:  \Alg(\Mat_1) \to s\Cat$ be  one of the functors  
$\cP\cT_*$ or $\cP\cH_*$. 
Given an object $E$ in $\cP(\R)$, the simplicial categories given by the homotopy fibre
\[
 \cP(\R[t])\by^h_{\cP(\R)}\{E\}       
\]
and the $2$-fibre
\[
  \cP(\R[t])\by^{(2)}_{\cP(\R)}\{E\}      
\]
are weakly equivalent.
\end{corollary}
\begin{proof}
By Proposition \ref{relalghgs}, $\cP(\R[t]/t^r) \to \cP(\R)$ is a  2-fibration in $s\Cat$. Moreover, the proof of Proposition \ref{strictlift} shows that the map
\begin{align*}
 \cP(\R[t]) \to &{\Lim_r}^{(2)} \cP(\R[t]/t^r)\\
  &\simeq \Lim_r [\cP(\R[t]/t^r)\by^{(2)}_{\cP(\R[t]/t^{r-1})} \cP(\R[t]/t^{r-1}) \by^{(2)}_{\cP(\R[t]/t^{r-2})} \ldots \by_{\cP(\R)}\cP(\R)]
\end{align*}
to the inverse $2$-limit is an equivalence, so $\cP(\R[t]) \to \cP(\R)$ is also a  2-fibration.

Therefore $\cP(\R[t])\by_{ \cP(\R)}^{(2)}\cP(\R) \to \cP(\R)$ is a fibration in the sense of Definition \ref{scatfibdef}, so
\begin{eqnarray*}
    \cP(\R[t])\by^h_{\cP(\R)}\{E\}   &\simeq&   \cP(\R[t])\by_{ \cP(\R)}^{(2)}\cP(\R)\by_{\cP(\R)}\{E\}\\
 & =& \cP(\R[t])\by_{ \cP(\R)}^{(2)}\{E\}\\
&=& \cP(\R[t])\by^{(2)}_{\cP(\R)}\{E\},
    \end{eqnarray*}
as required.
 \end{proof}

\subsubsection{$\SL_2$-splittings}

\begin{corollary}\label{allsplit}
 Every non-positively weighted  algebraic    mixed twistor (resp. mixed Hodge) structure $(X,x)_{\MTS}^{R, \mal}$ (resp. $(X,x)_{\MHS}^{R, \mal}$) on a pointed Malcev homotopy type $(X,x)^{R, \mal}$ admits a canonical $\SL_2$-splitting in the sense of Definition \ref{splitsl2}.
\end{corollary}
\begin{proof}
By  Theorem \ref{strictmodel}, we have an    $R$-equivariant $O(R)$-augmented CDGA $\sM$  in the category of ind-MTS (resp. ind-MHS) of non-negative weights, with $(X,x)_{\MTS}^{R, \mal}$ (resp. $(X,x)_{\MHS}^{R, \mal}$)  quasi-isomorphic in the category of  algebraic    mixed twistor (resp. mixed Hodge) structures to $\oSpec \xi(\sM)$.  

By Theorem \ref{STSequiv} (resp. Theorem \ref{SHSequiv}) and Lemma \ref{SHSalg}, there is a unique $R \by \bG_m$-equivariant (resp. $R \rtimes S$-equivariant) derivation $\beta\co \gr^W\sM \to (\gr^W\sM) \ten \row_2^{\sharp}O(\bA^2)(-1)$, with the corresponding object
\[
 O(\bA^1) \ten (\gr^W\sM)\ten O(\SL_2) \xra{\beta + \id \ten N }O(\bA^1) \ten(\gr^W\sM,W)\ten O(\SL_2)(-1)
\]
isomorphic to the object $M$ from the proof of Theorem \ref{strictmodel} (with $\gr^W\sM$ canonically isomorphic to $E$). 

In particular, it gives a $\bG_m \by R \by\bG_m$-equivariant (resp. $\bG_m \by R \rtimes S$-equivariant) isomorphism
\[
 \row_1^*\xi(\sM) \cong O(\bA^1) \ten (\gr^W\sM)\ten O(\SL_2).
\]
Since $\oSpec_{\bA^1 \by C^*} \sM$ is by construction quasi-isomorphic to $(X,x)_{\MTS}^{R, \mal}$ (resp. $(X,x)_{\MHS}^{R, \mal}$), with 
$\Spec \gr^W\sM$ quasi-isomorphic to $\ugr (X,x)_{\MTS}^{R, \mal}$ (resp. $\ugr (X,x)_{\MHS}^{R, \mal}$), this gives us a quasi-isomorphism 
\begin{align*}
 &\row_1^* \ugr (X,x)_{\MTS}^{R, \mal} \to  \bA^1 \by \Spec (\gr^W\sM)\by \SL_2\\
(\text{resp. } \quad
 &\row_1^* \ugr (X,x)_{\MHS}^{R, \mal} \to  \bA^1 \by \Spec (\gr^W\sM)\by \SL_2).
\end{align*}
\end{proof}

\begin{corollary}\label{qformalthing}
If a pointed Malcev homotopy type $(X,x)^{R, \mal}$ admits a non-positively weighted mixed twistor structure $(X,x)_{\MTS}^{R, \mal}$, then there is a canonical family 
\[
 \bA^1 \by (X,x)^{R, \mal} \simeq \bA^1 \by\ugr (X,x)_{\MTS}^{R, \mal}
\]
of quasi-isomorphisms over $\bA^1$.
\end{corollary}
\begin{proof}
 Take the fibre of the $\SL_2$-splitting 
\[
 \row_1^* \ugr (X,x)_{\MTS}^{R, \mal} \simeq \bA^1 \by \ugr (X,x)_{\MTS}^{R, \mal} \by \SL_2
\]
over $(1,1) \in \bA^1 \by C^*$. The fibre of $\SL_2 \to C^*$ over $1$ is $\left( \begin{smallmatrix} 1& 0 \\ \bA^1 & 0 \end{smallmatrix}\right)$, giving the family of quasi-isomorphisms.
\end{proof}

\subsubsection{Homotopy groups}

\begin{corollary}\label{mhspin2}\label{mtspin2}
Given a non-positively weighted  algebraic    mixed twistor (resp. mixed Hodge) structure $(X,x)_{\MTS}^{R, \mal}$ (resp. $(X,x)_{\MHS}^{R, \mal}$) on a pointed Malcev homotopy type $(X,x)^{R, \mal}$, there are natural  ind-MTS (resp. ind-MHS) on the the duals $(\varpi_n(X,x)^{\rho, \mal})^{\vee}$ of the  relative Malcev homotopy groups for $n\ge 2$, and on the Hopf algebra $O(\varpi_1(X,x)^{\rho, \mal})$.

These structures are compatible with the action of $\varpi_1$ on $\varpi_n$, with the Whitehead bracket and with the Hurewicz maps $\varpi_n(X^{\rho, \mal})\to \H^n(X, O(\Bu_{\rho}))^{\vee}$ ($n \ge 2$) and $\Ru\varpi_1(X^{\rho, \mal}) \to\H^1(X, O(\Bu_{\rho}))^{\vee}$, for  $O(\Bu_{\rho})$ as in Proposition \ref{propforms}. 
\end{corollary}
\begin{proof}
 By Corollary \ref{allsplit}, $(X,x)_{\MTS}^{R, \mal}$ (resp. $(X,x)_{\MHS}^{R, \mal}$) admits an $\SL_2$-splitting. Therefore the conditions of Theorem \ref{mhspin} are satisfied, giving the required result.
\end{proof}

Note that Theorems \ref{STSequiv}  and \ref{SHSequiv} now show that the various homotopy groups have associated objects in $\STS$ or $\SHS$, giving canonical $\SL_2$-splittings. These splittings will automatically be the same as those constructed in   Theorem \ref{pinsplit} from the splitting on the homotopy type. Explicitly, they give canonical isomorphisms
\[
 (\varpi_n(X,x)^{R, \mal})^{\vee} \ten \cS \cong (\gr^W \varpi_n(X,x)^{R, \mal})^{\vee} \ten \cS
\]
compatible with weight filtrations and with  twistor or Hodge filtrations, and similarly for $O(\varpi_1(X,x)^{\rho, \mal})$.

\subsection{Quasi-projective varieties}\label{qpin}

Fix a smooth projective complex variety $X$, a divisor $D$ locally of normal crossings, and set $Y:= X-D$. Let $j: Y \to X$ be the inclusion morphism. Take  a Zariski-dense representation $\rho\co \pi_1(Y,y) \to R(\R)$, for $R$ a reductive pro-algebraic group, with $\rho$ having unitary monodromy around local components of $D$.

\begin{definition}
Given a CDGA $A$ with $A^0=\R$, define
\[
 \pi_n(A):= \H_{n-1}G(A),
\]
where $G(A)$ is the quasi-free pro-finite-dimensional chain Lie algebra of Definition \ref{barwg}, given by the tangent space of the bar construction.
\end{definition}

\begin{corollary}\label{unitmtspin}
 There are natural  ind-MTS  on the the duals $(\varpi_n(Y,y)^{\rho, \mal})^{\vee}$ of the  relative Malcev homotopy groups for $n\ge 2$, and on the Hopf algebra $O(\varpi_1(Y,y)^{\rho, \mal})$.

These structures are compatible with the action of $\varpi_1$ on $\varpi_n$, with the Whitehead bracket and with the Hurewicz maps $\varpi_n(Y^{\rho, \mal})\to \H^n(Y, O(\Bu_{\rho}))^{\vee}$ ($n \ge 2$) and $\Ru\varpi_1(Y^{\rho, \mal}) \to\H^1(Y, O(\Bu_{\rho}))^{\vee}$.

Moreover, there are canonical $\cS$-linear isomorphisms
\begin{eqnarray*}
\varpi_n(Y^{\rho, \mal,y})^{\vee}\ten\cS &\cong& \pi_n(\bigoplus_{a,b}\H^{a-b} (X, \oR^bj_* O(\Bu_{\rho}))[-a], d_2)^{\vee}\ten\cS\\
O(\varpi_1(Y^{\rho, \mal},y)) \ten\cS &\cong& O(R \ltimes\pi_1(\bigoplus_{a,b}\H^{a-b} (X, \oR^bj_* O(\Bu_{\rho}))[-a], d_2)  )\ten\cS
\end{eqnarray*}
compatible with weight and twistor filtrations.
\end{corollary}
\begin{proof}
 This just combines  Theorem \ref{unitmts} (or  Theorem \ref{qmts} for a simpler proof whenever $\rho$ has trivial monodromy around the divisor)   with Corollary \ref{mtspin}. The splitting comes from Corollary \ref{allsplit}, making use of the isomorphism 
\[
 \ugr \varpi_n(Y,y)^{R, \mal}_{\MTS}= \gr^W\varpi_n(Y,y)^{R, \mal}.
\]
induced by the exact functor $\gr^W$  on MTS.
\end{proof}

\begin{corollary}\label{unitmhspin}
 If the local system on $X$ associated to any $R$-representation underlies a polarisable variation of Hodge structure, then there are natural  ind-MHS  on the the duals $(\varpi_n(Y,y)^{\rho, \mal})^{\vee}$ of the  relative Malcev homotopy groups for $n\ge 2$, and on the Hopf algebra $O(\varpi_1(Y,y)^{\rho, \mal})$.

These structures are compatible with the action of $\varpi_1$ on $\varpi_n$, with the Whitehead bracket and with the Hurewicz maps $\varpi_n(Y^{\rho, \mal})\to \H^n(Y, O(\Bu_{\rho}))^{\vee}$ ($n \ge 2$) and $\Ru\varpi_1(Y^{\rho, \mal}) \to\H^1(Y, O(\Bu_{\rho}))^{\vee}$.

Moreover, there are canonical $\cS$-linear isomorphisms
\begin{eqnarray*}
\varpi_n(Y^{\rho, \mal,y})^{\vee}\ten\cS &\cong& \pi_n(\bigoplus_{a,b}\H^{a-b} (X, \oR^bj_*O(\Bu_{\rho}))[-a], d_2)^{\vee}\ten\cS\\
O(\varpi_1(Y^{\rho, \mal},y)) \ten\cS &\cong& O(R \ltimes\pi_1(\bigoplus_{a,b}\H^{a-b} (X, \oR^bj_*O(\Bu_{\rho}))[-a], d_2)  )\ten\cS
\end{eqnarray*}
compatible with weight and  Hodge filtrations.
\end{corollary}
\begin{proof}
 This just combines Theorem \ref{unitmhs} (or  Theorem \ref{qmhs} for a simpler proof whenever $\rho$ has trivial monodromy around the divisor)  with Corollary \ref{mhspin}, together with the splitting of Corollary \ref{allsplit}.
\end{proof}

\begin{proposition}\label{unitmtspinen}
If the $(S^1)^{\delta}$-action on $ {}^{\nu}\!\varpi_1(Y,y)^{\red}$   descends to   $R$,  then for
 all $n$, the map
$
\pi_n(Y,y) \by S^1\to \varpi_n(Y^{\rho,\mal},y)_{\bT},
$
given by composing the map $\pi_n (Y,y) \to \varpi_n(Y^{\rho,\mal},y)$  with the $(S^1)^{\delta}$-action on $(Y^{\rho,\mal},y)_{\bT}$ from Proposition \ref{unitmtsenrich}, is continuous.
\end{proposition}
\begin{proof}
 The proof of Proposition \ref{kmtspinen} carries over to this generality.
\end{proof}

\begin{corollary}\label{unitmhspinanal}
Assume that the $(S^1)^{\delta}$-action on $ {}^{\nu}\!\varpi_1(Y,y)^{\red}$   descends to   $R$,
  and that the group $\varpi_n(Y,y)^{\rho,\mal}$  is finite-dimensional and spanned by the image of  $\pi_n(Y,y)$. Then  $\varpi_n(Y,y)^{\rho,\mal}$   carries a natural $\cS$-split mixed Hodge structure, which extends the mixed twistor structure of Corollary \ref{unitmtspin}. 
\end{corollary}
\begin{proof}
 The proof of Corollary \ref{kmhspinanal} adapts directly.
\end{proof}

\begin{remark}\label{cfmorgan}
 If we are willing to discard the Hodge or twistor structures, then Corollary \ref{qformalthing} gives a family 
\[
\bA^1 \by (Y^{\rho, \mal},y)\simeq \bA^1 \by \Spec (\bigoplus_{a,b}\H^{a-b} ( X, \oR^bj_*O(\Bu_{\rho}))[-a], d_2)
\]
of quasi-isomorphisms, and this copy of $\bA^1$ corresponds to $\Spec \cS$. 

If we pull back along the morphism $\cS \to \Cx$ given by $x \mapsto i$, the resulting complex quasi-isomorphism will preserve the Hodge filtration $F$ (in the MHS case), but not $\bar{F}$ or the real structure. This splitting is denoted by $b_i$ in Remark \ref{cfRMHS}, and comparison with \cite[Remark 1.3]{RMHS} shows that this is Deligne's functor $a_F$. 

Proposition \ref{morganhodge} adapts to show that when $R=1$, the mixed Hodge structure in Corollary \ref{unitmhspin} is the same as that of \cite[Theorem 9.1]{Morgan}.
Since $a_F$ was the splitting employed in \cite{Morgan}, we deduce that when $R=1$, the complex quasi-isomorphism at $i \in \bA^1$  (or equivalently at $\left( \begin{smallmatrix} 1& 0 \\ i & 0 \end{smallmatrix}\right) \in \SL_2$) is precisely the quasi-isomorphism of \cite[Corollary 9.7]{Morgan}.

Whenever the discrete $S^1$-action on $\varpi_n(Y,y)^{R, \mal}_{\MTS}$ (from Proposition \ref{unitmtsenrich}) is algebraic, 
it defines an algebraic mixed Hodge structure on  $\varpi_n(Y,y)^{R, \mal}$. In the projective case ($D= \emptyset$), \cite{KTP} constructed a discrete $\Cx^{\by}$-action on $\varpi_n(X,x)_{\Cx}$;
via Remark \ref{ktpcf}, the comments above show that whenever the $\Cx^{\by}$-action is algebraic,  it corresponds to the complex $I^{pq}$ decomposition of the mixed Hodge structure, with $\lambda \in  \Cx^{\by}$ acting on $I^{pq}$ as multiplication by $\lambda^p$.
\end{remark}

\subsubsection{Deformations of representations}

For $Y=X-D$ as above,  and some real algebraic group $G$, take  a  reductive representation $\rho\co \pi_1(Y,y) \to G(\R)$,  with $\rho$ having unitary monodromy around local components of $D$. Write $\g$ for the Lie algebra of $G$, and let $\ad \Bu_{\rho}$ be the local system of Lie algebras on $Y$ corresponding to the adjoint representation $\ad \rho \co \pi_1(Y,y) \to \Aut(\g)$. 

\begin{proposition}\label{defprop}
 The formal neighbourhood $\Def_{\rho}$ of  $\rho$ in the moduli stack $[\Hom(\pi_1(Y,y), G)/G]$ of representations is given by the  formal stack 
$
 [(Z,0)
/\exp(\H^0(Y,\ad \Bu_{\rho}))]
$, where $(Z,0)$ is the formal  germ at $0$ of the affine scheme $Z$ given by
\[
\{(\omega, \eta) \in \H^{1} ( X, j_*\ad \Bu_{\rho}) \oplus \H^{0} ( X, \oR^1j_*\ad \Bu_{\rho}) \,:\, d_2\eta +\half [\omega, \omega]=0, \,  [\omega, \eta]=0,\, [\eta,\eta]=0 )\}.
\]

The formal neighbourhood  $\fR_{\rho}$ of  $\rho$ in the rigidified moduli space $\Hom(\pi_1(Y,y), G)$ of framed representations is given by the formal   scheme
\[
 (Z,0)\by^{\exp(\H^0(Y,\ad \Bu_{\rho})) }\exp(\g),
\]
where $\exp(\H^0(Y,\ad \Bu_{\rho})) \subset \exp(\g) $ acts on $(Z,0)$ via the adjoint action.
\end{proposition}
\begin{proof}
Let $R$ be the Zariski closure of $\rho$. This satisfies the conditions of Corollary \ref{unitmtspin}, so we have 
an $\cS$-linear isomorphism
\[
 O(\varpi_1(Y^{\rho, \mal},y)) \ten\cS \cong O(R \ltimes\pi_1(\bigoplus_{a,b}\H^{a-b} (X, \oR^bj_*O(\Bu_{\rho}))[-a], d_2)  )\ten\cS
\]
of Hopf algebras. 

 Pulling back along any real homomorphism $\cS \to \R$ (such as $x \mapsto 0$) gives an isomorphism
\[
 \varpi_1(Y^{\rho, \mal},y) \cong O(R \ltimes\pi_1(\bigoplus_{a,b}\H^{a-b} (X, \oR^bj_*O(\Bu_{\rho}))[-a], d_2)  ).
\]

We now proceed as in   \cite[Remarks \ref{higgs-compare}]{higgs}. Given a real Artinian local ring $A = \R \oplus\m(A)$, observe that
\[
 G(A)\by_{G(\R)} R(\R) \cong \exp(\g \ten \m(A)) \rtimes R(\R).
\]
Since $\exp(\g\ten \m(A))$ underlies a  unipotent algebraic group, deformations of $\rho$ correspond to algebraic group homomorphisms
\[
 \varpi_1(Y^{\rho, \mal},y) \to \exp(\g \ten \m(A)) \rtimes R
\]
over $R$.

Infinitesimal inner automorphisms are given by conjugation by $\exp(\g \ten \m(A))$, and so \cite[Proposition \ref{htpy-meequiv}]{htpy} gives $\Def_{\rho}(A)$ isomorphic to
\[
[\Hom_R(\pi_1(\bigoplus_{a,b}\H^{a-b} (X, \oR^bj_*O(\Bu_{\rho}))[-a], d_2), \exp(\g \ten \m(A)))/\exp(\g \ten \m(A))^{R}],
\]
which is isomorphic to the groupoid of $A$-valued points of $[(Z,0)/\exp(\H^0(Y,\ad \Bu_{\rho}))]$.

The rigidified formal scheme $\fR_{\rho}$ is the groupoid fibre of  $\Def_{\rho}(A) \to B\exp(\g \ten \m(A)) $, which is just the set of $A$-valued points of $(Z,0)\by^{\exp(\H^0(Y,\ad \Bu_{\rho})) }\exp(\g)$, as in Proposition \ref{calchom}.
\end{proof}

\begin{remarks}
 The mixed twistor structure on $\varpi_1(Y^{\rho, \mal},y)$ induces a weight filtration on the pro-Artinian ring representing $\fR_{\rho}$. Since the isomorphisms of Corollary \ref{qformalthing} respect the weight filtration, the isomorphisms of  Proposition \ref{defprop} also do so.  
Explicitly, the ring $O(Z)$ has a weight filtration determined by setting $\H^{a-b} (X, \oR^bj_*O(\Bu_{\rho}))$ to be of weight $a+b$, so generators of $O(Z)$ have weights $-1$ and $-2$. The weight filtration on the rest of the space is then characterised by the conditions that $\g$ and $\H^0(Y,\ad \Bu_{\rho})$ both be of weight $0$.

Another interesting filtration is the pre-weight filtration $J$ of Proposition \ref{mochilemmamts}. The constructions transfer this to a filtration on  $\varpi_1(Y^{\rho, \mal},y)$, and the $\cS$-splittings (and hence Proposition \ref{defprop}) also respect $J$.  The filtration $J$ is determined by setting $\H^{a-b} (X, \oR^bj_*O(\Bu_{\rho}))$ to be of weight $b$, so generators of $O(Z)$ have weights $0$ and $-1$. We can then define $J_0Z:= \Spec O(Z)/J_{-1}O(Z)$, and obtain descriptions of  $J_0\Def_{\rho} \subset \Def_{\rho}$ and $J_0\fR_{\rho}\subset \fR_{\rho} $ by replacing $Z$ with $J_0Z$. These functors can be characterised as consisting of deformations for which the conjugacy classes of monodromy around the divisors remain unchanged --- these are the functors studied in \cite{Foth}.
\end{remarks}

\subsubsection{Simplicial and singular varieties}

As in \S \ref{unitsingsn}, let $X_{\bt}$ be a simplicial smooth proper complex variety, and $D_{\bt} \subset X_{\bt}$ a simplicial divisor  with normal crossings. Set $Y_{\bt}=X_{\bt} - D_{\bt}$,  assume that $|Y_{\bt}|$ is connected, and pick a point $y \in |Y_{\bt}|$. Let $j:|Y_{\bt}| \to |X_{\bt}|$  be the natural inclusion map. 

Take $\rho\co \pi_1(|Y_{\bt}|,y)\to R(\R)$  Zariski-dense,  and assume that for every local system $\vv$ on $|Y_{\bt}|$ corresponding to an $R$-representation, the local system $a_0^{-1}\vv$ on $Y_0$ is semisimple, with unitary monodromy around the local components of $D_0$.

\begin{corollary}\label{unitsingmtspin}\label{unitsingmhspin}
There are natural  ind-MTS  on the the duals $(\varpi_n(|Y_{\bt}|,y)^{\rho, \mal})^{\vee}$ of the  relative Malcev homotopy groups for $n\ge 2$, and on the Hopf algebra $O(\varpi_1(|Y_{\bt}|,y)^{\rho, \mal})$.

These structures are compatible with the action of $\varpi_1$ on $\varpi_n$, with the Whitehead bracket and with the Hurewicz maps $\varpi_n(|Y_{\bt}|^{\rho, \mal})\to \H^n(|Y_{\bt}|, O(\Bu_{\rho}))^{\vee}$ ($n \ge 2$) and $\Ru\varpi_1(|Y_{\bt}|^{\rho, \mal}) \to\H^1(|Y_{\bt}|, O(\Bu_{\rho}))^{\vee}$.

Moreover, there are canonical $\cS$-linear isomorphisms
\begin{eqnarray*}
\varpi_n(|Y_{\bt}|^{\rho, \mal,y})^{\vee}\ten\cS &\cong& \pi_n(\Th(\bigoplus_{p,q}\H^{p-q} (X_{\bt}, \oR^qj_* a^{-1}O(\Bu_{\rho}))[-p], d_1))^{\vee}\ten\cS\\
O(\varpi_1(|Y_{\bt}|^{\rho, \mal},y)) \ten\cS &\cong& O(R \ltimes\pi_1(\Th(\bigoplus_{p,q}\H^{p-q} ( X_{\bt}, \oR^qj_* a^{-1}O(\Bu_{\rho})[-p], d_1)  )))\ten\cS
\end{eqnarray*}
compatible with weight and twistor filtrations.

If  $a_0^{-1}\vv$ underlies a polarisable variation of Hodge structure on $Y_0$ for  all $\vv$ as above, then the ind-MTS above all become ind-MHS, with the $\cS$-linear isomorphisms above compatible with Hodge filtrations.
\end{corollary}
\begin{proof}
The proofs of Corollaries \ref{unitmtspin} and \ref{unitmhspin} carry over, substituting Theorems \ref{singunitmts} and \ref{singunitmhs} for Theorems \ref{unitmts} and \ref{unitmhs}.  
\end{proof}

\begin{corollary}\label{unitsingmhspinanal}
Assume that the $(S^1)^{\delta}$-action on $ {}^{\nu}\!\varpi_1(Y_0,y)^{\red}$   descends to   $R$,
  and that the group $\varpi_n(|Y_{\bt}|,y)^{\rho,\mal}$  is finite-dimensional and spanned by the image of  $\pi_n(|Y_{\bt}|,y)$. Then  $\varpi_n(|Y_{\bt}|,y)^{\rho,\mal}$   carries a natural $\cS$-split mixed Hodge structure, which extends the mixed twistor structure of Corollary \ref{unitsingmtspin}. 
\end{corollary}
\begin{proof}
 This is essentially the same as Corollary \ref{unitmhspinanal}, replacing Proposition \ref{unitmts} with Proposition \ref{singunitmtsenrich}.
\end{proof}

\begin{remark}\label{cfhain}
 When $R=1$, Proposition \ref{Hainhodge} adapts to show that the mixed Hodge structure of Corollary \ref{unitsingmhspin} agrees with that of \cite[Theorem 6.3.1]{Hainhodge}.
\end{remark}

\subsubsection{Projective varieties}

In Theorems \ref{mhsmal} and \ref{mtsmal}, explicit $\SL_2$ splittings were given for the mixed Hodge and mixed twistor structures on a connected compact K\"ahler manifold $X$. Since any MHS or MTS has many possible $\SL_2$-splittings, it is natural to ask whether the explicit splittings  are the same as the canonical splittings of Corollary \ref{allsplit}. Apparently miraculously, the answer is yes:

\begin{theorem}\label{cfksplit}
 The quasi-isomorphisms
\begin{align*}
 &\row_1^*  (X,x)_{\MTS}^{R, \mal} \simeq  \bA^1 \by \Spec (\ugr (X,x)_{\MTS}^{R, \mal} )\by \SL_2\\
\text{and } \quad &\row_1^*  (X,x)_{\MHS}^{R, \mal} \simeq  \bA^1 \by \Spec (\ugr (X,x)_{\MHS}^{R, \mal})\by \SL_2
\end{align*}
of Corollary \ref{allsplit} are homotopic to the corresponding quasi-isomorphisms of Theorems \ref{mhsmal} and \ref{mtsmal}.
\end{theorem}
\begin{proof}
Given a MTS or MHS $V$, an $\SL_2$-splitting $\row_1^*\xi(V) \cong (\gr^WV) \ten O(\SL_2)$ gives rise to a derivation $\beta \co \gr^WV \to \gr^WV\ten \Omega(\SL_2/C^*)$, given by differentiation with respect to  $\row_1^*\xi(V)$. Since $\Omega(\SL_2/C^*) \cong O(\SL_2)(-1)$, this $\SL_2$-splitting corresponds to the canonical $\SL_2$-splitting of Theorem \ref{STSequiv} or \ref{SHSequiv} if and only if $\beta( \gr^WV) \subset \gr^WV\ten \row_2^{\sharp}O(C)(-1)$.  

Now, the formality quasi-isomorphisms of Theorems \ref{mhsmal} and \ref{mtsmal} allow us to transfer the derivation $N\co \row_1^*\O((X,x)_{\MTS}^{R, \mal} ) \to \row_1^*\O((X,x)_{\MTS}^{R, \mal} )(-1)$ to an $N$-linear  derivation (determined up to homotopy)
\[
 N_{\beta}\co E\ten O(\SL_2) \to E\ten  O(\SL_2)(-1),
\]
for any fibrant cofibrant replacement $E$ for $O(\ugr (X,x)_{\MTS}^{R, \mal}) $, and similarly for $\O((X,x)_{\MHS}^{R, \mal} )$. Moreover, $\O((X,x)_{\MTS}^{R, \mal} )$ (resp.  $\O((X,x)_{\MHS}^{R, \mal} )$) is then quasi-isomorphic to the cone 
\[
\row_{1*}( E\ten O(\SL_2) \xra{N_{\beta}} E\ten  O(\SL_2)(-1)).
\]

If we write $N_{\beta}= \id \ten N + \beta$, for $\beta\co E \to E\ten  O(\SL_2)(-1)$, then  the key observation to make is that the formality quasi-isomorphism coincides with the canonical  quasi-isomorphism of Corollary \ref{allsplit} if and only if  for some choice of $\beta$ in the homotopy class,  we have
\[
 \beta(E) \subset E\ten \row_2^{\sharp}O(\bA^2)(-1) \subset E \ten O(\SL_2)(-1).
\]

Now, Remark \ref{gadata} characterises the homotopy class of  derivations $\beta$ in terms of minimal models, with $[\beta]= [\alpha + \gamma_x]$, where $\gamma_x$ characterises the basepoint, and $\alpha$  determines the unpointed structure. In Theorem \ref{archmonthm}, the operators  $\alpha$ and $\gamma_x$ are computed explicitly in terms of standard operations on the de Rham complex. 

For co-ordinates $\left(\begin{smallmatrix} u & v  \\ x & y \end{smallmatrix} \right) $ on $\SL_2$, it thus suffices to show that $\alpha$ and $\gamma_x$ are polynomials in $x$ and $y$. The explicit computation expresses these operators as expressions in $\tilde{D}= uD +vD^c$, $\tilde{D}^c= xD+yD^c$ and $h_i = G^2D^*D^{c*}\tilde{D}^c$, where $G$ is the Green's operator. However, each occurrence of $\tilde{D}$ is immediately preceded by either $\tilde{D}^c$ or by $h_i$. Since
\[
 \tilde{D}^c\tilde{D}= (xD+yD^c)(uD +vD^c )= (uy-vx)D^cD=D^cD,
\]
we deduce that $\alpha$ and $\gamma_x$ are indeed  polynomials in $x$ and $y$, so the  formality quasi-isomorphisms of Theorems \ref{mhsmal} and \ref{mtsmal} are just the canonical splittings of Corollary \ref{allsplit}.
\end{proof}

\bibliographystyle{alphanum}
\bibliography{references}

\end{document}